\documentclass[a4paper,12pt,twoside]{book}
\usepackage{float}
\usepackage{hyperref}
\usepackage{amsmath}
\usepackage{amssymb}
\usepackage{textcomp}
\usepackage{amsthm}
\usepackage{amsfonts}
\usepackage[all]{xy}
\usepackage{mathrsfs}
\usepackage{graphicx}
\usepackage{epstopdf}
\usepackage{tikz}
\usepackage{dsfont}
\usetikzlibrary{matrix}
\usepackage{verbatim}
\usepackage[latin1]{inputenc}
\usepackage{tikz}

\newtheorem{theo}{The\'or\`eme}[section]

\def\ind{{\mathop{\mathrm{ind}}\nolimits}}
\def\Ind{{\mathop{\mathrm{ind}}\nolimits}}
\def\Id{{\mathop{\mathrm{Id}}\nolimits}}
\def\A{\mathbf{A}}

\def\C{\mathbb{C}}
\def\lC{\mathscr{C}}
\def\D{\mathscr{D}}
\def\F{\mathcal{F}}
\def\G{\mathcal{G}}
\def\N{\mathbb{N}}
\def\P{\mathbb{P}}
\def\Q{\mathbb{Q}}
\def\R{\mathbb{R}}

\def\Z{\mathbb{Z}}

\def\a{\mathbf{a}}
\def\l{\ell}

\def\f{\mathbf{f}}
\def\c{\mathbf{c}}

\def\b{\mathbf{b}}
\def\K{\mathbf{K}}
\def\d{\mathbf{d}}
\def\e{\mathbf{e}}

\def\O{\mathcal{O}}
\def\line{\overline}
\def\tilde{\widetilde}

\def\Hom{\mathop{\mathrm{Hom}}\nolimits}
\def\Im{\mathop{\mathrm{Im}}\nolimits}

\def\hom{\mathop{\mathrm{Hom}}\nolimits}
\def\End{\mathop{\mathrm{End}}\nolimits}
\def\ker{\mathop{\mathrm{ker}}\nolimits}
\def\dim{\mathop{\mathrm{dim}}\nolimits}
\def\rank{\mathop{\mathrm{rank}}\nolimits}

\def\Supp{\mathop{\mathrm{Supp}}\nolimits}
\def\Lie{\mathop{\mathrm{Lie}}\nolimits}
\def\supp{\mathop{\mathrm{supp}}\nolimits}
\def\wt{\mathop{\mathrm{wt}}\nolimits}
\def\sym{\mathop{\mathrm{sym}}\nolimits}

\def\hat{\widehat}
\def\n{\underline{n}}

\def\remk{\noindent\textit{Remark:~}}

\newtheorem{prop}[theo]{Proposition}
\newtheorem*{prop*}{Proposition}
\newtheorem{def-prop}[theo]{Definition-Proposition}
\newtheorem{cor}[theo]{Corollary}
\newtheorem{lemma}[theo]{Lemma}
\newtheorem{teo}[theo]{Theorem}
\newtheorem*{teo*}{Theorem}
\newtheorem{definition}[theo]{Definition}
\newtheorem{definitions}[theo]{Definitions}

\newtheorem*{conj*}{Conjecture}

\newtheorem{notation}[theo]{Notation}
\newtheorem{example}[theo]{Example}

\everymath{\displaystyle}  

\input xy
\xyoption{all}
\title{Parabolic Induction and Geometry of Orbital Varieties for $GL(n)$}
\author{DENG Taiwang}

\begin{document}

\maketitle
%


\part*{Introduction}
\addcontentsline{toc}{part}{Introduction}

\clearpage

\pagestyle{plain}
\frontmatter
\pagenumbering{roman}

This thesis deals with the computation of the Jordan-H\"{o}lder decomposition
of a parabolic induced representation of $GL_n$ over a $p$-adic field $F$. 
Starting with irreducible cuspidal representations, Zelevinsky classified 
the irreducible representations in terms of multisegments 
\[
\a\mapsto L_{\a},
\]
where $L_{\a}$ is the irreducible representation of $GL_{n}(F)$ associated to the multiset $\a$, which is a set with multiplicities, of segments 
\[
\Delta_{\rho, r}=\{\rho, \rho\nu, \cdots, \rho\nu^{r-1}\},
\] 
where $\rho$ is an irreducible cuspidal representation of $GL_g(F)$, $n=rg$ and 
$\nu: GL_g(F)\rightarrow \C$ is the character given 
\[
x\mapsto |\det(x)|^{1/2}.
\]  
For example, $L_{\Delta_{\nu^{1-r}, r}}$ is the trivial representation of $GL_r(F)$. Given a
multisegment $\a=\{\Delta_1, \cdots, \Delta_s\}$ the total parabolic associated induced representation
is 
\[
\pi(\a)=L_{\Delta_1}\times L_{\Delta_2}\times \cdots \times L_{\Delta_s}
\]
 and one wants to compute the multiplicity $m(\b, \a)$ of $L_{\b}$ in $\pi(\a)$.

Zelevinsky introduced the geometry of nilpotent orbits and 
conjectured that the coefficients $m(\b, \a)$ is the value at $q=1$
of the Poincar\'e series $P_{\sigma(\a), \sigma(\b)}(q)$
where $\sigma(\a)$ and $\sigma(\b)$ are the associated orbits.
Moreover, he proved that these orbital varieties 
admit an open immersion into some Schubert varieties of type $A$.
This conjecture was proved by Chriss-Ginzburg and 
Ariki, see \cite{CG}, \cite{A}. 

\medskip

In the first part of this thesis, we are interested in another conjecture of Zelevinsky
stated in the last sentence of \S 8 of \cite{Z2}.

\begin{conj*}
The $m(\b,\a)$ depend only on 
natural relationships between segments of $\a$ and $\b$. 
\end{conj*}

Note:
\begin{itemize}
\item first that using types theory, the $m(\b,\a)$ are independent of the Zelevinsky lines
considered , cf. \cite{MV} for example, so that one is reduced to the case where
the cuspidal support of all the segment considered are contained in the Zelevinsky line
of the trivial representation.

\item Using this reduction, this conjecture can now be viewed as a special case of a 
conjecture of Lusztig about \emph{combinatorial invariance of Kazhdan-Lusztig polynomials}
which can be stated in these terms: let $x \leq y$ two elements of the symmetric group $S_n$,
the  Kazhdan-Lusztig polynomial $P_{x,y}(q)$ depends only on the poset structure of
$[x,y]:=\{ z \in S_n: x \leq z \leq y \}$.
\end{itemize}
The main application of the results of the this part of this thesis is then \emph{the proof of
the above conjecture of Zelevinsky}, cf. theorem \ref{teo: 4.4.5}: the results is already interesting
in the symmetric case, cf. the corollary \ref{coro-sym1}.

\medskip

Our approach rests on the use of some truncation functors
$$\a\mapsto \a^{(k)},$$
and the notion of partial derivativation
$$\D^k \hbox{ indexed by integers } k\in \Z,$$
which allows us, starting from  general multisegments $\a$ and $\b$,
to reduce to a symmetric situation where 
$\a$ and $\b$ are parametrized by $\sigma, \tau\in S_n$
for some $n$ usually less than the degree of $\a$. In this
symmetric case we obtain, using the result of Chriss-Ginzburg and Ariki, the equality 
\[
 m(\a_{\tau}, \a_{\sigma})=P_{\tau, \sigma}(1),
\]
where $P_{\tau, \sigma}$ is the Kazhdan Lusztig polynomial
associated to the permutations $\tau, \sigma\in S_n$.

Let us recall that these $m(\a_{\tau}, \a_{\sigma})$ are given, using Chriss-Ginzburg and Ariki, by Kazhdan-Lusztig
polynomials for the symmetric group $S_m$ where $m$ is the degree of $a$. So our formula can be also viewed as
equalities between Kazhdan-Lusztig polynomials for different symmetric groups: these equalities where also obtained
by Henderson \cite{H07}, but instead of using the Billey-Warrington cancellation for the symmetric
group, we investigate the 
geometry of nilpotent symmetric orbits. 

\remk using our truncation method, it should be possible to find a new algorithm for computing the general $m(\b, \a)$. 

\bigskip

\emph{In the second part} we give some applications of our method, the main aim is to give 
\emph{a formula for the computation of an induced representation} 
\[
L_{\a}\times L_{\b}=\sum m(\c, \b, \a)L_{\c}.
\]
in terms of the coefficients of the
''highest degree term'' of some explicit Kazhdan-Lusztig polynomials. For the moment 
we treat the case where $\b$ is a segment and leave the general case for future work. To give an impression, 
the most simple formula in the case where $\b=[k+1]$ from proposition \ref{prop: 8.1.5}, looks like
\[
L_{\a}\times L_{\b}=L_{\a+\b}+\sum_{\c\in \Gamma^{\l_k-1}(\a, k)}(\theta_k(\c, \a)-\theta_k(\c^{[k+1]_1}, \a+\b))L_{\c^{[k+1]_1[ k]_{\l_k-1}}}.
\]
where the $\theta_k(\c, \a)$ are defined thanks to partial derivative, cf. notation \ref{nota7-8-12}.
It would be interesting to compare our results with the known criteria of the irreducibility  for parabolic induced representations, cf.
\cite{ming1}, \cite{ming2} and \cite{jant}.

Moreover,
\begin{itemize}
\item in chapter 5, we obtain a geometric interpretation of the $5$ 
relations defining Kazhdan-Lusztig polynomials. 

\item \emph{In view of the conjecture of Lusztig}, 
which can be viewed as an generalization of Zelevinsky's conjecture,
in chapter 6, we give a classification of the posets $S(\a)=\{\b: \b\leq \a\}$,
in the sense of notation \ref{nota: 1.3.2}. We prove that they can be identified with
either an interval in the symmetric group $S_n$ or an interval in an double quotient
of $S_n$, which corresponds to parabolic orbits in a generalized flag variety.

\item Concerning partial derivation, in Chapter 7, using the Lusztig product of perverse sheaves (cf. \cite{Lu}), 
we give a geometric meaning
of the multiplicities appearing in the partial derivatives. In the general case we then obtain an explicit
formula for the derivative $\D^{k}(L_{\a})$, cf. corollary \ref{coro-fornula-derivative}.  
The main application is to calculate the coefficient $m(\c, \b, \a)$ in chapter 8.

%

\end{itemize}

\vspace{1pt}

\begin{center}
\rule{10cm}{.1pt}
\end{center}

\vspace{1pt}

\emph{Let us now give more details.} 
For a $p$-adic field $F$ and $g>1$, an irreducible admissible representation $\rho$
of $GL_{g}(F)$ is called cuspidal if for all proper parabolic
subgroup $P$, the corresponding Jacquet functor $J^G_{P}$ sends $\rho$ to $0$. We 
write 
$$\nu: GL_g(F)\rightarrow \C, \qquad\nu(x)=|\det(x)|^{1/2}$$ 
and for $k\geq 1$ and $\rho$ a cuspidal irreducible representation of $GL_g(F)$, 
we call the set $$\Delta_{\rho, k}=\{\rho, \rho\nu, \cdots, \rho\nu^{k-1}\}$$
a segment. For such a segment, the normalized induction functor
$$\ind_{P_{g,\cdots,g}}^{GL_{kg}(F)} (\rho \otimes \cdots \otimes \rho\nu^{k-1})$$
contains an unique irreducible sub-representation denoted by $L_{[\rho,\nu^{k-1}\rho]}$, where
$P_{g,\cdots,g}$ is the standard parabolic subgroup with Levi subgroup isomorphic to $k$ blocks of $GL_g$.
Then a multisegment is a multiset of segments that is a set with multiplicities.
For $i=1, \cdots, r$, let $\rho_i$ be an irreducible cuspidal representation of $GL_{n_i}(F)$ and for $k_i\in \N$, 
by definition, the multisegment
\[
 \a=\{\Delta_{\rho_i, k_i}: i=1, \cdots, r\}, 
\]
is of degree $\deg(\a)=\sum n_ik_i$.
In \cite{Z2}, the authors gave a parametrization $\a\mapsto L_{\a}$
of irreducible admissible representations of $GL_{n}(F)$
in terms of multisegments of degree $n$, where for a well ordered multisegment $\a$(cf. definition \ref{def: 6}),
the representation
$L_{\a}$
is the unique irreducible submodule of the parabolic induced representation
\[
 \pi(\a)=\Ind_{P}^{GL_{n}(F)}(L_{\Delta_{\rho_1, k_1}}\otimes \cdots \otimes L_{\Delta_{\rho_r, k_r}}).
\]
Now given two multisegments $\a$
and $\b$, one wants to determine the multiplicity 
$m(b, \a)$ of $L_{\b}$ in $\pi(\a)$.

Thanks to the Bernstein central decomposition, one is  reduced to the case 
where the cuspidal representation $\rho_i$ of $\a$ and $\b$ belongs to the 
same Zelevinsky line $\{\rho_0\nu^k: k \in \Z\}$. Zelevinsky also conjectured that 
$m(b, \a)$ is independent of $\rho_0$ and 
depends only on the relative position of $\a$ and $\b$: this conjecture now follows from the theory of types,  
cf. \cite{MV}. So one is reduced to the simplest case where $\rho_0$ is the trivial representation.

\bigskip

Let us now explain what is known about these coefficients $m(\b,\a)$ where the cuspidal support of
$\a,\b$ belongs the Zelevinsky line of the trivial representation.
First of all, it is proved in \cite{Z2} that there exists a poset structure on 
the set of multisegments such that 
$m_{\b, \a}>0$ if and only if $\b\leq \a$. And we let 
\[
  S(\a)=\{\b: \b\leq \a\}.
 \]
In \cite{Z3},  Zelevinsky introduced the nilpotent
orbit associated to a multisegment $\a$. 
More precisely, to a multisegment $\a$, one can 
associate $\varphi_{\a}: \Z\rightarrow \N$ 
with $\varphi_{\a}(k)$ the multiplicities
of $\nu^k$ appearing in $\a$. For each $\varphi$, 
$V_{\varphi}$ is a $\C$-vector spaces of dimension 
$\deg \varphi: =\sum_{k\in \Z}\varphi(k)$ with 
graded $k$-part of dimension $\varphi(k)$. 
Then $E_{\varphi}$ is the set of endomorphisms 
$T$ of degree $+1$, which admits 
a natural action of the group $G_{\varphi}=\prod_{k}GL(V_{\varphi, k})$.
Then the orbits of $E_{\varphi}$ under $G_{\varphi}$
are parametrized by multisegments $\a=\sum_{i\leq j}a_{ij}\Delta_{\nu^i, j-i+1}$ such that 
$\varphi=\varphi_{\a}$ consisting 
of $T$ with $a_{ij}$ Jordan cells starting from $V_{\varphi, i}$ to $V_{\varphi,j}$. 
We note $O_{\a}$ this orbit and we have the nice following property
\[
 \line{O}_{\a}=\bigsqcup_{\b\geq \a}O_{\b}.
\]
Now given a local system $\mathcal{L}_{\a}$
on $O_{\a}$, we can consider its 
intermediate extension $IC(\mathcal{L}_{\a})$ 
on $\line{O}_{\a}$ and its fiber at 
a geometric point $z_{\b}$ of $O_{\b}$ and form the 
Kazhdan-Lusztig polynomial 
\[
 P_{\a, \b}(q)=\sum_{i}q^{i/2}\dim_{\C}\mathcal{H}^i(IC(\mathcal{L}_{\a}))_{z_{\b}}.
\]
Zelevinsky then conjectured that $m_{\b, \a}=P_{\a, \b}(1)$ and 
call it the $p$-adic analogue of Kazhdan Lusztig Conjecture. 
This conjecture is a special case of a more general 
multiplicities formula proved by Chriss and Ginzburg in \cite{CG}, 
chapter 8.

\bigskip

\emph{In this work}, we first introduce the notion of a symmetric multisegment (cf. definition \ref{def: 2.1.5}), which is, 
roughly speaking,  a multisegment such that the 
beginnings and the ends of its segments are distinct and its segments admit non-empty intersections. 
We show that for a well chosen\footnote{Thanks to corollary \ref{coro-sym1} which is a particular case
of the Zelevinsky's conjecture, the results are independent of the choice of $\a_{\Id}$.}
symmetric multisegment $\a_{\Id}$, there is a natural bijection 
between the symmetric group $S_n$ to the set of symmetric multisegments $S(\a_{\Id})$, cf.  proposition \ref{teo: 2.3.2},
where 
$n$ is the number of  segments contained in  $\a_{\Id}$.

When we restrict to the geometry of the nilpotent orbits 
to the symmetric locus, 
%
%
we recover the  geometric situation of the Schubert varieties associated 
to $S_n$ and obtain that for two symmetric multisegment 
$\a_{\sigma}, \a_{\tau}$ associated to $\sigma, \tau\in S_n$, 
the coefficient $m_{\a_{\sigma}, \a_{\tau}}=P_{\sigma, \tau}(1)$. 

\bigskip

The next step \emph{in chapter 3} is to try to reach non symmetric 
cases, starting with a symmetric one. 
For example for $\a\geq \b$ be two multisegments and $\nu^k$ in 
the supercuspidal support of $\a$,  
one can eliminate every $\nu^k$ which appears at the end of 
some segments in $\a$ and $\b$ to obtain 
respectively a new pair of multisegments $\a^{(k)}, ~\b^{(k)}$ and 
try to prove that that $m(\b, \a)=m(\b^{(k)}, \a^{(k)})$. This 
result is almost true if we demand that $\b$ belongs to 
some subset $S(\a)_k$ of $S(\a)$, cf. Prop.\ref{cor: 3.2.3}.
The proof relies on the study of the geometry of nilpotent orbits and their links
with the Grassmannian, cf. the introduction of chapter 3.

%

\bigskip

I\emph{n chapter 4}, we iterate the process in chapter 3.
In fact, for  a multisegment $\a$ and $k_1, \cdots, k_r$
integers such that $\nu^{k_i}$ appears in the 
supercuspidal support of $\a$, let 
\[
 \a^{(k_1, \cdots, k_r)}=(((\a^{(k_{1})})\cdots)^{ (k_{r})}), 
\]
and 
\[
S(\a)_{k_{1}, \cdots, k_{r}}=
\{\c\in S(\a):
\c^{( k_{1}, \cdots , k_{i})}
\in  S(\a^{( k_{1}, \cdots , k_{i})})_{k_{i+1}}, \text{ for }i=1, \cdots, r\}. 
\]
Then we show that for $\b\in S(\a)_{k_{1}, \cdots, k_{r}}$,
we always have 
\[
 m(\b, \a)=m(\a^{(k_1, \cdots, k_r)}, \b^{(k_1, \cdots, k_r)}),
\]
Reciprocally, we show,  cf. proposition \ref{prop: 4.2.4}, that for any pair of multisegments
$\a>\b$, we can find $\a^{\sym}$ and 
$\b^{\sym}<\a^{\sym}$ such that 
\[
m(\b, \a)=m(\b^{\sym}, \a^{\sym}).
\]
In the end of chapter 4, following an example, we present an algorithm to find 
$(\a^{\sym}, \b^{\sym})$ . Finally the main application of the first part of this thesis, is, cf. theorem
\ref{teo: 4.4.5}, the proof of the Zelevinsky's conjecture stated before.

\vspace{1cm}

\emph{In the second part}, we consider the 
application of our result from the first four chapters.
\emph{In chapter 5}, as a first application, 
using the relation between symmetric groups and symmetric multisegments
we try to give a new proof of the fact that the Poincar\'e polynomial
$P_{\a_{\tau}, \a_{\sigma}}(q)$ of the intersection cohomology groups 
$\mathcal{H}^i(IC(\line{O}_{\a_{\tau}}))_{\a_{\sigma}}$ for 
\begin{itemize}
\item $\a_{\sigma}>\a_{\tau}$ a pair of symmetric multisegments with $\sigma, \tau\in S_n$ , 

\item where the index $\a_{\sigma}$ indicates that we localize at a point in $O_{\a_{\sigma}}$, 
\end{itemize}
satisfies the axioms defining the Kazhdan Lusztig polynomials for a Hecke algebra.

\bigskip

\emph{In Chapter 6}, we classify the poset $S(\a)$.
First of all, we single out the case where  the multisegment 
$\a$ contains segments with different beginnings and endings
and call it ordinary multisegment, cf. definition \ref{def: 2.1.1}.
In this case we prove that, as a poset,  
\[
 S(\a)\simeq S(\a^{\sym}, \a_{\min}^{\sym}):=\{\d\in S(\a^{\sym}): \d\geq \a_{\min}^{\sym}\},
\]
where $\a_{\min}$ is the minimal element in $S(\a)$ and $\a^{\sym}$(resp. $\a_{\min}^{\sym}$) is the symmetric multisegment associated to $\a$ (resp. $\a_{\min}$)
constructed in Chapter 4. Recall that in Chapter 2, we showed 
that $S(\a^{\sym})\subseteq S(\a_{\Id})$, 
for some $\a_{\Id}$, and $S(\a_{\Id})$
as a poset is isomorphic to $S_{n}$ with 
$n$ equal to the number of segments contained in $\a_{\Id}$. 
In this way, we identify the poset $S(\a)$ with some Bruhat 
interval in $S_n$, where $n$ is the number of segments contained in $\a$.

In the general case, as the ordinary case,
we can reduce to parabolic multisegments where 
a multisegment $\a$ is called parabolic if all of 
its segments contain a common point, cf. definition \ref{def: 6.2.5} and \ref{def: 6.2.22}.
Then all our construction for symmetric multisegments 
can be carried out with parabolic multisegments. Finally, we show that 
the poset $S(\a)$ is isomorphic to 
a Bruhat interval in $S_{J_2}\backslash S_n/S_{J_1}$,
where $J_i(i=1, 2)$ is a subset of generators and $S_{J_i}$ is 
the subgroup generated by $J_i$, see proposition
\ref{prop: 6.3.6} for details. 

\bigskip

\emph{In chapter 7}, if one is interested in calculating 
the multiplicities in $L_{\a}\times L_{\b}$, 
it might be interesting to first compute 
$\D^k(L_{\a})$. Using 
the formula of $\pi(\a)=\sum_{\b}m(\b, \a)L_{\b}$,
one is reduced to compute 
$$
 \D^k(\pi(\a))=\sum_{\b}n(\b, \a)L_{\b}
$$
for some coefficients $n(\b, \a)\geq 0$. 
As expected we can introduce a poset structure $\preceq_k$
on the set of multisegments so that 
$n(\b, \a)\geq 0\Leftrightarrow \b\preceq_k \a$, 
cf. proposition \ref{prop: 7.4.2}. Then using
the notion of Lusztig's product of two perverse sheaves
we prove, cf. proposition \ref{prop: 7.3.8}, that 
$n(\b, \a)$ is the value at $q=1$ of the Poincar\'e 
series of Lusztig product of two explicit perverse 
sheaves. In the parabolic case, 
we give an explicit description of this Lusztig product.
As a consequence, for case $\deg(\b)<\deg(\a)$, we show that the coefficient $n(\b, \a)$
is related to some $\mu(x, y)$, which is the coefficient of degree $\frac{1}{2}(\ell(y)-\ell(x)-1)$ in $P_{x, y}(q)$ defined to be zero if 
$\ell(y)-\ell(x)$ is even), where $x, y$ are elements in certain symmetric group and are related to $\a, \b$.

%

\bigskip

In the last chapter we use the computation of the partial derivatives in chapter 7 to give a recursive formula 
for the coefficients in the induced representation 
\[
L_{\a}\times L_{\b}=\sum m(\c, \b, \a)L_{\c}.
\]
It should be possible to treat the general case, but here we only consider the case where $\b$ is a segment.
The idea is to pass to lower degree by applying the partial derivatives.
The formulas are complicated, cf. proposition \ref{prop: 8.1.7}, even in the simplest case where $\b$ is a point.
It should be interesting to implement  the algorithm on a computer.

\tableofcontents

\mainmatter

\pagenumbering{arabic}

\renewcommand{\theequation}{\arabic{chapter}.\arabic{section}.\arabic{theo}}

\part{Multiplicities in 
Induced Representations}

\chapter{Induced Representations of \texorpdfstring{$GL_{n}$}{Lg}}

The aim of this section is to present our main object of study which are some 
integral coeffcients introduced by Zelevinsky, and defined by the formula \ref{eq: m(b, a)}, 
relating to some multisegments $\a, \b$ with cupsidal support contained in the  
Zelevinsky line associated to a cuspidal representation $\rho$. 

Recall that the 
set of irreducible representations of $GL_{n}$ breaks into pieces
according to the super-cuspidal support (Bernstein Center), and,  
thanks to the theory of types, we are reduced to study
the unipotent block, cf. \cite{MV}, that is induced representations with 
super-cuspidal support contained in the Zelevinsky line attach to the 
representation $\rho=1$.

Every unipotent irreducible representation is parametrized by a 
multisegment $\a$, that can be viewed as a function from the set 
of segments $\mathcal{C}$ to $\N$. For a multisegment $\a$, we 
denote by $L_{\a}$ the corresponding irreducible representation and 
$\pi(\a)$ the induced representation, cf. notations \ref{nota: 1.1.15}. 
The question is then to calculate the image of such an induced representation
in the associated Grothendieck group, that is to compute the multiplicity
$m(\b,\a)$ of $L_\b$ in $\pi(\a)$.

\medskip

To begin, let us fix some notations. 
Let $p$ be a prime number, $F/\Q_{p}$ be a finite extension. We fix an absolute
value $|.|$ on $F$ such that $|\varpi_{F}|=1/q$, where $\varpi_{F}$ is a uniformizer of $F$, and 
$q$ is the order of its residue field. For an integer $n \geq 1$, we denote by  $\nu$  the character of $GL_n(F)$ 
defined by $\nu(g)=|det(g)|$.

\section{Zelevinsky Classification}

\begin{notation}
We denote a partition of $n$ by
$\n=\{r_{1}, \cdots, r_{\alpha}\}$ with $\sum_{i=1}^{\alpha} r_{i}=n$. 
For a divisor $m$ of $n$, the partition $(m,\cdots,m)$ will be denoted $\underline{n_m}$.
We will also use the notation $\underline{n+m}=(n, m)$.
\end{notation}

\begin{definition}
For a partition $\n$, let 
$$P_{\n}=P_{\n}(F)=M_{\n}U_{\n}$$ 
be the corresponding  parabolic subgroup of $GL_{n}(F)$ with its decomposition into 
the product of its Levi subgroup $M_{\n}=GL_{r_{1}}(F)\times \cdots GL_{r_{\alpha}}(F)$ and 
its unipotent radical $U_{\n}$. Let $\delta_{P_{\n}}$ be the modular character of $P_{\n}$, given by
\[
 \delta_{P_{\n}}(-)=|det(ad(-)|_{\Lie U_{\n}})|^{-1}
\]

\end{definition}

For a topological group $G$, we recall that a representation $(\pi, V)$ of $G$ is 
\begin{itemize}
 \item \emph{smooth} if
for any vectors $v$, the stabilizer of $v$ in G is an open subgroup,  
\item
\emph{admissible} if for any open compact subgroup $K$ of $G$, 
$V^K=\{v:k.v=v, \forall k\in K\}$ is of finite dimension.
\end{itemize}

According to \cite{BZ1} theorem 4.1, a smooth representation of $GL_n(F)$ is of finite
length if and only if it is admissible and finitely generated.

\begin{definition}
For $\n=\{r_{1}, \cdots, r_{\alpha}\}$ and $\rho=\rho_1\otimes \cdots\otimes \rho_{r_{\alpha}}$
a smooth representation of $M_{\n}$, where the $\rho_{i}$ are representations 
of $GL_{r_{i}}(F)$, trivially extended to $P_{\n}$, we define
the normalized induction functor which associates to $ \rho$ the representation
$\pi=\Ind_{P_{\n}}^{GL_{n}(F)}(\rho)$
 of $G$ such that
\begin{displaymath}
\pi =\left\{ 
f: G\rightarrow V|\begin{array}{cc} 
                  &f(pg)=\delta_{P_{\n}}(p)^{-1/2}\rho(p)f(g),
 \forall p\in P_{\n}, f(gk)=f(g)\\ 
 &\text{ for all $k\in K$, with $K$ a certain open subgroup.}
 \end{array}\right\},
\end{displaymath}
here G acts on $f$ by $\pi(g)f(x)=f(xg)$.
\end{definition}

\begin{definition}
Let $(\pi, V)$ be a representation of $GL_{n}(F)$ and $P_{\n}$ a parabolic subgroup.
Let  $J_{P_{\n}}^{GL_{n}(F)}(\pi)$ be the Jacquet functor of $\pi$
defined by
\[
 J_{P_{\n}}^{GL_{n}(F)}(\pi)=V/V(U_{\n}),
\]
where $V(U_{\n})=\{u.v-u| u\in U_{\n}, v\in V\}$.
\end{definition}

\remk Both parabolic induction and Jacquet functor are additive exact functors between the category
 of smooth representations of $M_{\n}$ and $GL_{n}(F)$. Moreover, they preserve 
admissible representations and finitely generated representations.

\begin{prop} (cf.  \cite{DA} theorem 2.7, 4.1 and 5.3.)\label{prop: 1.1.5}
 For $\pi$ a smooth representation of $GL_{n}(F)$, and $\sigma$ a smooth
representation of $M_{\n}$, we have the following Frobenius reciprocity,
 \[
 \hom_{G}(\pi, \Ind_{P_{\n}}^{GL_{n}(F)}(\sigma))=\hom_{M_{\n}}(J_{P_{\n}}^{GL_{n}(F)}(\pi), \sigma\delta_{P_{\n}}^{-1/2}).
 \]
\end{prop}

\begin{definition}
 A smooth representation of $GL_{n}(F)$ is called cuspidal if for all nontrivial parabolic subgroup
$P_{\n}$ of $GL_n(F)$,
$$J_{P_{\n}}^{GL_{n}(F)}(\pi)=0.$$
We denote by $\lC_n$ the set of irreducible cuspidal representations of $GL_{n}(F)$, and
$$\lC=\coprod_{n \geq 1} \lC_n.$$
\end{definition}

\begin{prop} (cf. \cite{Z1} 4.1)\label{prop: 1.1.6}
Let $\pi$ be an irreducible representation of $GL_{n}(F)$, then there exists a partition 
$\n=\{r_{1}, \cdots, r_{\alpha}\}$
and a cuspidal representation $\rho=\rho_{1}\otimes \cdots \otimes \rho_{\alpha}$, 
of $M_{\n}$, 
such that $\pi$ can be embedded into $\Ind_{P_{\n}}^{GL_{n}(F)}(\rho)$. 
The set $ \{\rho_{1}, \cdots, \rho_{r}\}$ is determined by $\pi$ up to permutation, 
we call it \textit{the cuspidal support of $\pi$}.
\end{prop}
 
According to Harish Chandra, the study of irreducible representations of $GL_{n}$
is thus divided into two parts, the cuspidal representations and the parabolically induced representations.
We will not discuss here the classification of cuspidal representations of $GL_{n}(F)$, which rests on the theory of types 
for which the 
reader can refer to for example \cite{BK}. 

\begin{definition}\label{def: 1.1.8}
 By a multiset, we mean a pair $(S, r)$ where $S$ is a set 
 and $r: S\rightarrow \N$ is a map. We say $(S_{1}, r_{1})\subseteq (S_{2}, r_{2})$
 if $S_{1}\subseteq S_{2}$ and $r_{1}(s)\leq r_{2}(s)$ for all $s\in S_{1}$.  
 We define a bijection of multisets from $(S_1, r_1)$ to $(S_2, r_2)$ to be a bijection $\xi: S_1\rightarrow S_2$ satisfying 
 \[
 r_2(\xi(x))=r_1(x).
 \]
\end{definition}

\textbf{Convention: } Naturally, we write a multiset as a set with repetition. For example, 
for $S=\{a, b\}$ and $r(a)=2, r(b)=1$, then we write the multiset $(S, r)$ by
 $\{a, a, b\}$.

\begin{definitions} \label{def: 6}
\begin{itemize}
\item By a segment, we mean a subset $\Delta$ of $\lC$ of the form 
$\Delta=\{\rho, \nu\rho, \cdots, \nu\rho^{k}=\rho'\}$.
We denote it by $\Delta=[\rho, \rho']$ where $b(\Delta):=\rho$ is called its beginning  and 
$e(\Delta):=\rho'$ its end. Let $\Sigma^{univ}$ be the set of segments.

\item We say that two segments $\Delta_{1}$ and $\Delta_{2}$ are linked if
none of them is contained in the other and the union is again a segment.

\item For $\Delta_{1}=[\rho_{1},\rho_{1}']$
and $\Delta_{2}=[\rho_{2},\rho_{2}']$, we say $\Delta_{1}$ proceeds $\Delta_{2}$
if they are linked and $\rho_{2}=\nu^{k}\rho_{1}$ with $k>0$.

\item By a multisegment, we mean a finite multiset
$\a=\{\Delta_{1}, \cdots, \Delta_{r}\}$.  
Let  $\O^{univ}$ be the set of multisegments. 

\item We say a multisegment $\a=\{\Delta_{1}, \cdots, \Delta_{r}\}$ is
well ordered if for each pair of 
indices $i,j$ such that $i<j$, $\Delta_{i}$ does not proceeds $\Delta_{j}$.
\end{itemize}
\end{definitions}

\remk for a given multisegment, we may have several ways to arrange it to be 
a well ordered multisegment. 

\begin{notation} \label{nota-be}
 Let $\a=\{\Delta_{1}, \cdots, \Delta_{r}\}$. We call
 \[
  e(\a)=\bigl \{e(\Delta_{1}), \cdots, e(\Delta_{r}) \bigr \} \quad \hbox{ and } \quad
  b(\a)=\bigl \{b(\Delta_{1}), \cdots, b(\Delta_{r}) \bigr \}
 \]
respectively the end and the beginning of $\a$ as a multiset.
\end{notation}

\begin{def-prop} (\cite{Z2}3.1)
 Let $\rho$ be a cuspidal representation of $GL_{m}(F)$ then for $n=rm$
 $$
 \Ind_{P_{\underline{n_m}}}^{GL_{rm}(F)}(\rho\otimes \nu \rho\otimes \cdots \otimes \nu^{r-1}\rho )
 $$
 contains a unique irreducible sub-representation, denoted by $L_{[\rho, \nu^{m-1}\rho]}$.
\end{def-prop}

\begin{notation}
Let $\n=(r_{1}, \cdots ,r_{\alpha})$ be a partition. 
Let $\pi_i$ be a representation of $GL_{r_{i}}(F)$ for $i=1, \cdots, \alpha$.
Then we denote
\[
 \pi_{1}\times \cdots\times \pi_{\alpha}=\Ind_{P_{\underline{n}}}^{GL_{n}(F)}(\pi_{1}\otimes \cdots\otimes \pi_{\alpha}).
\] 
\end{notation}

\begin{prop} (\cite{Z2} Theorem 4.2)\label{prop: 1.1.11}
 Let $\Delta_{1}, \cdots, \Delta_{r}$ be segments, then the following two
 conditions are equivalent:
 \begin{description}
  \item [(1)]The representation $L_{\Delta_{1}}\times \cdots \times L_{\Delta_{r}}$ is irreducible.
  \item [(2)] For each $1\leq i,j\leq r$, $\Delta_{i}$ and $\Delta_{j}$ are
  not linked. 
 \end{description}
\end{prop}

 The following theorem gives a complete classification of the induced irreducible
representations of $GL_{n}(F)$ in terms of multisegments.

\begin{teo}\label{teo:1}
Let $\a=\{\Delta_{1}, \cdots, \Delta_{r}\}$ be a well ordered multisegment.
\begin{description}
\item[(1)] Then
the representation 
$$L_{\Delta_{1}}\times \cdots \times L_{\Delta_{r}}$$
contains an unique sub-representation, which we denote by $L_{\a}$.

\item[(2)] The representations $L_{\a}$ and $L_{\a'}$ are isomorphic if and only if $\a=\a'$ 
as well ordered multisegments, which means that there is a way to well order $\a'$ to obtain $\a$.

\item[(3)] Any irreducible representation of $GL_{n}(F)$ is isomorphic to some representation
of the form $L_{\a}$.
\end{description}
\end{teo}

\remk according to (2), the irreducible representation $L_\a$ does not depend on the well ordered form
of $\a$.

\begin{notation}\label{nota: 1.1.15}
From now on, for $\a=\{\Delta_{1}, \cdots, \Delta_{r}\}$ being well ordered,  we denote
\[
 \pi(\a)=L_{\Delta_{1}}\times \cdots \times L_{\Delta_{r}}.
\]

\end{notation}

\section{ Coefficients \texorpdfstring{$m(\b, \a)$}{Lg}}

\begin{notation}
We denote by $\mathcal{R}_{n}$ the Grothendieck group of the category
of finite length representations of $GL_{n}(F)$ and
$$\mathcal{R}^{univ}=\oplus_{n\geq 1} \mathcal{R}_{n}.$$
\end{notation}

\begin{prop}
The set $\mathcal{R}^{univ}$ is a bi-algebra with the multiplication $\mu$ and co-multiplication $c$ given by
  \[
   \mu(\pi_{1}\otimes \pi_{2})=\pi_{1}\times \pi_{2},\qquad 
   c(\pi)=\sum_{r=0}^{n} J^{GL_{n}(F)}_{P_{r,n-r}}(\pi).
  \]

\end{prop}
 
A consequence of theorem \ref{teo:1} is: 
 \begin{cor}\label{cor: 1.2.3}
 The algebra $\mathcal{R}^{univ}$ is a polynomial ring with indeterminates $\{L_{\Delta}: \Delta\in \Sigma^{univ}\}$.
 Moreover, as a $\Z$-module,
  the set $\{L_{\a}: \a\in \O^{univ}\}$ form a basis for $\mathcal{R}^{univ}$.
 \end{cor}

 \remk
 Note that this implies the Bernstein Center theorem, i.e, we have a decomposition
 \[
  \mathcal{R}^{univ}=\prod_{\rho}\mathcal{R}(\rho),
 \]
where $\rho$ runs through the equivalent classes of irreducible supercuspidal representations,
here we say two irreducible supercuspidal
representations are equivalent if they lie in the same Zelevinsky line, and $\mathcal{R}(\rho)$
is the subalgebra with support contained in the Zelevinsky line $\Pi_{\rho}$ generate by $\rho$.

 Using theorem \ref{teo:1}, let $\a=\{\Delta_{1}, \cdots, \Delta_{r}\} $ be a multisegment
 with support contained in some Zelevinsky line $\Pi_{\rho}$, then  we can write
 \addtocounter{theo}{1}
 \begin{equation}\label{eq: m(b, a)}
  \pi(\a)=\sum_{\b\in \O(\rho)} m(\b,\a)L_{\b}
 \end{equation}
where $\pi(\a)=\Delta_{1}\times \cdots \times \Delta_{r},~ m(\b,\a) \in \mathbb N$. One of the aims of this thesis is to give some new insights on these
$m(\b,\a)$.

\remk 
For our purpose, note that we can also rewrite the equation \ref{eq: m(b, a)} in the following form
\addtocounter{theo}{1}
 \begin{equation}\label{eq: m(b, a)1}
  L_{\a}=\sum_{\b\in \O(\rho)} \tilde{m}(\b,\a)\pi(\b).
 \end{equation}

The  simplest example is given by

\begin{prop}(cf. \cite{Z3} section 4.6 )\label{prop: 1.1.15}
 Let $\Delta_{1}$ and $\Delta_{2}$ be two linked segments, then 
 \[
  \Delta_{1}\times \Delta_{2}=L_{\a_{1}}+L_{\a_{2}}
 \]
with $\a_{1}=\{\Delta_{1}, ~\Delta_{2}\}$, $\a_{2}=\{\Delta_{1}\cup \Delta_{2}, ~\Delta_{1}\cap \Delta_{2}\}$.
\end{prop}

\remk
it is conjectured in \cite{Z2} 8.7 that the coefficient $m(\b,\a)$ depends only on the combinatorial relations of 
$\b$ and $\a$, and not on the specific cuspidal representation $\rho$. The 
independence of specific cuspidal representation can be showed by type theory,
see for example \cite{MV}. \textbf{In 
other words, as far as we are concerned with the coefficient $m(\b,\a)$, we can restrict
ourselves to the special case $\rho=1$, the trivial representation of $GL_{1}(F)$.}

\begin{definitions}\label{def: 1.2.5}
Let 
$$\Pi=\{\nu^k : k\in \Z\}$$ 
denote the Zelevinsky line of $\rho=1$.
We note 
\begin{itemize}
 \item 
$\Sigma$ the set of segments associated to $\Pi$,
\item
$\O$ the set of 
multisegments associated to $\Sigma$, 
\item
$\mathcal{R}$ the subalgebra of $\mathcal{R}^{univ}$ generate by the elements
 in $L_\a$ with $\a \in \O$,
 \item
$
\mathcal{C}=\{f:\Sigma\rightarrow \N \text{ with finite support}\},
$
\item
$\mathcal{S}=\{\varphi: \Z\rightarrow \N\}$.
\end{itemize}
\end{definitions}

\begin{notation}
For $i \leq j$, we will identify $L_{[\nu^i, \nu^j]} \in \mathcal R$ with $[i, j]$( for simplicity we let $[i]=[i, i]$). More generally 
we denote a multisegment $\a$ by $\sum_{i\leq j}a_{ij}[i,j]$.
\end{notation}

\begin{prop}\label{prop: 1.2.5}
By associating to $f\in \mathcal{C}$ the multisegment 
$$\sum_{\Delta\in \Sigma} f(\Delta)\Delta,$$
we can identify $\mathcal{C}$ with $\O$.
For every element $\b\in \O$, 
we set $f_{\b}$ for the associated function in $\mathcal{C}$.
\end{prop}

\begin{definition}\label{def: 2.2.2}
For a multisegment 
$$\a=\sum_{i\leq j}a_{ij}[i,j]$$ 
with $f_{\a}$ associated function in $\mathcal{C}$,
let 
\[
 \varphi_{\a}=\sum_{\Delta\in \a}f_{\a}(\Delta)\chi_{\Delta}\in \mathcal{S}.
\]
We call $\varphi_{\a}$
the weight of $\a$, and 
we call $\deg(\a)=\sum_{k\in \N} \varphi_{\a}(k)$ the degree of $\a$(or, the 
degree of $L_{\a}$).
\end{definition}

\begin{definition}\label{def: 1.2.9}
For $\varphi\in \mathcal{S}$, let $S(\varphi)$ be the 
set of multisegments with weight $\varphi$. 
\end{definition}

%



\section{A partial order on \texorpdfstring{$\O$}{Lg}}


\begin{definition}
For $\a$ a multisegment, by an elementary operation, we mean replacing two linked segments 
$\{\Delta_{1}, \Delta_{2}\}$
by $\{\Delta_{1}\cup \Delta_{2}, \Delta_{1}\cap \Delta_{2}\}$ in $\a$. 
\end{definition}

\begin{notation}\label{nota: 1.3.2}
Let $\b$ be a multisegment such that
it can be obtained from $\a$ by a series of elementary operations, then we say $\b\leq \a$. 
We denote
$$S(\a)=\{\b: \b\leq \a\}.$$
\end{notation}

\begin{definition}\label{def: 1.2.10}
 We define for $\b\leq \a$,
 \[
  \ell(\b,\a)=\max_{n}\{n: \a=\b_{0}\geq \b_{1}\cdots \geq\b_{n}=\b\},
 \]
and $\ell(a)=\ell(\a_{\min}, \a)$.
\end{definition}



\begin{definition}\label{def: 11}
We define the following total order relations on $\Sigma$:

\begin{displaymath} \left\{ \begin{array}{cc}
&[j, k]\prec  [m,n], \text{ if } k< n,\\
&[j,k]\prec  [m,n], \text{ if } j>m, n=k.
\end{array}\right. 
\end{displaymath}

\end{definition}



\begin{lemma}\label{lem: 3.0.8}
 Let $\b\in S(\a)$, then $\pi(\a)-\pi(\b)\geq 0$ in $\mathcal{R}$.
\end{lemma}
 
\begin{proof}
 By choosing a maximal chain of multisegments between $\a$ and 
 $\b$, we can assume that 
\[
 \a=\{\Delta_{1}, \cdots, \Delta_{r}\},
 \]
 \[
 \b=(\a\backslash\{\Delta_{j}, \Delta_{k}\})\cup \{\Delta_{j}\cap \Delta_{k}, \Delta_{j}\cup \Delta_{k}\}.
\]
Then by proposition \ref{prop: 1.1.15},
\[ 
 \pi(\a)=\pi(\b)+L_{\Delta_{1}}\times \cdots \times \hat{L}_{\Delta_{j}}
 \times\cdots \times \hat{L}_{\Delta_{k}}\times \cdots\times  L_{\Delta_{r}}
\times L_{\{\Delta_{j}, \Delta_{k}\}}
 \]
\end{proof}

\begin{prop}
 The set $S(\a)$ is a partially ordered finite set with unique minimal element $\a_{\min}$.
 Furthermore, $\a_{\min}$ is
 the unique multisegment in $S(\a)$ in which no segment is linked to
 the others.
\end{prop}

\remk in particular by proposition \ref{prop: 1.1.11} a multisegment $\a$ is minimal if and only if $\pi(\a)$ is irreducible.

\begin{proof}
For a proof of the partial orderness, we refer to \cite{Z2} 7.1. 
Let $X_{\a}:=\cup_{\Delta\in \a} \Delta$ be a subset 
of the Zelevinsky line $\Pi$.
Let $\varphi_\a$ be the weight function of $\a$.
Let $\Sigma(\a)$ be the set of segments with support in $X_{\a}$: this is
a finite set. For every $\Delta \in \Sigma(\a)$, we note
$\chi_{\Delta}$ the characteristic function of the set $\Delta$. Now 
we consider the set 
\[
 \Gamma(\a)=\{f\in \mathcal{C}:~ \varphi_\a= \sum_{\Delta\in \Sigma}f(\Delta)\chi_{\Delta} 
 \}.
\]
Then $\Gamma(\a)$ is a finite set. Clearly, for any $\b\in S(\a)$, we have 
$f_{\b}\in \mathcal{C}$ since the elementary operation does not change the weight function, note that $\b$ is uniquely determined by $f_{\b}$, 
so $S(\a)$ is finite since $\Gamma(\a)$ is finite.

We define $\a_{\min}=\{\Delta_{1}, \Delta_{2}, \cdots, \Delta_{r}\}$
with $\Delta_{1}\succeq \cdots \succeq \Delta_{r}$, where
for $\Delta_{0}=\emptyset$, we set $\Delta_{i}$ be the maximal segment with respect to the 
total order~$\prec$, such that $\chi_{\Delta_{i}}$ is supported in 
$\Supp(\varphi_\a-\chi_{\Delta_{0}}-\cdots - \chi_{\Delta_{i-1}})$.


We only need to show 
that for all $\b\in S(\a)$, we have $\a_{\min}\leq \b$. To see this,
we look at a maximal segment $\Delta'$ in $\b$, if it is linked to some
segments $\Delta''$, then we apply the elementary operation to them and
get $\b_{1}$. Now repeat the same procedure, in finite steps we get a 
multisegment $\b'\leq \b$ in which no segments are linked to the others.
It remains to show that $\b'=\a_{\min}$.

In fact, we have
\addtocounter{theo}{1}
\begin{equation}\label{equ: (1)}
 \varphi_\a=\sum_{\Delta\in \Sigma(\a)}f_{\a_{\min}}(\Delta)\chi_{\Delta} 
 =\sum_{\Delta\in \Sigma(\a)}f_{\b'}(\Delta)\chi_{\Delta}. 
\end{equation}

Let $\b'=\{\Delta_{1}', \cdots, \Delta_{t}'\}$ with 
$\Delta_{1}'\succeq \cdots \succeq \Delta_{t}'$.  Put $\Delta'_{0}=\emptyset $ and
suppose by induction that there is an $s$ with $1\leq s\leq \min\{r,~t\}$ such that 
for all $0\leq i<s$, $\Delta_{i}'= \Delta_{i}$. By construction, we  have $\Delta_{s}'\preceq  \Delta_{s}$
and we assume that $\Delta_{s}'\prec \Delta_{s}$. By the equality (\ref{equ: (1)}),
$e(\Delta_{s})=e(\Delta_{s}')$, 
then $\chi_{\Delta_{s}'}-\chi_{\Delta_{s}}$
is negative. Let $\Delta=\Delta_{s}\setminus \Delta_{s}'$.
Now by the equality (\ref{equ: (1)}),  there exists a minimal $i>s$ such that
the segment $\Delta_{i}'$ 
satisfies the property that $b(\Delta_{i}')\leq b(\Delta)\leq e(\Delta)\leq e(\Delta_{i}')$. 

But this implies that $\Delta_{s}'$ is linked
to $\Delta_{i}'$, contradiction. 
Therefore $\Delta_{s}'=\Delta_{s}$. We conclude by the same argument that
\[
 r=s,~\Delta_{i}'=\Delta_{i}, 1\leq i\leq r.
\]
\end{proof}

Concerning the coefficient $m(\b, \a)$, we have 
 
\begin{prop}(cf.\cite{Z2} 7.1)
 The coefficient $m(\b,\a)$ is
 \begin{itemize}
  \item nonzero if and only if $\b\leq \a$, and
  \item equal to $1$ if $\b=\a$. 
  \end{itemize}
\end{prop}
 

%
%
%
%
%
%

\section{Partial Derivatives}
In this section we show how to define some analogue of the Zelevinsky derivation. 
This section will not be used until Chapter 7 but some of the properties of partial derivation
will appear all along the text. 

\begin{definition}
 We define a left partial derivation with respect to index $i$ to be a morphism of algebras
 \begin{align*}
 &^i\D: \mathcal{R}\rightarrow \mathcal{R},\\
 &^i\D(L_{[j,k]})=L_{[j,k]}+\delta_{i,j}L_{[j+1,k]} \text{ if }(k>j),\\
 &^i\D(L_{[j]})=L_{[j]}+\delta_{[i],[j]}.
 \end{align*}
 Also we define a right partial derivation with respect to index $i$ to be a morphism of algebras
 \begin{align*}
 &\D^i: \mathcal{R}\rightarrow \mathcal{R}\\
 &\D^i(L_{[j,k]})=L_{[j,k]}+\delta_{i,k}L_{[j,k-1]} \text{ if }(j<k)\\
 &\D^i(L_{[j]})=L_{[j]}+\delta_{[i],[j]}.
 \end{align*}
\end{definition}

\begin{definition}
 We define 
 \[
  \D^{[i,j]}=\D^{j}\circ \cdots \circ \D^{i}
  \]
  \[
  ^{[i,j]}\D=(^{i}\D)\circ \cdots \circ ( ^{j}\D)
 \]
 And for $\c=\{\Delta_{1}, \cdots, \Delta_{s}\}$ with
 \[
  \Delta_{1}\preceq \cdots \preceq \Delta_{s},
 \]
we define 
\[
 \D^{\c}=\D^{\Delta_{1}}\circ \cdots \circ \D^{\Delta_{s}}
\]
and
\[
  ^{\c}\D=(^{\Delta_{s}}\D)\circ \cdots \circ (^{\Delta_{1}}\D).
\]
\end{definition}

\remk we recall that in \cite{Z1} 4.5, Zelevinsky defines a derivative $\mathscr{D}$ to be 
 an algebraic morphism 
 \[
  \mathscr{D}: \mathcal{R}\rightarrow \mathcal{R},
 \]
which plays a crucial role in Zelevinsky's classification theorem.

The relation between Jacquet functor and derivative is given by
\begin{prop}(cf. \cite{Z2}3.8)
 Let $\delta$ be the algebraic morphism such that $\delta(\rho)=1$ for all
 $\rho\in \lC$ and $\delta(L_\Delta)=0$ for all non cuspidal representations $L_\Delta$. Then
 \[
  \mathscr{D}=(1\otimes \delta)\circ c,
 \]
where $c$ is the co-multiplication.
\end{prop}

 The main advantage to work with partial derivatives instead of the derivative defined by
 Zelevinsky is that they are much more simpler but share the following positivity properties:
 
 \begin{teo}\label{teo: 3}
  Let $\a$ be any multisegment, then we have 
  $$\D^i(L_{\a})=\sum_{\b\in \O}n(\b,\a)L_{\b},$$
  such that $n(\b,\a)\geq 0$, for all $\b$.
 \end{teo}
 
\remk the same property of positivity holds for $^i\D$.

The theorem follows from the following two lemmas

\begin{definition}
For $i\in \Z$, let $\phi_{i}$ be the morphism of algebras defined by
\begin{align*}
 \phi_{i}: \mathcal{R}&\rightarrow \Z\\
 \phi_{i}([j,k])&=\delta_{[i], [j,k]}.
\end{align*}
\end{definition}

\begin{lemma}
For all multisegment $\a$, we have $\phi_{i}(L_{\a})=1$ if and
only if $\a$ contains no other segments than $[i]$, otherwise it is zero.
\end{lemma}

\begin{proof}
We prove this result by induction on the cardinality  of $S(\a)$, denoted by $|S(\a)|$. If $|S(\a)|=1$, then
$\a=\a_{min}$, hence $\phi_{i}(L_{\a})= \phi_{i}(\pi(\a))$, which is nonzero if and
only if $\a$ contains no other segments than $[i]$, and in latter case it is 1. 
Let $\a$ be a general multi-segment,
\[
 \pi(\a)= L_{\a}+\sum_{\b<\a}m(\b,\a)L_{\b}.
\]
Now $|S(\a)|>1$, we know that $\a$ is not minimal in $S(\a)$, hence $\a$
contains segments other than $[i]$, which implies $\phi_{i}(\pi(\a))=0$.

Since $|S(\b)|<|S(\a)|$ for any $\b<\a$, by induction, we know that $\phi_{i}(L_{\b})=0$
because $\b$ must contain segments other than $[i]$. So we are done.

\end{proof}

\begin{lemma}
 We have $\D^i= (1\otimes \phi_{i})\circ c$.
\end{lemma}

\begin{proof}
Since both are algebraic morphisms, we only need to check that they coincide on
generators. We recall the equation from \cite{Z2}, proposition 3.4
\[
 c(L_{[j,k]})=1\otimes L_{[j,k]}+\sum_{r=j}^{k-1}L_{[j,r]}\otimes L_{[r+1, k]}+L_{[j,k]}\otimes 1.
\]
   Now applying $\phi_{i}$,
\begin{align*}
 (1\otimes \phi_{i})c(L_{[j,k]})=&L_{[j,k]}+\delta_{i,k}L_{[j,k-1]} \text{ if }(k>j)\\
 (1\otimes \phi_{i})c(L_{[j]})=&L_{[j,k]}+\delta_{i,j},
 \end{align*}
 where $\delta_{i, j}$ is the Kronecker symbol. 
Comparing this with the definition of $\D^i$ yields the result.
\end{proof}

\remk We have the following relation between partial derivative and derivative of Zelevinsky.
 Let $e(\a)=\{[i_{1}], \cdots, [i_{\alpha}]: i_{1}\leq \cdots \leq i_{\alpha}\}$ be the end
of $\a$, then 
 \[
  \D(\a)=\D^{[i_{1}, i_{\alpha}]}(\a).
 \]

%
%
%
%
%
%
%
%
%

\chapter{Schubert varieties and KL polynomials}

In this chapter we recall some of the geometric constructions of Zelevinsky: 
the nilpotent orbital varieties and their relation with Schubert varieties.

\medskip

Concretely, for $\a, \b$ multisegments 
of degree $n$ such that the nilpotent
orbit $O_{\a}$ is included in the closure 
of $\line{O}_{\b}$, the germs of intersection
complexe $IC(\line{O_{\b}})$ at a generic point 
of $O_{\a}$ gives 
the Poincar\'e polynomial $P_{\a, \b}(q)$
and Zelevinsky conjectured that 
$$m_{\b, \a}=P_{\a, \b}(1)=P_{\sigma(\a), \sigma(\b)}(1)$$
viewed in the Schubert variety associated 
to the symmetric group $S_{n}$, where 
$\sigma(\a)$ and $\sigma(\b)$ are certain permutations
attached to $\a$ and $\b$.
This conjecture was proved by Chriss-Ginzburg \cite{CG}, 
and Ariki \cite{A}. 

\medskip

In the following, we study the case 
of symmetric multisegments in the 
sense of definition \ref{def: 2.1.5}.
The set of symmetric multisegment of some specific  weight $\varphi$
is indexed by $S_m$, where $m=\max_{k\in \Z}\varphi(k)$, which 
is in general strictly smaller than its degree= $\sum_{k}\varphi(k)$.
In this symmetric situation, 
we construct a fibration from the 
symmetric locus in the orbital varieties $E_{\varphi}$
to some smooth variety, where
the stratification of $E_{\varphi}$ gives rise to 
a stratification of the fibers.
And we show that the fiber is
isomorphic to some Schubert
variety of type $A_{m-1}$, which identifies the stratification of 
fiber with the stratification by Schubert cells.

\section{Symmetric multisegments}

Before we introduce the symmetric
multisegments, we present a type 
of multisegments which is more general and will be 
used in Chapter 6.

\begin{definition}\label{def: 2.1.1}
 We say a multisegment $\mathbf{a}$ is ordinary if  there exists no two segments in $\a$ that possesses the same beginning or end.
\end{definition}

\begin{example}
 Some typical examples of ordinary multisegments:
 let $\a=\{\Delta_{1}, \Delta_{2}, \Delta_{3}\}$, and 
 $\b=\{\Delta_{4}, \Delta_{5}, \Delta_{6}\}$
\[
 \Delta_{1}=[1,4],~\Delta_{2}=[2,5],~\Delta_{3}=[3,6],
\]
\[
 \Delta_{4}=[1,2], ~\Delta_{5}=[2, 4], ~\Delta_{6}.
=[4,5]\]

\begin{figure}[!ht]
\centering
\includegraphics{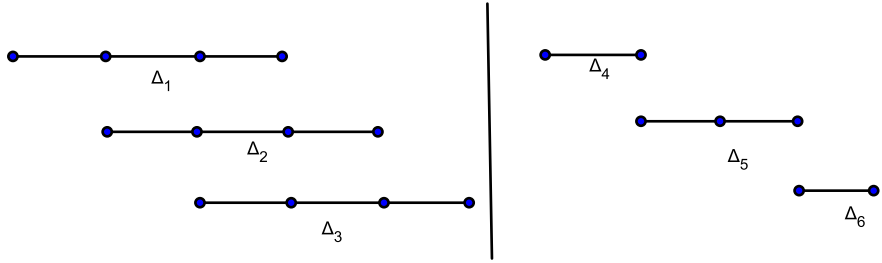}
\caption{\label{fig-multisegment2} Ordinary multi-segments}
\end{figure}
\end{example}

\begin{prop}
If $\a$ is ordinary then every $\b\leq \a$ is ordinary.
\end{prop}

\begin{proof}
From the definition, $\b$ is ordinary if and only if each element in $e(\b)$ and
$b(\b)$ appears with multiplicity one. We deduce from
the following lemma that $\b\leq \a$ is also ordinary.
\end{proof}

\begin{lemma}\label{lem: 2.1.3}
 Note that for $\b\leq \a$, we have 
$e(\b)\subseteq e(\a)$ and $b(\b)\subseteq b(\a)$(cf. notation \ref{nota-be}).
\end{lemma}

\begin{proof}
In fact, by transitivity, 
we only need to check this for case where $\b$ can be obtained from $\a$
by applying the elementary operation to the pair $\{\Delta_{1}\prec \Delta_{2}\}$.
Hence 
\[
 \b=\a\setminus \{\Delta_{1}, \Delta_{2}\}\cup \{\Delta_{1}\cup \Delta_{2}, ~\Delta_{1}\cap \Delta_{2}\}.
\]
Note that $e(\Delta_{1}\cup \Delta_{2})=e(\Delta_{2}), ~b(\Delta_{1}\cup \Delta_{2})=b(\Delta_{1})$,
and if $\Delta_{1}\cap \Delta_{2}\neq \emptyset$,  $e(\Delta_{1}\cap \Delta_{2})=e(\Delta_{1}), 
~b(\Delta_{1}\cap \Delta_{2})=b(\Delta_{2})$. Hence $b(\b)\subseteq b(\a), ~ e(\b)\subseteq 
e(\a)$. 
\end{proof}

\begin{definition}\label{def: 2.1.5}
 Let $\a=\{\Delta_{1}, \cdots, \Delta_{n}\}$  be ordinary.
We say that $\a$ is symmetric if 
\[
 \max\{b(\Delta_{i}): i=1,\cdots, n\}\leq \min\{e(\Delta_i): i=1,\cdots, n\}.
\]

 \end{definition}

 To explain the link with 
 the symmetric group, 
 we recall some basic facts about the 
 symmetric group $S_{n}$(cf. \cite{BF}).
Let $(i,j)$ be the transposition exchanging $i$ and $j$,
then 
\[
 S=\{\sigma_{i}:=(i,i+1): i=1,\cdots n-1\}
\]
form a system of generators of $S_{n}$.

\begin{definition}
For $w\in S_{n}$, its length $\ell(w)$ is the smallest integer $k$ such that 
\[
 w=s_{1}s_{2}\cdots s_{k},\text{ with } s_{i}\in S, \text{ for }i=1, \cdots, k.
\]

\end{definition}

On $S_{n}$, we have the famous 
Bruhat order which is defined as follow:
\begin{definition}
Let $T=\{wsw^{-1}: w\in S_{n}, s\in S\}$. For $u, w\in S_{n}$, 
\begin{description}
 \item [(i)]We write $\xymatrix{u\ar[r]^{t} &w}$, if $u^{-1}w=t \in T$ and $\ell(u)<\ell(w)$.
 \item [(ii)] We write $\xymatrix{u\ar[r]& w}$, if  $\xymatrix{u\ar[r]^{t} &w}$ for some $t\in T$.
 \item [(iii)] We write $u\leq w$ if there exists a sequence of $w_{i}\in S_{n}$, such that
 \[
  u\rightarrow w_{1}\rightarrow w_{2}\rightarrow \cdots \rightarrow w_{k}=w.
 \]

\end{description}
This defines a partial order on $S_{n}$, which is called the Bruhat order.
\end{definition}

\begin{prop}\label{teo: 2.3.2}
Let $\a_{\Id}=\{\Delta_{1}, \cdots, \Delta_{n}\}$ be symmetric,
such that 
\[
  b(\Delta_{1})< \cdots < b(\Delta_{n}),
\]

\[
  e(\Delta_{1})< \cdots < e(\Delta_{n}).
\]
Then for $w\in S_{n}$, the formula 
\begin{align*}
\Phi(w)&=\sum_{i=1}^{n} [b(\Delta_{i}), e(\Delta_{w(i)})]
\end{align*}
defines a bijection between $S_{n}$ and $S(\a_{\Id})$.
 Moreover, the order relation on $S(\a_{\Id})$ induces the inverse Bruhat order, i.e.,
 \[
  w\leq v\Leftrightarrow \Phi(w)\geq \Phi(v).
 \]
\end{prop}

\begin{example}
Let $n=3$ and
$\a_{\Id}=\{\Delta_{1}, \Delta_{2}, \Delta_{3}\}$ with  
\[
 \Delta_{1}=[1,4],~\Delta_{2}=[2,5],~\Delta_{3}=[3,6].
\]
Then $\Phi(\sigma_1)=\{\Delta_4, \Delta_5, \Delta_6\}$ with
\[
 \Delta_4=[1, 5], ~\Delta_5=[2, 4], ~\Delta_6=[3, 6].
\]
\begin{figure}[!ht]
\centering
\includegraphics{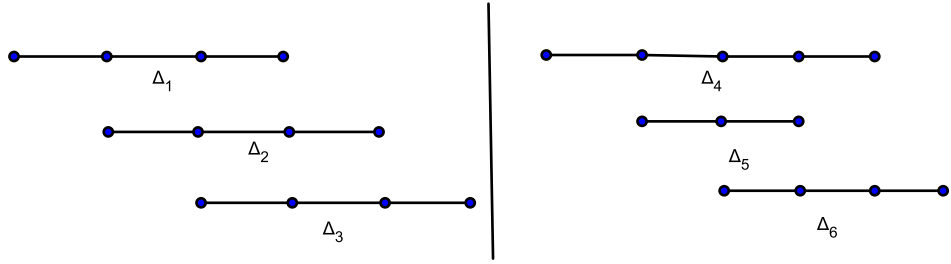}
\caption{\label{fig-multisegment19} Symmetric multi-segments}
\end{figure}
\end{example}

\begin{proof}
The injectivity is clear.
We observe that $\Phi(\Id)=\a_{\Id}$. 
We show now that $\Phi(w)\in S(\a_{\Id})$ for general $w$
and the partial order on $S(\a_{\Id})$ induces the inverse Bruhat order.

\begin{enumerate}
\item[(1)]For $v\leq w$,
 by the chain property of Bruhat order(cf. \cite{BF} Theorem 2.2.6), we have
\[
 v=w_{0}<w_{1}<\cdots <w_{\alpha}=w,
\]
such that $w_{\gamma}=\sigma_{i_{\gamma-1}, j_{\gamma-1}}w_{\gamma-1}$ for some
$i_{\gamma-1}< j_{\gamma-1}$
and $\ell(w_{\gamma})=\ell(w_{\gamma-1})+1$. 
Now by lemma 2.1.4 of \cite{BF}, we know that 

\[
 w_{\gamma-1}^{-1}(i_{\gamma-1})<w_{\gamma-1}^{-1}(j_{\gamma-1}).
\]
Hence the segments
$$[b(\Delta_{w_{\gamma-1}^{-1}(i_{\gamma-1})}), e(\Delta_{i_{\gamma-1}})]$$ 
$$[b(\Delta_{w_{\gamma-1}^{-1}(j_{\gamma-1})}), e(\Delta_{j_{\gamma-1}})]$$
are linked in $\Phi(w_{\gamma-1})$. Moreover, by performing the
elementary operation on the two segments, we obtain $\Phi(w_{\gamma})$, 
hence
\[
 \Phi(w_{\gamma-1})>Phi(w_{\gamma}).
\]
Again by transitivity of partial orders, we are done.
Note that we proved that all $\Phi(w)$ are in $S(\a_{\Id})$.
Moreover, for any $\b\in S(\a_{\Id})$, 
the fact that $\a_{\Id}$ is symmetric implies $b(\a_{\Id})=b(\b)$ since no segment 
is juxtaposed to the others. 
The same reason shows that $e(\a_{\Id})=e(\b)$.
Hence 
there is a unique $w\in S_{n}$ 
such that 
\[
\b=\sum_{i=1}^{n}[b(\Delta_{i}), e(\Delta_{w(i)})].
\]
This proves the surjectivity.

\item [(2)] Let $\Phi(w)\geq \Phi(v)$,
 we choose 
$$
 \Phi(w)=\Phi(w_{0})>\Phi(w_{1})>\cdots >\Phi(w_{\alpha})=\Phi(v)
$$
to be a maximal chain of multisegments, where 
$\Phi(w_{\gamma})$ is obtained from $\Phi(w_{\gamma-1})$ by performing the elementary operation 
to segments 
$$
\{[b(\Delta_{i_{\gamma-1}}), e(\Delta_{w_{\gamma-1}(i_{\gamma-1})})], 
\quad [b(\Delta_{j_{\gamma-1}}), e(\Delta_{w_{\gamma-1}(j_{\gamma-1})})]\}
$$
in $\Phi(w_{\gamma-1})$ with $i_{\gamma-1}<j_{\gamma-1}$. Then 
\[
 w_{\gamma-1}(i_{\gamma-1})<w_{\gamma-1}(j_{\gamma-1}).
\]

Hence 
$$
w_{\gamma}=\sigma_{w_{\gamma-1}(i_{\gamma-1}), w_{\gamma-1}(j_{\gamma-1})}w_{\gamma-1}.
$$
Note that in this case, we have either 
\[
 w_{\gamma}<w_{\gamma-1}
\]
or 
\[
 w_{\gamma}>w_{\gamma-1},
\]
by (1), the former implies $\Phi(w_{\gamma-1})<\Phi(w_{\gamma})$, 
contradiction to our assumption.

Hence we must have
\[
 w_{\gamma}>w_{\gamma-1}.
\]
We conclude by transitivity of partial order that $w<v$.
\end{enumerate}
\end{proof}

\section{Nilpotent Orbits}
In this section we shall introduce the nilpotent orbits constructed in 
\cite{Z3} and discuss their geometry and relations with multisegments.

\begin{definition}\label{def: 4.1.1}
\begin{description}
 \item[(1)] Let $\varphi\in \mathcal{S}$ (cf. Def. \ref{def: 1.2.5}) such that 
$\supp\varphi=\{1,\cdots,  h\}$. Let $V_{\varphi}=\oplus_{k\in \Z}V_{\varphi, k}$ be a 
$\Z$-graded $\C$ vector space such that $\dim V_{\varphi,k}=\varphi(k)$.
\item [(2)]
 Let $E_{\varphi}$ be the set of endomorphism $T$ of $V_{\varphi}$ of
 degree $1$, i.e. such that $TV_{\varphi, k}\subseteq V_{\varphi,k+1}$.
 \end{description}
\end{definition}

\remk(cf.\cite{Z3}, 1.8)
$G_{\varphi}(\C)=\prod_{k\in \Z}GL(V_{\varphi, k})$ acts on $E_{\varphi}$ by conjugation. For each element $T$ in $E_{\varphi}$, there exists a basis 
of $V_{\varphi}$ that
consists of homogeneous elements, under which $T$ is of the Jordan form . 

\begin{notation}
From now on, for simplicity,  in all circumstances, we will write $G_{\varphi}$ for $G_{\varphi}(\C)$, 
$GL_{n}$ for $GL_{n}(\C)$ and $M_{i,j}$ for $M_{i,j}(\C)$.
\end{notation}

\begin{lemma}\label{lem: 4.1.4}
By fixing a basis for each $V_{k}$, we have 
 \[
  E_{\varphi}\simeq M_{\varphi(2), \varphi(1)}\times \cdots \times M_{\varphi(h), \varphi(h-1)}
 \]
Here we suppose that $\supp\varphi \subseteq [1,h]$ and
$M_{k, \ell}$ denotes the vector space of matrices over $\C$ with
$k$ rows and $\ell$ columns. 
 \end{lemma}
\remk
In this case, the group
\[
 G_{\varphi}=GL_{\varphi(1)}\times \cdots \times GL_{\varphi(h)}
\]
acts by
\[
 (g_{1}, \cdots, g_{h}).(f_{1}, \cdots, f_{h-1})=(g_{2}f_{1}g_{1}^{-1}, g_{3}f_{2}g_{2}^{-1}, \cdots, g_{h}f_{h-1}g_{h-1}^{-1}).
\]
\begin{proof}
It follows directly from the definition of $E_{\varphi}$. 
\end{proof}

\begin{example}\label{ex: 2.2.2}
Consider the function $\varphi=\chi_{1}+2\chi_{2}+\chi_{3}\in \mathcal{S}$(cf. Def \ref{def: 11}),
where $\chi_{k}$ denotes the characteristic 
function of $k$.
To $\varphi$ we can attach the space $V_{\varphi}=V_{1}\oplus V_{2}\oplus V_{3}$ such that 
\[
 V_{1}=\C v_{1}, ~V_{2}=\C v_{2}\oplus \C v_{3}, ~ V_{3}=\C v_{4}.
\]
Consider the operator $T\in E_{\varphi}$, such that 
\[
 T(v_{1})=v_{2}-v_{3}, ~T(v_{2})=T(v_{3})=v_{4}.  
\]
Then by choosing a new basis 
\[
 v_{1}'=v_{1}, ~v_{2}'=v_{1}-v_{2}, ~v_{3}'=v_{1}+v_{3}, ~v_{4}'=2v_{4},
\]
we get 
\[
 T(v_{1}')=v_{2}', ~T(v_{2}')=0, T(v_{3}')=v_{4}',
\]
which gives the Jordan form $J_{T}$ of $T$
$$J_{T}=
\begin{pmatrix}
0 &0&0& 0\\
1&0&0& 0\\
0&0&0& 0\\
0&0&1&0
\end{pmatrix}
$$ 
\end{example}

\begin{prop}(cf.\cite{Z3}, 2.3)\label{prop: 2.2.4}
The orbits of $E_{\varphi}$ under
$G_{\varphi}$ are naturally parametrized by multisegments of weight $\varphi$.
\end{prop}
\begin{proof}
Let $\a=\sum_{i\leq j}a_{ij}[i,j]$ such that 
$\varphi_{\a}=\varphi$, then 
the orbit associated consists of 
the operators having exactly $a_{ij}$ Jordan cells starting from $V_{\varphi, j}$ and
ending in $V_{\varphi,i}$.
\end{proof}

\begin{notation}
We denote by $O_{\a}$ the orbit associated to the multisegment $\a$. 
\end{notation}

\begin{example}
We take the same function $\varphi=\chi_{1}+2\chi_{2}+\chi_3$ as in example \ref{ex: 2.2.2}.
Then the multisegments of weight $\varphi$ are listed below(cf. \cite{Z2} section 11.4)
\[
 \a_{\max}=\{[1], [2], [2], [3]\},~ \a_{\ell}=\{[1, 2], [2], [3]\}, 
\]
\[
 \a_{r}=\{[1], [2], [2, 3]\}, ~\a_{0}=\{[1, 2], [2, 3]\}, ~\a_{\min}=\{[1, 3], [2]\}.
\]
And the corresponding Jordan forms are given by 
$$
J_{\a_{\max}}=0, ~
J_{\a_{\ell}}=
\begin{pmatrix}
0 &0&0& 0\\
1&0&0& 0\\
0&0&0& 0\\
0&0&0&0
\end{pmatrix}
$$
$$
J_{\a_{r}}=
\begin{pmatrix}
0 &0&0& 0\\
0&0&0& 0\\
0&0&0& 0\\
0&0&1&0
\end{pmatrix}, ~
J_{\a_{0}}=
\begin{pmatrix}
0 &0&0& 0\\
1&0&0& 0\\
0&0&0& 0\\
0&0&1&0
\end{pmatrix},~
J_{\a_{\min}}=
\begin{pmatrix}
0 &0&0& 0\\
1&0&0& 0\\
0&0&0& 0\\
0&1&0&0
\end{pmatrix}.
$$ 
\end{example}

\begin{prop}(cf. \cite{Z3}, 2.2)
In $E_{\varphi}$, we have $\line{O}_{\b}=\coprod_{\a\geq \b}O_{\a} $. 
\end{prop}

\begin{definition}\label{def: 4.1.5}
 For any $T\in E_{\varphi}$, and $i\leq j$, denote by $T^{[i,j]}$ the
 composition map: 
 \begin{displaymath}
\xymatrix
{
V_{i}\ar[r]^{\hspace{-0.5cm }T} & V_{i+1}\cdots \ar[r]^{\hspace{0.5cm }T}&V_{j},\\
}
\end{displaymath}
we define
 \[
  r_{ij}(T)=\rank(T^{[i,j]}).
 \]
 \end{definition}
 \remk
 For $\a$ a multisegment, $r_{ij}(T)$ remains constant for any $T\in O_{\a}$, we denote
 it by $r_{ij}(\a)$.
 
We recall the following combinatorial results
\begin{prop}(cf. \cite{Z3}section 2.5)\label{prop: 4.2.3}
 Let $\a, \b$ be two multisegments such that 
 \[
  \varphi_{\a}=\varphi_{\b}.
 \]

  Then the following two conditions are equivalent:
 \begin{description}
  \item [(1)] $\b\leq \a$;
  \item [(2)] $r_{ij}(\a)\leq r_{ij}(\b)$ for all $i\leq j$.
 \end{description}
\end{prop}

In symmetric case, we have the following interesting description of
$r_{ij}$.
\begin{lemma}
Let $w\in S_n$. Then
we have
$ r_{i,j+n-1}(w):=r_{i, j+n-1}(\Phi(w))=\{k\leq i: w(k)\geq j\}.$
\end{lemma}

\begin{proof}
In fact, take 
\[
\a_{\Id}=\sum_{k=1}^{n}[k, k+n-1],
\]
and consider the bijection 
\[
\Phi: S_n\rightarrow S(\a_{\Id})
\]
with 
 \[
  \Phi(w)=\sum_{k=1}^{n}[k, w(k)+n-1].
\]
By definition, $r_{i,j+n-1}(w)$ is the number of segments in $L_{\Phi(w)}$
which contains $[i,j+n-1]$, hence is of the form $[k, w(k)+n-1]$ with
\[
 k\leq i,\quad w(k)\geq j.
\]
\end{proof}

Now combining with the proposition \ref{prop: 4.2.3}, gives the following
known results,
\begin{prop}(\cite{LM} Proposition 2.1.12)
In $S_{n}$,
$v\leq w\Leftrightarrow r_{ij}(v)\leq r_{ij}(w)$, for all $i\leq j$.  
\end{prop}

\section{Schubert Varieties and KL Polynomials}
%
Let $Y$ be an algebraic variety over $\C$.
\begin{definition}
By a stratification $\mathfrak{H}$ on $Y$, we mean a decomposition of $Y$ into 
locally closed smooth sub-varieties $Y_{i}$. An element of $\mathfrak{H}$ is 
called a stratum.
\end{definition}
\remk We require a variety to be irreducible.

\begin{definition}
 Let $D^{b}(Y)=D^b_c(Y)$ be the bounded derived category of  sheaves with values in 
complex vector spaces over $Y$. 
And let $D(Y)$ be the subcategory consisting of those complexes whose cohomology 
sheaves are constructible.
\end{definition}

Given a stratification $\mathfrak{H}$, we let $U_{\ell}$ denote the 
set of strata whose dimension is $\geq \ell$.

\begin{definition}(cf.\cite{de05} Remark 3.8.1)
Given a local system on the open stratum $U_{d}$ with $d=\dim(Y)$, we define inductively 
a complex $IC(Y, L)$ in $D(Y)$ as follows. 

We start by letting $IC(U_{d}, L): =L[\dim Y]$. Assuming that we already
defined $IC(U_{\ell+1}, L)$, let $j: U_{\l+1}\rightarrow U_{\l}$ be the 
open immersion, then we define
\[
 IC(U_{\ell}, L): = \tau_{\leq -\ell-1} Rj_{*}IC(U_{\ell+1}, L),
\]
here $\tau_{\leq k}$ is the truncation from the right in degree $k$.
In finite step, we get $IC(Y, L)$.
\end{definition}

\begin{notation}
 When we take $L=\C$, which is always the case for us, we denote 
 $IC(Y, \C)$ by  $IC(Y)$. In this case we denote 
 \[
  \mathcal{H}^{i}(Y):=\mathcal{H}^{i}(IC(Y)).
 \]
\end{notation}

\remk The cohomology sheaves $\mathcal{H}^{i}(Y)$ are locally 
constant over each stratum in $\mathfrak{H}$.

\begin{definition}
Let $n\geq 1$. By a Schubert variety of type $A_{n-1}$, we mean a closed sub-variety of the 
projective variety $GL_{n}/B_{n}$ which is stable under the 
multiplication by $B_{n}$ from the left, where $B_{n}$ is the Borel subgroup consisting of upper triangular 
matrices.
\end{definition}

\remk Let $V$ be a $\C$ vector space. Note that $GL_n$ acts transitively on the set of complete flags 
$\mathcal{F}(V):= \{(U^i: i=0, \cdots, n): 0=U^0\subset U^1\subset \cdots \subset U^n=V, \dim(U^i)=i\}$
and the stabilizer of a complete flag is given by a Borel subgroup. Hence by fixing a complete flag $(V^i: i=0, \cdots, n)$
and denoting its stabilizer by $B$, 
we  identify the variety $GL_n/B$ with $\mathcal{F}(V)$, in this way, 
we can consider the Schubert variety as a subset of $\mathcal{F}(V)$.

\begin{prop}(cf. \cite{C} page 148.)
We identify $S_{n}$ with the set of the permutation matrices in $GL_{n}$. Then
we have the Bruhat decomposition $GL_{n}=\coprod_{w\in S_{n}}B_{n}wB_{n}$.  
Moreover, we have 
\[
 \line{B_{n}wB_{n}}=\coprod_{v\leq w}B_{n}vB_{n}.
\]
\end{prop}

\begin{definition}
 We denote  $C_{w}: =B_{n}wB_{n}/B_{n}$ in $GL_{n}/B_{n}$ and the Schubert variety 
 $X_{w}=\line{C_{w}}$.
\end{definition}

Then for the Schubert variety $X_{w}$, we have a stratification given by 
$\mathfrak{H}=\{C_{v}: v\leq w\}$.

\begin{definition}
Let $v\leq w$, we define the Kazhdan Lusztig polynomial for the pair $v, w$:
\[
 P_{v, w}(q)=\sum_{i}q^{(i+d_{w})/2}\dim \mathcal{H}^i(X_{w})_{x_{v}},
\]
where $x_{v}$ is an element in $C_{v}$ and $d_{w}=\dim(X_{w})=\ell(w)$. 
\end{definition}

Concerning the intersection cohomology of Schubert varieties, we have the following purity theorem 
due to Kazhdan and Lusztig.

\begin{teo}(\cite{KL79})\label{teo: 2.3.9}
If $i+\ell(w)$ is odd, then the cohomology group 
$$\mathcal{H}^i(X_{w})=0.$$
\end{teo}

\remk This implies that  
$P_{v, w}(q)$ is a polynomial in $q$.

\section{Orbital Varieties and Schubert Varieties}
Note that on the orbital variety $\line{O}_{\b}$, we have a stratification
given by $\mathfrak{H}_{\b}=\{O_{\a}: \varphi_{\a}=\varphi_{\b}, \b\leq\a\}$. 

\begin{definition}
Let $\a$, $\b$ be two multisegments such that $\b\in S(\a)$. Then we define 
the polynomial
\[
 P_{\a, \b}(q)=\sum_{i}q^{(i+d_{\b})/2}\dim \mathcal{H}^i(\line{O}_{\b})_{x_{\a}},
\]
where $x_{\a}\in O_{\a}$ is an arbitrary point and $d_{\b}=\dim(O_{\b})$. 
We call it the Kazhdan Lusztig polynomial associated to $\{\a, \b\}$.
\end{definition}

\remk In \cite{Z4} Theorem 1, Zelevinsky showed that the varieties $\line{O}_{\b}$ are locally
isomorphic to some Schubert varieties of type $A_{m}$, where $m=\deg(\b)$. 
Hence again by theorem \ref{teo: 2.3.9}, we know that $P_{\a, \b}$ 
is a polynomial in $q$.

\vspace{0.5cm}
Here, we briefly recall Zelevinsky's results in \cite{Z4}.   
Let $\varphi$ be a function in $\mathcal{S}$(cf. Def. \ref{def: 1.2.5})
such that $\supp(\varphi)\subseteq [1, r]$.
We consider the 
flag variety 
$$
\F(\varphi)=\{0=U^0\subset U^1\subset \cdots U^r=V_{\varphi}: \dim(U^i/U^{i-1})=\varphi(i), 1\leq i\leq r\}
$$

We fix the standard flag 
\[
 F_{\varphi}=\{0=V_{\varphi}^0\subseteq V_{\varphi}^1\cdots \subseteq V_{\varphi}^r: V_{\varphi}^i=
 V_{\varphi, 1}\oplus \cdots \oplus V_{\varphi, i}\}\in \F(\varphi).
\]
\begin{definition}
Let $\G(\varphi)$ be the subset of $\F(\varphi)$ containing the elements $(U^i:0\leq i\leq r)\in \F(\varphi)$
such that $U^i\supseteq V_{\varphi}^{i-1}$ for $i=1, \cdots, r$. 
\end{definition}

Zelevinsky defined a map $\tau: E_{\varphi}\rightarrow \G(\varphi)$, 
by associating to $T\in E_{\varphi}$ the element $\tau(T)=(U^{i}: 0\leq i\leq r)$ such that 
\[
U^i=\{(v_{1},\cdots, v_{r})\in V_{\varphi, 1}\oplus \cdots \oplus V_{\varphi, r}: v_{j+1}=T(v_{j}), j\geq i\}.
\]

\begin{teo}(cf.\cite{Z4}Theorem 1)
 The morphism $\tau$ is an open immersion into the Schubert variety $\mathcal{G}(\varphi)$. 
\end{teo}

In fact, for $\b$ a multisegment of weight $\varphi$, we can describe explicitly the image of $O_{\b}$
in terms of Schubert cells 
in $\G(\varphi)$. Let $\b= \sum_{1\leq i\leq j\leq r}b_{ij}[i,j], X^{\b}=(x_{ij})$ with 
\begin{align*}
x_{ij}=& b_{ij}, \text{ \quad for } i\leq j\\
x_{ij}=&0, \text{ \quad for } i>j+1\\
x_{i,i-1}=&\sum_{n\leq i-1,  i\leq m}b_{nm}.
\end{align*}  

\begin{example}\label{ex: 2.3.13}
Let $\varphi=\chi_{[1]}+2\chi_{[2]}+\chi_{[3]}$,$\a=[1, 2]+[2,3],~ \b=[1,3]+[2]$.
And $ X_{\a}=(x_{ij}^{\a}), ~X_{\b}=(x_{ij}^{\b})$ be the matrix such that

$$
X^{\a}=
\begin{pmatrix}
0 &1&0\\
1&0&1\\
0&1&0
\end{pmatrix},~
X^{\b}=
\begin{pmatrix}
0 &0&1\\
1&1& 0\\
0&1& 0
\end{pmatrix}
$$ 
\end{example}

\begin{definition}
Let $\b$ be a multisegment of weight $\varphi$ and $$X^{\b}=(x_{i,j})_{1\leq i, j\leq r}.$$
We define $Y_{\b}$ to be the set of flags $$(U^i: i=0,1,\cdots, r)\in\G(\varphi)$$ such that 
\[
\dim((U^i\cap V_{\varphi}^j)/(U^i\cap V_{\varphi}^{j-1}+U^{i-1}\cap V_{\varphi}^j))=x_{ij},
 \text{ for all } 1\leq i,j\leq r.
\]
\end{definition}

\begin{example}
Let $\a$ be the multisegment in example \ref{ex: 2.3.13}. We have 
$Y_{\a}$ be the set of flags $(U^i: i=0, 1, 2, 3)$ such that 
\begin{align*}
 U^0&=0;\\
 \dim(U^1\cap V^1_{\varphi})&=x_{11}^{\a}=0\Rightarrow U^1\cap V^1_{\varphi}=0;\\
 \dim(U^1\cap V^2_{\varphi})&=x_{12}^{\a}=1\Rightarrow U^1\subseteq V^2_{\varphi};\\
 \dim(U^2\cap V^1_{\varphi})&=x_{21}^{\a}=1\Rightarrow U^2\supseteq  V^1_{\varphi}.
\end{align*}
And 
\[
  \dim(U^2\cap V^2_{\varphi}/(U^2\cap V^1_{\varphi}+U^1\cap V^2_{\varphi}))=x_{22}^{\a}=0,
\]
which implies
\[
 U^2\cap V^2_{\varphi}=U^2\cap V^1_{\varphi}+U^1\cap V^2_{\varphi};
\]
hence $U^2\cap V_{\varphi}^2=V_{\varphi}^1+U^1$, which 
is of dimension 2. 
 Hence $Y_{\a}$ is the set of flags $(U^i: i=0, 1, 2, 3)$ satisfying 
 \[
  U^0=0,~ U^1\cap V^1_{\varphi}=0,~ U^2\cap V_{\varphi}^2=V_{\varphi}^1+U^1,~ U^3=V_{\varphi}^3.
 \]
\end{example}

\begin{prop}(cf. \cite{Z4} Theorem 1.)
We have $O_{\b}=Y_{\b}\cap E_{\varphi}$. 
\end{prop}

\begin{example}
 Again, let $\a=[1, 2]+[2,3]$. 
 Let $T\in E_{\varphi}\cap Y_{\a}$, then 
 we have $\tau(T)=(U^i: i=0, 1, 2, 3)$, satisfying 
 \[
  U^0=0,~ U^1\cap V^1_{\varphi}=0,~ U^2\cap V_{\varphi}^2=V_{\varphi}^1+U^1,~ U^3=V_{\varphi}^3.
 \] 
 By definition,  if we write $T=(T_1, T_2)$ such that  
 \[
 T_i: V_{\varphi, i}\rightarrow V_{\varphi, i+1}, \quad i=1,2,
 \]
 then 
 $$U^1=\{(v, T_1v, T_2T_1v)\in V_{\varphi}: v\in V_{\varphi, 1}\},$$
 and $U^1\cap V^1_{\varphi}=0$ is equivalent to $T_1v\neq 0$. 
 Also,  we have 
 $$U^2=\{(v_1, v_2, T_2v_2)\in V_{\varphi}: v_1\in V_{\varphi, 1}, v_2\in V_{\varphi, 2}\}.$$ 
 Note that $U^2\cap V_{\varphi}^2=V_{\varphi}^1+U^1$  is equivalent to the following conditions
 \[
 U^1\subseteq V_{\varphi}^2,\quad  U^2\nsupseteq V_{\varphi, 2}. 
  \]
We know that $U^1\subseteq V_{\varphi}^2$ is equivalent to the fact that 
 for any $v\in V_{\varphi, 1}$, $(v, T_1v, T_2T_1v)\in V_{\varphi, 2}$, hence 
 $ T_2T_1v=0$. Furthermore,  we know that $U^2\nsupseteq V_{\varphi, 2}$ is equivalent to the fact that there exits 
 $v\in V_{\varphi, 2}$ such that $(0, v, T_2v)\notin V_{\varphi, 2}$, hence 
 $T_2v\neq 0$. 
 Hence we obtain that $T\in E_{\varphi}\cap Y_{\a}$ is equivalent to the following facts
 \[
 T_1\neq 0, \quad T_2T_1=0, \quad T_2\neq 0.
 \]
 The latter  is the same as to say that $T\in O_{\a}$. 
\end{example}

\begin{definition}
 Let $B_{i}(\varphi)= \{j: \sum_{m\leq i-1}\varphi(m)< j\leq \sum_{m\leq i}\varphi(m)\}$.
$$
S^{\b}=\{w\in S_{\deg(\b)}: Card(w(B_{i}(\varphi))\cap B_{j}(\varphi))=x_{ij}, 1\leq i,j\leq r\}.
$$
We denote by $w(\b)$ the unique element in $S^{\b}$ of maximal length.
\end{definition}

\begin{example}
In the example \ref{ex: 2.3.13}, we have 
\[
 B_{1}(\varphi)=\{1\}, B_{2}(\varphi)=\{2,3\}, B_{3}(\varphi)=\{4\}.
\]  
Let $\a=[1, 2]+[2,3]$. Then by definition
\[
 S^{\a}=\{w\in S_4: Card(w(B_i(\varphi))\cap B_j(\varphi))=x_{ij}^{\a}, 1\leq i, j\leq 3\}
\]
Therefore, for $w\in S^{\a}$, we have 
\[
 w(1), w(4)\in\{2, 3\}, ~\{w(2), w(3)\}\cap \{2, 3\}=\emptyset, ~
\]
therefore, 
\[
 \{w(1), w(4)\}=\{2, 3\}, ~\{w(2), w(3)\}=\{1, 4\},
\]
hence 
\[
 w=(13)(24), \text{ or } w=(12)(34),
\]
compare the length, we have 
\[
 w(\a)=(13)(24).
\]
The same method shows that
\[
 w(\b)=(1423)
\]

\begin{figure}[!ht]
\centering
\includegraphics{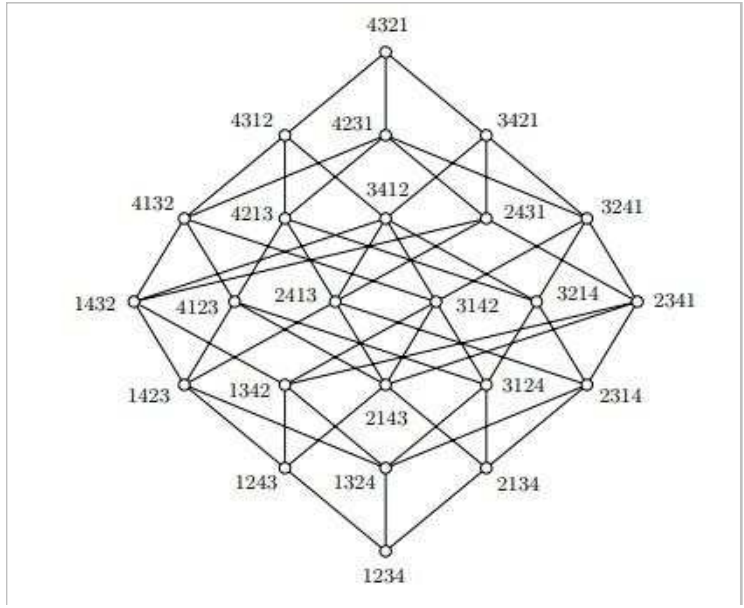}
\caption{\label{fig-multisegment12} Bruhat Order for $S_{4}$}
\end{figure} 
Note in the picture we denote a permutation
by its image.

\end{example}

\begin{teo}\label{cor: 6}(cf. \cite{Z4})
Let $\b'\geq \b$ such that $Y_{\b'}\subseteq \line{Y}_{\b}$, 
 we have
 \[
 P_{\b', \b}(q)=P_{w(\b'), w(\b)}(q).
 \]

\end{teo}

\begin{teo}(\cite{Z3}, \cite{CG})\label{teo: 4.1.5}
 Let $\mathcal{H}^{i}(\line{O}_{\b})_{\a}$ denote the stalk at a point
 $x\in O_{\a}$ of the $i$-th intersection cohomology sheaf of the variety $\line{O}_{\b}$.
 Then 
 \[
  m(\b,\a)=P_{\b, \a}(1).
 \]
 \end{teo}
\remk The intersection cohomology is nonzero only if $i+\dim(O_{\b})$ is even.

Hence $m(\b,\a)$ is the value
at $v=1$ of a certain Kazhdan Lusztig polynomial for the symmetric group $S_{m}$ with
$m=\deg (\b)$.

\remk 
Combining with theorem \ref{cor: 6}, this theorem gives a complete
calculation of the coefficients $m(\b, \a)$. But as we have seen,
this often involves elements in a huge symmetric group, which
is too clumsy. Moreover, 
another difficulty arise from the description
of the element $w(\b)$, which is not explicit.

\remk
In this chapter, 
for symmetric multisegments $\a$ and $\b$, 
we will give more concrete description about the coefficient $m_{\b, \a}$
in terms of elements in $S_{n}$ with $n$ equals to the number of segments
contained in $\a$, cf. corollary \ref{cor: 2.5.9}.
For general case, we will give use the reduction method from
chapter 4 to give a more concrete description.

\section{Geometry of Symmetric Nilpotent Orbits} 

For the moment, we consider a special case of symmetric multisegments, 
we assume that
\[                       
\a_{\Id}=\sum_{i=1}^{n}[i, n+i-1],\quad \varphi=\sum_{\Delta\in \a}f_{\a}(\Delta)\chi_{\Delta}.                      
\]
We remind that we already constructed a bijection
\[
 \Phi: S_{n}\rightarrow S(\a_{\Id})
\]
such that $\Phi(\Id)=\a_{\Id}$.

\begin{definition}
 Let 
 $$O_{w}=O_{\Phi(w)},\quad \hbox{ and }\quad  O_{\varphi}^{\sym}=\coprod_{w\in S_{n}} O_{w}\subseteq E_{\varphi}.$$
 Also, let $$\line{O}_{w}^{sym}=\line{O}_{w}\cap O_{\varphi}^{\sym}.$$
\end{definition}

\begin{definition}
 Let 
 \[
 \xymatrix
 { 
 E_{\varphi}\ar[d]^{ p_{\varphi}}&\hspace{-1cm} =M_{2,1}\times \cdots M_{n-1, n-2}\times M_{n,n-1}\times M_{n-1,n}
 \times \cdots \times M_{1,2}\\
 Z_{\varphi}&\hspace{-2.8cm}: =M_{2,1}\times \cdots M_{n-1, n-2}\times M_{n-1,n}\times \cdots \times M_{1,2}.
 }
 \]
\end{definition}
be the natural projection
with fiber $M_{n,n-1}$.

Now we want to describe the fiber of the 
restriction $p_{\varphi}|_{O_{\varphi}^{\sym}}$.
\begin{definition}
 We define $GL_{n,n-1}$ to be the subset of $M_{n, n-1}$ consisting
of the matrices of rank $n-1$.

\end{definition}

We denote by $p_{n}: M_{n, n}\twoheadrightarrow M_{n, n-1}$ the morphism
of forgetting the last column of elements in $M_{n, n}$.

\remk
Now by restriction to $GL_{n}$, we have the morphism
\[
 p_{n}: GL_{n}\twoheadrightarrow GL_{n,n-1},
\]
which satisfies the property that $p_n(g_1g_2)=g_1p_n(g_2)$ for $g_1, g_2\in GL_n$.

\begin{prop}\label{prop: 4.3.6}
The morphism 
\[
  p_{n}: GL_{n}\twoheadrightarrow GL_{n,n-1},
\]
is a fibration. 
Furthermore, it induces a bijection 
\[
p_{n}: B_{n}\backslash GL_{n}/B_{n}\rightarrow B_{n}\backslash GL_{n,n-1}/B_{n-1}.
\]
\end{prop}

\begin{proof}
To see that it is locally trivial, 
note that $p_n$ is $GL_n$ equivariant with $GL_n$ acting by multiplication
from the left. Since $GL_n$ acts transitively on itself, it acts also 
transitively on $GL_{n, n-1}$. Now $p_n$ is equivariant implies that 
all the fibers of $p_n$ are isomorphic. 
Let $H$ be the stabilizer of $p_n(\Id)$, then 
$GL_{n, n-1}\simeq GL_n/H$, it is a \'etale locally trivial 
fibration according to Serre \cite{S} proposition
 3.
 By Bruhat decomposition, every $g\in GL_{n}$ admits a decomposition
\[
 g=b_{1}wb_{2}, \quad b_{i}\in B_{n}, i=1,2, \quad w\in S_{n},
\]
here we identify $S_n$ with the set of permutation matrices in $GL_n$.
We can decompose $b_{2}=b_{3}v$, where $b_{3}\in GL_{n-1}$, which is identified 
with the direct summand in the Levi subgroup $GL_{n-1}\times \C^{\times}$,
and $v-\Id$ only contains non zero elements
in the last column, 
by definition, 
\[
 p_{n}(g)=b_{1}p_{n}(w)b_{3}.
\]
We obtain that $p_n$ induces 
\[
p_{n}: B_{n}\backslash GL_{n}/B_{n}\rightarrow B_{n}\backslash GL_{n,n-1}/B_{n-1}.
\]

It is bijective
because given $p_{n}(w)$, there is a unique way to reconstruct an
element which belongs to $S_{n}$.
\end{proof}

\begin{teo}\label{teo: 4.3.7}
The morphism
$$p_{\varphi}|_{O_{\varphi}^{\sym}}$$
is smooth with fiber $GL_{n,n-1}$. 
Moreover, the morphism $p_{\varphi}|_{O_{w}}: O_{w}\rightarrow p_{\varphi}(O_{\varphi}^{\sym})$ 
is surjective with fiber $B_{n}p_{n}(w)B_{n-1}$.
\end{teo}

\begin{proof}
Note that smoothness follows from that $p_{\varphi}: E_{\varphi}\rightarrow Z_{\varphi}$ 
is smooth and that $O_{\varphi}^{\sym}$ is open in $E_{\varphi}$. 
To see the rest of the properties, we fix an element $e_{w}$ in 
each orbit $O_{w}$ as follow.
Let $(v_{ij})(i=1, \cdots, 2n-1, j=1, \cdots, \varphi(i))$ be a basis of $V_{\varphi, i}$, 
and an element $e_{w}$
satisfying
\begin{displaymath}
    \left\{ \begin{array}{cc}
            e_{w}(v_{ij})&\hspace{-0.5cm}=v_{i+1,j}, \quad\text{ for } i<n-1\\
            e_{w}(v_{n-1, j})&\hspace{-3 cm}=v_{n, w(j)}, \\
            e_{w}(v_{ij})&\hspace{-0.5 cm}=v_{i+1,j-1}, \quad\text{ for } i\geq n.
           \end{array}\right.,
\end{displaymath}
here we put $v_{i,0}=0$. 
\begin{example} 
Let $w=(1,2)$, then by the strategy in the proof, 
$e_{w}$ is given by the following picture:
\begin{figure}[!ht]
\centering
\includegraphics{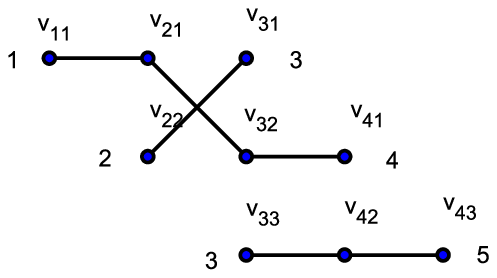}
\caption{\label{fig-multisegment13} }{Construction of $e_{w}$ in case $n=3$}
\end{figure} 
\end{example}
\vspace{0.5cm}
We claim that $e_{w}\in O_{w}$. 
In fact, it suffices to observe that 
\[
 e_{w}: v_{ii}\rightarrow \cdots\rightarrow v_{n-1,i}\rightarrow v_{n, w(i)}\rightarrow v_{n+1, w(i)-1}\rightarrow \cdots 
 v_{n+w(i)-1, 1},
\]
which by proposition
\ref{prop: 2.2.4}, implies that the multisegment indexing $e_{w}$ contains $[i,w(i)+n-1]$
for all $i=1, \cdots, n$, 
hence it must be $\Phi(w)$. 
Note that, by definition, we have 
\[
 p_{\varphi}(e_{\Id})=p_{\varphi}(e_{w}), ~\text{for all  }w\in S_{n}.
\]
Since $p_{\varphi}$ is compatible with the action of $G_{\varphi}$, 
we get 
\[
 p_{\varphi}(O_{\varphi}^{\sym})=p_{\varphi}(O_{w}),~ \text{for all }w\in S_{n},
\]
which implies that $p|_{O_{w}}$ is surjective. Now it remains to characterize
its fiber.
Let $T'\in p_{\varphi}(O_{\varphi}^{\sym})$,  then $p_{\varphi}^{-1}(T')\simeq M_{n, n-1}$
in $E_{\varphi}$. Moreover, for $T=(T_1, \cdots, T_{2n-2})\in p_{\varphi}^{-1}(T')$, then $T\in O_{\varphi}^{\sym}$ 
if and only if 
\[
T_{n-1}\in GL_{n, n-1}.
\]
Therefore, the map $T\mapsto T_{n-1}$ induces 
\[
p_{\varphi}^{-1}(T')\cap O_{\varphi}^{\sym}\simeq GL_{n, n-1}.
\]

Consider the variety $p_{\varphi}^{-1}(T')\cap O_w$. Note that 
since $G_{\varphi}$ acts transitively on  $p_{\varphi}(O_{\varphi}^{\sym})$, we may assume that $T'=p_{\varphi}(e_{\Id})$. 

\begin{lemma}\label{lem: 4.3.8}
The set of $f_{w}\in O_{w}$ satisfying 
\begin{displaymath}
    \left\{ \begin{array}{cc}
            f_{w}(v_{ij})&=v_{i+1,j}, \quad \text{ for } i<n-1\\
            f_{w}(v_{ij})&\hspace{-0.5cm}=v_{i+1,j-1}, \quad \text{ for } i\geq n.
           \end{array}\right.
\end{displaymath}
is in bijection with $B_{n}p_{n}(w)B_{n-1}$ via $p_{\varphi}^{-1}(p_{\varphi}(e_{\Id}))\cap O_{\varphi}^{\sym}\simeq GL_{n, n-1}$.
\end{lemma}

\begin{proof}
%

Now the element $f_{w}\in O_{w}$ is completely determined by 
the component 
\[
 f_{w, n-1}: V_{\varphi, n-1}\rightarrow V_{\varphi, n}.
\]
We know by proposition \ref{prop: 2.2.4} that 
$f_{w, n-1}$ is injective hence of rank $n-1$.
Hence we have 
$f_{w, n-1}\in GL_{n, n-1}$.

Now by proposition \ref{prop: 4.3.6} we get $B_{n}\backslash GL_{n,n-1}/B_{n-1}$
is indexed by $S_{n}$, it remains to see that $f_{w, n-1}$ is in the class indexed by
$p_{n}(w)$. 

Finally, we note that $p_{\varphi}$ is a morphism equivariant 
under the action of 
\[
 G_{\varphi}=GL_{1}\times GL_{2}\times \cdots \times GL_{n-1}\times GL_{n}\times \cdots \times GL_{2}\times GL_{1}.
\]
Since $G_{\varphi}$ acts transitively on $O_{w}$,  the image of $O_{w}$ is $G_{\varphi}. (p_{\varphi}(e_{w}))$, hence is 
$p_{\varphi}(O_{\Id})$. Now we prove that the stabilizer of 
$p_{\varphi}(e_{w})$ is $B_{n}\times B_{n-1}$.
Let $e_{\Id}=(e_{1}, \cdots, e_{ n-1}, e_{ n}, \cdots, e_{2n-2})$
with $e_{i}\in M_{ i, i+1}$ if $i<n$ and $e_i\in M_{i, i-1}$ if $i\geq n$. We have 
$$p_{\varphi}(e_{\Id})=(e_{1}, \cdots, e_{ n-2}, e_{n}, \cdots, e_{2n-2}).$$

Let $g=(g_{1}, \cdots, g_n, g_{n+1}, \cdots, g_{2n-1})$ such that 
$g.p_{\varphi}(e_{\Id})=p_{\varphi}(e_{\Id})$. Then by definition for $i<n-1$
we know that 
$g_{i+1}e_ig_i^{-1}=e_i$. We prove by induction on $i$ that 
$g_i\in B_i\in GL_i$ for $i\leq n-1$. For $i=1$, we have nothing to prove.
Now assume that $i\leq n-2$, and $g_i\in B_i$, we show that $g_{i+1}\in B_{i+1}$. 
Consider $$g_{i+1}e_ig_i^{-1}(g_i(v_{ij}))=g_{i+1}e_i(v_{ij})=g_{i+1}(v_{i+1, j}).$$
On the other hand,
by induction, we know that 
$$g_{i+1}e_ig_i^{-1}(g_i(v_{ij}))=e_i(g_i(v_{ij}))\in \oplus_{k\leq j}\C v_{i+1, k}.$$ 
Therefore we have $g_{i+1}\in B_{i+1}$. Actually, since $e_i$ is injective, 
the equality $e_i(g_i(v_{ij}))=g_{i+1}(v_{i+1, j})$, implies that $g_i$ is 
completely determined by $g_{i+1}$.
This shows that $g_{n-1}\in B_{n-1}$ it determines all $g_i$ for $i\geq n-1$. The same method proves that 
$g_{n}\in B_n$ and it determines all $g_i$ for $i\geq n$. 
We conclude that the fiber of the morphism
$p_{\varphi}|_{O_{w}}$ is isomorphic to 
$B_{n}p_{n}(w)B_{n-1}$.
\end{proof}
\end{proof}

\begin{cor}
We have for $v\leq w$ in $S_{n}$, and $X_{w}$ the closure of $B_{n}wB_{n}$ in $GL_{n}$,
\[
\dim \mathcal{H}^{i}(\line{O}_{w}^{sym})_{v}=\dim\mathcal{H}^{i}(X_{w})_{v},
\]
for all $i\in \Z$, here the index $v$ on the left hand side means that we localize at a generic point in $O_{v}$ and 
on the right hand side means that we localize at a generic point in $C_v$.
\end{cor}

\begin{proof}
Since $p_{\varphi}|_{O_{\varphi}^{\sym}}$ is a fibration with fiber $GL_{n,n-1}$ 
over $Z_{\varphi}$, we apply the smooth base change theorem to the following Cartesian diagram
\begin{displaymath}
\xymatrix 
{
GL_{n,n-1}\ar[r]\ar[d]&O_{\varphi}^{\sym}\ar[d]\\
p_{\varphi}(\Phi(\Id))\ar[r]&Z_{\varphi}.
}
\end{displaymath}
We get
$$
\dim \mathcal{H}^{i}(\line{O}_{w}^{sym})_{v}=\dim \mathcal{H}^{i}(\line{B_np_n(w)B_{n-1}})_{B_np_n(v)B_{n-1}}.
$$ 
Now apply proposition \ref{prop: 4.3.6}, we have 
\[
\dim \mathcal{H}^{i}(\line{B_np_n(w)B_{n-1}})_{B_np_n(v)B_{n-1}}=dim\mathcal{H}^{i}(X_{w})_{v}.
\]
\end{proof}

\begin{cor}\label{cor: 2.5.9}
 We have for $v\leq w$ in $S_{n}$,
 \[
  m_{\Phi(v), \Phi(w)}=P_{v, w}(1).
 \]
 \end{cor}
 
\begin{proof}
 This follows from the fact that 
 \[
  \dim \sum_i\mathcal{H}^{i}(X(w))_{v}=P_{v,w}(1)
 \]
(cf. \cite{KL}).
\end{proof}

\chapter{Descent of Degrees for Multisegment}

To attack the question of calculating 
the coefficient $m(\b, \a)$, this first naive idea, which can trace back to Zelevinsky \cite{Z2}, 
is to use the (partial) derivation.
If we  believe that for $\b\in S(\a)$, the coefficient $m(\b, \a)$  
only depends on the relative position  between the segments in $\a$, but not 
on the exact multisegment $\a$,
we should be allowed to do some sort of truncation on the multisegments
simultaneously
without changing the coefficient $m(\b, \a)$. It is reasonable to 
think that the partial derivative should
play the role of truncation. 

However, it is not true that we can always truncate. For example
if we take $\a=\{[1, 2], [2, 3]\}$ and we replace the segment $[2, 3]$ by $[2]$( truncate at the place 3), 
then we get $\a'=\{[1, 2], [2]\}$, this should not allowed because we changed the linkedness
relation between the two segments.
And simple calculation shows that 
\[
 \D^{3}(L_{\a})=L_{\a},
\]
we quickly notice that $\D^{3}(L_{\a})$ does not achieve its minimal degree
term $L_{\a'}$, which are supposed to appear. 

Such examples lead us to think that we can do truncation only when 
our partial derivative achieve its minimal degree terms.
More explicitly, we should avoid applying truncation to the
multisegments as $\a$ above. This gives us the hypothesis $H_{k}(\a)$(definition \ref{def: 3.1.3}).
And satisfying the hypothesis $H_{k}(\a)$ means that we can apply the truncation 
without changing the coefficients.

\section{Morphism for Descent of Degree of multisegment}

For a multisegment $\a$ and $k\in \Z$,  we will introduce a hypothesis called $H_{k}(\a)$ and 
let $S(\a)_{k}$ be the set of elements in $S(\a)$ satisfying the hypothesis $H_k(\a)$. We construct a multisegment $\a^{(k)}$ and 
a morphism $\psi_{k}: S(\a)_{k}\rightarrow S(\a^{(k)})$. We show that the morphism $\psi_{k}$
is surjective.

\begin{notation}\label{nota: 3.1.2}
 For $\Delta=[i,j]$ a segment, we put 
 \begin{align*}
 \Delta^{-}=&[i,j-1], \quad {^{-}\Delta}=[i+1,j],\\
 \Delta^{+}=&[i,j+1], \quad {^{+}\Delta}=[i-1,j].
 \end{align*}
\end{notation}

\begin{definition}\label{def: 2.2.3}
Let $k\in \Z$ and $\Delta$ be a segment, we define 
\begin{displaymath}
 \Delta^{(k)}=\left\{ \begin{array}{cc}
         &\Delta^-, \text{ if } e(\Delta)=k;\\
         &\hspace{-0.5cm}\Delta, \text{ otherwise }.
         \end{array}\right.
\end{displaymath}
For a multisegment 
 $
 \a=\{\Delta_{1}, \cdots, \Delta_{r}\},
 $
we define
  \[
  \a^{(k)}=\{\Delta_{1}^{(k)}, \cdots,\Delta_{r}^{(k)} \}.
 \]
\end{definition}

\begin{definition}\label{def: 3.1.3}
We say that the multisegment $\b\in S(\a)$ satisfies\textbf{ the hypothesis $H_{k}(\a)$} if
the following two conditions are verified 
\begin{description}
 \item[(1)]$\deg(\b^{(k)})=\deg(\a^{(k)})$;
 \item[(2)]there exists not a pair of linked segments $\{\Delta, \Delta'\}$ such that $e(\Delta)=k-1, ~e(\Delta')=k$.
\end{description}
\end{definition}

\begin{definition}
 Let 
\[
 \tilde{S}(\a)_{k}=\{\c\in S(\a): \deg(\c^{(k)})=\deg(\a^{(k)})\}.
\]
\end{definition}

\begin{lemma}\label{lem: 3.1.5}
Let $\c\in \tilde{S}(\a)_{k}$. Then  
\[
 \sharp\{\Delta\in \a: e(\Delta)=k\}=\sharp\{\Delta\in \c: e(\Delta)=k\}.
\] 
\end{lemma}

\begin{proof}
Note that 
\[
 \deg(\a)=\deg(\a^{(k)})+\sharp\{\Delta\in \a: e(\Delta)=k\}.
\]
Now that for $\c\in \tilde{S}(\a)_{k}$
\[
 \deg(\c)=\deg(\a),\quad \deg(\c^{(k)})=\deg(\a^{(k)}),
\]
we have 
\[
 \sharp\{\Delta\in \a: e(\Delta)=k\}=\sharp\{\Delta\in \c: e(\Delta)=k\}.
\]
\end{proof}

 \begin{lemma}\label{lem: 3.0.7}
 Let $k\in \Z$
 \begin{description}
  \item[(a)] For any $\b\in S(\a)$, we have $\deg(\b^{(k)})\geq \deg(\a^{(k)})$.
 \item[(b)]Let $\c\in \tilde{S}( \a)_{k}$, then for $\b\in S(\a)$ such that 
 $\b>\c$, we have $\b\in \tilde{S}(\a)_{k}$.
 \item[(c)]Let $\b\in \tilde{S}(\a)_k$, then $\b^{(k)}\in S(\a^{(k)})$. Moreover, if we suppose that $\a$ satisfies the hypothesis $H_{k}(\a)$
  and $\b\neq \a$, then 
  $$\b^{(k)}\in S(\a^{(k)})-\{\a^{(k)}\}$$.
  
 \item[(d)]Suppose that $\a$ does not verify the hypothesis $H_{k}(\a)$, then 
 there exists a $\b\in S(\a)$ satisfying the hypothesis $H_{k}(\a)$, such that
 $\b^{(k)}= \a^{(k)}$.
 \end{description}
 \end{lemma}

\begin{proof}
For (a), by lemma \ref{lem: 2.1.3} 
for any $\b\in S(\a)$, $e(\b)$ is a sub-multisegment
of $e(\a)$. And from $\b$ to $\b^{(k)}$, 
we replace those segments $\Delta$ such that $e(\Delta)=k$
by $\Delta^{-}$. Now (a) follows by counting the segments ending 
in $k$.

For (b), by (a), we have 
\[
 \deg (\a^{(k)})\leq \deg(\b^{(k)})\leq \deg(\c^{(k)}).
\]
The fact that $\c\in \tilde{S}( \a)_{k}$ implies that 
$ \deg (\a^{(k)})= \deg (\c^{(k)})$, hence 
$ \deg (\a^{(k)})= \deg (\b^{(k)})$ and $\b\in \tilde{S}( \a)_{k}$.

As for (c), 
suppose that 
$\deg(\b^{(k)})=\deg(\a^{(k)})$, we prove 
$\b^{(k)}<\a^{(k)}$.
Let
\[
 \a=\a_{0}>\cdots>\a_{r}=\b
\]
be a maximal chain of multisegments, then 
by $(b)$, we know 
$\deg(\a_{j}^{(k)})=\deg(\a^{(k)})$, for all $ j=1,\cdots, r$.
Our proof breaks into two parts.

\vspace{0.5cm}
(1)We show that 
\[
 \deg(\a_{j}^{(k)})=\deg(\a_{j+1}^{(k)})\Rightarrow \a_{j}^{(k)}\geq \a_{j+1}^{(k)}.
\]

Let $\a_{j+1}$ be obtained from $\a_{j}$
by applying the elementary operation to two linked segments $\Delta, \Delta'$.
\begin{itemize}
\item 
If none of them ends in $k$, then $\a_{j}^{(k)}$ contains
both of them. And we obtain $\a_{j+1}^{(k)}$ by applying 
the elementary operation to them.

If one of them ends in $k$, 
we assume $e(\Delta')=k$.

\item If 
$\Delta$ precedes $\Delta'$,  we know that if $e(\Delta)<k-1$,
$\Delta$ is still linked to $\Delta'^-$, and one obtains $\a_{j+1}^{(k)}$
by applying elementary operation to $\{\Delta, ~\Delta'^-\}$,
 otherwise $e(\Delta)=k-1$, which implies 
 $\a_{j+1}^{(k)}=\a_{j}^{(k)}$.

\item If $\Delta$ is preceded by $\Delta'$, 
then the fact that 
\[
 \deg(\a_{j+1}^{(k)})=\deg(\a_{j}^{(k)})
\]
implies $b(\Delta)\leq k$, hence
$\Delta'^{-}$ is linked to $\Delta$, and we obtain
$\a_{j+1}^{(k)}$ from $\a_{j}^{(k)}$ by applying elementary
operation to them.
\end{itemize}
Here we conclude that $\b^{(k)}\in S(\a^{(k)})$.
\vspace{0.5cm}

(2)Assuming that $\a$ satisfies the hypothesis $H_k(\a)$, we show that 
\[
 \a_{1}^{(k)}<\a^{(k)}.
\]
Let $\a_{1}$ be obtained from $\a$
by performing the elementary operation to $\Delta, \Delta'$.

We do it as in (1) but put $j=0$. Note that in (1), the only case where we can have 
$\a_{1}^{(k)}=\a^{(k)}$ is when 
$\Delta$ precedes $\Delta'$  and $ e(\Delta')=k, ~e(\Delta)=k-1$.
But such a case can not exist  since $\a$ verifies the hypothesis $H_{k}(\a)$.
Hence we are done.

Finally, for (d), we construct $\b$ in the following 
way. Suppose that $\a$ does not satisfy the 
hypothesis $H_{k}(\a)$, then there exists a pair of 
linked segments $\{\Delta, \Delta'\}$ such that 
\[
 e(\Delta)=k-1, \quad e(\Delta')=k,
\]
let $\a_{1}$ be the multisegment obtained by applying the
elementary operation to $\Delta$ and $\Delta'$. We have
\[
 \a_{1}^{(k)}=\a^{(k)}.
\]
If again $\a_{1}$ fails the hypothesis $H_{k}(\a)$, we repeat the same
construction to get $\a_{2}, \cdots$, since
\[
 \a>\a_{1}>\cdots. 
\]
In finite step, we get $\b$ satisfying the conditions in the theorem and 
\[
 \b^{(k)}=\a^{(k)}.
\]
\end{proof}
\remk Actually, the multisegment constructed in $(d)$ is unique, as we shall see 
later(proposition \ref{cor: 3.2.3}).

\begin{definition}
We define a morphism 
\[
 \psi_{k}:\tilde{S}( \a)_{k}\rightarrow S(\a^{(k)})
\]
by sending $\c$ to $\c^{(k)}$. 
\end{definition}

\begin{prop}\label{prop: 1.5.8}
The morphism $\psi_{k}$ is surjective. 
 
\end{prop}

\begin{proof}
Let $\d\in S(\a^{(k)})$, such that 
we have a maximal chain of multisegments, 
\[
 \a^{(k)}=\d_{0}>\cdots >\d_{r}=\d.
\]
By induction, we can assume that 
there exists $\c_{i}\in \tilde{S}(\a)_{k}$
such that $\c_{i}^{(k)}=\d_{i}$, for all 
$i<r$. Assume we obtain $\d$ from $\d_{r-1}$ by performing the 
elementary operation on the pair of linked segments $\{\Delta\prec \Delta'\}$.

\begin{itemize}
 \item If $e(\Delta)\neq k-1$ and $e(\Delta')\neq k-1$, then 
 we observe that the pair of segments are actually contained in $\c_{r-1}$.
 Let $\c_{r}$ be the multisegment obtained by performing the elementary operation 
 to them . We conclude that $\c_{r}^{(k)}=\d_{r}$, and $\c\in \tilde{S}(\a)_{k}$.
 \item If $e(\Delta)=k-1$, then $\Delta\in \c_{r-1}$ or $\Delta^+\in \c_{r-1}$ and $\Delta'\in \c_{r-1}$.
 The fact that $\d_{r-1}=\c_{r-1}^{(k)}$
implies that $k\notin e(\d_{r-1})$, hence $e(\Delta')>k$. Hence 
both $\Delta$ and $\Delta^{+}$ are linked to $\Delta'$. In either case we perform 
the elementary operation to get $\c_{r}$ such that $\c_{r}^{(k)}=\d$. 
\item If $e(\Delta')=k-1$, then $\Delta'\in \c_{r-1}$ or $\Delta'^+\in \c_{r-1}$ and $\Delta\in \c_{r-1}$.
The same argument as in the second case shows that there exists $\c_{r}$ such that 
$\c_{r}^{(k)}=\d$.
\end{itemize}
 
\end{proof}

Actually, the proof in proposition \ref{prop: 1.5.8} 
yields the following refinement.
\begin{cor}\label{cor: 1.5.9}
Let $\c\in \tilde{S}(\a)_{k}, \d\in S(\a^{(k)})$ such that 
\[
 \c^{(k)}>\d,
\]
then there exists a multisegment $\e\in \tilde{S}(\a)_{k}$ such that
\[
 \c>\e, ~ \e^{(k)}=\d.
\]
\end{cor}

\begin{proof}
Note that $\c\in \tilde{S}(\a)_k$ implies $\tilde{S}(\a)_k\supseteq \tilde{S}(\c)_k$.
Combine with the surjectivity of 
\[
\psi_k: \tilde{S}(\c)_k\rightarrow S(\c^{(k)}),
\]
we get the result.
\end{proof}

\begin{definition}\label{def: 2.3.2}
For $\a$ a multisegment, and $k\in \Z$
we define 
\[
 S(\a)_{k}=\{\c\in \tilde{S}(\a)_{k}:
 \c \text{ satisfies the hypothesis } H_{k}(\a) \}.
\]
\end{definition}

%
%
%

%

\begin{prop}\label{prop: 1.5.10}
 The restriction 
 \begin{align*}
  \psi_{k}: S(\a)_{k}&\rightarrow S(\a^{(k)})\\
  \c&\mapsto \c^{(k)} 
 \end{align*}
is also surjective.
\end{prop}

\begin{proof}
 For $\d\in S(\a^{(k)})$, by proposition \ref{prop: 1.5.8},
 we know that there exists $\c\in \tilde{S}(\a)_{k}$
 such that $\c^{(k)}=\d$. 
 But by $(d)$ in lemma \ref{lem: 3.0.7}, we know that 
 there exists $\c'\in S(\c)_{k}$ such that 
 $\c'^{(k)}=\c^{(k)}=\d$. 
 We conclude by the observation that if $\c\in \tilde{S}(\a)_{k}$, then 
 \[
  S(\c)_{k}\subseteq S(\a)_{k}.
 \] 
\end{proof}

Also, concerning the corollary \ref{cor: 1.5.9}, we have the 
following 

\begin{cor}\label{cor: 3.1.11}
Let $\c\in \tilde{S}(\a)_{k}$ and $\d\in S(\a^{(k)})$ such that
$\c^{(k)}>\d$. Then there exists a multisegment $\e\in S(\c)_{k}$
such that $\e^{(k)}=\d$.
\end{cor}

\begin{proof}
By corollary \ref{cor: 1.5.9}, we know that there exists 
an $\e'\in \tilde{S}(\c)_{k}$ such that $\e'^{(k)}=\d$.
By $(d)$ in lemma \ref{lem: 3.0.7}, we know that 
 there exists $\e\in S(\e')_{k}$ such that 
 $\e^{(k)}=\e'^{(k)}=\d$. 
 Hence we conclude by 
 the fact that
 if $\e'\in \tilde{S}(\a)_{k}$, then 
 \[
  S(\e')_{k}\subseteq S(\a)_{k}.
 \]

\end{proof}

\begin{definition}
Let $k\in \Z$ and $\Delta$ be a segment.
\begin{displaymath}
 {^{(k)}\Delta}=\left\{ \begin{array}{cc}
         &{^{-}\Delta}, \text{ if } b(\Delta)=k;\\
         &\Delta, \text{ otherwise }.
         \end{array}\right.
\end{displaymath}

Let
$$
\a=\{\Delta_{1}, \cdots, \Delta_{r}\},
$$
be a multisegment, we define 
  \[
  {^{(k)}\a}=\{{^{(k)}\Delta_{1}}, \cdots,{^{(k)}\Delta_{r}}, \}.
 \]
\end{definition}

\begin{definition}  We say that the multisegment $\b\in S(\a)$ satisfies \textbf{ the hypothesis ${_{k}H(\a)}$} if
 the following two conditions are verified 
\begin{description}
 \item[(1)]$\deg({^{(k)}\b})=\deg({^{(k)}\a})$;
 \item[(2)] there exists no pair of linked segments $\{\Delta, \Delta'\}$ such that 
 $$b(\Delta)=k, ~b(\Delta')=k+1.$$
\end{description}
\end{definition}

\remk 
There exists a version of lemma \ref{lem: 3.0.7} for ${^{(k)}a}$.
In the following sections, we will work exclusively with 
$\a^{(k)}$ and the hypothesis  $H_{k}(\a)$.
But all our results will remain valid if we replace $\a^{(k)}$ 
by  ${^{(k)}\a}$ and $H_{k}(\a)$ by ${_{k}H(\a)}$.

\section{Injectivity of \texorpdfstring{$\psi_{k}$}{Lg}: First Step}
By previous section, we know there exists $\c\in S(\a)_k$ such that $\c^{(k)}=(\a^{\min})_{\min}$, the 
minimal element in $S(\a^{(k)})$. In this section, we give an explicit construction of such a $\c$
and show that it is the unique multisegment in $S(\a)_k$ which is set to $(\a^{(k)})_{\min}$ by $\psi_k$.

\begin{itemize}
\item In proposition \ref{prop: 3.2.2}, we construct a multisegment $\c\in S(\a_1)_k$ such that $\c^{(k)}=(\a^{(k)})_{\min}$, 
where $\a_1$ is a multisegment such that $\a\in S(\a_1)$.

\item We prove that there exists a unique element in  $S(\a)_{k}$ which is sent to $(\a^{(k)})_{\min}$ by $\psi_k$. 

\item Then we apply the uniqueness result to $S(\a_{1})_{k}$ to prove that the constructed $\c$ before is in $S(\a)_k$.\footnote{
Here we use partial derivative to prove our result, but it can also be done 
in a purely combinatorial way, which is less elegant and more lengthy though.
}
\end{itemize}

\begin{notation}\label{nota: 3.2.1}
Let $\ell_{k}=f_{e(\a)}(k)$ (cf. Def.\ref{def: 1.2.5}).
\end{notation}


\begin{definition}
Let 
\[
 \a_{0}=\{\Delta\in (\a^{(k)})_{\min}: e(\Delta)=k-1\}.
\]
\end{definition}

\begin{prop}\label{prop: 3.2.2}
Let $\a_{0}=\{\Delta_{1}\succeq \cdots \succeq \Delta_{r}\}$. 
Let $\c$ be a multisegment such that 
\begin{description}
 \item[(1)]If $\varphi_{\a}(k-1)\geq \varphi_{\a}(k)$, then $r=\varphi_{\a}(k-1)-\varphi_{\a}(k)+\ell_{k}$.
 Let 
 \[
  \c=((\a^{(k)})_{\min}\setminus \a_{0})\cup\{\Delta_{1}^+\succeq \cdots \succeq \Delta_{\ell_{k}}^+\succeq \Delta_{m+1}\succeq \cdots \succeq \Delta_{r} \}.
 \]
\item[(2)]If $\varphi_{\a}(k)-\ell_{k}<\varphi_{\a}(k-1)<\varphi_{\a}(k)$, then $r=\varphi_{\a}(k-1)-\varphi_{\a}(k)+\ell_{k}$. Let 
\[
 \c=((\a^{(k)})_{\min}\setminus \a_{0})\cup \{\Delta_{1}^+\succeq \cdots \succeq \Delta_{r}^+\succ\underbrace{ [k]=\cdots =[k]}_{\l_k-r}\}
\]
\item[(3)]If $\varphi_{\a}(k-1)\leq \varphi_{\a}(k)-\ell_{k}$, then $\a_{0}=\emptyset$ and
\[
 \c=\a^{(k)}+\ell_{k}[k].
\]
Then $\c$ satisfies the hypothesis $ H_{k}(\c)$ and $\c^{(k)}=(\a^{(k)})_{\min}$.
\end{description}
\end{prop}

%

\begin{proof}
We show only the case $\varphi_{\a}(k-1)>\varphi_{\a}(k)$,
the proof for other cases is similar.
Note that we have the following equality
\[
 \varphi_{\a}(k-1)=\varphi_{(\a^{(k)})_{\min}}(k-1)=r+\sharp \{\Delta\in (\a^{(k)})_{\min}: \Delta\supseteq [k-1, k]\}.
\]
Moreover, $\varphi_{\a}(k-1)>\varphi_{\a}(k)$ implies 
that no segment in $(\a^{(k)})_{\min}$ starts at $k$ by minimality, hence we also have
\[
 \varphi_{\a}(k)=\varphi_{(\a^{(k)})_{\min}}(k)+\l_k=\sharp \{\Delta\in (\a^{(k)})_{\min}: \Delta\supseteq [k-1, k]\}+\l_k.
\]
Now comparing the two formulas gives
the equality 
$r=\varphi_{\a}(k-1)-\varphi_{\a}(k)+\ell_{k}$.
By definition we have $\c^{(k)}=(\a^{(k)})_{\min}$.
To check that $\c$ satisfies the hypothesis $H_{k}(\c)$, it suffices 
to note that $(\a^{(k)})_{\min}\setminus \a_{0}$ does not contain segment 
which ends in $k-1$.
\end{proof}

\begin{lemma}
 Let $\c\in S(\c)_{k}$ be a multisegment such that $\c^{(k)}$ is minimal. Then 
 if $\d\in S(\c)$ such that $\d^{(k)}=\c^{(k)}$, then $\c=\d$ 
\end{lemma}

\begin{proof}
Suppose that $\d<\c$ is a multisegment such that $\d^{(k)}=\c^{(k)}$. Consider the maximal chain of multisegments
\[
\c=\c_0>\cdots > \c_t=\d.
\]
Our assumption implies that $\c_i^{(k)}=\c^{(k)}$ for all $i=1, \cdots, t$ by lemma \ref{lem: 3.0.7}. Hence we can assume $t=1$ and consider 
$\d\in S(\c)$ to be a multisegment obtained by applying the elementary
operation to the pair of linked segments $\{\Delta\prec \Delta'\}$.
\begin{itemize}
 \item If $e(\Delta)\neq k, e(\Delta')\neq k$, then 
the pair $\{\Delta, \Delta'\}$ also appears in $\c^{(k)}$, 
contradict the fact that $\c^{(k)}$ is minimal.
\item If $e(\Delta')=k$, then by the fact that $\c\in S(\c)_{k}$, 
we know that $e(\Delta)<k-1$, which implies that the pair 
$\{\Delta, \Delta^-\}$ is linked and belongs to $\c^{(k)}$,contradiction.
\item If $e(\Delta')=k$ and $b(\Delta')<k+1$, then the 
pair $\{\Delta^-, \Delta'\}$ is still linked and belongs to 
$\c^{(k)}$, contradiction.
\end{itemize} 
Hence we must have $e(\Delta')=k$ and $b(\Delta')=k+1$, this
implies that $\deg(\d^{(k)})>\deg(\c^{(k)})$ and $\d\notin \tilde{S}(\c)_{k}$.
Finally, (b) of lemma \ref{lem: 3.0.7} implies that for all $\d<\c$, 
we have $\d\notin \tilde{S}(\c)_{k}$.
\end{proof}

\begin{prop}\label{prop: 3.2.4}
Let $\c\in S(\c)_{k}$ be a multisegment such that $\c^{(k)}$ is minimal. Then
the partial derivative $\D^{k}(L_{\c})$ contains in $\mathcal{R}$ a unique term 
of minimal degree $L_{\c^{(k)}}$, which appears with multiplicity one.
\end{prop}

\begin{proof}
Let 
$\c=\{\Delta_{1}, \cdots, \Delta_{r}\}$
such that $e(\Delta_{t})=k$ if and only if 
$t=i, \cdots, j$ with $i\leq j$.
Then 
\[
 \D^k(\pi(\c))=\Delta_{1}\times \cdots 
 \times \Delta_{i-1}\times (\Delta_{i}+\Delta_{i}^-)\times
 \cdots \times (\Delta_{j}+\Delta_{j}^-)\times \Delta_{j+1}\times 
 \cdots\times \Delta_{r}
\]
with minimal degree term given by 
\[
 \pi(\c^{(k)})=\Delta_{1}\times \cdots \times \Delta_{i-1}\times
 \Delta_{i}^-\times \cdots \times \Delta_{j}^-\times \Delta_{j+1}
 \times \cdots \times \Delta_{r}.
\]
The same calculation shows that for any $\d\in S(\c)$, 
the minimal degree terms in $\D^k(\pi(\d))$ is given 
by $\pi(\d^{(k)})$, whose degree is strictly greater
than that of $\c^{(k)}$ since by previous lemma we know that $\d\notin \tilde{S}(\c)_k$. Note that 
$\D^k(L_{\d})$ is a non-negative sum of irreducible representations
( Theorem \ref{teo: 3}), which cannot contain 
any representation of degree equal to that of $\c^{(k)}$, 
by comparing the minimal degree terms in 
$\D^k(\pi(\d))$ and $\sum_{\e\in S(\d)}m(\e, \d)\D^k(L_{\e})$.
Finally, comparing the minimal degree terms in 
 $\D^k(\pi(\c))$ and $\sum_{\e\in S(\c)}m(\e, \c)\D^k(L_{\e})$
 gives the proposition.
 
\end{proof}

\begin{prop}\label{prop: 3.2.5}
Let $\a$ be a multisegment. 
Then  $S(\a)_{k}$ contains a unique 
multisegment $\c$ such that $\c^{(k)}=(\a^{(k)})_{\min}$.
\end{prop}

\begin{proof}
Let 
$\a=\{\Delta_{1}', \cdots, \Delta_{s}'\}$
such that $e(\Delta_{t}')=k$ if and only if 
$n=i, \cdots, j$ with $i\leq j$.
Then 
\[
 \D^k(\pi(\a))=\Delta_{1}'\times \cdots 
 \times \Delta_{i-1}'\times (\Delta_{i}'+{\Delta_{i}'}^{-})\times
 \cdots \times (\Delta_{j}'+\Delta_{j}'^-)\times \Delta_{j+1}'\times 
 \cdots\times \Delta_{s}'
\]
with minimal degree term given by 
\[
 \pi(\a^{(k)})=\Delta_{1}'\times \cdots \times \Delta_{i-1}'\times
 \Delta_{i}'^-\times \cdots \times \Delta_{j}'^-\times \Delta_{j+1}'
 \times \cdots \times \Delta_{r}'.
\]
Note that in $\pi(\a^{(k)})$, 
$m((\a^{(k)})_{\min}, \a^{(k)})=1$(cf. \cite{Z3}). Now compare 
with the terms of minimal degree in 
$\sum_{\d\in S(\a)}m(\d, \a)\D^k(L_{\d})$ 
and apply the proposition \ref{prop: 3.2.5}
yields the uniqueness of $\c$ such that 
$\c^{(k)}=(\a^{(k)})_{\min}$.
\end{proof}

\begin{prop}
 Let $\c$ be the multisegment constructed in 
 proposition \ref{prop: 3.2.2}. Then $\c\in S(\a)$.
\end{prop}

\begin{proof}
Let $$\a_{1}=\a^{(k)}+m[k],$$ then we observe that 
$\a\in S(\a_{1})$.
Because of 
$\c\in S((\a^{(k)})_{\min}+m[k])$, 
we have $\c\in S(\a_{1})$.
Note that since $\deg ((\a_{1})^{(k)})=\deg (\c^{(k)})$,
the fact that $\c\in S(\c)_{k}$ implies that
$\c\in S(\a_{1})_{k}$. Now 
let $\d\in S(\a)_{k}$, then we have 
$\d\in S(\a_{1})_{k}$ since $\deg(\d^{(k)})=\deg (\a_{1}^{(k)})=\deg(\a^{(k)})$.
Assume furthermore that $\d^{(k)}$ is minimal, then 
by  proposition \ref{prop: 3.2.5}, we know that such a multisegment
in $S(\a_{1})_{k}$ is unique, which implies $\d=\c$.
\end{proof}

\begin{cor}\label{cor: 3.2.7}
Let $\c\in S(\a)_{k}$ such that $\c^{(k)}=(\a^{(k)})_{\min}$, then 
$\c$ is minimal in $\tilde{S}(\a)_{k}$. 
\end{cor}

\begin{proof}
 By corollary \ref{cor: 3.1.11}, we know that for any $\d\in \tilde{S}(\a)_{k}$,
 there exists a multisegment $\c'\in S(\a)_{k}$ 
 with $\c'^{(k)}=(\a^{(k)})_{\min}$, such that 
 $\d>\c'$. By uniqueness, we must have $\c=\c'$.
\end{proof}

\section{Geometry of Nilpotent Orbits: General Cases}

In this section, we show geometrically that 
the morphism 
\begin{align*}
 \psi_{k}: S(\a)_{k}&\rightarrow S(\a^{(k)})\\
 \c&\mapsto \c^{(k)}  
\end{align*}
is bijective, satisfying the properties
\begin{description}
 \item[(1)]For $\c\in S(\a)_{k}$, we have 
 $m(\c, \a)=m(\c^{(k)}, \a^{(k)})$.
 \item [(2)] The morphism $\psi_{k}$ preserves the order, i.e, 
 for $\c, \d\in S(\a)_{k}$, $\c>\d$ if and only if $\c^{(k)}>\d^{(k)}$.
\end{description}

To achieve this, firstly we consider the sub-variety
$X_{\a}^{k}= \coprod_{\c\in \tilde{S}( \a)_{k}}O_{\c}$, 
and construct a fibration $\alpha$ from $X_{\a}^{k}$
to $Gr(\ell_{k}, V_{\varphi_{\a}, k})$, the latter is the space of the 
$\ell_{k}$-dimensional subspace of $V_{\varphi_{\a}, k}$. 
Secondly, we construct an open immersion 
\[
 \tau_{W}: (X_{\a}^{k})_{W}\rightarrow Y_{\a^{(k)}}\times \Hom(V_{\varphi_{\a}, k-1}, W),
\]
where $(X_{\a}^{k})_{W}$ is the fiber over $W$ with respect to $\alpha$
and $Y_{\a^{(k)}}=\coprod_{\c\in S(\a^{(k)})}O_{\c}$.
Our main difficulty here lies in proving that $\tau_{W}$ is actually an 
open immersion. The idea is to  apply Zariski Main theorem, to do this, 
we have to prove the normality and irreducibility of both varieties. 
Irreducibility of $(X_{\a}^{k})_{W}$ follows from our results 
in previous section, and normality follows from the fibration $\alpha$ 
and the fact that orbital varieties are locally isomorphic to some 
Schubert varieties, by Zelevinsky, cf. \cite{Z4}.

Once we prove that $\tau_{W}$ is an open immersion.
All the desired properties of $\psi_{k}$ follow.

\vspace{0.5cm}
Here we fix a multisegment $\a$ and let $\varphi=\varphi_{\a}$.

\begin{definition}\label{def: 3.3.1}
\footnote{ In this section 
we only work with $X_{\a}^{(k)}$ instead of 
$\tilde{X}_{\a}^{(k)}=\coprod_{\b\in S(\a)_k}O_{\b}$. The reason can be seen from the simple example of the 
affine plane $\mathbb{A}^2$ endowed with the stratification 
\[
 X_{1}=\mathbb{A}^2-\mathbb{A}^1, \quad X_{2}=\mathbb{A}^1-pt, \quad X_{3}=pt. 
\]
If we are interested in $X_{1}\coprod X_{3}$, it is better to 
study $\mathbb{A}^2$, because there is no nontrivial directed extension
of $X_{1}$ by $X_{3}$. Instead, if we are interested in 
$X_{1}\coprod X_{2}$, we can study $\mathbb{A}^{2}-pt$, which is 
already a nontrivial extension.
}
\begin{itemize}
 \item 
Let 
$$
X_{\a}^{k}= \coprod_{\c\in \tilde{S}( \a)_{k}}O_{\c},  
$$ 
\item Let
$Y_{\a^{(k)}}= \coprod_{c\in S(\a^{(k)})}O_{\c}$.

\item For $\b>\c$ in $\tilde{S}(\a)_{k}$,
we define 
\[
 X_{\b, \c}^{k}=\coprod_{\b\geq \d\geq\c}O_{\d}.
\] 
\end{itemize}
\end{definition}

Let $\c\in \tilde{S}( \a)_{k}, T\in O_{\c}$, then 
 
\begin{lemma}
Let $\varphi=\varphi_{\a}$.
We have 
$\dim(\ker(T|_{V_{\varphi, k}}))=\sharp\{\Delta\in \a: e(\Delta)=k\}=\ell_{k}$(Notation \ref{nota: 3.2.1}),
which does not depend on the choice of $T$.
\end{lemma}

\begin{proof}
The fact $T\in O_{\c}$ implies 
\[
 \dim(\ker(T|_{V_{\varphi, k}}))=\sharp\{\Delta\in \c: e(\Delta)=k\}.
\]
Then our lemma follows from lemma \ref{lem: 3.1.5}.

\end{proof}

\begin{definition}
Let 
\[
 Gr(\ell_{k}, V_{\varphi})=\{W\subseteq V_{\varphi, k}: \dim(W)=\ell_{k}\},
\]
and for $W\in  Gr(\ell_{k}, V_{\varphi})$, let 
\[
 V_{\varphi}/W=V_{\varphi, 1}\oplus \cdots V_{\varphi, k-1}
 \oplus V_{\varphi, k}/W \oplus \cdots.
\]
Also, we denote by 
 \[
 p_{W}: V_{\varphi}\rightarrow V_{\varphi}/W
 \]
the canonical projection.
\end{definition}

\begin{definition}
We define
\[
 \tilde{Z}^{k}=\{(T, W): W\in Gr(\ell_{k}, V_{\varphi}), T\in End(V/W) \text{ of degree +1}\},
\]
and the canonical projection
\begin{align*}
 \pi:&\tilde{Z}^{k}\rightarrow Gr(\ell_{k}, V_{\varphi})\\
 &(T, W)\mapsto W.
\end{align*}

\end{definition}

\begin{prop}
 The morphism $\pi$ is a fibration with fiber 
 $$
 E_{\varphi_{\a^{(k)}}}(\text{ Def.\ref{def: 4.1.1}}).
$$
 \end{prop}

\begin{proof}
 This follows from the definition.
\end{proof}

\begin{definition}\label{def: 3.3.6}
Assume $\b, \c\in S(\a^{(k)})$.
\begin{itemize}
 \item Let
\[
 Z^{k, \a}=\{(T,W)\in \tilde{Z}^{k}: T\in Y_{\a^{(k)}} \}.
\]
\item Let
\[
 Z^{k, \a}_{\b, \c}=\{(T,W)\in \tilde{Z}^{k}: T\in \coprod_{\b\geq\d\geq \c}O_{\d}\}, 
 ~Z^{k, \a}_{\b}=\{(T,W)\in \tilde{Z}^{k}: T\in \coprod_{\d\geq \b}O_{\d}\}.
\]

\item Let
\[
 Z^{k, \a}(\c)=\{(T, W)\in Z^{k, \a}, T\in O_{\c}\}.
\]

\end{itemize}
\end{definition}

\remk 
 The restriction of $\pi$ to $Z^{k, \a}$ is a fibration 
 with fiber $ Y_{\a^{(k)}}$.

\begin{definition}\label{def: 3.3.7}
Now we define $T^{(k)}\in End(V/\ker(T|_{V_{\varphi, k}}))$ such that 
\begin{displaymath}
 T^{(k)}|_{V_{\varphi,i}}=\left\{ \begin{array}{cc}
&\hspace{-1cm}T|_{V_{\varphi, i}}, \text{ for }i \neq k, k-1,\\
&\hspace{-0.4cm}p_{T, k}\circ T|_{V_{\varphi, i}}, \text{ for }i=k-1\\
&\hspace{-1cm} T|_{V_{\varphi, i}}\circ p_{T, k}, \text{ for }i=k.
\end{array}
\right.
\end{displaymath}
where $p_{T, k}: V_{\varphi}\rightarrow V_{\varphi}/\ker(T|_{V_{\varphi, k}})$ is the canonical projection.
\end{definition}

This gives naturally an element $(T^{(k)}, \ker(T|_{V_{\varphi, k}}))$ in $Z^{k, \a}$.
We construct a morphism 
\[
 \gamma_{k}: X_{\a}^{k}\rightarrow Z^{k, \a}.
\]

by 
\[
 \gamma_{k}(T)=(T^{(k)}, \ker(T|_{V_{\varphi, k}})).
\]

\begin{definition}
 We define 
\[
 \alpha: X_{\a}^{k}\rightarrow Gr(\ell_{k}, V_{\varphi}), 
\]
with $\alpha(T)=\ker(T|_{V_{\varphi, k}})$.
\end{definition}

\remk 
 We have a commutative diagram
 \begin{displaymath}
  \xymatrix{
  X_{\a}^{k}\ar[d]^{ \alpha}\ar[r]^{\gamma_{k}}&Z^{k, \a}\ar[dl]^{\pi}\\
  Gr(\ell_{k}, V_{\varphi}).
  }
  \end{displaymath}
where $\gamma_{k}$ maps fibers to fibers.

\begin{prop}\label{prop: 4.6.10}
 The morphism $\alpha$ is a fiber bundle such that $\alpha|_{O_{\c}}$
 is surjective for any $\c\in \tilde{S}(\a)_{k}$.
\end{prop}

\begin{proof}
 We have to show that $\alpha$ is locally trivial. 
 We fix $W\in Gr(\ell_{k}, V_{\varphi})$
 Note that $GL_{\varphi(k)}$ acts transitively 
 on $Gr(\ell_{k}, V_{\varphi})$. Let 
 $P_{W}$ be the stabilizer of $W$. Then by Serre \cite{S} proposition
 3, we know that the principle bundle 
 $$
 GL_{\varphi(k)}\rightarrow GL_{\varphi(k)}/P_{W}
 $$
 is \'etale-locally trivial.
 Here the base $GL_{\varphi(k)}/P_{W}$ is isomorphic to $Gr(\ell_{k}, V_{\varphi})$.
 It is even Zariski-locally trivial because $P_{W}$ is parabolic, which is special 
 in the sense of Serre \cite{S}, $\S$ 4.
 Now we can write
 \begin{displaymath}
  \xymatrix{
 X_{\a}^{ k}\ar[d] & GL_{\varphi(k)}\times_{P_{W}}\alpha^{-1}(W)\ar[l]_{\hspace{-1.5cm}\delta}\ar[dl]\\
  Gr(\ell_{k}, V_{\varphi})&
  }
\end{displaymath}
 where 
 \[
  \delta([g, T])=g.T.
 \]
We claim that $\delta$ is an isomorphism.
In fact, for any $T\in X_{\a}^{ k}$, 
we choose $g\in GL_{\varphi(k)}$ such that 
\[
 g(\ker(T|_{V_{\varphi, k}}))=W.
\]
This implies $g.T\in \alpha^{-1}(W)$, thus
\[
 \delta([g^{-1}, g.T])=T.
\]
This shows the surjectivity. For injectivity, 
it is enough to show that 
$$
\delta([g,T])=g.T\in \alpha^{-1}(W)
$$
implies $g\in P_{W}$. But this is by definition of $P_{W}$.

The fact that $\alpha$ is locally trivial then can 
be deduced from that of  
$$GL_{\varphi(k)}\times_{P_{W}}\alpha^{-1}(W),$$
while the latter is a consequence of the fact that $GL_{\varphi(k)}$
is locally trivial over $Gr(\ell_{k}, V_{\varphi})$.

Finally, we want to show the surjectivity of the orbit $\alpha|_{O_{\c}}$. 
This is a consequence the fact that $GL_{\varphi(k)}$ acts transitively 
 on $Gr(\ell_{k}, V_{\varphi})$.
\end{proof}

\begin{prop}\label{prop: 3.2.11}
Let $\c\in \tilde{S}(\a)_{k}$. 
The restriction map
\[
 \gamma_{k}: O_{\c}\rightarrow Z^{k, \a}(\c^{(k)})
\]
is surjective.
\end{prop}

\begin{proof}
Let $(T_{0}, W)\in Z^{k, \a}(\c^{(k)})$. 
Consider 
\[
 m=\sharp\{\Delta\in \c: e(\Delta)=k, \deg(\Delta)\geq 2\}\leq \min\{\l_k, \dim(\ker(T_0|_{V_{\varphi, k-1}}))\}.
\]
We choose a splitting $V_{\varphi, k}=W\oplus V_{\varphi, k}/W$ and 
let $T': V_{\varphi, k-1}\rightarrow W$ be a linear morphism of rank $m$.
Finally, we define $T\in \gamma^{-1}_k((T_{0}, W))$ by letting
\[
 T|_{V_{\varphi, k-1}}=T'\oplus T_{0}|_{V_{\varphi, k-1}}, 
\]
\[
T|_{V_{\varphi, k}}=T_{0}|_{V_{\varphi, k}/W}\circ p_{W},
\]
\[
  T|_{V_{\varphi, i}}=T|_{V_{\varphi, i}},
 \text{ for } i\neq k-1, k.
\]
Let 
\[
\{\Delta\in \c: e(\Delta)=k, \deg(\Delta)\geq 2\}=\{\Delta_1, \cdots, \Delta_m\}, \quad b(\Delta_1)\leq \cdots \leq b(\Delta_m).
\]
We denote $W_i=T_0^{[b(\Delta_1), k-1]}(V_{\varphi, b(\Delta_1)})\cap \ker(T_0|_{V_{\varphi, k-1}})$, then
\[
W_1 \subseteq \cdots \subseteq W_r \subseteq \ker(T_0|_{V_{\varphi, k-1}}).
\]
Then we have $T\in O_\c$ if and only if 
\[
\dim(T'(W_i))-\dim(T(W_{i-1}))=\dim(W_i/W_{i-1}), \quad i=1, \cdots, m.
\]
Since such $T'$ always exists, we are done.
\end{proof}

\begin{notation}
 We fix $W\in Gr(\ell_{k}, V_{\varphi})$, and denote 
 $$(X_{\a}^{ k})_{W}, \quad (Z^{k, \a})_{W}$$ the fibers over $W$.
\end{notation}

\begin{prop}\label{prop: 3.3.13}
The fiber $(X_{\a}^{ k})_{W}$ is normal and irreducible as an algebraic variety
over $\C$.
\end{prop}

\begin{proof}
Note that since 
$\tilde{S}(\a)_{k}$ contains a unique minimal element $\c$, 
the variety $X_{\a}^{ k}$ is contained and is open in the 
irreducible variety $\line{O}_{\c}$. Now by 
\cite{Z4} theorem 1, we know that  $X_{\a}^{ k}$ is actually normal.

By proposition \ref{prop: 4.6.10}, we know that $\alpha$ is a fibration between two varieties 
$X_{\a}^{k}$ and $Gr(\ell_{k}, V_{\varphi})$.
The fact that both are normal and irreducible implies that the fiber $(X_{\a}^{ k})_{W}$ is normal 
and irreducible.
\end{proof}

\remk 
Note that by definition, we are allowed to identify 
$(Z^{k, \a})_{W}$ with $Y_{\a^{(k)}}$. This is what 
we do from now on.

\begin{definition}\label{def: 3.3.13}
 We choose a splitting $V_{\varphi, k}=W\oplus V_{\varphi, k}/W$
and denote by $q_{W}: V_{\varphi, k}\rightarrow W$ the projection.
 We define a morphism $\tau_{W}$ 
  $$
  \tau_{W}(T)=((\gamma_{ k})_{W}(T), q_{W}\circ T|_{V_{\varphi, k-1}}).
  $$
  \end{definition}
  
\remk
  Then we have the following commutative diagram
  \begin{displaymath}
  \xymatrix{
  (X_{\a}^{ k})_{W}\ar[r]^{\hspace{-2cm}\tau_{W}}\ar[d]^{(\gamma_{ k})_{W}} 
  &(Z^{k, \a})_{W}\times \Hom(V_{\varphi, k-1}, W)\ar[dl]^{s}\\
  (Z^{k, \a})_{W}
  }
  \end{displaymath}
 where $s$ is the canonical projection.

 \begin{lemma}\label{lem: 4.6.14}
 The morphism $\tau_{W}$ is injective. 
\end{lemma}

\begin{proof}
Note that any $T\in (X_{\a}^{ k})_{W}$ is determined by 
$(\gamma_{ k})_{W}(T)$ and $T|_{V_{\varphi, k-1}}$.
Furthermore, $T|_{V_{\varphi, k-1}}$ is determined by 
$p_{W}\circ T|_{V_{\varphi, k-1}}$ and $q_{W}\circ T|_{V_{\varphi, k-1}}$.
Since $p_{W}\circ T|_{V_{\varphi, k-1}}$ is a component of 
$(\gamma_{ k})_{W}(T)$, 
it is determined by $(\gamma_{ k})_{W}(T)$ and
$q_{W}\circ T|_{V_{\varphi, k-1}}$. This gives us the injectivity.
\end{proof}

\begin{lemma}
Let $\c\in S(\a)_{k}$ such that 
$\c^{(k)}=(\a^{(k)})_{\min}$. Then 
The image of $ O_{\c}\cap (X_{\a}^{k})_{W}$ is open 
in $O_{\c^{(k)}}\times \Hom(V_{\varphi, k-1}, W)$. 
\end{lemma}

\begin{proof}

Let $\c\in S(\a)_{k}$ such that
$\c^{(k)}=(\a^{(k)})_{\min}$.
We shall use the description in proposition 
\ref{prop: 3.2.2}.
We show that the image of 
\[
  O_{\c}\cap (X_{\a}^{k})_{W}
\]
is open in $O_{\c^{(k)}}\times \Hom(V_{\varphi, k-1}, W)$.

Let $T\in (O_{\c})_{W}$. We check case by case:
\begin{description}
 \item [(1)]If $\varphi(k-1)\leq \varphi(k)-\ell_{k}$, 
 the fact $\c^{(k)}=(\a^{(k)})_{\min}$ implies that 
 $T^{(k)}|_{V_{\varphi, k-1}}$ is injective.
 As a consequence
we have $\Im(T|_{V_{\varphi, k-1}})\cap W=0$. Hence for any element $T_{0}\in \Hom(V_{\varphi, k-1}, W)$
, we define $T_{0}\in O_{\c}$, such that 
 \[
  T_{0}|_{V_{\varphi, k-1}}=T_{0}\oplus T^{(k)}|_{V_{\varphi, k-1}},
 \]
which lies in the fiber over $(\gamma_{k})_{W}^{-1}((T^{(k)}, W))$. 
Since by proposition \ref{prop: 3.2.11}, every element in 
$O_{\c^{(k)}}$ comes from some element in $O_{\c}$, 
hence
\[
 \tau_{W}(O_{\c}\cap (X_{\a}^{k})_{W})=O_{\c^{(k)}}\times \Hom(V_{\varphi, k-1}, W),
\]
which is open.
 
\item[(2)]
If $\varphi(k)-\ell_{k}<\varphi(k-1)< \varphi(k)$,
the fact $\c^{ (k)}=(\a^{(k)})_{\min}$ implies that
the morphism 
\[
 T^{ (k)}|_{V_{\varphi, k-1}}
\]
contains a kernel of dimension 
\[
 \varphi(k-1)-\varphi(k)+\ell_{k}.
\]

Our description of $\c$ in proposition \ref{prop: 3.2.2} shows that in this case  
$$
\dim(\Im(T|_{V_{\varphi, k-1}})\cap W)=\varphi(k-1)-\varphi(k)+\ell_{k}.
$$

In this situation, given an element $T_{0}\in \Hom(V_{\varphi, k-1}, W)$ 
 we define $T'\in E_{\varphi}$, such that 
 \[
  T'|_{V_{\varphi, k-1}}=T_{0}\oplus T^{(k)}|_{V_{\varphi, k-1}},
 \]
\[
 T'|_{V_{\varphi, k}}=T^{(k)}|_{V_{\varphi, k}/W}\circ p_{W},
\]
\[
 T'|_{V_{\varphi, i}}=T^{(k)}, \text{ for } i\neq k-1, k.
\]

By construction and proposition 
\ref{prop: 2.2.4}, we know that 
$T'\in O_{\c}$ if and only if $T'|_{V_{\varphi, k-1}}$ is injective, since no segment in $\c$ ends in $k-1$, as described in 
proposition \ref{prop: 3.2.2}. 
And this is equivalent to say
\[
 T_{0}|_{\ker(T^{( k)}|_{V_{\varphi, k-1}})}
\]
is injective. 
This is an open condition, hence $O_{\c}\cap (X_{\a}^{ k})_{W}$ 
is open in $O_{\c^{( k)}}\times \Hom(V_{\varphi, k-1}, W)$.

\item[(3)]If $\varphi(k-1)\geq \varphi(k)$, then by proposition
\ref{prop: 3.2.2}
\[
 \c^{ (k)}=(\a^{( k)})_{\min}
\]
implies 
\[
 \Im(T|_{V_{\varphi, k-1}})\supseteq W.
\]
Recall the notation from proposition \ref{prop: 3.2.2} 
\[
 \a_{0}=\{\Delta_{1}\succeq \cdots \succeq \Delta_{r}\}.
\]
with $r=\varphi(k-1)-\varphi(k)+\ell_{k}$.
Then 
\[
 \c=((\a^{(k)})_{\min}\setminus \a_{0})\cup 
 \{\Delta_{1}^+\succeq \cdots \succeq \Delta_{\ell_{k}}^+\succeq \Delta_{\ell_{k}+1}
 \succeq \cdots \succeq \Delta_{r}\}.
\]

Let $T_{0}\in \Hom(V_{\varphi, k-1},W)$, we define $T'\in E_{\varphi}$
\[
 T'|_{V_{\varphi, k-1}}=T_{0}\oplus T^{(k)}|_{V_{\varphi, k-1}},
 \]
\[
 T'|_{V_{\varphi, k}}=T^{(k)}|_{V_{\varphi, k}/W}\circ p_{W},
\]
\[
 T'|_{V_{\varphi, i}}=T^{(k)}, \text{ for } i\neq k-1, k.
\]
Consider the following flag over $V_{\varphi, k-1}$,
\[
 \ker(T^{(k)}|_{\varphi, k-1})=V_{r}
 \supseteq \cdots \supseteq V_{1}\supseteq V_{0}=0,
\]
where $V_{i}=\Im((T^{(k)})^{\Delta_{i}})\cap \ker(T^{(k)}|_{\varphi, k-1})$, with $i=1, \cdots, r$, for the notation $(T^{(k)})^{\Delta}$, we refer to 
definition \ref{def: 4.1.5}.

Now by proposition \ref{prop: 2.2.4}, we know that 
$T'\in O_{\c}$ if and only if 
\[
 \dim(T_{0}(V_{i}))-\dim(T_{0}(V_{i-1}))
 =\dim(V_{i}/V_{i-1}),
\]
for $i=1, \cdots, \ell_{k}$.
In fact, if $V_{i}\neq V_{i-1}$, then 
\[
 \dim(V_{i}/V_{i-1})=\sharp\{j: \Delta_{j}=\Delta_{i}\}.
\]
And by construction, if $i\leq \ell_{k}$, by proposition \ref{prop: 2.2.4}, 
the fact that $\c$ contains $\Delta_{i}^+$ implies that if $T'\in O_{\c}$, 
\[
 \dim(T_{0}(V_{i}))-\dim(T_{0}(V_{i-1}))
 =\dim(V_{i}/V_{i-1}).
\]
The converse holds by the same reason.

Again, this is an open condition, which proves that 
$O_{\c}\cap (X_{\a}^{k})_{W}$ is open in 
$O_{\c^{(k)}}\times \Hom(V_{\varphi, k-1}, W)$. 
\end{description}
 
\end{proof}

\begin{prop}\label{prop: 4.6.14}

The morphism $\tau_{W}$ is an open immersion.

\end{prop}

\begin{proof}
To see that it is open immersion, we shall use Zariski's main theorem.
Since all Schubert varieties are normal, we observe that 
\[
(Z^{k, \a})_{W}\times \Hom(V_{\varphi, k-1}, W)
\]
are normal by theorem 1 of \cite{Z4}. Also, by 
proposition \ref{prop: 3.3.13}, we know that 
$(X_{\a}^{k})_{W}$ is irreducible and normal,
hence $\tau_W$ is an open immersion.
\end{proof}

\begin{prop}\label{prop: 4.6.15}
Let $\c\in \tilde{S}(\a)_{k}$. Then 
$\c\in S(\a)_{k}$ if and only if
\[
 O_{\c}\cap (X_{\a}^{k})_{W}
\]
is open in 
\[
 (O_{\c^{(k)}}\times \Hom(V_{\varphi, k-1}, W)).
\]

\end{prop}

\begin{proof}
We already showed that  
\[
 O_{\c}\cap (X_{\a}^{k})_{W}
\]
is a  sub-variety of 
\[
 O_{\c^{(k)}}\times \Hom(V_{\varphi, k-1}, W).
\]
Moreover, we know that 
\[
  (O_{\c^{(k)}}\times \Hom(V_{\varphi, k-1}, W))\cap (X_{\a}^{k})_{W}
\]
is open in
\[
  O_{\c^{(k)}}\times \Hom(V_{\varphi, k-1}, W)
\]
since $\tau_{W}$ is open.  
Finally, by proposition \ref{prop: 3.2.11}, 
\begin{align*}
 &( O_{\c^{(k)}}\times \Hom(V_{\varphi, k-1}, W))\cap (X_{\a}^{k})_{W}\\
&=\coprod_{\d \in \tilde{S}(\a)_{k}, \d^{(k)}=\c^{(k)}}O_{\d}\cap (X_{\a}^{k})_{W}.
 \end{align*}
The variety $( O_{\c^{(k)}}\times \Hom(V_{\varphi, k-1}, W))\cap (X_{\a}^{k})_{W}$ 
 is irreducible because $( O_{\c^{(k)}}\times \Hom(V_{\varphi, k-1}, W))$ is irreducible, hence the 
 stratification $\coprod_{\d \in \tilde{S}(\a)_{k}, \d^{(k)}=\c^{(k)}}O_{\d}\cap (X_{\a}^{k})_{W}$
 by locally closed sub-varieties can only contain one term which is open, 
 from the point of view of Zariski topology.
Since for any element 
\[
 \d'\in \{\d \in \tilde{S}(\a)_{k}, \d^{(k)}=\c^{(k)}\},
\]
by (d) of lemma \ref{lem: 3.0.7}, we know that 
there exists $\c'\in S(\a)_{k}$ such that $\d'>\c'$. Hence we conclude that
\[
\{\d \in \tilde{S}(\a)_{k}, \d^{(k)}=\c^{(k)}\},
\]
contains a unique minimal element, which lies in $S(\a)_{k}$.
Now our proposition follows. 
\end{proof}

\begin{cor}\label{cor: 4.6.16}
Let $\a$ be a multisegment  and 
$$\c\in S(\a)_{k},$$
then 
\[
 P_{\a, \c}(q)=P_{\a^{(k)}, \c^{(k)}}(q).
 \]
\end{cor}

\begin{proof}
First of all, by proposition \ref{prop: 4.6.10}
and Kunneth formula, we know that 
\[
 \mathcal{H}^{j}(\line{O}_{\c})_{\a}=\mathcal{H}^{j}(\line{O}_{\c}\cap (X_{\a}^{(k)})_{W})_{\a},
\]
the localization being taken at a point in $O_{\a}\cap (X_{\a}^{(k)})_{W}$.
Now by proposition \ref{prop: 4.6.14} and proposition \ref{prop: 4.6.15}
, we may regard $\line{O}_{\c}\cap (X_{\a}^{(k)})_{W}$
as an open subset of $\line{O}_{\c^{(k)}}\times Hom(V_{\varphi, k-1}, W)$, hence
\[
 \mathcal{H}^{j}(\line{O}_{\c}\cap (X_{\a}^{(k)})_{W})_{\a}=
 \mathcal{H}^{j}(\line{O}_{\c^{(k)}}\times Hom(V_{\varphi, k-1}, W))_{\a^{(k)}}
 \]
and Kunneth formula implies that the latter is equal to 
\[
 \mathcal{H}^{j}(\line{O}_{\c^{(k)}})_{\a^{(k)}}.
\]
\end{proof}

\begin{cor}\label{cor: 4.6.17}
Let $\d\in S(\a)$  such that
 \[
  \d^{(k)}= \a^{(k)},
 \]
and 
 \[
  \c\in S(\a)_{k},
 \]
then $\c< \d$, and 
\[
 P_{\a, \c}(q)=P_{\d, \c}(q).
\]
\end{cor}

\begin{proof}
By corollary \ref{cor: 3.1.11}, we know that there exists $\c'\in S(\a)_{k}$
such that 
\[
 \d>\c', ~\c'^{(k)}=\c^{(k)}.
\]
And proposition \ref{prop: 4.6.15} implies $\c'=\c$.
Finally, applying the corollary \ref{cor: 4.6.16} to 
the pairs $\{\a, \c\}$ and $\{\d, \c\}$ yields the 
result.
\end{proof}

\section{Conclusion}
In this section, we draw some conclusions from what we 
have done before, espectially the properties related to $\psi_{k}$.

\begin{prop}\label{cor: 3.2.3}
 The map
 \begin{align*}
 \psi_{k}: S(\a)_{k} &\rightarrow S(\a^{(k)})\\
\c &\mapsto \c^{(k)}
 \end{align*}
is bijective. Moreover, 
\begin{itemize}
 \item  for $\c\in  S(\a)_{k}$
\[
 m(\c,\a)=m(\c^{(k)}, \a^{(k)}).
\]
\item for $\b, \c\in S(\a)_{k}$, we have $\b>\c$ if and only if
$\b^{(k)}> c^{(k)}$.
\end{itemize}
\end{prop}

\begin{proof}
By proposition \ref{prop: 4.6.15},
we know that $\psi_{k}$ is injective. 
Surjectivity is given by proposition \ref{prop: 1.5.10}.

For $\c\in  S(\a)_{k}$,
\[
 m(\c,\a)=m(\c^{(k)}, \a^{(k)})
\]
is by corollary \ref{cor: 4.6.16} by putting $q=1$, 
and applying theorem \ref{teo: 4.1.5}.

Finally, for $\b, \c\in S(\a)_{k}$, if $\b>\c$, then 
$\c\in S(\b^{(k)}, \b)$, and 
by lemma \ref{lem: 3.0.7}, we know that $\b^{(k)}>\c^{(k)}$. 
Reciprocally, if $\b^{(k)}>\c^{(k)}$, by 
proposition \ref{prop: 4.6.15}, we know that $\line{O}_{\b}\subseteq \line{O}_{\c}$, hence 
$\b>\c$. 

\end{proof}

\begin{cor}\label{lem: 2.3.4}
We have
\begin{itemize}
\item 
\addtocounter{theo}{1}
\begin{equation}\label{equ: (2)}
 \pi(\a^{(k)})=
 \sum_{\c\in S(\a)_{k}}m(\c,\a)L_{\c^{(k)}},
\end{equation}
\item let $\b\in S(\a)$ such that $\b$ satisfies the hypothesis $H_{k}(\a)$
and $\b^{(k)}=\a^{(k)}$, then 
\[
 m(\b, \a)=1, ~S(\a)_{k}=S(\b)_{k}.
\]
\end{itemize}
\end{cor}

\begin{proof}

The first part follows from the fact that 
$\psi_{k}$ is bijective and 
$m(\c, \a)=m(\c^{(k)}, \a^{(k)})$.
For the second part of the lemma, we note that 
$L_{\b^{(k)}}=L_{\a^{(k)}}$
appears with multiplicity one in $\pi(\a^{(k)})$, then equation (\ref{equ: (2)})
implies $m(\b, \a)=m(\b^{(k)}, \a^{(k)})=1$.
To see that $ S(\a)_{k}= S( \b)_{k}\subseteq S(\b)$.
Note that we have $S(\b)_{k}\subseteq S(\a)_{k}$ and two bijection 
\[
 \psi_{k}: S(\a)_{k}\rightarrow S(\a^{(k)}),
\]
\[
 \psi_{k}: S(\b)_{k}\rightarrow S(\b^{(k)})=S(\a^{(k)}),
\]
Hence comparing the cardinality gives 
$S(\a)_{k} =S(\b)_k$.
\end{proof}

\section{Minimal Degree Terms in Partial Derivatives}

\begin{prop}\label{teo: 3.0.6}
\begin{description}
\item[(i)] 
Suppose that $\a$ satisfies the hypothesis $H_{k}(\a)$.\\
Then $\D^{k}(L_{\a})$ contains  in $\mathcal{R}$ a unique 
irreducible representation of minimal degree, which is
 $L_{\a^{(k)}}$,
 and it appears with multiplicity one.
\item[(ii)]If $\a$ fails to satisfy the hypothesis $H_{k}(\a)$, then \\
 $L_{\a^{(k)}}$
 will not appear in $\D^{k}(L_{\a})$, 
 and the irreducible representations appearing
 are all of degree $>\deg(\a^{(k)})$. 
\end{description}
\end{prop}

\begin{proof}
Let $\a=\{\Delta_1\preceq \cdots \preceq \Delta_r\}$, such that 
\[
 e(\Delta_1)\leq \cdots<e(\Delta_i)=\cdots=e(\Delta_j)<\cdots \leq e(\Delta_r),
\]
with $k=e(\Delta_i)$.

We prove the proposition by induction on $\ell(\a)$(cf. definition \ref{def: 1.2.10}).
For, $\ell(\a)=0$, which means that $\a=\a_{\min}$, 
in this case $\a$ satisfies the $H_{k}(\a)$, and 
\[
 \D^{k}(L_{\a})=\D^{k}(\pi(\a))=
 \Delta_{1}\times \cdots \times (\Delta_{i}+\Delta_{i}^{-})\times \cdots 
 \times (\Delta_{j}+\Delta_{j}^{-}) \times \cdots  
\]
which contains
\[
 L_{\a^{(k)}}=\pi(\a^{(k)})=\Delta_{1}\times \cdots \times \Delta_{i}^{-}
 \times \cdots \Delta_{j}^{-}\times \cdots .
\]
Hence we are done.

For general $\a$, we have refer to the lemma \ref{lem: 3.0.7}.
 
We write
\addtocounter{theo}{1}
\begin{equation}\label{eq: 1}
\pi(\a)=L_{\a}+ \sum_{\b<\a}m(\b,\a)L_{\b}. 
\end{equation}
Now applying $\D^{k}$ to both sides and consider only the lowest degree
terms,
on the left hand side, we get
\addtocounter{theo}{1}
\begin{equation}\label{eq: 2}
\pi(\a^{(k)})=\Delta_{1}\times \cdots \times \Delta_{i-1}\times \Delta_{i}^{-}\times \cdots \times \Delta_{j}^{-}\times \cdots \Delta_{r}.
\end{equation}
By theorem \ref{teo: 3}, both sides are positive sum 
of irreducible representations, then 

\begin{itemize}
\item 
If  $\a$ satisfies the hypothesis $H_{k}(\a)$, 
on the right hand side, from our lemma \ref{lem: 3.0.7} and induction, we know that for all $\b<\a$,
$\D^{k}(L_{\b})$ does not contain
$L_{\a^{(k)}}$ as subquotient, hence
$\D^{k}(L_{\a})$
must contain $L_{\a^{(k)}}$ with multiplicity one. 
We have to show that it does not contain 
other subquotients of $\pi(\a^{(k)})$. 
Note that by induction, we have the following formula
\[
 \pi(\a^{(k)})=X+\sum_{\c\in S(\a)_{k}\setminus \a}m(\c, \a)L_{\c^{(k)}},
\]
where $X$ denotes the minimal degree terms in $\D^{k}(L_{\a})$. 
Now apply corollary \ref{lem: 2.3.4},  we conclude that $X=L_{\a^{(k)}}$.

\item 
Now if $\a$ fails to satisfy the hypothesis $H_{k}(\a)$, $\a\notin S(\a)_k$,
by proposition \ref{cor: 3.2.3} and induction, we know that there exists $\b\in S(\a)_k$, such that
$\a^{(k)}=\b^{(k)}$ and $\D^{k}(L_{\b})$ contains $L_{\a^{(k)}}$ as a subquotient with
multiplicity one.

Now by the lemma \ref{lem: 3.0.8}, $\pi(\a)-\pi(\b)$ is a positive sum 
of irreducible representations which contain $L_{\a}$: by the positivity of partial derivative, 
we know that we obtain a positive sum of irreducible representations after
applying $\D^{k}$. Now
$$
\D^{k}(\pi(\a)-\pi(\b))=\pi(\a^{(k)})-\pi(\b^{(k)})+\text{ higher degree terms}
$$ 
contains only terms of degree
$>\deg(\a^{(k)})$, so does $\D^{k}(L_{\a})$.

\end{itemize}
This finishes our induction.

\end{proof}

\begin{cor}\label{cor: 3.5.2}
Let $\a$ be a multisegment such that $\varphi_{e(\a)}(k)=1$. Then 
\begin{itemize}
 \item If $\a\in S(\a)_k$, then $\D^k(L_{\a})=L_{\a}+L_{\a^{(k)}}$. 
 
 \item If $\a\notin S(\a)_k$, then $\D^k(L_{\a})=L_{\a}$.
\end{itemize}
 
\end{cor}

\begin{proof}
First of all, we observe 
that the highest degree term in $\D^k(L_{\a})$ is given by 
$L_{\a}$. In fact, we have 
\[
 \D^k(\pi(\a))=\D^k(L_{\a})+\sum_{\b<\a}m(\b, \a)\D^k(L_{\b}),
\]
meanwhile we have 
\[
 \D^k(\pi(\a))=\pi(\a)+\text{ lower terms. }
\]
By induction on $\ell(\a)$ we conclude that the highest degree terms
in $\D^k(L_{\a})$ is $L_{\a}$.

If $\a\in S(\a)_k$, then proposition \ref{teo: 3.0.6} implies that 
the minimal degree term of $\D^k(L_{\a})$, but since 
$\deg(\a^{(k)})=\deg(\a)-1$, therefore we must have 
\[
 \D^k(L_{\a})=L_{\a}+L_{\a^{(k)}}.
\]
On the contrary, if  $\a\notin S(\a)_k$, then 
by (ii) of the proposition \ref{teo: 3.0.6}, we know that 
all irreducible representations appearing in 
 $\D^k(L_{\a})$ are of degree $>\deg(\a^{(k)})=\deg(\a)-1$, which implies 
 \[
 \D^k(L_{\a})=L_{\a}.
\]

\end{proof}

\chapter{Reduction to symmetric cases}

- In the first paragraph of this chapter, we generalize the construction of chapter 3 by iterating the truncation 
functor to obtain for $\c_1, \c_2$ two multisegments, 
the truncation ${^{(\c_1)}\b^{(\c_2)}}$ of a multisegment
$\b$. 

- Then we give an algorithm
to, starting from two multisegments $\a$ and $\b\in S(\a)$,
construct two symmetric multisegments 
$\a^{\sym}$ and $\b^{\sym}\in S(\a^{\sym})$ such that we have
the following equality
\[
 m(\b, \a)=m(\b^{\sym}, \a^{\sym}).
\]

- Then  we study some examples and we show how
our algorithm works for finding the coefficient $m(\b, \a)$.

- Finally, in the last paragraph, we give a proof of the Zelevinsky's conjecture stated in the introduction.

\section{Minimal Degree Terms}

The goal of this section is to define the 
set $S( \a)_{\d}\subseteq S(\a)$ and describe some of 
its properties.

\begin{definition}
Let  $(k_{1}, \cdots, k_{r})$ be a sequence of integers.
We define 
\[
\a^{(k_{1}, \cdots, k_{r})}=(((\a^{(k_{1})})\cdots)^{ (k_{r})}).
\]
\end{definition}

\begin{notation}
Let $\Delta=[k, \ell]$, we denote 
\[
 \a^{ (\Delta)}=\a^{(k, \cdots, \ell)}.
\]
More generally, for $\d=\{\Delta_{1}\preceq \cdots \preceq \Delta_{r}\}$,
 let
\[
 \a^{ (\d)}=(\cdots ((\a^{ (\Delta_{r})})^{ (\Delta_{r-1})})\cdots)^{ (\Delta_{1})}.
\]
\end{notation}

\begin{definition}
Let $(k_{1}, \cdots, k_{r})$ be a sequence of integers 
, then we define 
\[
S(\a)_{k_{1}, \cdots, k_{r}}=
\{\c\in S(\a):
\c^{( k_{1}, \cdots , k_{i})}
\in  S(\a^{( k_{1}, \cdots , k_{i})})_{k_{i+1}}, \text{ for }i=1, \cdots, r\}. 
\]
and 
\[
 \psi_{k_{1}, \cdots, k_{r}}:
 S(\a)_{k_{1}, \cdots, k_{r}}\rightarrow S(\a^{(k_{1}, \cdots, k_{r})}),
\]
sending $\c$ to $\c^{(k_{1}, \cdots, k_{r})}$.
\end{definition}

\begin{notation} 
Let $\d=\{\Delta_{1}\preceq \cdots \preceq \Delta_{r}\}$ such 
that $\Delta_{i}=[k_{i}, \ell_{i}]$. 
We denote 
\[
 S(\a)_{\d}: =S(\a)_{k_{r},\cdots, \ell_{r}, k_{r-1}, \cdots, k_{1}, \cdots, \ell_{1}}
\]
and 
\[
 \psi_{\d}: =\psi_{k_{r},\cdots, \ell_{r}, k_{r-1}, \cdots, k_{1}, \cdots, \ell_{1}}.
\]
\end{notation}

\begin{prop}\label{prop: 3.2.17}
Let $(k_{1}, \cdots, k_{r})$ be a sequence of integers . Then the set 
$S(\a)_{k_{1}, \cdots, k_{r}}$ is non-empty.
In fact, we have a bijective morphism 
\[
 \psi_{k_{1}, \cdots, k_{r}}: 
 S(\a)_{k_{1}, \cdots, k_{r}}\rightarrow S(\a^{(k_{1}, \cdots, k_{r})}).
 \]
 Moreover,
 \begin{description}
  \item[(1)] For $\c\in S(\a)_{k_{1}, \cdots, k_{r}}$,
  we have 
  \[
   m(\c, \a)=m(\c^{(k_{1}, \cdots, k_{r})}, \a^{(k_{1}, \cdots, k_{r})}).
  \]
\item[(2)] For $\b, \c \in S(\a)_{k_{1}, \cdots, k_{r}}$, 
then $\b>\c$ if and only if $\b^{(k_{1}, \cdots, k_{r})}>\c^{(k_{1}, \cdots, k_{r})}$.
 \item[(3)] We have
 \[
  \pi(\a^{(k_{1}, \cdots, k_{r})})
  =\sum_{\c\in S(\a)_{k_{1}, \cdots, k_{r}}}m(\c, \a)L_{\c^{(k_{1}, \cdots, k_{r})}}.
 \]
\item[(4)] Let $\b\in S(\a)_{k_{1}, \cdots, k_{r}}$ and 
$\b^{(k_{1}, \cdots, k_{r})}=\a^{(k_{1}, \cdots, k_{r})}$, then 
\[
 S(\a)_{k_{1}, \cdots, k_{r}}=S( \b)_{k_{1}, \cdots, k_{r}}.
\]

 \end{description}
 
\end{prop}

\begin{proof}
Injectivity follows from the fact 
\[
 \psi_{k_{1}, \cdots, k_{r}}=\psi_{k_{r}}\circ \psi_{k_{r-1}}\circ \cdots \circ \psi_{k_{1}}.
\]

For surjectivity, let $\d\in S(\a^{(k_{1}, \cdots, k_{r})})$,
we construct $\b$ inductively such that 
$\psi_{k_{1}, \cdots, k_{r}}(\b)=\d$. Let $\a_{r}=\d$, 
assume that we already construct $\a_{i}\in  S(\a^{( k_{1}, \cdots,  k_{i})})_{k_{i+1}}$ satisfying
that 
$$\a_{i}^{( k_{i+1} ,\cdots , k_{j})}\in 
S( \a^{( k_{1}, \cdots , k_{j}}))_{k_{j+1}}$$
for all $i< j\leq r$ and $\a_{i}^{( k_{i+1}, \cdots, k_{r})}=\d$.

Note that by the bijectivity of the morphism 
\[
 \psi_{ k_{i}}: S( \a^{( k_{1}, \cdots , k_{i-1})})_{k_{i}}\rightarrow 
 S(\a^{(k_{1}, \cdots, k_{i})}),
\]
there exists a unique $\a_{i-1}\in S( \a^{( k_{1}, \cdots , k_{i-1})})_{k_{i}}$, such that 
\[
\a_{i-1}^{ (k_{i})}=\a_{i}.
\]
Finally, take $\b=\a_{0}\in S(\a)_{k_{1}, \cdots, k_{r}}$.
We show (1)
by induction on $r$. The case for $r=1$ is by proposition
\ref{cor: 3.2.3}. For general $r$, by induction
\[
 m(\c, \a)=m(\c^{(k_{1}, \cdots, k_{r-1})}, \a^{(k_{1}, \cdots, k_{r-1})}),
\]
and now apply the case $r=1$ to the pair 
$\c^{(k_{1}, \cdots, k_{r-1})}, \a^{(k_{1}, \cdots, k_{r-1})}$
gives 
\[
 m(\c^{(k_{1}, \cdots, k_{r-1})}, \a^{(k_{1}, \cdots, k_{r-1})})=
 m(\c^{(k_{1}, \cdots, k_{r})}, \a^{(k_{1}, \cdots, k_{r})}).
\]
Hence 
\[
 m(\c, \a)=m(\c^{(k_{1}, \cdots, k_{r})}, \a^{(k_{1}, \cdots, k_{r})}).
\]
Also, to show (2), it suffices 
to apply successively the proposition \ref{cor: 3.2.3}.
And (3) follows from the bijectivity of $\psi_{k_{1}, \cdots, k_{r}}$
and (1). 
As for (4), we know by definition, 
\[
 S(\a)_{k_{1}, \cdots, k_{r}}\supseteq 
 S( \b)_{k_{1}, \cdots, k_{r}}.
\]
 We know that any for $\c\in S(\a)_{k_{1}, \cdots, k_{r}}$,
we have $\c^{(k_{1}, \cdots, k_{r})}\leq \b^{(k_{1}, \cdots, k_{r})}$, by (2), 
this implies that $\c\leq \b$. Hence we are done.
\end{proof}

Similarly, we have 
\begin{definition}
Let $(k_{1}, \cdots, k_{r})$ be a sequence of integers, 
then we define 
\[
_{k_{r}, \cdots, k_{1}}S(\a)=
\{\c\in S(\a):
{^{( k_{i}, \cdots , k_{1})}\c}
\in  {_{k_{i+1}}S({^{( k_{i}, \cdots , k_{1})}\a})}, \text{ for }i=1, \cdots, r\}. 
\]
and 
\[
 _{k_{r}, \cdots, k_{1}}\psi:
 {_{k_{r}, \cdots, k_{1}}S}(\a)\rightarrow S(^{(k_{r}, \cdots, k_{1})}\a),
\]
sending $\c$ to $^{(k_{r}, \cdots, k_{1})}\c$.
\end{definition}

\begin{notation} 
Let $\d=\{\Delta_{1},\cdots , \Delta_{r}\}$ such 
that $\Delta_{i}=[k_{i}, \ell_{i}]$ with $k_{1}\leq \cdots \leq k_{r}$
We denote 
\[
 _{\d}S(\a): =_{k_{r},\cdots, \ell_{r}, k_{r-1}, \cdots, k_{1}, \cdots, \ell_{1}}S(\a),
\]
and 
\[
 _{\d}\psi: =_{k_{r},\cdots, \ell_{r}, k_{r-1}, \cdots, k_{1}, \cdots, \ell_{1}}\psi.
\]
\end{notation}

\remk
Let $k_{1}, k_{2}$ be two integers. 
In general, we do not have 
\[
{ _{k_{2}}(S(\a)_{k_{1}})}=(_{k_{2}}S(\a))_{k_{1}}.
\]
For example, let $k_{1}=k_{2}=1$, $\a=\{[1], [2]\}$, then 
\[
 { _{k_{2}}(S(\a)_{k_{1}})}=\{\a\}, ~(_{k_{2}}S(\a))_{k_{1}}=\{[1, 2]\}.
\]

\begin{notation}\label{nota: 4.1.8}
We write for multisegments $\d_{1}, \d_{2}, \a$,
\[
 {_{\d_{2}}S(\a)_{\d_{1}}}: =({_{\d_{2}}S(\a))_{\d_{1}}}, ~
 S(\a)_{\d_{1}, \d_{2}}: =(S(\a)_{\d_{1}})_{\d_{2}}.
\]
and 
\[
 {_{\d_{2}}\psi_{\d_{1}}}: =({_{\d_{2}}\psi)_{\d_{1}}},
 ~\psi_{\d_{1}, \d_{2}}: =(\psi_{\d_{1}})_{\d_{2}}
\] 
And for $\b\in S(\a)$, 
\[
 ^{(\d_{2})}\b^{(\d_{1})}: =(^{\d_{2}}\b)^{(\d_{1})},
 ~\b^{(\d_{1}, \d_{2})}: =(\b^{(\d_{1})})^{(\d_{2})}.
\]

\end{notation}

\section{Reduction to symmetric case}

Now we return to the main question, i.e., the calculation of 
the coefficient $m(\c, \a)$ for $\c\in S(\a)$. Before we go into the details, 
we describe our strategies: 
\begin{description}
 \item[(i)]Find a symmetric multisegment, denoted by $\a^{\sym}$, 
 such that $L_{\a}$ is the minimal degree term in some partial derivative of 
 $L_{\a^{\sym}}$.
 \item[(ii)]For $\c\in S(\a)$, find $\c^{\sym }\in S(\a^{\sym})$, such that 
 we have $m(\c, \a)=m(\c^{\sym}, \a^{\sym})$.
\end{description}

\begin{prop}\label{cor: 5}
Let $\a$ be any multisegment, then there exists an ordinary multisegment $\b$, and
two  multisegments $\c_{i}, i=1,2$ such that 
\[
 \b\in {_{\c_{2}}S(\b)_{\c_{1}}}, ~ \a={^{(\c_{2})}\b^{(\c_{1})}}
\]
\end{prop}

\begin{proof}
Let $\a=\{\Delta_{1}, \cdots, \Delta_{r}\}$ be such that
\[
 \Delta_{1}\preceq \cdots \preceq \Delta_{r},
\]
and  
\[
 e(\Delta_{1})\leq  \cdots < e(\Delta_{j})=\cdots =e(\Delta_{i})< e(\Delta_{i+1})\leq  \cdots,
\]
such that $\Delta_{j}$ is the smallest multisegment in $\a$ such that $e(\Delta_{j})$ appears 
in $e(\a)$ with multiplicity greater than 1. Let 
$\Delta^{1}=[e(\Delta_{i})+1, \ell]$ be a segment, where $\ell$ is the maximal integer such that for any 
$m$ such that 
$e(\Delta_{i})\leq m\leq \ell-1$,
there is a segment in $\a$ which ends in $m$. Let $\a_{1}$ 
be the multisegment obtained 
by 
replacing $\Delta_{i}$ by $\Delta_{i}^+$,  and all $ \Delta \in \a$ such that $e(\Delta)\in (e(\Delta_{i}), \ell]$ by $\Delta^+$.
 Now we continue the previous
construction with $\a_{1}$ to get $\a_{2}\cdots$, until we get a 
multisegment $\a_{r_{1}}$ such that
$e(\a_{r_{1}})$ contains no segment with multiplicity greater
than 1. Let
\[
 c_{1}=\{\Delta^{1}, \Delta^{2}, \cdots, \Delta^{r_{1}}\}.
\]
Note that by construction, we have
\[
 \Delta^{1}\prec \Delta^{2}\prec \cdots \prec \Delta^{r_{1}}.
\]
And we show that $\a_{r_{1}}\in S(\a_{r_{1}})_{\c_{1}}$.
Note that 
\[
 \a_{i}=\a_{r_1}^{(\Delta^{r_{1}}, \cdots, \Delta^{i+1})},
\]
by induction on $r_1$, we can assume  that $\a_{1}\in S(\a_{r_{1}})_{\Delta^{r_{1}}, \cdots, \Delta^{2}}$
and show that 
$\a\in S(\a_{1})_{\Delta^{1}}$.
We observe that in $\a_{1}$, by construction, 
with the notations above, $\Delta_{j}, \cdots, \Delta_{i-1}$ are the only segments
in $\a_{1}$ that ends in $e(\Delta_{i})$, and $\Delta_{i}^+$ is the only segment
in $\a_{1}$ that ends in $e(\Delta_{i})+1$. Hence we conclude that 
$\a_{1}\in S(\a_{1})_{e(\Delta_{i})+1}$. And for $e(\Delta_{i})+1<m\leq \ell$, we know that 
$\a_{1}^{(e(\Delta_{1})+1, \cdots, m-1)}$ does not contain a segment which ends in $m-1$, hence 
$\a_{1}^{(e(\Delta_{1})+1, \cdots, m-1)}\in S(\a_{1}^{(e(\Delta_{1})+1, \cdots, m-1)})_{m}$.
We are done by putting $m=\ell$.

Now same construction can be applied to show that there exists a multisegment
$\a_{r_{2}}$ such that $b(\a_{r_{2}})$ contains no segment with multiplicity greater
than 1, and 
\[
 c_{2}=\{^{1}\Delta, \cdots, ^{r_{2}}\Delta\},
\]
such that
\[
  \a_{r_{2}}\in {^{\c_{2}}S(\a_{2})}, ~\a_{r_{1}}={^{(\c_{2})}\a_{r_{2}}}
\]
as minimal degree component.

Note that in this way we construct an
ordinary multisegment $\b=\a_{r_{2}}$, 
\[
 \b\in {_{\c_{2}}S(\b)_{\c_{1}}}, ~ \a={^{(\c_{2})}\b^{(\c_{1})}}
\]
\end{proof}

To finish our strategy (i), we are reduced to consider
the case of ordinary multisegments.

\begin{prop}\label{prop: 4.2.2}
Let $\b$ be an ordinary multisegment, then there exists a symmetric multisegment $\b^{\sym}$, and
a multisegment $\c$ such that 
such that 
\[
 \b^{\sym}\in S(\b^{\sym})_{\c},~ \b=\b^{\sym, ~(\c)}.
\]
\end{prop}

\begin{proof}
In general $\b$ is not symmetric, i.e,
we do not have $\min\{e(\Delta): \Delta\in \b\}\geq \max\{b(\Delta): \Delta\in \b\}$.
Let 
\[
 \b=\{\Delta_{1}, \cdots, \Delta_{r}\},  \quad b(\Delta_{1})> \cdots >b( \Delta_{r}).
\]
so that
\[
 b(\Delta_{1})=\max\{b(\Delta_i): i=1, \cdots, r\}.
\]
If $\b$ is not symmetric, let
$\Delta^{1}=[\ell, b(\Delta_{1})-1] $ with $\ell$ maximal satisfying 
that for any $m$ such that $\ell-1\leq m\leq b(\Delta_{1})$, there is a segment in $\b$ starting in $m$.
We construct $\b_{1}$ by replacing every segment $\Delta$ in 
$\b$ ending in $\Delta^{1}$ by $^{+}\Delta$.
Repeat this construction with $\b_{1}$ to get $\b_{2}\cdots $,
until we get $\b^{\sym}=\b_{s}$, which is symmetric.
Let $\c=\{\Delta^{1}, \cdots, \Delta^{s}\}$, then as before, we have 
\[
 \b^{\sym}\in {_{\c}S(\b^{\sym})}, ~ \b= ~^{(\c)}(\b^{\sym}).
\]
\end{proof}

As a corollary, we know that 

\begin{cor}\label{cor: 4.2.3}
For any multisegment $\a$, we can 
find a symmetric multisegment $\a^{\sym}$ and three multisegments $\c_{i}, i=1, 2, 3$, such that 
\[
 \a^{\sym}\in {_{\c_{2}, \c_{3}}S(\a^{\sym})_{\c_{1}}}, ~
 \a={^{(\c_{2}, \c_{3})}\a^{\sym, (\c_{1})}}.
\]
\end{cor}

Now applying proposition \ref{prop: 3.2.17}
\begin{prop}\label{prop: 4.2.4}
The morphism 
\[
{_{\c_{2}, \c_{3}}\psi_{\c_{1}}}: {_{\c_{2}, \c_{3}}S(\a^{\sym})_{\c_{1}}}
 \rightarrow S(\a)
\]
is bijective, and  for $\b\in S(\a)$, there exists a unique $\b^{\sym}\in S(\a^{\sym})$ such that
\[
 m(\b, \a)=m(\b^{\sym}, \a^{\sym}).
\]
\end{prop}

\section{Examples }

In this section we shall give some examples to illustrate the idea of 
reduction to symmetric case. 

We first take $\a=\{[1], [2], [2], [3]\}$ to show how to reduce 
a general multisegment to an ordinary multisegment. 
The procedure is showed in the following picture. 

\begin{figure}[!ht]
\centering
\includegraphics{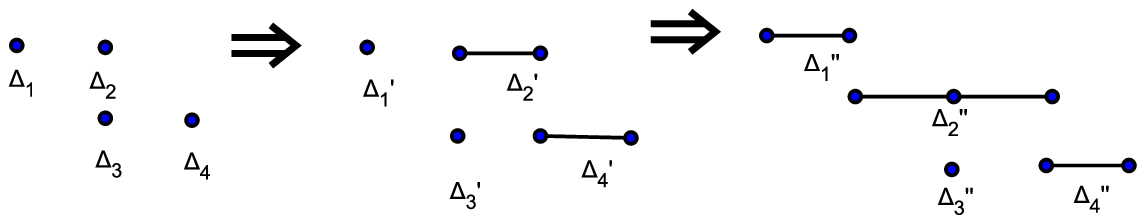}
\caption{\label{fig-multisegment4} }
\end{figure} 
Here we have $\a_{2}=\{[0, 1], [1,3], [2], [3, 4]\}$, 
such that 

\[
 \a_{2}\in {_{[0,1]}S(\a_{2})_{[3,4]}}, ~\a={^{([0,1])}\a_{2}^{([3,4])}}
\]

Next, we reduce the ordinary multisegment $\a_{2}$ to 
a multisegment $\a^{\sym}$, as is showed in the 
following picture.

\begin{figure}[!ht]
\centering
\includegraphics{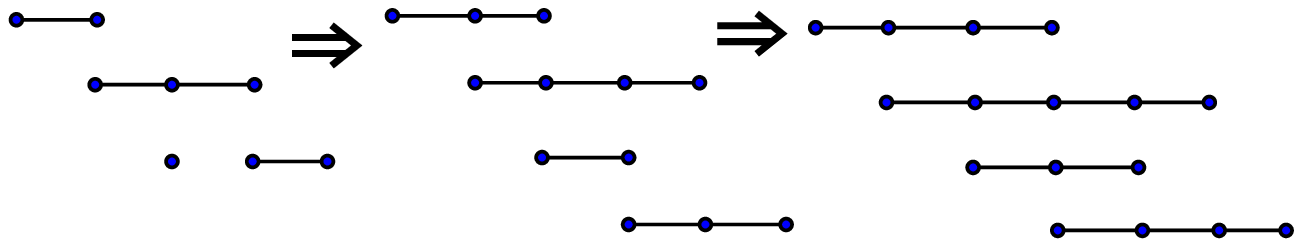}
\caption{\label{fig-multisegment5} }
\end{figure} 

Here,we have 
\[
 \a^{\sym}=\{[0, 3], [1, 5], [2, 4], [3, 6]\}=\Phi(w)
\]
where $w=\sigma_{2}\in S_{4}$.

Now we take $\b=\{[1, 2], [2, 3]\}$, we want to find 
$\b^{\sym}\in S(\a^{\sym})$ such that $m(\b, \a)=m(\b^{\sym}, \a^{\sym})$. 
Actually, following the procedure in Figure 2 above, we have 

\begin{figure}[!ht]
\centering
\includegraphics{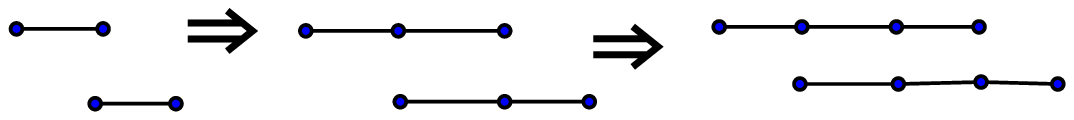}
\caption{\label{fig-multisegment6} }
\end{figure} 

Here we get $\b_{2}=\{[0, 3], [1], [2, 4]\}$. Again, follow the procedure
in Figure 3 above gives

\begin{figure}[!ht]
\centering
\includegraphics{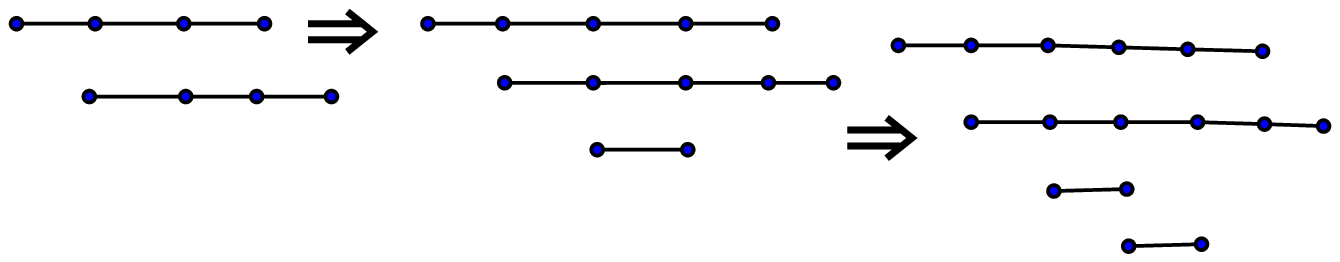}
\caption{\label{fig-multisegment7} }
\end{figure} 

Hence we get  
\[
 \b^{\sym}=\{[0, 5], [1, 3], [2, 6], [3, 4]\}=\Phi(v)
\]
with $v=(13)(24)\in S_{4}$. From \cite{Z2} section 11.3, we know
that $m(\b, \a)=2$, hence we get $m(\b^{\sym}, \a^{\sym})=2$.

\remk
We showed in section 2 that 
\[
 m(\b^{\sym}, \a^{\sym})=P_{v, w}(1),
\]
where $P_{v,w}(q)$ is the Kazhdan Lusztig polynomial associated to $v, w$.
One knows that $P_{v,w}(q)=1+q$, hence $P_{v,w}(1)=2$.

As we have seen, to each multisegment, we have (at least) two different 
ways to attach a Kazhdan Lusztig polynomial:

(1)To use the Zelevinsky construction as described in 
section 4.2. \\
(2)To first construct an associated symmetric multisegment, 
and then attach the corresponding Kazhdan Lusztig polynomial. 

\vspace{0.5cm}
\remk 
In general, for $\a>\b$, (1) gives a polynomial $P_{\a, \b}^{Z}$ which is a
Kazhdan Lusztig polynomial for the symmetric group $S_{\deg(\a)}$. And (2)
gives a polynomial $P_{\a, \b}^{S}$, which is a KL polynomial for a symmetric 
group $S_{n}$ with $n\leq \deg(\a)$. 
It may happen that $n=\deg(\a)$.
By corollary \ref{cor: 4.6.16}, we always have $P_{\a, \b}^{Z}=P_{\a, \b}^{S}$.

\begin{example}
Consider $\a=\{1, 2, 2, 3\}, \b=\{[1,2], [3,4]\}$, then by 
\cite{Z3} section 3.4, we know that $P_{\a, \b}^{Z}=1+q$.
And the symmetrization of $\a$ and $\b$ are given by
\[
 \a^{\sym}=\Psi((2,3)), \quad \b^{\sym}=\Psi((1,3)(2,4)).
\]
Hence $P_{\a, \b}^{S}=P_{(2,3), (1,3)(2,4)}=1+q$,
which is 
the Kazhdan Lusztig polynomial for the pair $((2,3), (1,3)(2,4))$ in $S_{4}$
\end{example}

\section{Proof of the Zelevinsky's conjecture}
 \begin{definition}
 The relation type between 2 segments $\{\Delta, \Delta'\}$ is one of the following
 \begin{itemize}
  \item $\Delta$ cover $\Delta'$ if $\Delta\supseteq \Delta'$;
 \item linked  but not juxtaposed if $\Delta$ does not cover $\Delta'$ and $\Delta\cup \Delta'$ is a segment but $\Delta\cap \Delta'\neq \emptyset$;
 \item juxtaposed if $\Delta\cup \Delta'$ is a segment but $\Delta\cap \Delta'=\emptyset$;
 \item unrelated if $\Delta\cap \Delta'=\emptyset$ and $\Delta, \Delta'$ are not linked.
 \end{itemize}

 \end{definition}

\begin{definition}
Two multisegments 
$$\a=\{\Delta_1,  \cdots,  \Delta_r\} \qquad \hbox{and} \qquad 
\a'=\{\Delta_1', \cdots, \Delta'_{r'}\}$$ 
have the same relation type if 
\begin{itemize}
\item $r=r'$;
\item there exists a bijection 
\[
\xi: \a\rightarrow \a'
\]
 of multisets which preserves the partial order $\preceq$ and relation type of segments and induces bijection of multisets 
 \[
 e(\xi): e(\a)\rightarrow e(\a'), \quad b(\xi): b(\a)\rightarrow b(\a'). 
 \]
\end{itemize} 
satisfying 
\[
e(\xi)(e(\Delta))=e(\xi(\Delta)), \quad b(\xi)(b(\Delta))=b(\xi(\Delta)).
\]
\end{definition}

\begin{lemma}
Let $\a$ and $\a'$ be of the same relation type induced by $\xi$. 
Let $\{\Delta_1\preceq \Delta_2\}$ be linked in $\a$.  Denote by $\a_1$($\a_1'$, resp.) the multisegment obtained by applying the elementary
operation to $\{\Delta_1, \Delta_2\}$( $\{\xi(\Delta_1), \xi(\Delta_2)\}$, resp.). 
Then $\a_1$ and $\a_1'$
also have the same relation type.
\end{lemma}
\begin{proof}
We define a bijection
\[
\xi_1: \a_1\rightarrow \a_1'
\]
by
\[
\xi_1(\Delta_1\cup \Delta_2)=\xi(\Delta_1)\cup \xi(\Delta_2), \quad \xi_1(\Delta_1\cap \Delta_2)=\xi(\Delta_1)\cap \xi(\Delta_2)
\]
and 
\[
\xi_1(\Delta)=\xi(\Delta), \quad \text{ for all } \Delta\in \a\setminus \{\Delta_1, \Delta_2\}.
\]
It induces a bijection between the end multisets $e(\a_1)$ and $e(\a_1')$ as well as the beginning multisets 
$b(\a_1)$  and $b(\a_1')$.
Also the morphism $\xi$ preserves the partial order follows from the fact that for 
$x, y\in e(\a)$ such that $x\leq y$,  then $e(\xi_1)(x)=e(\xi)(x)\leq e(\xi_1)(y)=e(\xi)(y)$( The same fact holds for  $b(\xi_1)$).
Finally, it remains to show that $\xi_1$ respects the relation type. 
Let $\Delta\preceq \Delta'$ be two segments in $\a_1$, if non of them is contained in $\{\Delta_1\cup \Delta_2, \Delta_1\cap \Delta_2\}$, 
then $\xi_1(\Delta)=\xi(\Delta)$ and $\xi_1(\Delta')=\xi(\Delta')$  and they are in the same relation type as $\{\Delta, \Delta'\}$ by assumption.
For simplicity, we only discuss the case where $\Delta=\Delta_1\cup \Delta_2$ 
but $\Delta'$ is not contained in $\{\Delta_1\cup \Delta_2, \Delta_1\cap \Delta_2\}$, other cases are similar. 
\begin{itemize}
\item If $\Delta'$ cover $\Delta$, then $\Delta$ cover $\Delta_1$ and $\Delta_2$,  hence
$\xi_1(\Delta)=\xi(\Delta)$ cover $\xi(\Delta_1)$ and $\xi(\Delta_2)$, which implies $\xi_1(\Delta')$ covers $\xi_1(\Delta)$.
\item If $\Delta'$ is linked to $\Delta$ but not juxtaposed, then either $\Delta'$ covers $\Delta_2$ and linked to $\Delta_1$, 
or $\Delta'$ is linked to $\Delta_2$ but not juxtaposed. In both cases we have $\xi(\Delta')$ is linked to $\xi(\Delta_1)\cup \xi(\Delta_2)$ and 
not juxtaposed. 
\item If $\Delta'$ is juxtaposed to $\Delta$, then $\Delta'$ is juxtaposed to $\Delta_2$ since $\Delta_2\succeq \Delta_1$. 
Therefore $\xi(\Delta')$ is juxtaposed to $\xi(\Delta_2)$ which implies $\xi_1(\Delta')$ is juxtaposed to the segment $\xi_1(\Delta)$.
\item If $\Delta'$ is unrelated to $\Delta_1\cup \Delta_2$, then it is unrelated to both $\Delta_1$ and $\Delta_2$ with $\Delta_2\preceq \Delta'$, 
this implies that $\xi(\Delta')$ is unrelated to $\xi(\Delta_1)\cup \xi(\Delta_2)$.  
\end{itemize}

\end{proof}

\remk As every element $\b\in S(\a)$ is obtained from $\a$ by a sequence of elementary operations,
we can define an application of poset
$$
\Xi: S(\a)\longrightarrow S(\a').
$$

\begin{lemma}
 The application $\Xi$ is well defined and bijective.
 \end{lemma}
 \begin{proof}
 We give a new definition of $\Xi$ in the following way. For $\b\in S(\a)$, we define
 \[
 \Xi(\b)=\{[b(\xi)(b(\Delta)), e(\xi)(e(\Delta))]: \Delta\in \b\}
 \]
 such a definition is independent of the choice of elementary operations.
 It remains to see that it  coincides with the one using elementary operation. 
 In fact, let $\a_1$ be a multisegment obtained by applying the elementary operation to 
 the pair of segments $\{\Delta_1\preceq \Delta_2\}$, then by our original definition of $\Xi$, it
sends $\a_1$ to $\a_1'$ in the previous lemma. Now 
 by the new definition, we have  $\Xi(\a_1)$ given by
 \[
\{\xi(\Delta):   \Delta\in \a\setminus \{\Delta_1, \Delta_2\}\}\cup \{[b(\xi)(b(\Delta_1)), b( \xi)(b(\Delta_2))], [b(\xi)(b(\Delta_2)), b( \xi)(b(\Delta_1))]\}.
 \]
 By our definition of $\xi$, we get 
 \[
 [b(\xi)(b(\Delta_1)), b( \xi)(b(\Delta_2))]=\xi(\Delta_1)\cup \xi(\Delta_2),
 \]
 and 
 \[
 [b(\xi)(b(\Delta_2)), b( \xi)(b(\Delta_1))]=\xi(\Delta_1)\cap \xi(\Delta_2).
 \]
 
 Hence we conclude that $\Xi$ is well defined.  
 Note that by our definition, $\xi$ is invertible, which gives $\xi^{-1}$, and in the same way we can construct $\Xi^{-1}$.
 Now we have 
 \[
 \Xi\Xi^{-1}=\Id, \quad \Xi^{-1}\Xi=\Id
 \]
 by our definition above using $b(\xi)$ and $e(\xi)$.  This shows that $\Xi$ is bijective.

 \end{proof}

%
%

\begin{teo}\label{teo: 4.4.5}
For $\a$ and $\a'$ having the same relation type,  then for $\b\in S(\a)$ with $\b'=\Xi(\b)$, we have 
$$
m(\b, \a)=m(\b', \a').
$$
\end{teo}

\begin{proof}
First of all, we consider the case where $\a$ and $\a'$ are symmetric multisegments.
Let $\a=\Phi(w)$ by fixing a map
\[
\Phi: S_n\rightarrow S(\a_{\Id}).
\]
Now since $\a$ and $\a'$ have the same relation type, we know that $\a'=\Phi'(w)$
for some fixe map 
\[
\Phi': S_n\rightarrow S(\a'_{\Id}).
\]
Finally,  let $\a=\{\Delta_1, \cdots, \Delta_n\}$ and $\a'=\{\Delta_1', \cdots, \Delta_n'\}$ such that 
\[
b(\Delta_1)<\cdots< b(\Delta_n),  \quad \Delta_i'=\xi(\Delta_i).
\]
Without loss of generality, we assume that $b(\Delta_1)=b(\Delta_1')$. 
We can assume that $b(\Delta_i)=b(\Delta_{i-1})+1$. In fact, if 
$b(\Delta_i)>b(\Delta_{i-1})+1$, then by replacing $\Delta_i$ by $^{+}\Delta_i$ , we get a new symmetric multisegment $\a_1$
which has the same relation type as $\a$. Moreover, let $\b\in S(\a)$ and $\b_1$ be the corresponding multisegment in $S(\a_1)$, then 
\[
m(\b, \a)=m(\b_1, \a_1)
\]
by proposition \ref{cor: 3.2.3}. 
We note that the equality
\[
m(\b_1, \a_1)=m(\b', \a')
\]
implies that
\[
m(\b', \a')=m(\b, \a).
\]
Therefore it suffices to prove the theorem for $\a_1$ and $\a'$.
From now on, let $b(\Delta_i)=b(\Delta_{i-1})+1$ and $b(\Delta_i)=b(\Delta_i')$. 
The same argument shows that we can furthermore assume that 
\[
e(\Delta_{w^{-1}(i)})=e(\Delta_{w^{-1}(i-1)})+1, \quad e(\Delta'_{w^{-1}(i)})=e(\Delta'_{w^{-1}(i-1)})+1.
\]
Now if $e(\Delta_{w^{-1}(1)})<e(\Delta'_{w^{-1}(1)})$, 
then consider the truncation functor  $\a'\mapsto \a'^{(e(\Delta_{w^{-1}(1)})+1, \cdots, e(\Delta_{w^{-1}(1)}))}$, 
the latter is a symmetric multisegment having the same relation type as $\a'$, and 
\[
m(\b', \a')=m(\b'^{(e(\Delta_{w^{-1}(1)})+1, \cdots, e(\Delta_{w^{-1}(1)}))}, \a'^{(e(\Delta_{w^{-1}(1)})+1, \cdots, e(\Delta_{w^{-1}(1)}))})
\] 
by proposition \ref{prop: 3.2.17}. Repeat the same procedure, 
in finite step, we find $\c$, such that 
\[
\a=\a'^{(\c)}
\]
and 
\[
m(\b, \a)=m(\b', \a'). 
\]
by  proposition \ref{prop: 3.2.17}.

\remk an interesting application of this computation is given in the corollary \ref{coro-sym1}.

For general case, note that in section 4.4, we construct a symmetric multisegment $\a^{\sym}$  and three multisegments $\c_i, i=1, 2, 3$
such that 
\[
 \a^{\sym}\in {_{\c_{2}, \c_{3}}S(\a^{\sym})_{\c_{1}}}, ~
 \a={^{(\c_{2}, \c_{3})}\a^{\sym, (\c_{1})}}.
\]
(cf. Corollary \ref{cor: 4.2.3}).  
The same for $\a'$, we have 
\[
 \a'^{\sym}\in {_{\c'_{2}, \c'_{3}}S(\a'^{\sym})_{\c'_{1}}}, ~
 \a'={^{(\c'_{2}, \c'_{3})}\a'^{\sym, (\c'_{1})}}.
\]

\end{proof}

\begin{lemma}
The two multisegment $\a^{\sym}$ and $\a'^{\sym}$ have the same relation type. And let 
$\Xi^{\sym}: S(\a^{\sym})\rightarrow \a'^{\sym}$ be the bijection constructed above, then we have 
the following commutative diagram
\begin{displaymath}
\xymatrix{
{_{\c_{2}, \c_{3}}S(\a^{\sym})_{\c_{1}}}\ar[r]^{\Xi^{\sym}}\ar[d]^{{_{\c_{2}, \c_{3}}\psi_{\c_{1}}}}& {_{\c'_{2}, \c'_{3}}S(\a'^{\sym})_{\c'_{1}}}\ar[d]^{_{\c'_{2}, \c'_{3}}\psi_{\c'_{1}}}\\
S(\a)\ar[r]^{\Xi}& S(\a').
}
\end{displaymath}
\end{lemma}

Admitting the lemma, we have 
\[
m(\b, \a)=m(\b^{\sym}, \a^{\sym}), \quad m(\b', \a')=m(\b'^{\sym}, \a'^{\sym})
\]
by proposition \ref{prop: 4.2.4}. Now by what we have proved before and the above lemma, we have 
\[
m(\b^{\sym}, \a^{\sym})=m(\b'^{\sym}, \a'^{\sym}),
\]
which implies $m(\b, \a)=m(\b', \a')$.

\begin{proof}
Note that by construction we know that the number of segments in $\a^{\sym}$ is the same as that of $\a$. 
Let $\a^{\sym}=\{\Delta_1\preceq \cdots \preceq  \Delta_r\}$, then 
$\a=\{{^{(\c_{2}, \c_{3})}\Delta_1^{ (\c_{1})}}\preceq \cdots \preceq {^{(\c_{2}, \c_{3})}\Delta_r^{(\c_{1})}}\}$.
Also let $\a'^{\sym}=\{\Delta'_1\preceq \cdots \preceq  \Delta'_r\}$.
We define 
\begin{align*}
\xi^{\sym}:& \a^{\sym}\rightarrow \a'^{\sym}\\
  & \Delta_i \mapsto\Delta'_i.
\end{align*}
This automatically induces bijections
\[
e(\xi^{\sym}): e(\a^{\sym})\rightarrow e(\a'^{\sym}), \quad b(\xi^{\sym}): b(\a^{\sym})\rightarrow b(\a'^{\sym}),
\]
since all of them are sets.  Note that we definitely have 
\[
\xi({^{(\c_{2}, \c_{3})}\Delta_i^{ (\c_{1})}})={^{(\c'_{2}, \c'_{3})}\Delta_i'^{ (\c'_{1})}}
\]

It remains to show that $\xi^{\sym}$ preserve the relation type.  
Let $i\leq j$.  Then $\Delta_i$ and $\Delta_j$ are linked if and only if one of the following happens 
\begin{itemize}
\item ${^{(\c_{2}, \c_{3})}\Delta_i^{ (\c_{1})}}$ and ${^{(\c_{2}, \c_{3})}\Delta_j^{ (\c_{1})}} $ are linked, juxtaposed or not;
\item ${^{(\c_{2}, \c_{3})}\Delta_i^{ (\c_{1})}} $ and ${^{(\c_{2}, \c_{3})}\Delta_j^{ (\c_{1})}}$ are unrelated.
\end{itemize}

And $\Delta_j$ covers $\Delta_i$ if and only if ${^{(\c_{2}, \c_{3})}\Delta_j^{ (\c_{1})}}$ covers ${^{(\c_{2}, \c_{3})}\Delta_i^{ (\c_{1})}}$. 
Since $\xi$ preserves relation types, this shows that $\xi^{\sym}$ also preserves relation types.  Hence we conclude that 
$\a^{\sym}$ and $\a'^{\sym}$ have same relation type.  To see that the map $\Xi^{\sym}$ sends ${_{\c_{2}, \c_{3}}S(\a^{\sym})_{\c_{1}}}$ to 
${_{\c'_{2}, \c'_{3}}S(\a'^{\sym})_{\c'_{1}}}$, consider $\b\in S(\a)$ and its related element $\b^{\sym}\in {_{\c_{2}, \c_{3}}S(\a^{\sym})_{\c_{1}}}$. 

\medskip

- First of all, we assume that $l(\b)=1$, i.e. $\b$ can be obtained from $\a$ by applying the elementary operation to the pair 
$\{{^{(\c_{2}, \c_{3})}\Delta_i^{ (\c_{1})}}, {^{(\c_{2}, \c_{3})}\Delta_j^{ (\c_{1})}}\}(i<j)$.  Let 
$\tilde{\b}$ be the element in $S(\a^{\sym})$ obtained by applying the elementary operation to the pair of segments $\{\Delta_i, \Delta_j\}$ in $\a^{\sym}$. 
Then we have 
\[
\b={^{(\c_{2}, \c_{3})}\tilde{\b}^{(\c_{1})}}.
\]
Let $\tilde{\b}'=\Xi^{\sym}(\tilde{\b})$.  By construction, we have
\[
\b'=\Xi(\b)={^{(\c'_{2}, \c'_{3})}\tilde{\b}'^{(\c'_{1})}}.
\]
Now consider 
\[
\tilde{\b}_0=\tilde{\b}>\cdots >\tilde{\b}_n=\b^{\sym}
\]
be a maximal chain of multisegments and let $\tilde{\b}'_i=\Xi^{\sym}(\tilde{\b}'_i)$, then  
\[
\tilde{\b}'_0>\cdots >\tilde{\b}'_n.
\]
Let 
\[
\tilde{\b}_{i}=\{\Delta_{i, 1}\preceq \cdots \preceq \Delta_{i, r_i}\},\quad \tilde{\b}'_{i}=\{\Delta_{i, 1}'\preceq \cdots \preceq \Delta_{i, r_i}'\}.
\]
We prove by induction that 
\[
\b'={^{(\c'_{2}, \c'_{3})}\tilde{\b}_i'^{(\c'_{1})}}.
\]
We already showed the case where $i=0$. Assume that we have 
\[
\b'={^{(\c'_{2}, \c'_{3})}\tilde{\b}_j'^{(\c'_{1})}}
\]
for $j<i$.  Suppose that $\tilde{\b}_i$ is obtained from $\tilde{\b}_{i-1}$ by applying the elementary operation to 
the pair of segments $\{\Delta_{i-1, \alpha_{i-1}}\preceq \Delta_{i-1, \beta_{i-1}}\}$.  
We deduce from the fact $\tilde{\b}_i\geq \b^{\sym}$ that we are in one of the following situatios
\begin{itemize}
\item  ${^{(\c_{2}, \c_{3})}\Delta_{i-1, \alpha_{i-1}}^{(\c_{1})}}=\emptyset$ or ${^{(\c_{2}, \c_{3})}\Delta_{i-1, \beta_{i-1}}^{(\c_{1})}}=\emptyset$;
\item  $b({^{(\c_{2}, \c_{3})}\Delta_{i-1, \beta_{i-1}}^{(\c_{1})}})=b({^{(\c_{2}, \c_{3})}\Delta_{i-1, \alpha_{i-1}}^{(\c_{1})}})$;
\item $e({^{(\c_{2}, \c_{3})}\Delta_{i-1, \beta_{i-1}}^{(\c_{1})}})=e({^{(\c_{2}, \c_{3})}\Delta_{i-1, \alpha_{i-1}}^{(\c_{1})}})$.
\end{itemize}
According the our assumption that $\tilde{\b}'_i=\Xi^{\sym}(\tilde{\b}'_i)$, we have 
\[
\xi({^{(\c_{2}, \c_{3})}\Delta_{i-1, j}^{(\c_{1})}})={^{(\c_{2}, \c_{3})}\Delta_{i-1, j}'^{(\c_{1})}},
\]
therefore the pair $\{{^{(\c_{2}, \c_{3})}\Delta_{i-1, \alpha_{i-1}}'^{(\c_{1})}}, {^{(\c_{2}, \c_{3})}\Delta_{i-1, \beta_{i-1}}'^{(\c_{1})}}\}$
also satisfies one of the listed properties above.  And this shows that $\tilde{\b}'_i$ is sent to $\b'$ by $_{\c'_{2}, \c'_{3}}\psi_{\c'_{1}}$. 
Therefore by proposition \ref{prop: 4.6.15}, we know that 
\[
\b_n'\geq \b'^{\sym}.
\]
Conversely, we have 
\[
\Xi^{\sym-1}(\b'^{\sym})\geq \b^{\sym}.
\]
Combine the two inequalities to get
\[
\Xi^{\sym}(\b^{\sym})= \b'^{\sym}.
\]

- The general case where $\l(\b)>1$, we can choose a maximal chain of multisegments 
 \[
 \a=\a_0>\cdots >\a_{\l(\b)}=\b.
 \]
 Let $\a_i'=\Xi(\a_i)$, by assumption, we can assume that for $i<\l(\b)$, we have
 \[
 \Xi^{\sym}(\a_i^{\sym})=\a_i'^{\sym}.
 \]
 By considering the set $S(\a_{\l(\b)-1})$, we are reduce to the case where $\l(\b)=1$. Hence we are done.
 \end{proof}
 
\begin{cor} \label{coro-sym1}
Let $\a_{\Id}$ be a symmetric multisegment associated to the identity in $S_n$ and 
\[
\Phi: S_n\rightarrow S(\a_{\Id}).
\] 
Then 
\[
m(\Phi(v), \Phi(w))=P_{w, v}(1).
\]
\end{cor} 
\begin{proof}
The special case where
\[
\a_{\Id}=\sum_{i=1}^{n}[i, i+n-1]
\]
is already treated in corollary \ref{cor: 2.5.9}.
The general case can be deduced from the theorem above.
\end{proof}

\part{Applications }

\chapter{Geometric Proof of KL Relations}

For $n\geq 1$, recall that the permutation group $S_n$ of $\{1, \cdots, n\}$
and that $S=\{\sigma_i=(i, i+1): i=1, \cdots, n-1\}$ is a set of generators.
It is followed from \cite{KL79} that the following properties characterize
a unique family of polynomials $P_{x,y}(q)$ of $\Z[q]$ for $x, y\in S_n$ 
 \begin{description}
  \item[(1)] $P_{x,x}=1$ for all $x\in S_{n}$;
  \item[(2)]if $x<y$ and $s\in S$,  are such that $sy<y$, $sx>x$,
  then $P_{x,y}=P_{sx, y}$;
   \item[(3)]if $x<y$ and $s\in S$, are such that $ys<y$, $xs>x$,
  then $P_{x,y}=P_{xs, y}$;
   \item[(4)]if $x<y$ and $s\in S$, are such that $sy<y$, $sx<x$,
   and $x$ is not comparable to $sy$, 
  then $P_{x,y}=P_{sx, sy}$;
   \item[(5)]if $x<y$ and $s\in S$, are such that $sy<y$, $sx<x$, and $x<sy$,
  then
  \[
   P_{x,y}=P_{sx, sy}+qP_{x, sy}-\sum_{x\leq z <sy,sz<z}q^{1/2(\ell(y)-\ell(z))}\mu(z, sy)P_{x,z},
  \]
here $\mu(z, sy)$ is the coefficient of degree $1/2(\ell(sy)-\ell(z)-1)$ in $P_{z, sy}$ defined to be zero if 
$\ell(sy)-\ell(z)$ is even).
 \end{description}

In this chapter , we shall
prove by using our results in section 3.3 that the polynomial
\[
 P_{x,y}(q): =q^{\frac{1}{2}(\dim(O_{\Phi(y)})-\dim(O_{\Phi(x)}))}\sum_{i}q^{\frac{1}{2}i}\mathcal{H}^{i}(\line{O}_{\Phi(y)})_{\Phi(x)}
\]
satisfies the first 4 conditions and 
we give an interpretation geometric for the fifth condition
which will be used in the Chapter 7.

\remk
The condition $P_{x,x}=1$ is trivial.

The set up for through this chapter is the following.
Assume that $k, k_1\in \N$ such that $1<k_1\leq n, k=n+k_1-1$, and $\a_{\Id}$ be a multisegment such that we have 
an isomorphism
\[
 \Phi: S_n\rightarrow S(\a_{\Id}).
\]
Note that we have $n<k\leq 2n-1$.

\section{Relation (2) and (3)}

Since the relation (2) and (3) are symmetric to each other, 
we only prove (2). By \cite{BF} (1.26), 
the conditions 
\[
 \sigma_{k_1-1}w>w, ~\sigma_{k_1-1}v<v.
\]
are equivalent to 
\[
 w^{-1}(k_1-1)<w^{-1}(k_1), \quad v^{-1}(k_1-1)>v^{-1}(k_1).
\]

\begin{prop}
Let $\a=\Phi(w), \c=\Phi(v)\in S(\a)$, such that 
\[
 w^{-1}(k_1-1)<w^{-1}(k_1), \quad v^{-1}(k_1-1)>v^{-1}(k_1),
\]
then 
\[
 P_{w, v}(q)=P_{\sigma_{k_1-1}w, v}(q).
\]
\end{prop}

\begin{proof}

Suppose that 
\[
 \Phi(\Id)=\{\Delta_1\preceq \cdots \preceq \Delta_n\}.
\]
Let $\b=\Phi(\sigma_{k_1-1}w)$, then 
\begin{align*}
 \b=&\sum_{j}[b(\Delta_j), e(\Delta_{\sigma_{k_1-1}w(j)})]\\
   =&\sum_{j}[b(\Delta_{w^{-1}\sigma_{k_1-1}(j)}), e(\Delta_j)]\\
   =&\sum_{j\neq k_1-1, k_1}[b(\Delta_{w^{-1}(j)}), e(\Delta_j)]+
   [b(\Delta_{w^{-1}(k_1-1)}), e(\Delta_{k_1})]+[b(\Delta_{w^{-1}(k_1)}), e(\Delta_{k_1-1})].
\end{align*}

Note that 
\[
 e(\Delta_{k_1-1})=n+k_1-2=k-1, ~e(\Delta_{k_1})=n+k_1-1=k,
\]

then $\b^{(k)}=\a^{(k)}$.
Now applying the corollary \ref{cor: 4.6.17} 
gives the result.
\end{proof}

\section{Relation (4)}

Let $\a=\Phi(\Id),~ \varphi=\varphi_{\a}$,
As in section 3.3, we know that for fixed $W$, by proposition
\ref{prop: 4.6.14}, 
we have an open immersion
\[
 \tau_{W}:  (X_{\a}^{k})_{W}\rightarrow 
 (Z^{k, \a})_{W}\times \Hom(V_{\varphi, k-1}, W).
\]
\begin{definition}
By composing with the canonical projection $$(Z^{k, \a})_{W}\times \Hom(V_{\varphi, k-1}, W)
\rightarrow (Z^{k, \a})_{W},$$ we have a morphism
\[
 \phi_W:  (X_{\a}^{k})_{W}\rightarrow (Z^{k, \a})_{W}.
\]
\end{definition}

\begin{prop}\label{prop: 4.4.5}
For any $ \b=\Phi(w)\in S(\a)_{k}$, we have 
 \[
  \psi_k^{-1}(\b^{(k)})=\{\b, \b'=\Phi(\sigma_{k_1-1}w)\}.
 \]
Moreover, $\phi_W$ is a fibration such that 
\begin{description}
\item[(1)] We have an isomorphism $\phi_{W}^{-1}(O_{\b^{(k)}})\simeq (\C^2-\{0\})\times \C^{2n-k-1}$.
\item[(2)] We have $\phi_{W}^{-1}(O_{\b^{(k)}})\cap O_{\b'}\simeq \C^{\times}\times \C^{2n-k-1}$.

\end{description}
\end{prop}

\begin{proof}
Note that we have 
\[
 \psi_k^{-1}(\b^{(k)})\subseteq S(\b^{(k)}+[k]),
\]
we observe that 
\[
 S(\b^{(k)}+[k])\cap S(\a)=S(\b'),
\]
Since $\b$ is minimal in $\psi_k^{-1}(\b^{(k)})$ 
(See Prop. \ref{prop: 4.6.15}), we have 
\[
 \psi_k^{-1}(\b^{(k)})=\{\b, \b'\}.
\]

 Then consider the restricted morphism
\[
 \phi_W: (O_{\b}\cup O_{\b'})_{W}\rightarrow O_{\b^{(k)}}.
\]
Let $T\in O_{\b}\cup O_{\b'}, ~T_{0}\in \Hom(V_{\varphi, k-1}, W)$.
Define $T'\in E_{\varphi}$ by
\[
 T'|_{V_{\varphi, k-1}}=T_{0}\oplus T^{(k)}|_{V_{\varphi, k-1}},
 \]
\[
 T'|_{V_{\varphi, k}}=T^{(k)}|_{V_{\varphi, k}/W}\circ p_{W},
\]
\[
 T'|_{V_{\varphi, i}}=T^{(k)}, \text{ for } i\neq k-1, k.
\]
 We know that $\dim(W)=\ell_{k}=1$, and 
for $\dim(\ker(T^{(k)}|_{V_{\varphi, k-1}}))=2$. 
Now let 
\[
 \Delta_{1}<\Delta_{2}
\]
be the two segments in $\b^{(k)}$ which ends in 
$k-1$.
And we consider the following flag
\[
 V_{0}=\ker(T^{(k)}|_{V_{\varphi,k-1}})\supseteq 
 V_{1}=\Im(T^{(k)})^{\Delta_{2}}\cap \ker(T^{(k)}|_{V_{\varphi,k-1}}).
\]
And we have $\dim(V_{1})=1$. 
Then for $T'\in O_{\b}\cup O_{\b'}$, it is necessary and sufficient that 
\[
T_{0}(V_{0})\neq 0.
\]
This amounts to give a nonzero element in $\Hom(V_{0}, W)\simeq \C^2$, 
which proves that the fiber  $\phi_W^{-1}(T^{(k)})\simeq (\C^{2}-{0})\times \C^{2n-k-1}$,
where the factor $\C^{2n-k-1}$ comes from the fact that $\dim(V_{\varphi, k-1})=2n-(k-1)=2n-k+1$. 
As for $T'\in O_{\b'}$, it is necessary and sufficient that 
\[
  T_{0}(V_{1})=0, ~T_{0}(V_{0})\neq 0,
\]
which amounts to give a zero element in $\Hom(V_{0}/V_{1}, W)\simeq \C$.
Hence $\phi_W^{-1}(T^{(k)})\cap O_{\b'}\simeq\C^{\times}\times \C^{2n-k-1}$. To see that 
$\phi_W$ is a fibration,
 fix $V\subseteq V_{\varphi, k-1}$ such that 
 $\dim(V)=2$. Consider the sub-scheme of $Z_{W}^k$ given by 
\[
  U_{V}=\{T\in Z_{W}^k: \ker(T|_{V_{\varphi, k-1}})=V\}.
\]
Note that since $\dim(V_{\varphi, k-1})=\dim(V_{\varphi, k}/W)+2$, the 
fact that $\dim(\ker(T|_{V_{\varphi,k-1}}))=2$ implies that $U_{V}$
is actually open in $Z_{W}^k$. 
In this case 
\[
 \phi_W^{-1}(U_{V})=U_{V}\times (\Hom(V, W)-\{0\})\times \Hom(V_{\varphi, k-1}/V, W).
\]

\end{proof}

\begin{prop}
 Let $\b=\Phi(w), \c=\Phi(v)\in S(\a)$, such that 
\[
 w^{-1}(k_1-1)>w^{-1}(k_1), \quad v^{-1}(k_1-1)>v^{-1}(k_1), \quad w<v, 
\]
and $w$ is not comparable with $\sigma_{k_1-1}v$,
then 
\[
 P_{w, v}(q)=P_{\sigma_{k_1-1}w, \sigma_{k_1-1}v}(q).
\]
\end{prop}

\remk As before, our conditions are equivalent to 
\[
 \sigma_{k_1-1}w>w, ~ \sigma_{k_1-1}v>v.
\]

\begin{proof}
Note that our assumption implies that both $\b$ and $\c$ are in 
$S(\a)_{k}$. 
Let $b'=\Phi(\sigma_{k_1-1}w), \c'=\Phi(\sigma_{k_1-1}v)
$. Then $\b'>\c'$.

For $\b>\d>\c$, we must have $\d=\Phi(\alpha)$ with $\sigma_{k_1-1}\alpha<\alpha$.
In fact,  $\sigma_{k_1-1}\alpha>\alpha$ would imply $\d>\c'$ by 
lifting property of Bruhat order (cf. \cite{BF} proposition 2.2.7). 
Now that we have $\b>\d>\c'$, contradicting to our assumption that $\b$
is not comparable to $\c'$.
Let $\d'=\Phi(\sigma_{k_1-1}\alpha)$. Note that we create actually by this way 
construct a morphism between the sets 
$$\rho: \{\d: \b\geq \d\geq \c\}\rightarrow \{\d': \b'\geq \d'\geq \c'\}$$
sending $\d$ to $\d'$.
\begin{lemma}
The morphism $\rho$ is a bijection.  
\end{lemma}

\begin{proof}
Let $\e'=\Phi(\beta)\in S(\b')$ with $\e'>\c'$. 
We show that $\sigma_{k_1-1}\beta>\beta$.
In fact, assume that $\sigma_{k_1-1}\beta<\beta$.
Then the lifting property of Bruhat order implies 
$\b>\e'>\c'$, which is a contradiction to the fact that 
$\b$ is not comparable to $\c'$.
Hence we have $\e=\Phi(\sigma_{k_1-1}\beta)<\e'$.
Moreover, since $\sigma_{k_1-1}w<\beta<\sigma_{k_1-1}v$,
and $w>\sigma_{k_1-1}w,~ v>\sigma_{k_1-1}v$, we have 
\[
 w<\sigma_{k_1-1}\beta<v, 
\]
hence $\b>\e>\c$. This proves the surjectivity.
The injectivity is clear from the definition.

\end{proof}

As a corollary, we have 

\begin{lemma}
 The restricted morphism 
 \[
  \phi_W: X_{\b', \c'}^{k}\rightarrow Z_{b^{(k)}, \c^{(k)}}^k(cf. \text{ Def. \ref{def: 3.3.6}})
 \]
is a fibration with fibers isomorphic to $\C^{\times}\times \C^{n-k}$.
\end{lemma}

\begin{proof}
Since $\phi_W$ is a composition of 
$\tau_W$, which is an open immersion,  and a canonical
projection, to show that it is a fibration, it suffices
to show that all of its fibers are isomorphic to $\C^{\times}\times \C^{n-k}$.
This follows from proposition \ref{prop: 4.4.5}
and the fact that for any $\d'\in S(\b')$ we have $\d'\notin S(\a)_{k}$.
\end{proof}

Hence we get
\[
 P_{\b', \c'}(q)=P_{\b^{(k)}, \c^{(k)}}(q).
\]
Now we are done by applying corollary \ref{cor: 4.6.17},
i.e,
\[
 P_{\b^{(k)}, \c^{(k)}}(q)=P_{\b, \c}(q).
\]
Hence we are done.
\end{proof}

\section{Relation (5)}

Finally, we arrive at the relation (5).
We will give an interpretation of this relation in terms of 
the decomposition theorem (See \cite{BBD}).

\begin{definition}
Let 
\begin{align*}
 \mathfrak{Z}_{W}=&\{(T,z)\in Z^{k, \a}_{W}\times 
 \Hom(V_{\varphi, k-1}^*, W^*):  \text{ and }z \\
 &\text{ factors through the canonical projection  } 
 V_{\varphi, k-1}^*\rightarrow \ker(T|_{V_{\varphi, k-1}})^*\}.
\end{align*}
\end{definition}

\begin{prop}\label{prop: 4.4.10}
 The canonical projection $\mathfrak{Z}_{W}\rightarrow Z^{k, \a}_{W}$
 turns $\mathfrak{Z}_{W}$ into a vector bundle of rank 2 over $Z_{W}^k$. 
\end{prop}
 
 \begin{proof}
 Note that we have $\dim(\ker(T|_{V_{\varphi, k-1}}))=2$ and 
 $\dim(W)=1$. 
 Note that by taking dual,  as a scheme,  $\mathcal{Z}_{W}$ is isomorphic 
 to the scheme parametrize the data $(T, z)\in Z_{W}^k\times V_{\varphi, k-1}$
 such that $z\in \ker(T|_{V_{\varphi, k-1}})$.
 Fix $V\subseteq V_{\varphi, k-1}$ such that 
 $\dim(V)=2$. Consider the sub-scheme of $Z_{W}^k$ given by 
\[
  U_{V}=\{T\in Z_{W}^k: \ker(T|_{V_{\varphi, k-1}})=V\}.
\]
As is showed in proposition \ref{prop: 4.4.5},  $U_{V}$
is actually open in $Z_{W}^k$. 
Using the previous interpretation of $\mathfrak{Z}_{W}^k$, we observe that 
the open set $U_{V}$ trivializes the projection $\mathfrak{Z}_{W}\rightarrow Z^{k, \a}_{W}$.

\end{proof}

\begin{definition}
Let $\mathcal{Z}_{W}^k= Proj_{Z^{k, \a}_W}(\mathfrak{Z}_W)$
be the projectivization of the vector bundle $\mathfrak{Z}_{W}\rightarrow Z_{W}^k$.
And we shall denote the structure morphism by $\kappa_{W}^k: \mathcal{Z}_{W}^k\rightarrow Z_{W}^k $.
\end{definition}

\begin{definition}
From now on, we fix a pair of non-degenerate bi-linear forms
\[
 \zeta_{k-1}: V_{\varphi, k-1}\times V_{\varphi, k-1} \rightarrow \C, 
 ~\zeta_{k}: V_{\varphi, k}\times V_{\varphi, k} \rightarrow \C.
\]
which allows us to have an identification $\eta_{i}: V_{\varphi, i}\simeq V_{\varphi, i}^*$,
for $i=k-1, k$. 
\end{definition} 
 
\remk Here our definition $X_{\a}^k$ depends on the choice of $V_{\varphi}$.
If we choose  $V_{\varphi}'$ such that 
\[
 V_{\varphi,i}'=V_{\varphi, i}, \text{ for } i\neq k-1, k, ~
 V_{\varphi, k-1}'=V_{\varphi, k-1}^*, ~V_{\varphi, k}'=V_{\varphi, k}^*,
\]
we can get $X_{\a}^{k}(V_{\varphi}')$, which is isomorphic to 
$X_{\a}^k$ after we choose an isomorphism $V_{\varphi, k-1}^*\simeq V_{\varphi, k-1}$
and $V_{\varphi, k}^*\simeq V_{\varphi, k}$. This is what we do here. 
Note that once we fix $V_{\varphi, k-1}^*\simeq V_{\varphi, k-1}$ and 
$V_{\varphi, k}^*\simeq V_{\varphi, k}$. Our morphism $\eta_{i}$ will 
become an inner automorphism, but in general we have $\eta_k(W)=W^*\neq W$.

\begin{definition}
Let $T\in (X_{\a}^k)_{W}$, then we define 
\[
 \lambda: (X_{\a}^k)_{W}\rightarrow (X_{\a}^k)_{\eta_k(W)},
\]
by letting
\[
 \lambda(T)|_{V_{\varphi, k-2}}=\eta_{k-1}\circ T|_{\varphi, k-2} 
\]

\[
 \lambda(T)|_{V_{\varphi, k-1}}=\eta_k\circ T|_{\varphi, k-1}\circ \eta_{k-1}^{-1},
\]
\[
 \lambda(T)_{V_{\varphi, k}}=T|_{\varphi, k}\circ \eta_k^{-1},
\]
and 
\[
 \lambda(T)_{V_{\varphi, i}}=T|_{\varphi, i}, \text{ for } i\neq k-2, k-1, k.
\]
\end{definition}

\begin{lemma}
We have 
$\ker(\lambda(T)|_{V_{\varphi, k}})=\eta_k(W)$,
and $$\ker(\lambda(T)^{(k)}|_{V_{\varphi, k-1}})=\eta_{k-1}(\ker(T|_{V_{\varphi,k-1}})).$$ 
\end{lemma}

\begin{proof}
The fact  $\ker(\lambda(T)|_{V_{\varphi, k}})=\eta_k(W)$ follows from definition.
Note that 
\[
 \ker(T^{(k)}|_{V_{\varphi, k-1}})=\{v\in V_{\varphi, k-1}:  T(v)\in W\}=T|_{V_{\varphi, k-1}}^{-1}(W).
\]
Since
\[
 (\lambda(T)|_{V_{\varphi, k-1}})^{-1}(\eta_k(W))=\eta_{k-1}(T|_{V_{\varphi, k-1}})^{-1}(W)
 =\eta_{k-1}(\ker(T|_{V_{\varphi,k-1}})),
\]
hence 
\[
  \ker (\lambda(T)^{(k)}|_{V_{\varphi, k-1}})=\eta_{k-1}(\ker(T|_{V_{\varphi,k-1}})).
\]

\end{proof}

\begin{definition}
We define
\[
\xi_W: (X_{\a}^k)_W\rightarrow \mathcal{Z}_{W}^k,
\]
for $T\in (X_{\a}^k)_W$, then 
\[
 \xi_W(T)=(T^{(k)}, \lambda(T)|_{\ker(\lambda(T)^{(k)}|_{V_{\varphi, k-1}})}).
\]
 
This is well defined since  
\[
 \lambda(T)|_{\ker(\lambda(T)^{(k)}|_{V_{\varphi, k-1}})}\in\Hom(\ker(\lambda(T)^{(k)}|_{V_{\varphi, k-1}}),\eta_k(W) ),
\]
and
$$
 \Hom(\ker(\lambda(T)^{(k)}|_{V_{\varphi, k-1}}),\eta_k(W) )
\simeq \Hom(\ker(T^{(k)}|_{V_{\varphi, k-1}})^*, W^*),
$$ 
and $\lambda(T)|_{\ker(\lambda(T)^{(k)}|_{V_{\varphi, k-1}})}\neq 0$.
 
\end{definition}

\begin{prop}
 The morphism $\xi_{W}$ is a fibration with fibers isomorphic to $\C^{\times}\times \C^{n-k}$.
\end{prop}

\begin{proof}
Let $V\subseteq V_{\varphi, k-1}$ be a subspace such that $\dim(V)=2$. 
Consider the open sub-scheme of $\mathfrak{Z}_{W}$
\[
 U_{1,V}=\{(T, z)\in \mathfrak{Z}_{W}: z\neq 0,~ \ker(T|_{V_{\varphi, k-1}})=\eta_{k-1}^{-1}(V)\}.
 \]
 \[
 U_{V}=\{T\in Z_{W}^k, ~\ker(T|_{V_{\varphi, k-1}})=\eta_{k-1}^{-1}(V)\}.
 \]
Let $\tilde{U}_{1, V}$ be the image of $U_{1, V}$
in $\mathcal{Z}_{W}^k$ by the canonical projection.
As indicated in the proof of  proposition \ref{prop: 4.4.10}, 
the set $U_{V}$ trivialize the morphism $\mathfrak{Z}_{W}\rightarrow Z_{W}^k$,
hence 
\[
 \tilde{U}_{1, V}\simeq U_{V}\times (\Hom(V, \eta_{k}(W))-\{0\}/\C^{\times})
\]

Note that we have 
\[
 \Hom(V, \eta_{k}(W))\simeq \Hom(\eta_{k-1}^{-1}(V), W).
\]
And by proposition \ref{prop: 4.6.14} and proposition \ref{prop: 4.4.5}, 
we have the following isomorphism
\[
 \xi_{W}^{-1}(\tilde{U}_{1, V})\simeq U_{V}\times (\Hom(\eta_{k-1}^{-1}(V), W)-0)\times 
 \Hom(V_{\varphi, k-1}/\eta_{k-1}^{-1}(V), W).
\]
Hence for any $(T, z)\in \mathcal{Z}_{W}^k$ such that 
$\ker(T|_{V_{\varphi, k-1}})=\eta_{k-1}^{-1}(V)$
, let $U_{2, V}$ be an open 
subset of $(\Hom(\eta_{k-1}^{-1}(V), W)-0)/\C$ which trivializes the bundle
\[
 (\Hom(\eta_{k-1}^{-1}(V), W)-0)\rightarrow (\Hom(\eta_{k-1}^{-1}(V), W)-0)/\C,
\]
then the open sub-scheme $U_{V}\times U_{2, V}$ of $\tilde{U}_{V, 1}$ trivialize the 
morphism $\phi_W$ as a neighborhood of $(T, z)$.
 
\end{proof}

\begin{definition}
Let $\b>\c$ be two elements in $\tilde{S}(\a)_{k}$, then we define
\[
 \mathcal{Z}^k_{\b, \c}=\xi_{W}((X^k_{\b, \c})_W).
\]
And 
\[
 \mathcal{Z}^k(\b)=\xi_W((O_{\b})_W).
\]
\end{definition}

\begin{definition}
 Let $w<v$ be two elements in $S_{n}$ such that $\sigma_{k_1-1}v<v$. We define 
 \[
 R(w, v)_{k_1}=\{z: w\leq z<\sigma_{k_1-1} v, \sigma_{k_1-1}z<z\}.
 \]
And we denote $R(\Id, v)_{k_1}$ by $R(v)_{k_1}$. 
\end{definition}
 
Now let $\b=\Phi(w)$, $\c=\Phi(v)$ such that 
\[
 w(k_1-1)>w(k_1),~ v(k_1-1)>v(k_1).
\]
And let $\b'=\Phi(\sigma_{k_1-1}w), \c'=\Phi(\sigma_{k_1-1}v)$. We assume that 
\[
 \b>\c, ~\b>\c',
\]
which coincide with the assumption in relation (5) at the beginning of this chapter.

Now we apply the decomposition theorem to the projective morphism
\[
 \kappa_W: \mathcal{Z}^k_{\b', \c'}\rightarrow Z^{k, \a}_{\b^{(k)}, \c^{(k)}, W}.
\]
which asserts that there exists a finite collection of triples $(\d_{i}, L_{i}, h_{i}: i=1, \cdots, r)$, 
with $\d\in S(\a)_{k}$, ~$\b^{(k)}\leq \d_{i}^{(k)}<\c^{(k)}$, where $L_{i}$ is a vector spaces over $\C$, 
such that 
\addtocounter{theo}{1}
\begin{equation}\label{eq: (4)}
 R(\kappa_W)_{*}IC(\mathcal{Z}_{\b', \c'}^k)=IC(Z^{k, \a}_{\b^{(k)}, \c^{(k)}, W})\oplus_{i=1}^{r}IC(Z^{k, \a}_{\b^{(k)}, \d^{(k)}_{i}, W}, L_{i})[h_{i}].
\end{equation}

Now localize at a point $x_{\b^{(k)}}\in O_{\b^{(k)}}$, we have know that 
the Poincar\'e series of $(IC(Z^{k, \a}_{\b^{(k)}, \c^{(k)},W}))_{x_{\b^{(k)}}}$ is given by 
$P_{\b, \c}(q)=P_{w, v}(q)$. And

\begin{lemma}
The Poincar\'e series of $R\Gamma(\kappa_W^{-1}(x_{\b^{(k)}}), IC(\mathcal{Z}_{\b', \c'}^k))$ 
is given by $P_{\sigma_{k_1-1}w, \sigma_{k_1-1}v}(q)+qP_{w, \sigma_{k_1-1}v}(q)$,
where $\Gamma$ is the functor of taking global sections.
\end{lemma}

\begin{proof}
 Note that by assumption, we have 
 \[
  \kappa_W^{-1}(x_{\b^{(k)}})\simeq \P^1
 \]
such that $ \kappa_W^{-1}(x_{\b^{(k)}})\cap \mathcal{Z}^k(\b')=\{pt\}$ and 
$\kappa_W^{-1}(x_{\b^{(k)}})\cap \mathcal{Z}^k(\b)\simeq \P^1-\{pt\}$.
And we have the following exact sequence 
\[
0\rightarrow  IC(\mathcal{Z}_{\b', \c'}^k)|_{pt}\rightarrow IC(\mathcal{Z}_{\b', \c'}^k)\rightarrow 
 IC(\mathcal{Z}_{\b', \c'}^k)|_{\kappa_W^{-1}(x_{\b^{(k)}})-\{pt\}}\rightarrow 0.
\]
Taking the Poincar\'e series gives the result.

\end{proof}

Now it is clear that our equation (\ref{eq: (4)})
will give rise to an equation of the form as that in (5) 
in the introduction of this chapter.
Comparing the two equations, we get 

\begin{prop}\label{prop: 5.3.12}
The collection of triples $(\d_{i}, L_{i}, h_{i}: i=1, \cdots, r)$ are given by 
\begin{description}
 \item [(1)]We have $\{\d_{i}: i=1, \cdots, r\}=\{z\in R(w, v)_{k}: \mu(z, \sigma_{k-1}v)\neq 0\}$.
 \item [(2)] If $\d_{i}=\Phi(z)$, then $L_{i}\simeq \C^{\mu(z, \sigma_{k-1}v)}$.
 \item [(3)] If $\d_{i}=\Phi(z)$, then $h_i=\l(v)-\l(z)$.
\end{description} 
\end{prop}

\begin{proof}
Note that 
the Poincar\'e series of the intersection complex 
 $IC(Z^{k, \a}_{\b^{(k)}, \d^{(k)}_{i}, W}, L_{i})[h_{i}]$
is 
\[
 \dim(L_i)q^{1/2h_i}P_{w, z_i}(q),
\]
where $\d_i=\Phi(z_i)$. 
Now compare the polynomials given by \ref{eq: (4)}
and the relation (5)  in the beginning of
this chapter, we get our results.
\end{proof}

\remk Note that one should be able to deduce the above results from a general statement about the 
decomposition theorem. We leave this for future work. 

\remk It seems that we have done here 
may be generalized to give the normality of for general $\line{O}_{\b}$
instead of using the results of Zelevinsky.

\chapter{Classification of Poset \texorpdfstring{$S(\a) $}{Lg}}

Let $\a$ be a multisegment and $S(\a)=\{\b\leq \a\}$ the associated poset defined in  
\ref{nota: 1.3.2}. The aim of this chapter is to identify the poset 
structure of $S(\a)$.

In the first section we consider the case where 
$\a$ is ordinary and prove that $S(\a)$  is an 
interval in $S_m\simeq B\backslash GL_m/B$, where 
$m$ is the number of segments in $\a$. and $B$ is the Borel subgroup.

In the general case we identify $S(\a)$ with an interval 
in a parabolic quotient $S_{J_1}\backslash S_m/S_{J_2}$ of $S_m$ 
given in section 2 related to the double quotient $P_{J_1}\backslash GL_m/ P_{J_2}$,
where $P_{J_1}$ and $P_{J_2}$ are parabolic subgroups.

\section{Ordinary Case}

Our goal in this section is to prove that 
for general ordinary multisegment $\a$, 
the set $S(\a)$ is isomorphic to some Bruhat interval $[x, y]$ for 
$x, y\in S_{n}$, where $n$ depends on $\a$.

\begin{lemma}
Assume that $\b\in S(\b)_{k}$ such that $\b$ and $\b^{(k)}$ are both ordinary.
Let $\c\in S(\b)_k$. Then for $\d\in S(\b)$ and $\d>\c$, we have $\d\in S(\b)_k$. 
\end{lemma}

\begin{proof}
It suffices to show that $\d$ satisfies the hypothesis $H_{k}(\b)$.
Note that $e(\d)=\{e(\Delta): \Delta\in \d\}$ is a set because $\d$ is ordinary 
and by lemma \ref{lem: 2.1.3} we have $e(\d)\subseteq e(\b)$ .

Note that  $k-1\notin e(\b)$ since $\b\in S(\b)_k$ and $\b^{(k)}$ is ordinary
, therefore it is not in $e(\d)$ either. Hence to show that $\d\in S(\b)_k$
 hence it is equivalent to show that $k\in e(\d)$.
Since $\c\in S(\d)$, we know that $e(\c)\subseteq e(\d)$ .
Now that $k\in e(\c)$, we conclude that $k\in e(\d)$. We are done. 
\end{proof}

Now let $\b'\in S(\b)_k$ such that $\psi_k(\b')=(\b^{(k)})_{\min}$, then 
\begin{lemma}\label{lem: 6.1.3}
We have 
\[
 S(\b)_k=\{\c\in S(\b): \c\geq\b'\}.
\] 
\end{lemma}

\begin{proof}
 By the lemma above, we know that 
 $S(\b)_k\supseteq \{\c\in S(\b): \c\geq\b'\}$.
 We conclude that we have equality since $\psi$ preserve the order.
\end{proof}

\begin{prop}
Assume that $\a$ is ordinary. Then 
\begin{description}
\item[(1)] 
There exists a symmetric multisegment $\a^{\sym}$ such that 
\[
 S(\a)\simeq S(\a^{\sym})_{k_{r}, \cdots, k_{1}}.
\]

\item[(2)] There exists an element $\a'\in S(\a^{\sym})$ such that 
\[
 S(\a^{\sym})_{k_{r}, \cdots, k_{1}}=\{\c\in S(\a^{\sym}): \c\geq \a'\}.
\]

\end{description}
\end{prop}

\begin{proof}
Note that $(1)$ follows directly from proposition \ref{prop: 4.2.2}
and $(2)$ follows from applying successively lemma \ref{lem: 6.1.3} to 
the sequence obtained in 
the lemma below.  
\end{proof}

\begin{lemma}
There exists a sequence of multisegments  $\a_{0}=\a, \cdots,\a_{r}=\a^{\sym}$ such that 
$\a^{\sym}$ is symmetric, with $\a_{i}\in S(\a_{i})_{k_{i}}$ and $\a_{i-1}=\a_{i}^{(k_{i})}$
for some $k_{i}$. 
Moreover, $\a_{i}$ is ordinary for
all $i=1, \cdots, r$

\end{lemma}

\begin{proof}
Recall that  in proposition \ref{prop: 4.2.2} that every ordinary  multisegment
$\a$ can be obtained as 
\[
 \a=\a_{0}, \a_{1}, \cdots, \a_{r},
\]
where $\a_{r}$ is symmetric, with $\a_{i}\in S(\a_{i})_{k_{i}}$ and $\a_{i-1}=\a_{i}^{(k_{i})}$
for some $k_{i}$.

The statement (1) follows directly from proposition \ref{prop: 4.2.2}. 
Note that the ordinarity of $\a_{i}$'s follows from construction.
 
\end{proof}

\section{The parabolic KL polynomials}

For fixed $n \in \N$ and a pair of elements in $S_{n}$,  we can associate  
a Kazhdan Lusztig Polynomial $P_{x,y}(q)$. We know also that the coefficients
of such a polynomial are given by the dimensions of the intersection cohomology
of corresponding Schubert varieties in $GL_{n}/B$. 

Similar construction can give rise to a polynomial related to the Poincar\'e series 
of the intersection cohomology of the Schubert varieties in $GL_{n}/P$, where 
$P$ is a standard parabolic subgroup. This has been done in 
Deodhar \cite{D2} for general Coxeter System $(W, S)$. However, as indicated in 
the same article, in our case where $G=GL_{n}$, this is not so interesting because we 
have a good fibration $G/P\rightarrow G/B$, so basically everything boils down to 
the Borel case.

In this section, 
for certain multisegment $\a$, we shall relate the set $S(\a)$ to 
the orbits in $GL_{n}/P$, where the multiplicities appear to be the corresponding
Parabolic Kazhdan Lusztig Polynomials.

\begin{notation}
Let $S=\{\sigma_{i}: i=1, \cdots, n-1\}$ be a set of generators for $S_{n}$. 
For $J\subseteq S$, let $S_{J}=<J>$ be the subgroup generated by $J$
and $S_n^{J}=\{w\in S_{n}: ws>w \text{ for all } s\in J\}$. 
\end{notation}

\begin{prop}(cf. \cite{BF} Prop. 2.4.4)\label{prop: 5.2.2}
 We have 
 \begin{description}
  \item [(1)]\(S_{n}=\coprod_{w\in S_n^{J}}wS_{J};\)
  \item [(2)]for $w\in S_n^{J}$ , and $x\in S_{J}$, $\ell(wx)=\ell(w)+\ell(x)$ .
 \end{description}

\end{prop}
\remk Now we can identify $S_n^{J}$ with $S_{n}/S_{J}$, hence it is 
in bijection with the Borel orbits in $GL_{n}/P$, where $P$ is the 
parabolic subgroup determined 
by $J$.

\begin{notation}
 Let $\a_{\Id}^J=\{\Delta_{1}, \cdots, \Delta_{n}\}$ such that 
 \[
  e(\Delta_{1})<\cdots< e(\Delta_{n}), 
 \]
and 
\[
  b(\Delta_{1})\leq \cdots\leq b(\Delta_{n}),
\]
such that 
\[
 b(\Delta_{i})=b(\Delta_{i+1}) \text{ if and only if } \sigma_{i}\in J
\]
and $b(\Delta_{n})\leq e(\Delta_{1})$.
\end{notation}
\begin{example}\label{ex: 6.2.4}
Let $n=4$, and  $J=\{\sigma_1, \sigma_3\}$, then we can choose  
\[
 \a_{\Id}^{J}=[1, 3]+[1, 4]+[2, 5]+[2, 6].
\] 
\end{example}

\begin{definition}\label{def: 6.2.5}
We call a multisegment $\a\in S(\a_{\Id}^J)$ a multisegment of parabolic type $J$.
\end{definition}

\begin{prop}\label{prop: 6.2.5}
 For $w\in S_n^{J}$, let $\a^J_{w}=\sum [b(\Delta_{i}), e(\Delta_{w(i)}) ]$, then $\a_{w}^J\in S(\a_{\Id}^J)$.
\end{prop}

\begin{example}
Let $\a_{\Id}^J$ as in example \ref{ex: 6.2.4}. For $w=\sigma_1\sigma_2$, then 
\[
 \a_w^J=[1, 4]+[1, 5]+[2, 3]+[2, 6].
\]

\end{example}

\begin{proof}
 We proceed by induction on $|J|$. If $|J|=0$, we are 
 in the symmetric case, so we are done by Proposition \ref{teo: 2.3.2}.
 And in general, let $J=J_{1}\cup \{\sigma_{i_{0}}\}$ with 
 $i_{0}= \min\{i: \sigma_{i}\in J\}$ and 
 $i_{1}=\max\{i: b(\Delta_{i})=b(\Delta_{i_{0}})\}$.
 
 Let $\a_{1}= \{\Delta_{1}^{1}, \cdots, \Delta_{n}^{1}\}$, such that 
\begin{align*}
 \Delta_{i}^{1}&=^{+}(\Delta_{i}), \text{ for } i\leq i_{0}, \\
 \Delta_{i}^{1}&=\Delta_{i}, \text{ otherwise. }( \text{ cf. Nota. \ref{nota: 3.1.2}}) .
\end{align*}
\begin{example}
Let $\a_{\Id}^J$ be a multisegment as in example \ref{ex: 6.2.4}. 
Then 
\[
 \a_1=[0, 3]+[1, 4]+[2, 5]+[2, 6].
\]

\end{example}

Let $\a_{\Id}^{J_{1}}=\a_{1}$ with
\begin{displaymath}
 b(\Delta_i^1)=
 \left\{\begin{array}{cc}
 b(\Delta_i)-1, &\text{ for }i\leq i_{0},\\
b(\Delta_i), &\text{ for }i> i_0.
\end{array}\right.
\end{displaymath}

Then we  have
\[
 \a_{\Id}^J={^{(b(\Delta_{1}^1), \cdots, b(\Delta_{i_{0}}^1))}\a_{1}}.
\] 
Let $w_1=(i_1, \cdots, i_0+1, i_0)$, then $w_1\in S_n^{J_1}$.
Note that we have also $ww_{1}\in S_n^{J_{1}}$, since 
\[
 ww_{1}(i)=w(i-1)<ww_{1}(i+1)=w(i), \text{ for }i=i_{0}+1, \cdots, i_{1}-1.
\]

Then by induction, we know that 
\[
 \a_{ww_1}^{J_1}=\sum_{i}[b(\Delta_{i}^1), e(\Delta_{ww_1(i)}^1)]\in S(\a_{1}).
\]
\begin{example}
 Let $\a_{\Id}^J$ as in the previous example. Then $i_1=2$, and $J_1=\{\sigma_3\}$.
 In this case, we have $w_1=\sigma_1$ and $ww_1=\sigma_1\sigma_2\sigma_1$, with 
 \[
  \a_{ww_1}^{J_1}=[0, 5]+[1, 4]+[2, 3]+[3, 6].
 \]
\end{example}

Moreover,
\[
 \a_{w}^{J}={^{(b(\Delta_{1}^1), \cdots, b(\Delta_{i_{0}}^1))}\a_{ww_1}^{J_{1}}}.
\]
The result, that is the fact $\a_w^J\in S(\a_{\Id}^J)$ follows from the next lemma.

\end{proof}

\begin{lemma}\label{lem: 6.2.7}
We have 
\[
 \a_{ww_1}^{J_{1}}\in {_{b(\Delta_{1}^1), \cdots, b(\Delta_{i_{0}}^1)}S(\a_{1})}.
\] 
\end{lemma}

\begin{proof}
In fact, let $\a_{1, 0}=\a_{\Id}^J$ and 
 for $j\leq i_{0}$, 
$\a_{1, j}= \{\Delta_{1, j}, \cdots, \Delta_{n,j}\}$, such that 
\begin{align*}
 \Delta_{i, j}&=^{+}(\Delta_{i}), \text{ for } i\leq j, \\
 \Delta_{i, j}&=\Delta_{i}, \text{ otherwise. }
\end{align*}
Then we have  
$\a_{1,j}={^{(b(\Delta_{j+1}^1), \cdots, b(\Delta_{i_{0}}^1))}\a_{1}}$,
for $j=0, 1, \cdots, i_{0}$.
For $ j<i_0-1$, let
\[
 \b_{j}=\sum_{j< i\leq i_{0}} [b(\Delta_{i}^{1})+1, e(\Delta_{w(i)}^{1}) ]+
 \sum_{i> i_{0},\text{ or } i\leq j} [b(\Delta_{i}^{1}), e(\Delta_{w(i)}^{1})],
\]
and $\b_{i_0}=\a_{ww_1}^{J_1}$
so that $\b_{j}={^{(b(\Delta_{j+1}^1), \cdots, b(\Delta_{i_{0}}^1))}\a_2}$.
We show that  $\b_{j}\in {_{b(\Delta_{j}^1)}S(\a_{1, j})}$ by induction on $j$. 
\begin{description}
 \item[(1)] For $j=i_{0}$,   we have 
 $$b(\Delta^{1}_{i_{0}})=b(\Delta^{1}_{i_{0}+1})-1=\cdots =b(\Delta^{1}_{i_{1}-1})-1=b(\Delta^1_{i_1})-1.$$
And $ww_1(i_0)>ww_1(i_1)>ww_{1}(i_1-1)>\cdots >ww_1(i_{0}+1)$, hence
 $$e(\Delta^{1}_{ww_1(i_{0})})>e(\Delta_{ww_1(i_1)}^1)>e(\Delta_{ww_1(i_1-1)}^1)>\cdots >e(\Delta_{ww_1(i_0+1)}^1),$$
 because $w\in S^{J}$.
This
implies  that $\b_{i_0}$ satisfies 
the hypothesis  $(_{b(\Delta^{1}_{i_{0}})}H(\a_{1,i_0}))$.
\item[(2)] For general $j\leq i_{0}-1$, 
By induction, we may assume that $\b_{j+1}\in {_{b(\Delta_{j+1}^1)}S(\a_{1, j+1})}$. 
Now to show  $\b_{j}\in {_{b(\Delta_{j}^1)}S(\a_{1, j})}$ , 
 we know that
$b(\Delta_{j}^{1})+1<b(\Delta_{j+1}^{1})+1$ in $\b_{j+1}$
(we have inequality by assumption on $i_{0}$), which proves that $\b_{j}\in {_{b(\Delta_{j}^1)}S(\a_{1, j})}$ .
 Hence we are done.
\end{description} 

\end{proof}

\begin{lemma}\label{lem: 6.2.8}
Let $J=\{\sigma_{i_{0}}\}\cup J_1$ such that 
$i_0=\min\{i: \sigma_{i}\in J\}$. Let
$i_1\in \Z$ be the maximal integer satisfying for $i_0\leq i<i_1$ 
we have $\sigma_{i}\in J$.
Then 
\[
 S_{J}^{J_1}=\{w_i: i=1, \cdots, i_1-i_0+1\}
\]
with 
\[
 w_i=(i_1-i+1, \cdots, i_0+1, i_0)\in S_J.
\]
As a consequence, we have 
\[
 S_n^{J_1}=\coprod_{i}S_n^Jw_i.
\]

\end{lemma}

\begin{proof}
By proposition \ref{prop: 5.2.2}, we only need to show that
$S_J=\coprod_{j}w_jS_{J_1}$ and $w_j\in S^{J_1}$. 
The fact that $w_j\in S^{J_1}$ follows from 
\[
 w_j(i)=i-1, \text{ for }i=i_0+1, \cdots, i_1-j+1 , ~w_j(i_0)=i_1-j+1,
\]
and $w_j(i)=i$ for $i\notin \{i_0, \cdots, i_1-j+1\}$.
Finally, to see that $S_J=\coprod_{j}w_jS_{J_1}$, we compare the cadinalities.
Let $J_0=\{\sigma_{i}: i=i_0 \cdots, i_1-1 \}$, then
\[
 S_J\simeq S_{J_0}\times S_{J\setminus J_0},
\]
\[
 S_{J_1}\simeq S_{J_0\setminus \{\sigma_{i_0, i_0+1}\}} \times S_{J\setminus J_0}.
\]
Hence $\sharp{S_J}/\sharp{S_{J_1}}= \frac{\sharp{S_{J_0}}}{\sharp{S_{J_0\setminus \{\sigma_{i_0, i_0+1}\}}}}
=(i_1-i_0+1)!/(i_1-i_0)!=i_1-i_0+1$.
Finally, 
by proposition \ref{prop: 5.2.2}, we know that 
\[
 S_n=\coprod_{v\in S_n^J}vS_J=\coprod_{j=i_0}^{i_1-i_0+1}\coprod_{v\in S_n^J}vw_{j}S_{J_1}=\coprod_{j}S_n^Jw_jS_{J_1}.
\]

\end{proof}

Keeping the notations of proposition \ref{prop: 6.2.5}, we have
\begin{lemma}\label{lem: 6.2.9}
For $i=1, \cdots, i_1-i_0+1$, we have 
\[
 \a_w^J={^{(b(\Delta_{1}^1), \cdots, b(\Delta_{i_{0}}^1))}\a_{ww_i}^{J_{1}}}.
\] 
\end{lemma}

\begin{proof}
Note that by definition
We have 
\[
 \a_{ww_j}^{J_1}=\sum_{i}b(\Delta_{i}^1), e(\Delta_{ww_j(i)}^1)].
\]
As noted before, we have 
\[
 b(\Delta_i^1)=b(\Delta_i)-1, \text{ for }i\leq i_{0}, ~b(\Delta_i^1)=b(\Delta_i), \text{ for }i> i_0.
\]
Also,we observe that 
$e(\Delta_{i}^1)=e(\Delta_i)$.
Hence 
\[
 {^{(b(\Delta_{1}^1), \cdots, b(\Delta_{i_{0}}^1))}\a_{ww_i}^{J_{1}}}=
 \sum_{i}b(\Delta_{i}), e(\Delta_{ww_j(i)})].
\]
It remains to see that we have 
\[
 \sum_{i=i_0}^{i_1-j+1}b(\Delta_{i}), e(\Delta_{ww_j(i)})]=\sum_{i_0}^{i_1-j+1}b(\Delta_{i}), e(\Delta_{w(i)})]
\]
since $b(\Delta_{i_0})=\cdots=b(\Delta_{i_1-j+1})$.
Hence we have 
\[
  \a_w^J={^{(b(\Delta_{1}^1), \cdots, b(\Delta_{i_{0}}^1))}\a_{ww_i}^{J_{1}}}.
\] 
\end{proof}

\begin{definition}
As in the symmetric cases, we have the following map 
\begin{align*}
 \Phi_{J}: S_n^{J}&\rightarrow S(\a_{\Id}^J)\\
 w&\mapsto \a_{w}^J.
\end{align*}
\end{definition}

\begin{prop}\label{prop: 6.2.10}
The morphism $\Phi_{J}$ is bijective and translate the inverse Bruhat order
on $S_n^{J}$ to the order on $S(\a_{\Id}^J)$.
\end{prop}

\begin{proof}
Again, we do this by induction on $|J|$. If $|J|=0$, 
we are in the symmetric case, so everything is done in section 2.3.
In general, we keep the notation in the proposition \ref{prop: 6.2.5}.
We have $J=J_{1}\cup \{\sigma_{i_{0}}\}$.
And as we proved above, 
\[
 \a_{ww_1}^{J_{1}}\in {_{b(\Delta_{1}^1), \cdots, b(\Delta_{i_{0}}^1)}S(\a_{1})}.
\]

Also, we note that the morphism 
${_{b(\Delta_{1}^1), \cdots, b(\Delta_{i_{0}}^1)}\psi}$
sends $\Phi_{J_{1}}(ww_1)$ to 
$\Phi_{J}(w)$ for $w\in J$, as is proved in the proposition above. 
Therefore 
\[
 \Phi_J={_{b(\Delta_{1}^1), \cdots, b(\Delta_{i_{0}}^1)}\psi}\circ \Phi_{J_1},
\]
and the injectivity of follows from that of ${_{b(\Delta_{1}^1), \cdots, b(\Delta_{i_{0}}^1)}\psi}$
and induction on $J_1$.
For surjectivity,
let $\b\in S(\a_{\Id}^J)$, by surjectivity of the map 
\[
 {_{b(\Delta_{1}^1), \cdots, b(\Delta_{i_{0}}^1)}\psi}: 
 {_{b(\Delta_{1}^1), \cdots, b(\Delta_{i_{0}}^1)}S(\a_{1})}\rightarrow S(\a_{\Id}^J),
\]
we know that there exists a $w'\in S^{J_1}$, such that 
$\Phi_{J_1}(w')\in  {_{b(\Delta_{1}^1), \cdots, b(\Delta_{i_{0}}^1)}S(\a_{1})}$, and 
is sent to $\b$ by $ {_{b(\Delta_{1}^1), \cdots, b(\Delta_{i_{0}}^1)}\psi}$.
By lemma \ref{lem: 6.2.8},  every $w'\in S^{J_{1}}$ can be write as 
$ww_j$ for some $w\in S^J$ and $w_j\in S^{J_1}$. Now by lemma \ref{lem: 6.2.9},
\[
 \b=\a_{w}^J.
\]
Note that for $w>w'$ in $S^J$, then $ww_1>w'w_1$ in $S^{J_1}$, hence by induction
\[
 \Phi_{J_1}(ww_1)<\Phi_{J_1}(ww_1),
\]
we get 
\[
 \Phi_{J_1}(w)<\Phi_{J_1}(w),
\]
since the morphism ${_{b(\Delta_{1}^1), \cdots, b(\Delta_{i_{0}}^1)}\psi}$
preserves the order.

\end{proof}

\begin{prop}
Let $v_{1}, v_{2}\in S^{J}$, then we have 
\[
 P_{\Phi_{J}(v_{1}), \Phi_{J}(v_{2})}(q)=P_{v_{1}, v_{2}}^{J}(q)
\]
where on the right hand side is the parabolic KL polynomial indexed by $v_{1}, v_{2}$.
\end{prop}

\begin{proof}
As is proved in \cite{D2}, we have $P_{v_{1}, v_{2}}^{J}(q)=P_{v_{1}v_{J}, w_{2}v_{J}}(q)$, 
where $v_{J}$ is the maximal element in $S_{J}$. So it suffices to show that we have 
the equality $P_{\Phi_{J}(v_{1}), \Phi_{J}(v_{2})}(q)=P_{v_{1}v_{J}, v_{2}v_{J}}(q)$.
Also, from lemma \ref{lem: 6.2.7}, we know that 
\[
 \Phi_{J}(v_1)= {_{b(\Delta_{1}^1), \cdots, b(\Delta_{i_{0}}^1)}\psi}(\Phi_{J_1}(v_1w_1)),
\]
where $w_1$ is described in lemma \ref{lem: 6.2.8}.
Hence we have 
\[
 P_{\Phi_{J_{1}}(v_1w_1), \Phi_{J_{1}}(v_{2}w_1)}(q)=
 P_{\Phi_{J}(v_{1}), \Phi_{J}(v_{2})}(q)
\]
by corollary \ref{cor: 4.6.16}.

By induction, we have 
\[
 P_{\Phi_{J_{1}}(v_{1}w_1), \Phi_{J_{1}}(v_{2}w_1)}(q)
 =P_{v_{1}w_1v_{J_{1}}, v_{1}w_1v_{J_{1}}}(q).
\]
Now to finish, we have to show $v_{J}=w_1v_{J_{1}}$.
But we know that 
\[
 S_J=\coprod_{j}w_jS_{J_1}
\]
with $w_1=\max\{w_j: j=1, \cdots, i_1-i_0+1\}$, we surely have
\[
 v_{J}=w_1v_{J_{1}}.
\]
\end{proof}

More generally, for $J_{i}\subseteq S$, $i=1,2$, we can consider
the $P_{J_{1}}$ orbit in $GL_{n}/P_{J_{2}}$. We state the related result 
without proving. 

\begin{definition}
Let $S_n^{J_{1}, J_{2}}=\{w\in S_{n}: s_{1}vs_{2}>v \text{ for all } s_{i}\in J_{i}, i=1,2\}$. 
\end{definition}

\begin{definition}
Let $v\in S_n^{J_1, J_2}$. We define 
\[
 S_{J_1}^{J_2, v}=\{w\in S_{J_1}: ws>w, \text{ for all } s\in S_{J_1}\cap vS_{J_2}v^{-1}\}.
\]
 
\end{definition}

\remk If we let $M_J$ be the Levi subgroup of $P_J$, then the set 
$S_{J_1}^{J_2, v}$ corresponds to the Borel orbits in $M_{J_1}/(M_{J_1}\cap vM_{J_2}v^{-1})$.

\begin{prop}\label{prop: 6.2.13}
 We have 
 \begin{description}
  \item [(1)]$S_{n}= \coprod_{v\in S_n^{J_{1}, J_{2}}}S_{J_{1}}vS_{J_{2}}$;
  \item [(2)]$\ell(xvy)=\ell(v)+\ell(x)+\ell(y)$ for $v\in S^{J_{1}, J_{2}}$ , $x\in S_{J_{1}}^{J_2, v}, y\in S_{J_{2}}$.
  \item [(3)]The $P_{J_{1}}$ orbits in $GL_{n}/P_{J_{2}}$ are indexed 
   by $S_n^{J_{1}, J_{2}}$.
 \end{description}  
\end{prop}

\begin{definition}
 For $v_{1}, v_{2}\in S^{J_{1}, J_{2}}$ such that $v_{1}\leq v_{2}$, we let $P^{J_{1}, J_{2}}_{v_{1}, v_{2}}(q)$ be the 
 Poincar\'e series of  the localized intersection cohomology 
 $$\mathcal{H}^{\bullet}(\line {P_{J_{1}}v_{2}P_{J_{2}}})_{v_{1}P_{J_{2}}}.$$
\end{definition}

\begin{lemma}
For $v_{1}, v_{2}\in S^{J_{1}, J_{2}}$ such that $v_{1}\leq v_{2}$,
 we have 
 \[
  P^{J_{1}, J_{2}}_{v_{1}, v_{2}}(q)=P_{w_1, w_2}(q),
 \]
where $w_i$ is the element of maximal length in $S_{J_{1}}v_iS_{J_2}$.
\end{lemma}

\begin{notation}\label{nota: 6.2.21}
 Let $\a_{\Id}^{J_1, J_2}=\{\Delta_{1}, \cdots, \Delta_{n}\}$ such that 
 \[
  e(\Delta_{1})\leq \cdots\leq  e(\Delta_{n}), 
 \]
 such that 
\[
 e(\Delta_{i})=e(\Delta_{i+1}) \text{ if and only if } \sigma_{i}\in J_{1}
\]
and 
\[
  b(\Delta_{1})\leq \cdots\leq b(\Delta_{n}),
\]
such that 
\[
 b(\Delta_{i})=b(\Delta_{i+1}) \text{ if and only if } \sigma_{i}\in J_{2}
\]
and $b(\Delta_{n})\leq e(\Delta_{1})$.
\end{notation}

\begin{definition}\label{def: 6.2.22}
We call a multisegment $\a\in S(\a_{\Id}^{J_1, J_2})$ a multisegment of parabolic type $(J_{1}, J_{2})$.
\end{definition}

\begin{lemma}\label{lem: 6.2.3}
 For $w\in S^{J_{1}, J_{2}}$, let $\a_{w}^{J_1, J_2}=\sum [b(\Delta_{i}), e(\Delta_{w(i)}) ]$, then 
 $\a_{w}^{J_1, J_2}\in S(\a_{\Id}^{J_1, J_2})$.
Therefore we have an application
\begin{align*}
 \Phi_{J_{1}, J_{2}}: S^{J_{1}, J_{2}}&\rightarrow S(\a_{\Id}^{J_1, J_2})\\
 w&\mapsto \a_{w}^{J_1, J_2}.
\end{align*}
\end{lemma}

\begin{prop}\label{prop: 6.2.4}
The morphism $\Phi_{J_{1}, J_{2}}$ is bijective and translate the inverse Bruhat order
on $S^{J_{1}, J_{2}}$ to the order on $S(\a_{\Id}^{J_1, J_2})$.
\end{prop}

\begin{prop}
Let $w_{1}, w_{2}\in S_n^{J_{1},J_{2}}$, then we have 
\[
 P_{\Phi_{J_{1}, J_{2}}(w_{1}), \Phi_{J_{1}, J_{2}}(w_{2})}(q)=P_{w_{1}, w_{2}}^{J_{1}, J_{2}}(q)
\]
where on the right hand side is the parabolic KL polynomial indexed by $w_{1}, w_{2}$.
\end{prop}

\begin{example}
 We are now ready to interpret the following results
(due to Zelevinsky, see \cite{Z3} Section 3.3): let $\a=k[0, 1]+(n-k)[1,2]$
 then $\a$ corresponding to the identity in $S_n^{J, J}$ with 
 \[
  J=\{\sigma_{i}: i\neq k\}. 
 \]
Note that in this case, we have $GL_{n}/P_{J}$ is the 
Grassmanian $G_{k}(\C^{n})$, where as the $P_{J}$
orbits correspond to the stratification, for $r\leq r_{0}=\min\{k, n-k\}$
and fixed $\C^k\in G_{k}(\C^n)$, 
\[
 X_{r}=\{U\in G_{k}(\C^{n}): \dim(U\cap \C^{k})= k-r\}
\]
with $\line{X_{r}}=\coprod_{r'\leq r}X_{r'}$.
\end{example}

\remk 
There is another way to obtain  the result of this section , i.e., by direct
geometric construct, as in section 4.3, where we prove the same result
for symmetric case. In this situation, instead having the flag variety
$G/B$ in the fibers, we will find $G/P_{J}$ in the fibers. 
There is one advantage in this geometric construction, i.e, by employing 
the same proof as in section 4.4, one can get a resolution for $G/P_{J}$
by pulling back that of the corresponding orbit variety. This 
shows for example, that the resolution can not be small when the associated quiver 
is of type $A_{n}$, $n\geq 3$, by the example constructed by Zelevinsky for 
flag variety, which does not admit any small resolution. We remark 
that the resolution is always small for type $A_{2}$, as is proved by Zelevinsky.

\remk Note that in \cite{Z3}, Zelevinsky constructed a small resolution 
for the $\line{O}_{\a}$ with $\a=\{[1, 2], [2, 3]\}$, which corresponds
to a Schubert varieties of 2-step . Now with our interpretation, we should 
be able to construct a small resolution for all 2-step Schubert varieties.
We return to this question later.

\remk With the help of partial derivative which we will develop in 
next section, we will be able to give inverse parabolic KL polynomials
combining results of this section, which is described in \cite{D2}
. See next section for more details.

\section{Non Ordinary Case}

In this section, for a general multisegment $\a$, we will relate the poset $S(\a)$ 
to a Bruhat interval $[x,y]$ with $x<y$ in some $S_{r}^{J_1, J_2}$. 

Now let $\a$ be a multisegment. 
First of all, we decide 
the set $J_1, J_2$. 

\begin{definition}\label{def: 6.3.1}
We define two sets $J_1(\a), J_2(\a)$.
\begin{itemize}
 \item 
Let $b(\a)=\{k_1\leq \cdots\leq k_r\}$.
Then let $J_2(\a)\subseteq S_r$ be the set such that
 $\sigma_{i}\in J_2(\a)$ if and only if $k_i=k_{i+1}$. 
\item Let $e(\a)=\{\ell_1\leq \cdots \leq \ell_r\}$.
Then let $J_1(\a)\subseteq S_r$ be the set such that $\sigma_{i}\in J_1(\a)$ if and only if $\ell_{i}=\ell_{i+1}$.
\end{itemize}
\end{definition}


Keeping the notations in definition \ref{def: 6.3.1},
\begin{prop}\label{prop: 6.3.2}
There exists a unique $w\in S_r^{J_1(\a), J_2(\a)}$, such that 
\[
 \a=\sum_{j}[k_j, \ell_{w(j)}].
\]
\end{prop}

\begin{proof}
We observe that there exists an element $w'\in S_r$, such that 
\[
\a=\sum_j [k_j, \ell_{w'}(j)]. 
\]
Now by proposition \ref{prop: 6.2.13},
we know that there exists $w'=w_{J_1(\a)}ww_{J_2(\a)}$
with $w_{J_i(\a)}\in S_{J_i(\a)}$ for $i=1, 2$ and $w\in S_r^{J_1(\a), J_2(\a)}$.
Now we only need to prove that 
\[
 \a=\sum_j [k_j, \ell_{w}(j)].
\]
In fact, by definition of $J_i(\a), i=1,2$, we know that
\[
 k_j=k_{v(j)},  \text{ for all } v\in S_{J_{2}(\a)},
\]
\[
 \ell_{j}=\ell_{v(j)}, \text{ for all } v\in S_{J_{1}(\a)}.
\]
Hence 
\begin{align*}
 \a&=\sum_j [k_j, \ell_{w_{J_1(\a)}ww_{J_{2}(\a)}(j)}]\\
   &=\sum_j [k_j, \ell_{w(\a)w_{J_2(\a)}(j)}]\\
   &=\sum_j [k_{w_{J_2(\a)}^{-1}w^{-1}(j)}, \ell_{j}]\\
   &=\sum_j [k_j, \ell_{w(j)}].
\end{align*}
\end{proof}
 
Next we show how to reduce a general multisegment $\a$
to a multisegment $\a_{w}^{J_{1}(\a), J_2(\a)}$ of parabolic 
type $(J_1(\a), J_2(\a))$  without changing the poset structure $S(\a)$.

\begin{prop}\label{prop: 6.3.3}
Let $\a$ be a multisegment, then there exists a multisegment  $\c$, 
and a multisegment $\a_{w}^{J_{1}(\a), J_2(\a)}$  of parabolic type $(J_1(\a), J_2(\a))$, such that 
\[
 \a_{w}^{J_{1}(\a), J_2(\a)}\in S(\a_{w}^{J_{1}(\a), J_2(\a)})_{\c}, ~\a=(\a_{w}^{J_{1}(\a), J_2(\a)})^{(\c)}
\]

\end{prop}
\begin{proof}
 
 In general $\a$ is not of parabolic type, i.e,
we do not have $\min\{e(\Delta): \Delta\in \a\}\geq \max\{b(\Delta): \Delta\in \a\}$.
Now we show how to construct  $\a_{w}^{J_{1}(\a), J_2(\a)}$ .
 
 In fact, let 
\[
 \a=\{\Delta_{1}, \cdots, \Delta_{n}\},  \Delta_{1}\prec \cdots \prec \Delta_{n}.
\]
Then
\[
 e(\Delta_{1})=\min\{k: i=1, \cdots, n\}.
\]
If $\a$ is not of parabolic type, let
$\Delta^{1}=[e(\Delta_{1})+1, \ell] $ with $\ell$ maximal satisfying 
that for any $m$ such that $ e(\Delta_{1})\leq m\leq  \ell-1$, there is a segment in $\a$ ending in $m$.
We construct $\a^{1}$ by replacing every segment $\Delta$ in 
$\a$ ending in $\Delta^{1}$ by $\Delta^{+}$.
Repeat this construction with $\b_{1}$ to get $\a^{2}\cdots $,
until we get $\a^{s}$, which is of parabolic type.
Let $\c=\{\Delta^{1}, \cdots, \Delta^{s}\}$, then we do as in  proposition 
\ref{cor: 5} to get 
\[
 \a^s \in S(\a^s)_{\c},~ \a=(\a^s)^{(\c)}.
\] 
Note that by our construction we have 
\[
 J_1(\a^i)=J_1(\a),~ J_2(\a^i)=J_2(\a),
\]
for $i=1, \cdots, s$. 
\end{proof}

\begin{lemma}
Assume that $\a\in S(\a)_{k}$ such that 
\[
 J_1(\a)=J_1(\a^{(k)}), ~J_2(\a)=J_2(\a^{(k)}).
\]
Let $\c\in S(\b)_k$. Then for $\d\in S(\b)$ and $\d>\c$, we have $\d\in S(\b)_k$. 
\end{lemma}

\begin{proof}
It suffices to show that $\d$ satisfies the hypothesis $H_{k}(\a)$.
Note that $e(\d)\subseteq e(\a)$ by lemma \ref{lem: 2.1.3}.
Assume that $k\in e(\a)$ to avoid triviality.
Now that $k-1\notin e(\a)$ since $\a\in S(\a)_k$ and 
\[
 J_1(\a)=J_1(\a^{(k)}), ~J_2(\a)=J_2(\a^{(k)}),                                                  
\]
 so it is also not in $e(\d)$. Hence to show that $\d\in S(\b)_k$
 hence it is equivalent to show that $\varphi_{e(\d)}(k)=e_{e(\a)}(k)$.
Since $\c\in S(\d)$, we know that $e(\c)\subseteq e(\d)$ hence 
$\varphi_{e(\d)}\leq \varphi_{e(\d)}(k)$.
Now that 
$\c\in S(\a)_k$ implies 
$\varphi_{e(\c)}=\varphi_{e(\a)}$, we conclude that $\varphi_{e(\d)}(k)=e_{e(\a)}(k)$. 
We are done. 
\end{proof}

Now let $\a'\in S(\a)_k$ such that $\psi_k(\a')=(\a^{(k)})_{\min}$, then 
\begin{lemma}\label{lem: 6.3.5}
We have 
\[
 S(\a)_k=\{\c\in S(\a): \c\geq\a'\}.
\] 
\end{lemma}

\begin{proof}
 By the lemma above, we know that 
 $S(\a)_k\supseteq \{\c\in S(\a): \c\geq\a'\}$.
 We conclude that we have equality since $\psi$ preserve the order.
\end{proof}

\begin{prop}\label{prop: 6.3.6}
Assume that $\a$ is a multisegment. Then 
\begin{description}
\item[(1)] 
There exists a multisegment 
$\a_{w}^{J_{1}(\a), J_2(\a)}$ of parabolic type $(J_1(\a), J_2(\a))$ and 
a sequence of integers $k_1, \cdots, k_r$
such that 
\[
 S(\a)\simeq S(\a_{w}^{J_{1}(\a), J_2(\a)})_{k_{r}, \cdots, k_{1}}.
\]

\item[(2)] There exists an element $\a'\in S(\a_{w}^{J_{1}(\a), J_2(\a)})$ such that 
\[
 S(\a_{w}^{J_{1}(\a), J_2(\a)})_{k_{r}, \cdots, k_{1}}=\{\c\in S(\a_{w}^{J_{1}(\a), J_2(\a)}): \c\geq \a'\}.
\]

\end{description}
\end{prop}
 
\begin{proof}
Note that (1) follows from  proposition \ref{prop: 6.3.3} and 
proposition \ref{prop: 3.2.17}.
And (2) follows from 
applying the lemma \ref{lem: 6.3.5} successively to the lemma below. 
\end{proof}
 
\begin{lemma}
There exists a sequence of multisegments  $\a_{0}=\a, \cdots,\a_{r}=\a_{w}^{J_{1}(\a), J_2(\a)}$ such that 
$\a_{w}^{J_{1}(\a), J_2(\a)}$ is of parabolic type $(J_1(\a), J_2(\a))$,   $\a_{i}\in S(\a_{i})_{k_{i}}$ and $\a_{i-1}=\a_{i}^{(k_{i})}$
for some $k_{i}$. 
Moreover, 
\[
 J_{1}(\a_i)=J_1(\a), ~J_2(\a)=J_2(\a)
\]
for
all $i=1, \cdots, r$.  
\end{lemma}

\begin{proof}
 This follows from our construction in the 
 proof of proposition \ref{prop: 6.3.3}.
\end{proof}

\chapter{Computation of Partial Derivatives}

In this chapter, 
we study the problem of computing the partial derivatives
$\D^k(L_{\a})$ of the irreducible representation 
$L_{\a}$ attached to a multisegment $\a$. The idea is to use these computations
to calculate the multiplicities in the induced representation  $L_{\a} \times L_{\b}$, cf. the next chapter.
Recall that we have already given a way of computing 
$L_{\a}$ as a sum, cf. (\ref{equ: (1)})
\[
 L_{\a}=\sum_{\b}\tilde{m}_{\b, \a}\pi(\a).
\]

So one is reduced to the calculate 
\[
 \D^k(\pi(\a))=\sum_{\b} n_{\b, \a} L_{\b},~ \qquad n_{\b, \a}\geq 0.
\]


As for the coefficient $m_{\b, \a}$, we first  introduce a new poset structure $\preceq_k$
on the set of multisegments so that  we have the equivalence 
between $n_{\b, \a}>0$ and $\b\preceq_k \a$, cf. proposition \ref{prop: 7.4.2}. 

The principal result of this chapter 
is the interpretation of the coefficient 
$n_{\b, \a}$ as the value at $q=1$ of some 
Poincar\'e series of the Lusztig product of 
two explicit perverse sheaves on orbital varieties, cf. proposition
\ref{prop: 7.3.8}. 



In 7.4, we compute these Lusztig products as the push forward by a projection 
$\beta''$, cf. corollary \ref{cor: 7.4.19}, of some concrete perverse sheaf on 
an orbital variety. In \S 7.6 we first study the geometry of the  case where the 
multisegments are of Grassmanian type. In this case the projection $\beta''$ is simply
cf. proposition \ref{prop: 7.6.8}, the natural projection 
\[
 GL_n/P\rightarrow GL_n/P'
\]
with $P\subseteq P'$ two parabolic subgroups.
The geometry of the parabolic case is treated in \S 7.7: the constructions and proofs are the same as the
Grassmanian type.

Finally in the last section 7.8, we obtain a complete formula for 
$\D^k(L_{\a})$ in the general case, cf. corollary \ref{coro-fornula-derivative}


\section{New Poset Structure on Multisegments}

In this section we define a new poset structure $\preceq_k$ depending on an 
integer $k$ on
the set of multisegments and show that the term $L_{\b}$ appears in
$\D^k(\pi(\a))$ if and only if $\b\preceq_k \a $.

\begin{definition}
For a well ordered multisegment $\a=\{\Delta_1, \cdots, \Delta_s\}$
with $\Delta_1\preceq \cdots \preceq \Delta_s$,
let 
\[
 \a(k):=\{\Delta\in \a: e(\Delta)=k\}=\{\Delta_{i_{0}}, \Delta_{i_0+1}, \cdots, \Delta_{i_1}\}. 
\]
Now let $\Gamma \subseteq \a(k)$, let 
\[
 \a(k)_\Gamma: =(\a(k)\setminus \Gamma)\cup \{\Delta^{(k)}: \Delta\in \Gamma\},
\]
and 
\[
 \a_\Gamma: =(\a\setminus \a(k))\cup \a(k)_\Gamma.
\] 

We say $\b\preceq_k \a$ if there exist a multisegment $\c\in S(\a)$ such that 
\[
 \b\leq \a_\Gamma
\]
for some $\Gamma$.  
\end{definition}

\begin{lemma}
We have 
\addtocounter{theo}{1}
\begin{equation}\label{eq: 7.41}
 \D^k(\pi(\a))=\pi(\a)+\sum_{\Gamma\subseteq \a(k), \Gamma\neq \emptyset}\pi(\a_\Gamma).
\end{equation}  
\end{lemma}

\begin{proof}
 
Let 
\[
\a=\{\Delta_1, \cdots, \Delta_r, \Delta_{r+1}, \cdots, \}.     
\]
Then 
\[
 \pi(\a)=\prod_{i=1}^{r}L_{\Delta_i} \times \prod_{i>r}L_{\Delta_i}
\]
 and 
 \begin{align*}
  \D^k(\pi(\a))&=\prod_{i=1}^{r}(L_{\Delta_i}+L_{\Delta_i^{(k)}})\times \prod_{i>r}L_{\Delta_i}\\
               &=\pi(\a)+\sum_{\Gamma\subseteq \a(k), \Gamma\neq \emptyset}\pi(\a_{\Gamma}).
\end{align*}
\end{proof}

\begin{prop}\label{prop: 7.4.2}
Let  
\addtocounter{theo}{1}
\begin{equation}
\D^k(\pi(\a))=\sum_{\b}n_{\b, \a}L_{\b}.
\end{equation}
Then $n_{\b, \a}>0$ if and only if $\b\preceq_k \a$.  
\end{prop}

\begin{proof}
Let $\b\preceq \a$, 
then by definition we have 
$\b\leq \a_\Gamma$ for some $\Gamma$. 
Therefore $m_{\b, \a_\Gamma}>0$, now 
we have $n_{\b, \a}>0$ by equation 
\ref{eq: 7.41}. Conversely, 
if $n_{\b, \a}>0$, then by 
equation \ref{eq: 7.41}, we know that 
 $\b\leq \a_\Gamma$ for some $\Gamma$. 
\end{proof}

\begin{cor}\label{cor: 7.4.3}
We have $\b\preceq_k \a$ 
 if and only if $\D^k(\pi(\a))-\pi(\b)\geq 0$
 in $\cal{R}$. 
\end{cor}

\begin{proof}
We keep the notations in the proof of  
proposition \ref{prop: 7.4.2}.  
We know that $\b\preceq_k \a$ implies
$\b\leq \a_\Gamma$ for some $\Gamma\subseteq \a(k)$. 
By lemma \ref{lem: 3.0.8}, we know that 
$\b\leq \a_\Gamma$ implies that 
$\pi(\a_\Gamma)-\pi(\b)\geq 0$ in $\cal{R}$. 
Since $\D^k(\pi(\a))-\pi(\a_\Gamma)\geq 0$ by 
equation (\ref{eq: 7.41}), we have $\D^k(\pi(\a))-\pi(\b)\geq 0$.
Conversely, if $\D^k(\pi(\a))-\pi(\b)\geq 0$, 
we have $n(\b, \a)>0$, hence $\b\preceq_k \a$
by proposition \ref{prop: 7.4.2}.
\end{proof}

\begin{prop}\label{prop: 7.3.5}
For any $\b\preceq_k \a$, there exists $\c\in S(\a)$,
and some subset $\Gamma\subseteq \c(k)$,  
such that 
\[
 \b=\c_{\Gamma}.
\]
Conversely, if $\b=\c_{\Gamma}$ for some $\c\in S(\a)$,
 then $\b\preceq_k \a$.
\end{prop}

\begin{proof}
For the converse part, suppose $\c\neq \a$,
by equation \ref{eq: 7.41}, we have $\D^k(\pi(\c))-\pi(\b)\geq 0$ in $\mathcal{R}$.
By lemme \ref{lem: 3.0.8}, we 
know that $\pi(\a)-\pi(\c)\geq 0$ in $\cal{R}$, hence $\D^k(\pi(\a))-\D^k(\pi(\c))\geq 0$
by theorem \ref{teo: 3}. Therefore $n_{\b,\a}>0$. 
Hence we have $\b\preceq_k \a$.

For the direct part, suppose that 
$\b\preceq_k \a$, hence $\b<\a_{\Gamma_1}$ for some $\Gamma_1$. We prove by induction on $\ell(\b, \a_T)$.
If $\ell(\b, \a_{\Gamma_1})=0$, then $\b=\a_{\Gamma_1}$, we are done.
Now let $\b<\d\leq \a_{\Gamma_1}$ such that 
$\ell(\b, \d)=1$, by induction, we know that 
\[
 \d=\c'_{\Gamma_0}, 
\]
for some $\c'\in S(\a)$.
Note that by replacing $\c'$ by $\a$, we can assume that 
$\d=\a_{\Gamma_1}$ and $\ell(\b, \a_{\Gamma_1})=1$.
 
By definition, we know that $\b$ is obtained by 
applying the elementary operation to a pair of segments $\{\Delta\preceq \Delta'\}$
in $\a_T$. Now we set out to construct $\c$. 

\begin{itemize}
 \item If $\{\Delta, \Delta'\}\subseteq \a\setminus \{\Delta^{(k)}: \Delta\in \Gamma_1\}\subseteq \a$,
 let $\c$ be the multisegment obtained by applying the elementary operations to 
 $\{\Delta, \Delta'\}$. And we have 
 \[
  \b=\c_{\Gamma_1}.
 \]
\item If $\{\Delta, \Delta'\}\cap \{\Delta^{(k)}: \Delta\in \Gamma_1\}=\{\Delta'\}$, 
then $\{\Delta, \Delta'^{+}\}\in \a$
let $\c$ be the 
multisegment obtained by applying the elementary operations to 
 $\{\Delta, \Delta'^+\}$. Then 
 let 
 \[
  \Gamma=(\Gamma_1\setminus \{\Delta'^+\})\cup\{\Delta\cup \Delta'^+\}
 \]
and we have 
 \[
  \b=\c_{\Gamma}.
 \]
\item If $\{\Delta, \Delta'\}\cap \{\Delta^{(k)}: \Delta\in \Gamma_1\}=\{\Delta\}$, 
then $\{\Delta^{+}, \Delta'\}\in \a$
let $\c$ be the 
multisegment obtained by applying the elementary operations to 
 $\{\Delta^{+}, \Delta'\}$. Then let 
 \[
  \Gamma=(\Gamma_1\setminus \{\Delta^+\})\cup\{\Delta\cap \Delta'\}
 \]
 and we have 
  \[
  \b=\c_{\Gamma}.
 \]
\end{itemize}
Hence we are done.

\end{proof}

\begin{prop}
The relation $\preceq_k$ defines a 
poset structure on $\O$. 
\end{prop}

\begin{proof}
By definition we have $\a\preceq_k \a$ 
for any $\a\in \O$. 
Suppose $\a_1\preceq_k \a_2, \a_2\preceq_k \a_3$, 
we want to show that $\a_1\preceq_k \a_3$. 
By proposition \ref{prop: 7.3.5}, there exists $\c\in S(\a_2)$
and $\Gamma_{1}\subseteq \c(k)$
, such that 
\[
 \a_1=\c_{\Gamma_1}.
\]
 Note that 
by corollary \ref{cor: 7.4.3},  
the fact $\a_2\preceq_k \a_3$ implies $\D^k(\pi(\a_3))-\pi(\a_2)\geq 0$.
Hence we have $n(\a_3, \c)>0$, therefore $\c\preceq_k \a_3$ by proposition \ref{prop: 7.4.2}.  
In turn, we know that there exists a multisegment $\c'\in S(\a_3)$
and $\Gamma_2\subseteq \c'(k)$, such that 
\[
 \c=\c'_{\Gamma_2}.
\]
Since we have $\c(k)\subseteq \c'(k)$, we take 
\[
 \Gamma_3: =\Gamma_1\cup \Gamma_2\subseteq \c'(k).
\]
Now we get 
\[
 \a_1=\c'_{\Gamma_3},
\]
which implies $\a_1\preceq_k \a_3$ by proposition \ref{prop: 7.3.5}. 
Finally, if $\a\preceq_k \b$ and $\b\preceq_k \a$, then by definition
we have $\a=\b$. 

\end{proof}

\begin{definition}
We let 
\[
\Gamma(\a, k)=\{\b: \b\preceq_k \a\}.
\]
\end{definition}

\section{Canonical Basis and Quantum Algebras}

In this section, following \cite{LNT}, we recall 
the results of Lusztig on canonical basis, the relation of quantum algebras 
and the algebra $\mathcal{R}$. We are especially interested in 
the construction of 
a product of perverse sheaves over orbital varieties defined by Lusztig \cite{Lu}, 
which is closely related to the product defined by induction in $\mathcal{R}$.


\begin{definition}
Let $\N^{(\Z)}$ be the semi-group of sequences $(d_j)_{j\in \Z}$ of non negative integers 
which are zero for all but finitely many $j$.   
Let $\alpha_i$ be the element whose $i$-th term is 1 and other terms are zero.
\end{definition}

\begin{definition}
 We define a symmetric bilinear form on $\N^{(\Z)}$ given by 
 \begin{displaymath}
 (\alpha_i, \alpha_j)=
 \left\{\begin{array}{cc}
 2, &\text{ for } i=j;\\
 -1,  &\text{ for } |i-j|=1;\\
 0,   &\text{ otherwise }.
   \end{array}\right. 
 \end{displaymath}
\end{definition}

\begin{definition}
Let $q$ be an indeterminate and $\Q(q^{1/2})$ be the fractional field of $\Z[q^{1/2}]$.
Let $U_q^{\geq 0}$ be the $Q(q^{1/2})$-algebra generated by the elements $E_i$ and $K_i^{\pm 1}$ for $i\in \Z$
with the following relations:
\begin{align*}
K_iK_j=K_jK_i,~K_iK_i^{-1}=1;&\\
K_iE_i=q^{1/2(\alpha_i, \alpha_j)}E_iK_i ;&\\
E_iE_j=E_jE_i, &\text{ if } |i-j|>1;\\
E_i^2E_{j}-(q^{1/2}+q^{-1/2})E_iE_jE_i+E_jE_i^2=0, &\text{ if }|i-j|=1.
\end{align*}
and let $U^{+}$ be the subalgebra generated by 
the $E_i$'s. 
\end{definition}

\remk This is the $+$ part of the quantized enveloping algebra $U$ associated by Drinfeld
and Jimbo to the root system $A_{\infty}$ of $SL_{\infty}$. And for $q=1$, this specializes to the 
classical enveloping algebra of the nilpotent radical of a Borel subalgebra.

\begin{definition}\label{def: 7.2.4}
We define a new order on the set of segments $\Sigma$

\begin{displaymath} \left\{ \begin{array}{cc}
&[j, k]\lhd  [m,n], \text{ if } k< n,\\
&[j,k]\rhd  [m,n], \text{ if } j<m, n=k.
\end{array}\right. 
\end{displaymath}
We also denote $[j,k]\lhd [m, n]$ or $[j, k]=[m, n]$ by $\unlhd [m, n]$. 
\end{definition}

\begin{lemma}
The algebra $U_q^+$ is $\N^{(\Z)}$-graded via the weight function $\wt(E_i)=\alpha_i$. 
Moreover, for a given weight $\alpha$, the homogeneous component of $U_q^+$ with weight $\alpha$ 
is of finite dimension, and its basis are naturally parametrized by the multisegments of the same 
weight.
\end{lemma}

\begin{proof}
Let $\a=\sum_{s=1}^rm_{i_s, j_s} [i_s, j_s]$ be a multisegment of weight $\alpha$, note that 
here we identify the weight $\varphi_{i}$ with $\alpha_i$, and that 
\[
 [i_1, j_1]\unlhd \cdots \unlhd [i_r, j_r] ( \text{ cf. Def. \ref{def: 7.2.4}} )
\]
Then we associate to $\a$ the element
\[
 (E_{j_1}\cdots E_{i_1})\cdots (E_{j_r}\cdots E_{i_r}).
\] 
\end{proof}

\begin{notation}
For $x\in U^+$ be an element of degree $\alpha$, we will denote $\wt(x)=\alpha$. 
\end{notation}

\begin{example}
For $i\leq j$, let $\alpha_{ij}=\alpha_i+\cdots+\alpha_j$. Consider the homogeneous components 
of $U^+$ with weight $\alpha=2\alpha_{12}$, whose basis is given by 
 \[
  E_1E_2E_1E_2, ~E_1E_1E_2E_2. 
 \]
The element $ E_1E_2E_1E_2$ is parametrized by the 
multisegment $[1]+[1,2]+[2]$, while $E_1E_1E_2E_2$ is parametrized
by the multisegment $2[1]+2[2]$.
\end{example}

In \cite{Lu}, Lusztig has defined certain bases for $U_q^+$
associated to the orientations of a Dynkin diagram, called PBW( Poincar\'e-Birkhoff-Witt) basis, 
which
specializes to the classical PBW type bases.
Following \cite{LNT}, we describe the PBW-basis

\begin{definition}
We define 
\[
 E([i])=E_i, ~E([i, j])=[E_j[\cdots [E_{i+1}, E_i]_{q^{1/2}}\cdots ]_{q^{1/2}}  ]_{q^{1/2}},
\]
where $[x, y]_{q^{1/2}}=xy-q^{-1/2(\wt(x), \wt(y))}yx$.  
More generally, let 
$\a=\sum_s a_{i_s, j_s}[i_s, j_s]$ be a multisegment, such that 
\[
 [i_1, j_1]\unlhd \cdots \unlhd [i_r, j_r](  \text{ cf. Def. \ref{def: 7.2.4}}), 
\]
we define 
\[
 E(\a)=\frac{1}{\prod_s [a_{i_s,j_s}]_{q^{1/2}}!}E([i_1,j_1])^{a_{i_1, j_1}}\cdots E([i_r, j_r])^{a_{i_r,j_r}},
\]
here $[m]_{q^{1/2}}=\frac{q^{1/2m}-q^{-1/2m}}{q^{1/2}-q^{-1/2}}$ for $m\in \Z$ and 
$[m]_{q^{1/2}}!=[m]_{q^{1/2}}[m-1]_{q^{1/2}}\cdots [2]_{q^{1/2}}$.

\end{definition}

\begin{definition}
Let $x\mapsto \line{x}$ be the involution defined as the unique ring  automorphism of $U_q^+$ defined by 
\[
 \line{q^{1/2}}=q^{-1/2},~ \line{E_i}=E_i.
\]
\end{definition}

\begin{prop}(cf. \cite{Lu})
Let $\mathcal{L}:=\bigoplus_{\a\in \O}\Z[q^{1/2}]E(\a)\subseteq U_q^+$. Then 
there exists a unique $\Q(q^{1/2})$-basis $\{G(\a): \a\in \O\}$ of $U_q^+$ such 
that 
\[
 \line{G(\a)}=G(\a), ~ G(\a)=E(\a) \text{ modulo }q^{1/2}\mathcal{L}.
\]
This is called Lusztig's canonical basis. 
\end{prop}

Lusztig also gave a geometric description 
of his canonical basis in terms of 
the orbital varieties $\line{O}_{\a}$.

\begin{definition}
Let $\A$ be the group ring of $\line{\Q}^*_{\l}$ over $\Z$.
Let $\K_{\varphi}$ be the Grothendieck group 
over $\A$ of the category of constructible, $G_{\varphi}$-equivariant
$\Q_{\l}$ sheaves over $E_{\varphi}$,  considered as a variety over a 
finite field $\mathbb{F}_q$. 
\end{definition}

\begin{lemma}(cf. \cite{Lu})
The $\A$-module $K_{\varphi}$ admits  a basis  $\{\gamma_{\a}: \a\in S(\varphi)\}$
indexed by the $G_{\varphi}$ orbits of $E_{\varphi}$, where $\gamma_{\a}$ corresponds 
to the constant  sheaf $\line{\Q}_{\l}$ on the orbit $O_{\a}$, extending by $0$ to the complement.
\end{lemma}

\begin{definition}\label{def: 7.4.12}
Let $\varphi=\varphi_1+\varphi_2\in \mathcal{S}$.
We define a diagram of varieties
\addtocounter{theo}{1}
\begin{equation}\label{di: 7.3}
\xymatrix
{
E_{\varphi_1}\times E_{\varphi_2}& E'{\ar[l]_{\hspace{0.8cm}\beta}}\ar[r]^{\beta'}&E''\ar[r]^{\beta''}&E_{\varphi},
}
\end{equation}
where 
\begin{align*}
 E'':=&\{(T, W): W=\bigoplus W_i,~ W_i \subseteq V_{\varphi, i},~ T(W_i)\subseteq W_{i+1}, ~\dim(W_i)=\varphi_2(i)\},\\
 E':=& \{(T, W, \mu, \mu'): (T, W)\in E'',~ \mu: W\simeq V_{\varphi_2}, ~\mu': V_{\varphi}/W\simeq V_{\varphi_1}\},
\end{align*}
and 
\[
 \beta''((T, W))=W, ~ \beta'((T, W, \mu, \mu'))=(T, W), ~\beta((T, W, \mu, \mu'))=(T_1, T_2),
\]
such that 
\[
 T_1=\mu' \circ T\circ \mu'^{-1}, ~ T_2=\mu\circ T\circ \mu^{-1}.
\]
\end{definition}

\begin{prop}(cf. \cite{Lu})
The group $G_{\varphi}\times G_{\varphi_1}\times G_{\varphi_2}$ 
acts naturally on the varieties in the diagram (\ref{di: 7.3})
with $G_{\varphi}$ acting trivially on $E_{\varphi_1}\times E_{\varphi_2}$
and $G_{\varphi_1}\times G_{\varphi_2}$ acting trivially on $E_{\varphi}$.
And all the maps there are compatible with such actions. Moreover,
we have 
\begin{description}
 \item[(1)]The morphism $\beta'$ is 
 a principle $G_{\varphi_1}\times G_{\varphi_2}$-fibration.
 \item[(2)]The morphism $\beta$ is a locally trivial trivial 
 fibration with smooth connected fibers.
 \item[(3)]The morphism $\beta''$ is proper.
\end{description} 
\end{prop}

\begin{example}\label{ex: 7.4.14}
Let $\varphi_1=\chi_{1}, ~\varphi_{2}=\chi_2$. Then $\varphi=\chi_1+\chi_2$
and 
\[
 E_{\varphi_1}=E_{\varphi_2}=0, ~ E_{\varphi}=\line{\mathbb{F}}_{q}.
\]
Moreover, we have 
\[
 E''=\{(T, W): W=V_{\varphi_2}, T\in \line{\mathbb{F}}_q\}\simeq \line{\mathbb{F}}_q, 
\]
and 
\[
 E'=\{(T, W, \mu, \mu'): (T, W)\in E'', \mu, \mu'\in \line{\mathbb{F}}_q^{\times}\}
 \simeq \line{\mathbb{F}}_q\times (\line{\mathbb{F}}_q^{\times})^2.
\]
\end{example}

\begin{cor}(cf. \cite{Lu})\label{cor: 7.4.15}
Let $\a\in \O(\varphi_1), \a'\in \O(\varphi_2)$.
There exists a simple perverse sheaf( up to shift )  $\mathcal{P}$ such that 
\[
  \beta^*(IC(\line{O}_{\a})\otimes IC(\line{O}_{\a'}))=\beta'^{*}(\mathcal{P}). 
\]
\end{cor}

\begin{example}
As in example \ref{ex: 7.4.14}, 
let $\a=\{[1]\}, \a'=\{[2]\}$, then 
\[
 IC(\line{O}_{\a})=\line{\Q}_{\l}, ~IC(\line{O}_{\a'})=\line{\Q}_{\l}.
\]
Hence if we let 
\[
 \mathcal{P}=\line{\Q}_{\l},
\]
then 
\[
 \beta^{*}(IC(\line{O}_{\a})\otimes IC(\line{O}_{\a'}))=\beta''^{*}(\mathcal{P}).
\]

\end{example}

\begin{definition}
 We define a multiplication  
\[
  IC(\line{O}_{\a})\star IC(\line{O}_{\a'})=\beta''_{*}(\mathcal{P}).
 \]
\end{definition}

\begin{example}
As in the example \ref{ex: 7.4.14}, we have 
\[
 IC(\line{O}_{\a})\star IC(\line{O}_{\a'})=\beta''_{*}(\mathcal{P})=IC(E_{\varphi}),
\]
note that here $\beta''$ is an isomorphism. 
\end{example}

\begin{prop}(cf. \cite{Lu})
 Let $\a\in \O(\varphi_1), \a'\in \O(\varphi_2)$.
We associate to the intersection cohomology complex $IC(\line{O}_{\a})$
\[
 \tilde{\gamma}_{\a}=\sum_{\b\geq \a}p_{\b, \a}(q)\gamma_{\b},
\]
where $p_{\b, \a}(q)$ is the formal alternative sum 
of eigenvalues of the Frobenius map on the stalks of the cohomology sheaves 
of $IC(\line{O}_{\a})$ at any $\mathbb{F}_q$ rational point of $O_{\b}$. 
Moreover, the multiplication $\star$ gives a  $\A$-bilinear map 
\[
 \K_{\varphi_1}\times \K_{\varphi_2}\rightarrow \K_{\varphi},
\]
which defines an associative algebra structure over $\K=\bigoplus_{\varphi}\K_{\varphi}$. 
\end{prop}

To relate the algebra $K$ and $U^{\geq 0}$
\begin{prop}(\cite{Lu} Prop. 9.8, Thm. 9.13)\label{prop: 7.4.20}

\begin{itemize}
 \item
The elements $\gamma_{i}:=\gamma_{[i]}$ for all $i\in \Z$
generate the algebra $\K$ over $\A$.
 
 \item Let $U^{\geq 0}_{\A}=U_q^{\geq 0}\otimes_{\Z} \A$. 
 Then we have a unique $\A$-algebra morphism 
 $\Gamma: \K\rightarrow U_{\A}^{\geq 0}$ such that 
 \[
  \Gamma(\gamma_{j})=K_j^{-j}E_{j}; 
 \]
for all $j\in \Z$.  Moreover, for $\varphi\in \mathcal{S}$, let
\[
 S(\varphi)=\sum_{i\in \Z}(\varphi(i)-1)\varphi(i)/2- \sum_{i\in \Z}\varphi(i)\varphi(i+1).
\]
Then there is an $\A$-linear map $\Theta: K_{\varphi}\rightarrow U_{\A}^{+}$, such that 
\[
 \Gamma(\xi)=q^{1/2S(\varphi)}K(\varphi)\Theta(\xi),
\]
where $K(\varphi)=\prod_{i\in \Z}K_{i}^{-i\varphi(i)}$. 
 \item 
 We have 
 \[
 \Gamma(\gamma_{\c})=
 q^{1/2(r-\delta_{\c})}K(\varphi_{\c})E(\c), 
 \]
where 
\[
 r=\sum_{i}\varphi_{\c}(i)(\varphi_{\c}(i)-1)(2i-1)/2-\sum_{i}i\varphi_{\c}(i-1)\varphi_{\c}(i), 
\]
and $\delta_{\c}$ is the co-dimension of the orbit $O_{\c}$ in $E_{\varphi_{\c}}$.

 \item We have 
 \[
  \Theta(\gamma_{\a})=q^{1/2\dim(O_{\a})}E(\a),~ \Theta(\tilde{\gamma}_{\a})=q^{1/2\dim(O_{\a})}G(\a).
 \]
Hence 
\[
 G(\a)=\sum_{\b\geq \a}P_{\b, \a}(q)E(\b). 
\]

\end{itemize}
\end{prop}

\begin{prop}\label{prop: 7.4.21}
The canonical basis of $U_q^+$ are almost orthogonal with 
respect to a scalar product introduced by Kashiwara \cite{K}, 
which are given by 
\[
 (E(\a), E(\b))=\frac{(1-q)^{\deg(\a)}}{\prod_{i\leq j}h_{a_{ij}}(q)}\delta_{\a, \b},
\]
where $\a=\sum_{i\leq j}a_{ij}[i,j]$,
$h_{k}(z)=(1-z)\cdots (1-z^k)$ and $\delta$ is the Kronecker symbol(\cite{LTV}).
And we have 
\[
 (G(\a), G(\b))=\delta_{\a, \b}\text{  mod }q^{1/2}\A.
\]

\end{prop}

\begin{notation}
We denote by $\{E^*(\a)\}$ and $\{G^*(\a)\}$ 
the dual basis of $\{E(\a)\}$ and $\{G(\a)\}$ with 
respect to the Kashiwara scalar product.
\end{notation}

\begin{prop}(cf. \cite{LNT})
Let $\a=\sum_{i\leq j}a_{ij}[i,j]$. Then 
 \begin{itemize}
  \item We have 
 \[
  E^*(\a)=\frac{\prod\limits_{i\leq j}h_{a_{ij}}(q)}{(1-q)^{\deg(\a)}}E(\a)=
  \overrightarrow{\prod\limits_{ij}}q^{1/2\binom{a_{ij}}{2}}E^*([i,j])^{a_{ij}},
 \]
here the product is taken with respect to the order $\leq $. 

\item And 
  \[
   E^*(\a)=\sum_{\b\leq \a}P_{\a, \b}(q)G^*(\b).
  \]

 \end{itemize} 
\end{prop}

\begin{example}
 Let $\a=[1]+[2]$. Then 
 \[
  E^*(\a)=E([1])E([2])=E(\a),
 \]
 and 
 \[
  G^*([1,2])=E^*([1,2]),
 \]
\[
 G^*(\a)=E^*(\a)-q^{1/2}E^*([1,2]).
\]
\end{example}

Finally, we establish the relation of 
between the algebras $\cal{R}$ and $U^+$.

\begin{definition}
 Let $B$ be the polynomial algebra generated by 
 the set of coordinate functions $\{t_{ij}: i< j\}$.
 Following \cite{LNT}, we write $t_{ii}=1$, 
 $t_{ij}=0$ if $i>j$, and indexed the non-trivial $t_{i,j}$'s 
 by segments, namely, $t_{[ij]}=t_{i,j-1}$ for $i<j$. 
\end{definition}

Now by corollary \ref{cor: 1.2.3}, we have the following 

\begin{prop}\label{prop: 7.4.25}
 We have an algebra isomorphism $\phi: B\simeq \cal{R}$ by identifying 
 $t_{[ij]}$ with $L_{[ij]}$ for all $i<j$.  
\end{prop}

\begin{definition}
 Let $B_q$ be the quantum analogue of $B$ generate by 
 $\{T_{ij}: i<j\}$, where $T_{ij}$ is considered as 
 the $q$-analogue of $t_{ij}$. Also, we write $T_{ii}=1$
 and $T_{ij}=0$ if $i>j$. And we will indexed the non-trivial 
 $T_{ij}$ by $T_{[i,j-1]}$. The generators $T_s$'s satisfies 
 the following relations (cf. \cite{BZ2}). Let $s>s'$ be two segments.
 Then 
 \begin{displaymath}
 T_{s'}T_{s}=
 \left\{ \begin{array}{cc}
    q^{- 1/2(\wt(s'), \wt(s))}T_sT_{s'}+(q^{-1/2}-q^{1/2})T_{s\cap s'}T_{s\cup s'},
    &\text{  if } s \text{ and } s' \text{ are linked, }\\
    q^{-1/2(\wt(s'), \wt(s))}T_sT_{s'}, \text{ otherwise }.
   \end{array} \right.
\end{displaymath}
\end{definition}

\begin{prop}(cf. \cite{LNT} Section 3.5)\label{prop: 7.4.23}
There exist an algebra isomorphic morphism 
\[
 \iota: U_q^+\rightarrow B_q, 
\]
given by $\iota(E^*([i,j]))=T_{[i,j]}$. Moreover, for $\a=\sum_{i\leq j}a_{ij}[i,j]$,
we have 
\[
 \iota(E^*(\a))=\overrightarrow{\prod_{i\leq j}}q^{1/2\binom{a_{ij}}{2}}T_{[i,j]}^{a_{ij}},
\]
here the multiplication is taken with respect to the order $<$.
\end{prop}

\begin{example}
Let $\a=[1]+[2]$, then 
\[
 \iota(E^*(\a))=T_{[1]}T_{[2]}.
\]

\end{example}

\begin{prop}
By specializing at $q=1$, the dual canonical basis $\{G^*(\a): \a\in \O\}$ gives rise
to a well defined basis for $B$, denoted by $\{g^*(\a): \a\in \O\}$. Moreover, 
the morphism $\phi$ sends $g^*(\a)$ to $L_{\a}$ for all $\a\in \O$. 
\footnote{ It is surprising 
that an isomorphism in the commutative world is governed by a non-commutative
one, such phenomenon also happens in the theory of periods, where a period 
be a 
complex number whose real and imaginary parts are values
of absolutely 
convergent integrals of rational functions with rational coefficients, over
domains in $\R_n$ given by polynomial inequalities with rational coefficients, then 
there is a conjecture saying that two rational functions give the same period if and 
only if they can be transformed to each other according to three simple rules, see 
\cite{KZ} chapter 1.} 
\end{prop}

\section{Partial Derivatives and Poincar\'e's series}

In this section we will deduce a geometric description for the 
partial derivatives, using results of last section.

\begin{definition}
Kashiwara \cite{K} introduced some $q$-derivations $E'_i$ in $\End(U_q^{+})$
for all $i\in \Z$ satisfying 
\[
 E_i'(E_j)=\delta_{ij}, ~E_i'(uv)=E_i'(u)v+q^{-1/2(\alpha_i, \wt(u))}uE_i'(v). 
\]
\end{definition}

\begin{example}
Simple calculation shows that  
\[
E_i'(E([j,k]))=\delta_{i, k}(1-q)E([j, k-1]), 
\]
 by taking dual, we get 
 \[
  E_i'(E^*([j,k]))=\delta_{i, k}E^*([j, k-1]), 
 \] 
\end{example}

\begin{prop}
We have 
\[
 (E_i'(u), v)=(u, E_iv),
\]
where $( , )$ is the scalar product introduced in 
proposition \ref{prop: 7.4.25}. 
\end{prop}

Note that by identifying 
the algebra $U_q^{+}$ and 
$B_q$ via $\iota$, we get a version of  
$q$-derivations in $\End(B_q)$. 

\begin{definition}
By specializing at $q=1$, 
the $q$ derivation $E_i'$ gives a derivation $e_i'$ of the algebra $B$ by  
 \[
  e_i'(t_{[jk]})=\delta_{ik}t_{[j,k-1]}, 
  ~ e_i'(uv)=e_i'(u)v+ue_i'(v). 
 \] 
\end{definition}

\begin{prop}\label{prop: 7.6.5}
Let 
\[
 D^i: =\sum_{n=0}^{\infty}\frac{1}{n!}{e_i'}^n.
\]
Then the morphism $D^i: B\rightarrow B$ is an 
algebraic morphism. Moreover, if we identify 
the algebras $\mathcal{R}$ and $B$ via $\phi$, then 
the morphism $D^i$ coincides with  the partial derivative
$\D^i$.  

\end{prop}

\begin{proof}
For $n\in \N$, we have 
\[
 e'^n(uv)=\sum_{r+s=n}\binom{n}{r}e_{i}'^r(u)e_{i}'^s(v), 
\]
therefore 
\[
 D^i(uv)=\sum_{n=0}^{\infty}\frac{1}{n!}\sum_{r+s=n}\binom{n}{r}e_{i}'^r(u)e_{i}'^s(v)=D^i(u)D^i(v). 
\]
Finally, to show that $D^i$ and $\D^i$ coincides, 
it suffices to prove that 
\[
 \phi\circ D^i(t_{[j,k]})=\D^i\circ \phi(t_{[j,k]}),
\]
but we have 
\[
 D^i(t_{[j,k]}=t_{[j,k]}+\delta_{i,k}t_{[j, k-1]},
\]
 and 
 \[
  \D^i(L_{[j,k]})=L_{j,k}+\delta_{i,k}L_{[j, k-1]}.
 \]
Therefore, we have 
\[
 \phi\circ D^i(t_{[j,k]})=\D^i\circ \phi(t_{[j,k]}).
\]  
\end{proof}

\remk Without specializing at $q=1$, 
the operator $D^i$ is not an algebraic morphism.  
To get an algebraic morphism at the level of 
$U_q^+$, one should consider not only 
the summation of the iteration of $e_i'$'s
but all the derivations, which 
gives rise to an embedding into the  quantum shuffle algebras, cf. \cite{L}.

Next we show how to determine $\D^i(L_{\a})$
by the algebra $\K$ of Lusztig.

\begin{lemma}\label{lem: 7.6.6}
Let $n\in \N$, and $ \d\in \O$. Then we have  
\[
 E_i^nG(\b)=\sum_{\d}(E_i^nG(\b), E^*(\d))E(\d)=\sum_{\d}(G(\b), E_i'^nE^*(\d))E(\d),
\]
where $( , )$ is the Kashiwara scalar product. 
Moreover, for each $\b$ such that 
$(G(\b), E_i'^nE^*(\d))\neq 0$, we have 
\[
 \wt(\d)=\wt(\b)+n\alpha_i.
\]

\end{lemma}

\begin{proof}
This is by definition. 
\end{proof}

\begin{cor}
Let $\b\preceq_k \d$ 
such that $wt(\d)=\wt(\b)+n\alpha_i$. 
Then $L_{\b}$ appears as a factor of $\frac{1}{r!}e_i'^r(\pi(\d))$
if and only if $r=n$. 

\end{cor}

\begin{proof}
We know that for each $\b\preceq_k \d$, 
the representation $L_{\b}$ is a factor of $\D^i(\pi(\d))$.  
Now by proposition \ref{prop: 7.6.5}, 
$\D^i=\sum_{r}\frac{1}{r!}e_i'^r$, moreover, 
by lemma \ref{lem: 7.6.6}, factors of $\frac{1}{r!}e_i'^r(\pi(\d))$
always have weight $\wt(\d)-r\alpha_i$. 
Therefore we are done. 
\end{proof}

\begin{prop}\label{prop: 7.3.8}
Let $\b\preceq_k \a$, then
there exists $\c\in S(\a)$ such that $\c=\b+\l[k]$. Then 
\[
 n_{\b, \a}=\sum_{i}\dim\mathcal{H}^{2i}( IC(\line{O}_{\l[k]})\star IC(\line{O}_{\b}))_{\a}.
\] 
\end{prop}

\begin{proof}
Note that by proposition \ref{prop: 7.4.20}, we have  
\begin{align*}
 \Gamma(\tilde{\gamma}_{\l[k]}\star \tilde{\gamma}_{\b})&=
 \Gamma(\tilde{\gamma}_{\l[k]})\Gamma(\tilde{\gamma}_{\b})\\
 &=q^{1/2(S(\varphi_{\l[k]})+S(\varphi_{\b}))}K(\varphi_{\l[k]})K(\varphi_{\b})\Theta(\tilde{\gamma}_{\l[k]})
 \Theta(\tilde{\gamma}_{\b})\\
 &=q^{1/2(S(\varphi_{\l[k]})+S(\varphi_{\b}))}K(\varphi_{\l[k]}+\varphi_{\b})q^{1/2(\dim(O_{\l[k]})+\dim(O_{\b}))}G(\l[k])
 G(\b).
 \end{align*}
Since we have 
\[
S(\varphi_{\l[k]})=\dim(O_{\l[k]})=0, ~G(\l[k])=E(\l[k])=\frac{1}{[\l]_{q^{1/2}}!}E_{k}^{\l},
~\varphi_{\l[k]}+\varphi_{\b}=\varphi_{\a},
\]
so 
\[
 \Gamma(\tilde{\gamma}_{\l[k]}\star \tilde{\gamma}_{\b})=
 \frac{1}{[\l]_{q^{1/2}}!}q^{1/2(S(\varphi_{\b})+\dim(O_{\b}))}K(\varphi_{\a})E_{k}^{\l}G(\b).
\]
And 
\begin{align*}
 \Gamma(\gamma_{\d})&=q^{1/2S(\varphi_{\d})}K(\varphi_{\d})\Theta(\gamma_{\d})\\
 &=q^{1/2(S(\varphi_{\d})+\dim(O_{\d}))}K(\varphi_{\d})E(\d).
\end{align*}
Now write
\[
 \tilde{\gamma}_{\l[k]}\star \tilde{\gamma}_{\b}=\sum_{\b\preceq_k \d, \varphi_{\d}=\varphi_{\a}}p_{\d, \b}(q)\gamma_{\d},
\]
with 
\[
 p_{\d, \b}(q)=\sum_{i}q^{i}\mathcal{H}^{2i}(IC(\line{O}_{\l[k]})\star IC(\line{O}_{\b}))_{\d}.
\]
Applying $\Gamma$ gives 
\begin{multline*}
 \frac{1}{[\l]_{q^{1/2}}!}q^{1/2(S(\varphi_{\b})+\dim(O_{\b}))}K(\varphi_{\a})E_{k}^{\l}G(\b)= \\
 \sum_{\b\preceq_k \d, \varphi_{\d}=\varphi_{\a}}p_{\d, \b}(q)q^{1/2(S(\varphi_{\d})+\dim(O_{\d}))}K(\varphi_{\d})E(\d).
\end{multline*}
Hence 
\[
 E_{k}^{\l}G(\b)=[\l]_{q^{1/2}}!\sum_{\b\preceq_k \d, \varphi_{\d}=\varphi_{\a}}p_{\d, \b}(q)
 q^{1/2(S(\varphi_{\d})+\dim(O_{\d})-S(\varphi_{\b})-\dim(O_{\b}))}E(\d),
\]
now compare with lemma \ref{lem: 7.6.6},
we get 
\[
 (G(\b), E_i'^nE^*(\d))=[\l]_{q^{1/2}}!p_{\d, \b}(q)
 q^{1/2(S(\varphi_{\d})+\dim(O_{\d})-S(\varphi_{\b})-\dim(O_{\b}))}.
\] 
Finally, 
we write 
\[
 \frac{1}{[\l]_{q^{1/2}}!}E_i'^nE^*(\d)=\sum_{\b}n_{\b, \d}(q)G^{*}(\b),
\]
by applying the scalar product, we get 
\[
 n_{\b, \d}(q)=(G(\b), \frac{1}{[\l]_{q^{1/2}}!}E_i'^nE^*(\d))
 =p_{\d, \b}(q)
 q^{1/2(S(\varphi_{\d})+\dim(O_{\d})-S(\varphi_{\b})-\dim(O_{\b}))}.
\]
Hence, by specializing at $q=1$, we have 
\[
 n_{\b, \d}=p_{\d, \b}(1).
\]
Now take $\d=\a$, we get the formula 
in our proposition.
\end{proof}

\section{A formula for  Lusztig's product}

In this section we will find a geometric way to calculate Lusztig's product in special case, which allows us 
to determine the partial derivatives in the following sections.

\begin{definition}
Let $k\in \Z$.
We say that  $\a$ satisfies the assumption 
$(\A_k)$ if it satisfies the following conditions
\footnote{Since here we only work with the partial derivative 
$\D^k$ with $k\in \Z$, 
for every multisegment, we can always use the reduction 
method to increase the length of segments from the left, so 
that at some point we 
arrive at the situation of our assumption $(\A_k)$, therefore we do not lose the 
generality.}
\begin{description}
 \item[(1)] We have 
\[
 \max\{b(\Delta): \Delta\in \a\}+1<\min\{e(\Delta): \Delta\in \a\}.
\]
\item[(2)] 
Moreover, we have 
$\varphi_{e(\a)}(k)\neq 0$ and $\varphi_{e(\a)}(k+1)=0$.

\end{description}

\end{definition}

\begin{lemma}\label{lem: 7.6.1}
Let 
$\a$ be a multisegment satisfying 
the assumption $(\A_k)$. 
Then $\a$ is of parabolic type. Moreover, 
The set $S(\varphi_{\a})$ contains a unique maximal
element satisfying the assumption $(\A_k)$, denoted by $\a_{\Id}$. 
\end{lemma}

\begin{proof}
Let $b(\a)=\{k_1\leq \cdots\leq k_r\}$,
$e(\a)=\{\ell_1\leq \cdots \leq \ell_r\}$. 

Then by proposition \ref{prop: 6.3.2}, 
we know that there exists an element 
$w\in S_r^{J_1(\a), J_2(\a)}$, such that 
\[
 \a=\sum_{j}[k_j, \ell_{w(j)}].
\]
Let 
\[
 \a_{\Id}=\sum_{j}[k_j, \ell_j],
\]
now by proposition \ref{prop: 6.2.4}, we know that 
$\a\leq \a_{\Id}$. 
Finally, $\a_{\Id}$ depends only 
on $b(\a)$ and $e(\a)$, not on $\a$, 
which shows that $\a_{\Id}$ is the maximal element
in $S(\varphi_{\a})$ satisfying the assumption $(\A_k)$. 
\end{proof}

\begin{lemma}\label{lem: 7.6.2}
Suppose that $\a$ is a multisegment satisfying the hypothesis $(\A_k)$, then 
 
 \begin{description}
  \item[(1)] $\tilde{S}(\a)_k=S(\a)$;

\item[(2)] we have 
 \[
  X_{\a}^k=Y_{\a}=\coprod_{\c\in S(\a)}O_{\c}.
 \]
  \end{description}
\end{lemma}

\begin{proof}
Note that by assumption  
\[
 \max\{b(\Delta): \Delta\in \a\}<\min\{e(\Delta): \Delta\in \a\}.
\]
This ensures that for any $\c\in S(\a)$, 
we have $\varphi_{e(\c)}(k)=\varphi_{e(\a)}(k)$,
hence by definition $\c\in \tilde{S}(\a)_{k}$. 
This proves (1), and (2) follows from (1).
\end{proof}

\begin{lemma}\label{lem: 7.6.3}
Let $\a$ be a multisegment 
satisfying the assumption $(\A_k)$
 and $\a=\a_{\Id}$. 
 Let $\ell\in \N$ such that $\ell\leq \varphi_{e(\a)}(k)$
 and   $\varphi\in \mathcal{S}$ such that 
 \[
  \varphi+\ell \chi_{[k]}=\varphi_{\a}.
 \]
Then for $\b\in S(\varphi)$, we have 
$\b\preceq_k \a$ if and only if $\b^{(k)}\leq \a^{(k)}$
and $\varphi_{e(\b)}(k-1)=\ell+\varphi_{e(\a)}(k-1)$.
 
\end{lemma}

\begin{proof} 
Let 
$\b\in S(\varphi)$ such that 
$\b\preceq_k \a$, then 
by proposition \ref{prop: 7.3.5},
we know that 
$\b=\c_{\Gamma} $ for some 
$\c\in S(\a)$ and $\Gamma\subseteq \c(k)$.
Therefore 
\[
 \b^{(k)}=\c^{(k)}\leq \a^{(k)}
\]
by the lemma above. And by definition of 
$\c_{\Gamma}$, we know that 
\[
 \varphi_{e(\b)}(k-1)=\ell+\varphi_{e(\c)}(k-1).
\]
Now applying the fact that 
$\a$ satisfies the assumption $(\A_k)$, we deduce that 
\[
 \varphi_{e(\c)}(k-1)=\varphi_{e(\a)}(k-1).
\]
Conversely, let $\b\in S(\varphi)$ be a multisegment  such that 
$\b^{(k)}\leq \a^{(k)}$ and 
$\varphi_{e(\b)}(k-1)=\ell+\varphi_{e(\a)}(k-1)$.

We deduce from $\b^{(k)}\leq \a^{(k)}$ that 
\[
 \b\leq \a^{(k)}+\varphi_{e(\b)}(k)[k],
\]
from which we obtain
\[
 \varphi_{\b}=\varphi_{\a^{(k)}}+\varphi_{e(\b)}(k)\chi_{[k]}.
\]
By assumption, we know that 
\[
 \varphi_{\b}+\ell\chi_{[k]}=\varphi_{\a}.
\]
Combining with the formula
\[
 \varphi_{\a}=\varphi_{\a^{(k)}}+\varphi_{e(\a)}(k)\chi_{[k]},
\]
we have 
\[
 \varphi_{e(\a)}(k)=\varphi_{e(b)}(k)+\ell.
\]
Now that for any 
$\Delta\in \a$, if $e(\Delta)=k$, 
then $b(\a)\leq k-1$. Therefore we have 
\[
 \varphi_{e(\a^{(k)})}(k-1)=\varphi_{e(\a)}(k-1)+\varphi_{e(\a)}(k).
\]
Applying the formula
$\varphi_{e(\b)}(k-1)=\ell+\varphi_{e(\a)}(k-1), ~\varphi_{e(\b^{(k)})}(k-1)=\varphi_{e(\a^{(k)})}(k-1)$, we get 
\[
  \varphi_{e(\b^{(k)})}(k-1)=\varphi_{e(\b)}(k-1)+\varphi_{e(\b)}(k).
\]
Such a formula implies that for $\Delta\in \b$, 
if $e(\Delta)=k$, then 
$b(\Delta)\leq k-1$. 

Let $b(\a)=\{k_1\leq \cdots\leq k_r\}$,
$e(\a)=\{\ell_1\leq \cdots \leq \ell_r\}$. 
The assumption that $\a=\a_{\Id}$ implies that 
\[
 \a=\sum_i [k_i, \ell_i]
\]
Suppose that 
\[
 \a(k)=\{[k_i,\ell_i]: i_0\leq i\leq i_1\}.
\]
Take 
$\Gamma=\{[k_i, \ell_i]: i_0+\ell\leq i\leq i_1\}$ and 
\[
 \a'=\a_{\Gamma}.
\]
Then $\a'\preceq_k \a$.  
Note that $\a'$ is a multisegment of 
parabolic type which 
corresponds to the identify 
in some symmetric group, 
cf. notation \ref{nota: 6.2.21}. 
Finally, proposition \ref{prop: 6.2.4} 
implies that $\b\in S(\a')$.

\end{proof}

\begin{lemma}\label{lem: 7.7.2}
Assume that $\a$ is 
a multisegment satisfying $(\A_k)$.
Let $r\leq \varphi_{e(\a)}(k)$ and 
$\d=\a+r[k+1]$. Then 
we have 
$X_{\d}^{k+1}=Y_{\d}$ and for a 
fixed subspace $W$ of $V_{\varphi_{\d}, k+1}$
of dimension $r$,
the open immersion 
\[
 \tau_W: (X_{\d}^{k+1})_{W}\rightarrow (Z^{k+1, \d})_W\times \Hom(V_{\varphi_{\d}, k}, W)
\]
is an isomorphism.
   
\end{lemma}

\begin{proof}
Note that our assumption on $\a$ 
ensures that $X_{\d}^{k+1}=Y_{\d}$ since 
we have $\d_{\min}\in \tilde{S}(\d)_{k+1}$. 
It suffices to show that $\tau_W$ is surjective.
Let $(T^{(k)}, T_0)\in (Z^{k+1, \d})_W\times \Hom(V_{\varphi_{\d}, k}, W)$,
by fixing a splitting $V_{\varphi_{\d}, k+1}=W\oplus V_{\varphi_{\d}, k+1}/W$, we define 
 \[
  T'|_{V_{\varphi_{\d}, k}}=T_{0}\oplus T^{(k+1)}|_{V_{\varphi_{\d}, k}},
 \]
\[
 T'|_{V_{\varphi_{\d}, k+1}}=T^{(k+1)}|_{V_{\varphi_{\d}, k+1}/W}\circ p_{W},
\]
\[
 T'|_{V_{\varphi_{\d}, i}}=T^{(k+1)}, \text{ for } i\neq k, k+1,
\]
where $p_w: V_{\varphi_{\d}}\rightarrow  V_{\varphi_{\d}, k}/W$ is the 
canonical projection.
Then we have $T'\in Y_{\d}$ hence $T'\in (X_{\d}^{k+1})_{W}$. 
Now since by construction we have $\tau_W(T')=(T^{(k)}, T_0)$, we are done.
\end{proof}

\begin{definition}\label{def: 7.4.6}
Assume that $\a$ is 
a multisegment satisfying $(\A_k)$
and $\d=\a+r[k+1]$ for some $r\leq \varphi_{e(\a)}(k)$.
Let $\mathfrak{X}_{\d}$ be the open sub-variety of 
$X_{\d}^{k+1}$ consisting of those orbits 
$O_{\c}$ with $\c\in S(\d)$, such that  
 $\varphi_{e(\c)}(k)+r=\varphi_{e(\a)}(k)$. 
\end{definition}

\begin{definition}
Let $V$ be a vector space and $\ell_1<\ell_2<\dim(V)$ be two integers.
We define 
\[
 Gr(\ell_1, \ell_2, V)=\{(U_1, U_2): U_1\subseteq U_2\subseteq V, \dim(U_1)=\ell_1, \dim(U_2)= \ell_2\}.
\]

\end{definition}

\begin{definition}
Let $\ell$ be an integer and 
$\a$ be a multisegment.
We let 
\[
E''_{\a}=\{(T', W'): T'\in Y_{\a}, W'\in Gr(\ell, \ker(T'|_{V_{\varphi_{\a}, k}}))\} .
\]
Note that we have a canonical morphism 
\[
 \alpha': E''_{\a}\rightarrow Gr(\ell, \varphi_{e(\a)}(k), V_{\varphi_{\a}, k})
\]
sending $(T',W')$ to $(W', \ker(T'|_{V_{\varphi_{\a}, k}}))$.

\end{definition}

\begin{prop}\label{prop: 7.7.6}
The morphism $\alpha'$ is a fibration.   
\end{prop}

\begin{proof}
The morphism $\alpha'$ is equivariant under the action of $GL(V_{\varphi_{\a}, k})$.
The same proof as in proposition \ref{prop: 4.6.10} shows that 
the morphism $\alpha'$ is actually a $P_{(U_1, U_2)}$ bundle, where
$P_{(U_1, U_2)}$  is a subgroup of $GL(V_{\varphi_{\a}, k})$ which fixes the given 
element $(U_1, U_2)$.  
Now we take a Zariski neighborhood $\mathfrak{U}$ of 
$(U_1, U_2)$ over which we have the trivialization
\[
 \gamma: \alpha'^{-1}(\mathfrak{U})\simeq \alpha'^{-1}((U_1, U_2))\times \mathfrak{U}, 
\]
such an isomorphism comes from a section 
\[
 s: \mathfrak{U} \rightarrow GL(V_{\varphi_{\a}, k}),
 ~s((U_1, U_2))=Id,
\]
by $\gamma((T, W'))=[(g^{-1}T, g^{-1}W'), \alpha'((T, W'))]$, where $g=s(\alpha'((T, W')))$.
We remark that the existence of the section $s$
is guaranteed by local triviality of  $GL(V_{\varphi_{\a}, k})\rightarrow GL(V_{\varphi_{\a}, k})/P_{(U_1, U_2)}$,
cf.  \cite{S}, $\S$ 4.  
\end{proof}

\begin{prop}\label{prop: 7.7.7}
Assume that $\a$ is 
a multisegment satisfying $(\A_k)$
and $\d=\a+r[k+1]$ for some $r\leq \varphi_{e(\a)}(k)$.
Let $\ell\in \N$ such that $r+\ell=\varphi_{e(\a)}(k)$
and $W$ a subspace of 
$V_{\varphi_{\d}, k+1}$ 
such that $\dim(W)=r$.
We have a canonical projection
 \[
  p: (\mathfrak{X}_{\d})_W\rightarrow E''_{\a}
 \]
where for $T\in (\mathfrak{X}_{\d})_W$ with 
$\tau_W(T)=(T_1, T_0)\in (Z^{k+1, \d})_W\times \Hom(V_{\varphi_{\d}, k}, W)$, we define
$p(T)=(T_1, \ker(T_0|_{W_1}))$,  
where $W_1=\ker(T_1|_{V_{\varphi_{\d}, k}})$
(Note that here we identify $(Z^{k+1, \d})_W$ with $Y_{\a}$, see the remark after proposition \ref{prop: 3.3.13} ). 
Moreover, 
let $U_1\subseteq U_2\subseteq V_{\varphi_{\d}, k}$ be subspaces
such that $\dim(U_1)=\ell,~ \dim(U_2)=\varphi_{e(\a)}(k)$, then 
$p$ is a fibration 
with fiber 
$$\{T\in \Hom(V_{\varphi_{\d}, k}, W): \ker(T|_{U_2})=U_1 \}.$$
\end{prop}

\begin{proof}
We show that $p$ is well defined.
Since by definition of $(\mathfrak{X}_{\d})_W$ we know that 
\[
 \dim(W)+\dim(\ker(T_0|_{\ker(T_1|_{V_{\varphi_{\d}, k}})}))=\dim(\ker(T_1|_{V_{\varphi_{\d}, k}})),
\]
hence to see that 
\[
 \ell =\dim(\ker(T_0|_{\ker(T_1|_{V_{\varphi_{\d}, k}})})),
\]
it suffices to show that 
\[
 \varphi_{e(\a)}(k)=\dim(\ker(T_1|_{V_{\varphi_{\d}, k}})),
\]
this follows from the fact that $\a=\d^{(k+1)}$. 
Finally, we show that $p$
is a fibration. 
Note that by definition, the fiber of $p$
is isomorphic to $$\{T\in \Hom(V_{\varphi_{\d}, k}, W): \ker(T|_{U_2})=U_1 \}.$$

So it suffices to show that it is locally trivial.
To show this, we consider the open subset $\mathfrak{U}$ in $E''_{\a}$ as constructed 
in the proof of proposition \ref{prop: 7.7.6}.

Now we construct a trivialization 
for $p$
\[
 \varrho: p^{-1}(\alpha'^{-1}(\mathfrak{U}))\rightarrow \alpha'^{-1}(\mathfrak{U})\times 
 \{T\in \Hom(V_{\varphi_{\d}, k}, W): \ker(T|_{U_2})=U_1 \}
\]
with $\varrho(T)=[(T_1, W'), g^{-1}(T_0)]$,
where $g=s((W', W_1))$, $W_1=\ker(T_1|_{V_{\varphi_{\d}, k}})$.
Note that given 
$$[(T_1, W'), T_0]\in \alpha'^{-1}(\mathfrak{U})\times 
 \{T\in \Hom(V_{\varphi_{\d}, k}, W): \ker(T|_{U_2})=U_1 \},$$
take $W_1=\ker(T_1|_{V_{\varphi_{\d}, k}})$   
then $(W', W_1)\in \mathfrak{U}$, 
hence $g=s((W', W_1))$ exists.
Let $T_0'=gT_0$. Then $T=\tau_W^{-1}((T_1, T_0'))\in  p^{-1}(\alpha'^{-1}(\mathfrak{U}))$.

\end{proof}

\begin{definition}
Let $\ell+\dim(W)=\varphi_{e(\a)}(k)$.
We define $ \mathfrak{Y}_{\a}$ to be the set of pairs $(T, U)$ satisfying
 \begin{description}
  \item[(1)] $U\in Gr(\ell, V_{\varphi_{\a},k})$,
  $T\in \End(V_{\varphi_{\a}}/U) \text{ of degree } 1$;
  \item[(2)] $ T\in O_{\b}$ for some $\b\preceq_k \a$.
 \end{description}
 And we have a canonical 
 projection 
 \[
  \sigma: E''_{\a}\rightarrow  \mathfrak{Y}_{\a}
 \]
for $(T, U)\in  E''_{\a}$, we associate
\[
 \sigma((T, U))=(T', U)
\]
where 
 $T'\in \End(V_{\varphi_{\a}}/U)$
is the quotient of $T$. 
Also, we have a morphism 
\[
 \sigma': \mathfrak{Y}_{\a}\rightarrow Gr(\ell, \varphi_{e(\a)}(k), V_{\varphi_{\a}, k}),
\]
by 
\[
 \sigma'((T, U))=(U, \pi^{-1}(\ker(T|_{V_{\varphi_{\a},k}/U}))).
\]
where $\pi: V_{\varphi_{\a}, k}\rightarrow V_{\varphi_{\a},k}/U$
be the canonical projection.
\end{definition}

\begin{lemma}
 We have for $T\in \mathfrak{Y}_{\a}$,
\begin{description}
\item [(1)] $\gamma_k(T)\in Z^{k,\a}$;
\item[(2)] $T|_{V_{\varphi_{\a}, k-1}} \text{ is surjective }$;
\item[(3)] $\dim(\ker(T|_{V_{\varphi_{\a},k}/U}))=\dim(W)$; 
  
 \end{description}
 for $\gamma_k$, see definition \ref{def: 3.3.7}.
\end{lemma}

\begin{proof}
(1)To show $\gamma_k(T)\in Z^{k, \a}$, it suffices to show that 
 for $\b\preceq_k \a$, we have $\b^{(k)}\leq \a^{(k)}$. 
 Note that by proposition \ref{prop: 7.3.5},  
 there exists $\c\in S(\a)$ and $\Gamma\subseteq \c(k)$,
such that 
 \[
  \b=\c_{\Gamma}.
 \]
Now by lemma \ref{lem: 7.6.2}, 
we have $\c\in \tilde{S}(\a)_k$, which
implies that 
\[
 \b^{(k)}=\c^{(k)}\leq \a^{(k)}.
\]
(2)By definition 
for any $T\in \mathfrak{Y}_{\a}$, we have $Y\in O_{\b}$
for some $\b\preceq_k \a$. 
By the fact that   
$\a$ satisfies the assumption $(\A_k)$, 
we know that any $\c\in S(\a)$
satisfies $(\A_k)$, hence 
\[
 \b=\c_{\Gamma}
\]
cannot contain a segment which starts at 
$k$, therefore $T|_{V_{\varphi_{\a}, k-1}}$ is surjective. 
(3) 
Note that from the definition of $\mathfrak{Y}$, 
we know that for  $T\in \mathfrak{Y}_{\a}$, we have 
$T\in O_{\b}$ for some $\b\preceq_k \a$. 
Now it follows 
\[
 \ker(T|_{V_{\varphi_{\a}, k}/U })=\varphi_{e(\a)}(k)-\ell=\dim(W).
\]

\end{proof}

\begin{prop}
Let $\a$ be a 
multisegment satisfying 
the assumption $(\A_k)$.
Then the morphism $\sigma'$ is a fibration. 
Moreover, if we assume that $\a=\a_{\Id}$, 
cf. lemma \ref{lem: 7.6.1},
 then the morphism $\sigma$ is also a
fibration.  
\end{prop}

\begin{proof}
We first show that 
$\sigma'$ is locally trivial.
We observe that the group 
$GL(V_{\varphi_{\a}, k})$ acts both
on the source and target of $\sigma'$ 
in such a way that $\sigma'$ is $GL(V_{\varphi_{\a}, k})$-equivariant.
As in the proof of  
proposition \ref{prop: 7.7.6}, 
let $\mathfrak{U}\subseteq Gr(\ell, \varphi_{e(\a)}(k), V_{\varphi_{\a}, k})$
be a neighborhood of a given element $(U_1, U_2)$
 such that we have a section 
\[
 s: \mathfrak{U}\rightarrow GL(V_{\varphi_{\a}, k}),~ s((U_1, U_2))=Id.
\]
 Note that in this case we have a natural trivialization of 
 $\sigma'$ by
 \[
  \sigma': \beta'^{-1}(\mathfrak{U})\simeq \mathfrak{U}\times \beta'^{-1}((U_1, U_2))
 \]
by $\sigma'((T, U))=[(U,  \pi^{-1}(\ker(T|_{V_{\varphi_{\a}, k}}))), g^{-1}((T, U))]$
with $g=s((U,  \pi^{-1}(\ker(T|_{V_{\varphi_{\a}, k}}))))$. 
Finally, we show that $\sigma$ 
 is surjective and locally trivial. 

 We observe that 
$\alpha'=\sigma'\sigma$ and $\sigma$ preserves fibers. 
 Now we fix a neighborhood $\mathfrak{U}$ as above and 
 get a commutative diagram
 \begin{displaymath}
  \xymatrix{
  \alpha'^{-1}(\mathfrak{U})\ar[d]\ar[r]^{\hspace{-1cm}\gamma }& \mathfrak{U}\times \alpha'^{-1}((U_1, U_2))\ar[d]^{\delta}\\
  \sigma'^{-1}(\mathfrak{U})\ar[r]^{\hspace{-1cm}\gamma'}& \mathfrak{U}\times \sigma'^{-1}((U_1, U_2))
  }
 \end{displaymath}
where $\delta([x, T])=[x, \sigma(T)]$ for any $x\in \mathfrak{U}$
and $T\in \alpha'^{-1}((U_1, U_2))$. Therefore to show that 
$\sigma $ is locally trivial , it suffices to show that it is locally trivial  
when restricted to the fiber $\alpha'^{-1}((U_1, U_2))$.
Note that we have 
\[
 \alpha'^{-1}((U_1, U_2))\simeq\{T\in Y_{\a}: \ker(T|_{V_{\varphi_{\a}, k}})=U_2\}\simeq (X_{\a}^{k})_{U_2}
 \hookrightarrow Y_{\a^{(k)}}\times \Hom(V_{\varphi_{\a}, k-1}, U_2)
\]
and 
\begin{align*}
\sigma'^{-1}((U_1, U_2))&\simeq \{T: T\in \End(V_{\varphi_{\a}}/U_1) \text{ of degree }1, 
\ker(T|_{V_{\varphi_{\a}, k}/U_1})=U_{2}/U_1, \\
   & T\in O_{\b}, \text{ for some } \b\preceq_k \a\} \hookrightarrow Y_{\a^{(k)}}\times \Hom(V_{\varphi_{\a}, k-1}, U_2/U_1).
\end{align*}
Note that the canonical morphism  
\[
 \Hom(V_{\varphi_{\a}, k-1}, U_2)\rightarrow \Hom(V_{\varphi_{\a}, k-1}, U_2/U_1)
\]
is a fibration. Hence to show that 
\[
 \alpha'^{-1}((U_1, U_2))\rightarrow \sigma'^{-1}((U_1, U_2))
\]
is a fibration, it suffices to show that $\sigma|_{ \alpha'^{-1}((U_1, U_2))}$ is surjective with isomorphic fibers everywhere . 
Let $(T, U_1)\in \sigma'^{-1}((U_1, U_2))$ with 
$$\tau_{U_2/U_1}(T)=(T_0, q_0)\in Y_{\a^{(k)}}\times \Hom(V_{\varphi_{\a}, k-1}, U_2/U_1),$$ 
where $\tau_{U_2/U_1}$ is the morphism in definition \ref{def: 3.3.13}.  
We fix a splitting $U_2\simeq U_2/U_1\oplus U_1$. Now to give $(T', U_1)\in \sigma^{-1}((T, U_1))$
amounts to give $q_1\in \Hom(V_{\varphi_{\a}, k-1}, U_1)$ such that 
\[
 \tau_{U_2}(T')=(T_0, q_0\oplus q_1).
\]
Note that by lemma \ref{lem: 7.6.3},
the condition $\a=\a_{\Id}$ implies that 
$T'$ lies in $(X_{\a}^{k})_{U_2}$ if and only if 
$q_1$ satisfies 
\begin{align*}
 \dim(\ker(q_0\oplus q_1|_{\ker(T_0|_{V_{\varphi_{\a}, k-1}})}))=\varphi_{e(\a)}(k-1),
\end{align*}
which is an open condition. Therefore $\sigma$
is surjective. By definition of $\mathfrak{Y}_{\a}$, we know that 
\[
 \dim(\ker(q_0|_{\ker(T_0|_{V_{\varphi_{\a}, k-1}})}))=\varphi_{e(\a)}(k-1)+\ell,
\]
therefore if we denote $W_1=\ker(q_0|_{\ker(T_0|_{V_{\varphi_{\a}, k-1}})})$, 
then $q_1$ satisfies that 
\[
 \dim(\ker(q_1|_{\ker(T_0|_{V_{\varphi_{\a}, k-1}})})\cap W_1)=\varphi_{e(\a)}(k-1).
\]
Such a condition is independent of the 
pair $(T_0, q_0)$ since we always have 
$\dim(\ker(T_0|_{V_{\varphi_{\a}, k-1}}))=\varphi_{e(\a)}(k-1)+\varphi_{e(\a)}(k)$
and $\dim(W_1)=\varphi_{e(\a)}(k-1)+\ell$.
\end{proof}

We return to the morphism $p$ and $\sigma$.

\begin{lemma}\label{lem: 7.7.10}
Note that an element of $G_{\varphi_d}$ stabilizes $(\mathfrak{X}_{\d})_W$
if and only if it stabilizes $W$. Let $G_{\varphi_{\d}, W}$ be the 
stabilizer of $W$, then for $\c\leq \d$, and $T\in O_{\c}\cap (\mathfrak{X}_{\d})_W$, we have 
\[
 O_{\c}\cap  (\mathfrak{X}_{\d})_W=G_{\varphi_{\d}, W}T.
\]
\end{lemma}
\begin{proof}
 Recall that from proposition \ref{prop: 4.4.10}, we have 
 \begin{displaymath}
  \xymatrix{
 X_{\d}^{ k+1}\ar[d] & GL_{\varphi_{\d}(k+1)}\times_{P_{W}}\alpha^{-1}(W)\ar[l]_{\hspace{-1.5cm}\delta}\ar[dl]\\
  Gr(\ell_{k+1}, V_{\varphi_{\d}})&
  }
\end{displaymath}
 where $\ell_{k+1}=\varphi_{e(\d)}(k+1)$. Note that we have 
 \[
  G_{\varphi_{\d}, W}=\cdots\times G_{\varphi_{\d},k}\times P_{W}\times G_{\varphi_{\d},k+2}\times \cdots ,
 \]
where $G_{\varphi_{\d}, i}=GL(V_{\varphi_{\d}, i})$. 
From this diagram we observe that the orbits there is a one to one 
correspondance between the $G_{\varphi_{\d}}$ orbits 
on $X_{\d}^{k+1}$ and $G_{\varphi_{\d}, W}$ orbits on $\alpha^{-1}(W)$.  
Finally, since $\mathfrak{X}_{\d}$ is an open subvariety 
consisting of $G_{\varphi_{\d}}$ orbits, we are done.
\end{proof}

\begin{definition}
The canonical projection 
\[
 \pi: V_{\varphi_{\d}}\rightarrow V_{\varphi_{\d}}/W
\]
induces a projection 
\[
 \pi_{*}: G_{\varphi_{\d}, W}\rightarrow G_{\varphi_{\a}}, 
\]
where we identify $V_{\varphi_{\d}}/W$ with $V_{\varphi_{\a}}$.
\end{definition}

\begin{prop}
The morphism $p$ is equivariant under the action of 
$G_{\varphi_{\d}, W}$ and $G_{\varphi_{\a}}$ via $\pi_{*}$, i.e, 
\[
 p(gx)=\pi_{*}(g)p(x).
\]
Moreover, it induces a one to one correspondance between orbits.
\end{prop}

\begin{proof}
Note that for 
$T\in (\mathfrak{X}_{\d})_W$, such that 
$\tau_W(T)=(T_1, T_0)\in (Z^{k+1, \d})_W\times \Hom(V_{\varphi_{\d}, k}, W)$, 
let
\[
U_1=\ker(T_1|_{V_{\varphi_{\d}, k}}),~ U_0=\ker(T_0|_{U_1})   
\]
we have 
\[
 p(T)=(T_1,U_0).
\]
Now it follows from the definition that 
we have 
\[
 p(gT)=\pi_{*}(g)p(T).
\]
Hence $p$ sends orbits to orbits. 
It remains to show that the pre-image of an orbit
is an orbit instead of unions of orbits.

We proved in proposition \ref{prop: 7.7.7} that 
\[
 p^{-1}p(T)= \{(T_1, q): q\in \Hom(V_{\varphi_{\d}, k}, W), \ker(q|_{U_1})=U_0\},
\]
 note that here we identify elements of $(\mathfrak{X}_{\d})_{W}$ with 
 its image under $\tau_W$.
 Let $(T_1, q)\in p^{-1}p(T)$. Then 
 we want to find $g\in G_{\varphi_{\d}, W}$ such that $g(T_1, T_0)=(T_1, q)$. 
 Note that 
 by fixing a splitting $V_{\varphi_{\d},k+1}=W\oplus V_{\varphi_{\d},k+1}/W$,
 we can choose $g\in G_{\varphi_{\d}}$ such that 
 $g_{i}=Id\in GL(V_{\varphi_{\d}, i})$ for all $i\neq k+1$, and 
 \begin{displaymath}
  g_{k+1}=\begin{pmatrix}
     g_1 & g_{12}\\
     0& Id_{V_{\varphi_{\d},k+1}/W}
    \end{pmatrix}\in P_W,
 \end{displaymath}
where $g_1\in GL(W)$,  
and $g_{12}\in \Hom(V_{\varphi_{\d}, k+1}/W, W)$. 
By hypothesis, we know that the restrictions of $q$ 
and $T_0$ to 
$U_1$ are surjective with kernel $U_0$, so we 
can choose $g_1\in GL(W)$, such that 
\[
 g_1T_0(v)=q(v), \text{ for all } v\in U_1.
\]
Finally, for $v_1\in V_{\varphi_{\d},k+1}/W$, by our assumption 
at the beginning of this section on $\a$, we know that 
$T_1|V_{\varphi_{\d}, k}$ is surjective, hence there exists $v\in V_{\varphi_{\d}, k}$
such that $T_1(v)=v_1$.
Then we define  
 \[
  g_{12}(v_1)=q(v)-g_1T_0(v).
 \]
We check that this is well defined, i.e, for another $v'\in V_{\varphi_{\d}, k}$ such 
that $T_1(v')=v_1$, we have 
 \[
  q(v)-g_1T_0(v)=q(v')-g_1T_0(v'),
 \]
this is the same as to say that 
 \[
  q(v-v')=g_1T_0(v-v').
 \]
We observe that $T_1(v-v')=0$, hence $v-v'\in U_1$, now $ q(v-v')=g_1T_0(v-v')$
follows from our definition of $g_1$. Under such a choice, we have 
 \[
  g((T_1, T_0))=(T_0, q).
 \]
Hence we are done.
\end{proof}

\begin{prop}\label{prop: 7.6.17}
The morphism  $\sigma$ is equivariant under 
the action of $G_{\varphi_{\a}}$. 
Assume that $\a$ is a multisegment 
which satisfies the assumption $(\A_k)$. 
Let $\varphi\in \mathcal{S}$ such that 
\[
 \varphi+\l\chi_{[k]}=\varphi_{\a},
\]
where $\chi$ is the characteristic function.
Then there exists a one to one correspondance 
between the orbits of $\mathfrak{Y}_{\a}$ and 
the set 
\[
 S:=\{\b\in S(\varphi): \b\preceq_k \a\}.
\]
Moreover, for each orbit $\mathfrak{Y}(\b)$
indexed by $\b$, $\sigma^{-1}(\mathfrak{Y}(\b))$
is irreducible hence 
contains a unique orbit in $E''_{\a}$ as (Zariski) open subset.

\end{prop}

\begin{proof}
 The fact that 
 $\sigma$ is equivariant under the 
 action of $G_{\varphi_{\a}}$ follows 
 directly from the definition.
 To show that the orbits of 
 $\mathfrak{Y}$ under $G_{\varphi_{\a}}$
 is indexed by $S$, consider the morphism 
 \[
  p': \mathfrak{Y}_{\a}\rightarrow Gr(\ell, V_{\varphi_{\a}, k}), ~(T, U)\mapsto U
 \]

 As in the proposition \ref{prop: 4.4.10}, we have 
 the following diagram
 \begin{displaymath}
  \xymatrix{
\mathfrak{Y}_{\a} \ar[d]^{p'}& GL_{\varphi_{\a}(k)}\times_{P_{U}} p'^{-1}(U)\ar[l]_{\hspace{-1.5cm}\delta}\ar[dl]\\
  Gr(\ell, V_{\varphi_{\a}, k})&
  }
\end{displaymath}
 which shows that $p'$ is a $GL_{\varphi_{\a}, k}$ bundle. 
 Moreover, the same proof as in lemma \ref{lem: 7.7.10} shows 
 that the orbits of $\mathfrak{Y}$ are in 
 in one to one correspondance with that of 
 the fibers 
 \begin{align*}
  p'^{-1}(U)\simeq \{T\in \End(V_{\varphi_{\a}}/U): &T \text{ is of degree }1, 
  T\in O_{\b} \text{ for some } \b\preceq_k \a
  \},
 \end{align*}
under the action of stabilizer $G_{\varphi_{\a}, U}$ of $U$. 
 Let $\varphi\in \mathcal{S}$ be the 
 such that $\varphi+\ell \chi_{[k]}=\varphi_{\a}$. 
 Then by identifying $V_{\varphi}$ with 
 $V_{\varphi_{\a}}/U$, we can view 
 $p'^{-1}(U)$ as an open subvariety of $E_{\varphi}$. 
 Note that we are identifying orbits with orbits by the canonical projection 
 \[
  G_{\varphi_{\a}, U}\rightarrow G_{\varphi}. 
 \]
Now it follows that 
the fibers are parametrized by the set $S$. 
Finally, let $\b\in S$. We have to show that 
$\sigma^{-1}(\mathfrak{Y}(\b))$ is irreducible, which
is a consequence of the following lemma.

\end{proof}

\begin{lemma}\label{lem: 7.6.18}
Let $\a, \b$ be the  multisegments as above. Then 
there exists a 
bijection between the set  
\[
 Q(\a, \b)=\{\c\in S(\a): \b=\c_{\Gamma} \text{ for some } \Gamma\subseteq \c(k)\},
\]
and the orbits in $\sigma^{-1}(\mathfrak{Y}(\b))$
which respects the poset structure,  given by 
\[
 \c\mapsto E''_{\a}(\c^{\sharp}),
\]
where for $\b=\c_{\Gamma}$,
\[
  \c^{\sharp}=(\c\setminus \c(k))\cup \Gamma\cup \{\Delta^+: \Delta\in \c(k)\setminus \Gamma\},
\]
and $E''_{\a}(\c^{\sharp})$ is the orbit indexed by $\c^{\sharp}$.
Moreover, the set $Q(\a, \b)$
contains a unique minimal element.
\end{lemma}
\remk 
 
We remark that $S(\varphi)$ contains a 
unique maximal element.

\begin{proof}

Recall that we constructed in proposition 
\ref{prop: 7.7.7} a morphism $p$, consider the composition
\[
 (\mathfrak{X}_{\d})_W \xrightarrow{p} E_{\a}''\xrightarrow{\sigma} \mathfrak{Y}_{\a},
\]
which sends $(O_{\c})_W$ to $\mathfrak{Y}(\b)$, where 
$\b=\c^{(k, k+1)}$ for $\c \in S(\d)$. Hence we have 
\[
 \b=(\c^{(k+1)})_{\Gamma}
\]
for $\Gamma=\{\Delta\in \c: e(\Delta)=k\}$. Note that 
$\c\in S(\d)$ implies that  
 $\c^{(k+1)}\leq \a=\d^{(k+1)}$.
 Conversely, for $\c\in Q(\a, \b)$, 
 such that 
 \[
  \b=\c_{\Gamma},
 \]
there is a unique element 
\[
 \c'=\c^{\sharp}
\]
in $S(\d)$ such that $O_{\c'}\subseteq \mathfrak{X}_{\d}$ and $\c=\c'^{(k+1)}$.
Therefore we conclude that 
there is a bijection between 
the $G_{\varphi_{\a}}$-orbits in $\sigma^{-1}(\mathfrak{Y}(\b))$
and $Q(\a, \b)$.  

Finally, 
for 
\[
 \varphi_{\a}=\varphi_{\b}+\ell\chi_{[k]},
\]
we show
by induction on $\ell$
that the set 
$Q(\a, \b)$
contains a unique minimal element.

For case 
$\ell=1$, let 
\[
 \b(k): =\{\Delta\in \b: e(\Delta)=k\}=\{\Delta_1\preceq \cdots \preceq \Delta_h\}.
\]
and $\c_i=(\b\setminus \Delta_i)\cup \Delta_i^+$.
Then 
\[
 Q(\a, \b)\subseteq \{\c_i: i=1, \cdots, h\},
\]
and $\c_h$ is minimal in the latter, which 
implies that $\c_h\in Q(\a, \b)$ and is minimal. 
In general,
let 
\[
 \varphi=\varphi_{\b}+\chi_{[k]}.
\]
Note that there exists 
$\c'\in S(\varphi) $
satisfying the assumption $(\A_k)$
and $\Gamma'\subseteq \c'(k)$ 
such that 
\[
 \b=\c'_{\Gamma'}. 
\]
In fact, by assumption, we know that 
\[
 \b=\c_{\Gamma}
\]
for some $\c\in S(\a)$ and $\Gamma\subseteq \c(k)$. 
Let 
\[
 \Gamma\supseteq \Gamma_1,
\]
such that $\ell=\sharp \Gamma=\sharp \Gamma'+1$ and
\[
 \c'=\c_{\Gamma_1},   
\]
then we have 
\[
 \b=\c'_{\Gamma\setminus \Gamma_1}.
\]
Now we apply our induction to 
the case 
 \[
  Q_1: =\{\c\in S(\varphi): 
  \c \text{ satisfies the assumption } (\A_{k}),
  \b=\c_{\Gamma} \text{ for some } \Gamma\subseteq \c(k)\},
 \]
from which we know that there exists a unique minimal element $\c_1$ in $ Q_1$. 
Now by assumption, we know that 
\[
 \b_1\leq \c'\preceq_k \a,
\]
and by induction, we know that the set 
$Q(\a, \b_1)$ contains a unique element $\b_2$.
We claim that $\b_2$ is minimal in $Q(\a, \b)$. 
In fact, let $\e\in Q(\a, \b)$, then 
\[
 \b=\e_{\Gamma'}
\]
for some $\Gamma'\subseteq \e(k)$. 
Again let 
\[
 \Gamma_1'\subseteq \Gamma',~ \e'=\e_{\Gamma_1'}
\]
such that $\ell=\sharp \Gamma'=\sharp \Gamma_1'+1$.
Now we obtain 
\[
 \e'\in Q_1, ~ \b=\e'_{\Gamma'\setminus \Gamma_1'}.
\]
By minimality of $\c_1$, we know that 
\[
 \c_1\leq \e'.
\]
Note that this implies $\c_1\preceq \e'$, and 
by transitivity of poset relation, we get 
$\c_1\preceq_k \e$. Now we apply 
proposition \ref{prop: 7.3.5} to get 
\[
 \c_1=\f_{\Gamma''}, 
\]
for some $\f\in S(\e)$ and $\Gamma''\subseteq \f(k)$.
Again we deduce from induction that 
\[
 \f\geq \c_2.
\]
Hence $\c_2\leq \e$. 
\end{proof}

Now we return to the calculation
of product of perverse sheaves, cf. 
corollary \ref{cor: 7.4.15}.

\begin{cor}\label{cor: 7.4.19}
Let $\a$ be a multisegment 
satisfying the assumption $(\A_k)$
and $\b\preceq_k \a$ such that  
\[
 \varphi_{\a}=\varphi_{\b}+\ell\chi_{[k]}.
\]
Let $\c$ the minimal element 
in $Q(\a, \b)$ and $E''_{\a}(\c)$ be the 
$G_{\varphi_{\a}}$ orbit indexed by $\c$
in $E_{\a}''$.  
Then we have 
\[
 IC(\line{O}_\b)\star IC(\line{O}_{\l [k]})=\beta''_{*}(IC(\line{E_{\a}''(\c^{\sharp})})). 
\]

\end{cor}
 
\begin{proof}

First of all, by
definition 
\[
 E''=\{(T, U): T\in E_{\varphi_{\a}}, ~T(U)=0, ~ \dim(U)=\ell \},
\]
therefore we have 
\[
 E''_{\a}\subseteq E''.
\]
Furthermore, the variety $ E''_{\a}$ is open in $E''$. 
In fact, consider the canonical
morphism 
\[
 \beta'': E''\rightarrow E_{\varphi_{\a}},
\]
then $E''_{\a}=\beta''^{-1}(Y_{\a})$. 
Since $Y_{\a}$ is open in $E_{\varphi_{\a}}$,
we know that $ E''_{\a}$ is open in $E''$.
Now we have two morphisms 
\begin{align*}
\sigma\beta': & \beta'^{-1}(E''_{\a})\rightarrow \mathfrak{Y}_{\a},\\
\beta: & E'\rightarrow E_{\varphi_{\b}}\times E_{\varphi_{\ell[k]}}\simeq E_{\varphi_{\b}}.
\end{align*}

We claim that $\beta^{-1}(O_{\b})\cap  \beta'^{-1}(E''_{\a})=\beta'^{-1}\sigma^{-1}(\mathfrak{Y}(\b))$,
where $\mathfrak{Y}(\b)$ is the orbit in $\mathfrak{Y}(\b)$ under the 
action of $G_{\varphi_{\a}}$. 

By definition of $\beta$, we know that 
\begin{align*}
 \beta^{-1}(O_{\b})\cap  \beta'^{-1}(E''_{\a})=\{(T, W, \mu, \mu'): &\mu: W\simeq V_{\varphi_{\ell[k]}}, ~\mu': V_{\varphi_{\a}}/W\simeq V_{\varphi_{\b}},\\
 &T\in O_{\f} \text{ for some } \f\in S(\a), \b\preceq_k \f
 \}.
\end{align*}
Now by definition of $\sigma$ and $\beta'$, 
we know that 
$\beta^{-1}(O_{\b})\cap  \beta'^{-1}(E''_{\a})=\beta'^{-1}\sigma^{-1}(\mathfrak{Y}(\b))$.
Now by proposition \ref{prop: 7.6.17},
$\sigma^{-1}(\mathfrak{Y}_{\b})$ 
contains $E_{\a}''(\c^{\sharp})$ as the unique open 
sub-orbit, where $\c$ is the minimal
element in $Q(\a, \b)$.
Therefore we conclude that 
\[
  \beta'^{*}(IC(\line{E''_{\a}(\c^{\sharp})}))=\beta^*(IC(\line{O_{\b}})\otimes IC(E_{\varphi_{\ell[k]}})).
\]
Now by definition
\[
  IC(\line{O}_\b)\star IC(\line{O}_{\l [k]})=\beta''_{*}(IC(\line{E_{\a}''(\c^{\sharp})})).
\]

\end{proof}

\section{Multisegments of Grassmanian Type}

In order to precisely describe the previous corollary
concerning Lusztig's product in the Grassmanian case 
in the next section, 
we  generalize the construction
in section 3.3 to get more general results concerning the 
the set $S(\a)$ for general multisegment $\a$.

Let $V$ a $\C$ vector space of dimension $r+\l$  and $Gr_r(V)$ be the variety of 
$r$-dimensional subspaces of $V$. 

\begin{definition}
By a partition of $ \l$, we mean a sequence $\lambda=(\ell_1, \cdots, \ell_r)$ for some $r$, where $\ell_i\in \N$
, $0\leq \ell_1\leq \cdots \ell_r\leq \ell$.
And for $\mu=(\mu_1,\cdots, \mu_s)$ be another partition, we say $\mu\leq \lambda$
if and only if $\mu_i\leq \lambda_i$ for all $i=1, \cdots, $. Let $\mathcal{P}(\ell, r)$
be the set of partitions of $\ell$ into $r$ parts.

\end{definition}

\begin{definition}
We fix a complete flag 
\[
 0=V^0\subset V^1\subset \cdots \subset V^{r+\ell}=V.
\]
This flag provides us a stratification of the variety $Gr_r(V)$ by Schubert varieties, labeling  by partitions 
, denoted by $\line{X}_{\lambda}$, 
\[
\line{X}_{\lambda}=\{U\in Gr_r(V): \dim(U\cap V^{\ell_i+i})\geq i, \text{ for all } i=1, \cdots, r \}.
\]
\end{definition}

\begin{lemma}(cf. \cite{Z5})
We have 
\[
 \mu\leq \lambda \Longleftrightarrow \line{X}_{\mu}\subseteq \line{X}_{\lambda}.
\]
And the Schubert cell 
\[
 X_{\lambda}=\line{X}_{\lambda}-\sum_{\mu<\lambda}\line{X}_{\mu}
\]
is open in $\line{X}_{\lambda}$.
\end{lemma}

\begin{definition}
Let $\Omega^{r, \ell}$ be the set 
\begin{align*}
 \Omega^{r, \l}=\{(a_1, \cdots, a_m; b_0, \cdots, b_{m-1}): &\sum_i a_i=r, \sum_j, b_j=\l,\\ 
 &\text{ for } 0<i<m, a_i>0, b_i>0\}.
\end{align*}
\end{definition}

\begin{lemma}(cf. \cite{Z5})\label{lem: 7.2.4}
There exists a bijection
\[
 \Omega^{r, \l}\rightarrow \mathcal{P}(\ell, r),
\]
which sends $(a_1, \cdots, a_m; b_0, \cdots, b_{m-1})$ to a partition of $\ell$
given by $~b_0, ~b_0+b_1, \cdots,~, b_0+\cdots+b_{m-1}$, and that the elements $b_0+\cdots+b_{i-1}$ figures in $\lambda$
with multiplicity $a_i$. 

\end{lemma}

\begin{notation}
 From now on, we will also write 
\[
 \lambda=(a_1, \cdots, a_m; b_0, \cdots, b_{m-1}),
\]
with notations as in the previous lemma.
\end{notation}

%

We introduce the formula in \cite{Z5} to calculate the 
Kazhdan Lusztig polynomials for Grassmannians. 

\begin{definition}
Let $\lambda=(a_1, \cdots, a_{m}; b_0, \cdots, b_{m-1})$ be a partition.
Following \cite{Z5}, we represent a partition as a broken line in the 
plane $(x, y)$, i.e, the graph of the piecewise-linear function $y=\lambda(x)$
which equals $|x|$ for large $|x|$, has everywhere slope $\pm 1$, and whose 
ascending and decreasing segments are precisely $ b_0, \cdots, b_{m-1}$ 
and $a_1, \cdots, a_m$, respectively. 
Moreover, we call the local maximum and minimum of the graph $y=\lambda(x)$ 
the peaks and depressions of $\lambda$.

\end{definition}

\begin{figure}[!ht]
\centering
\includegraphics{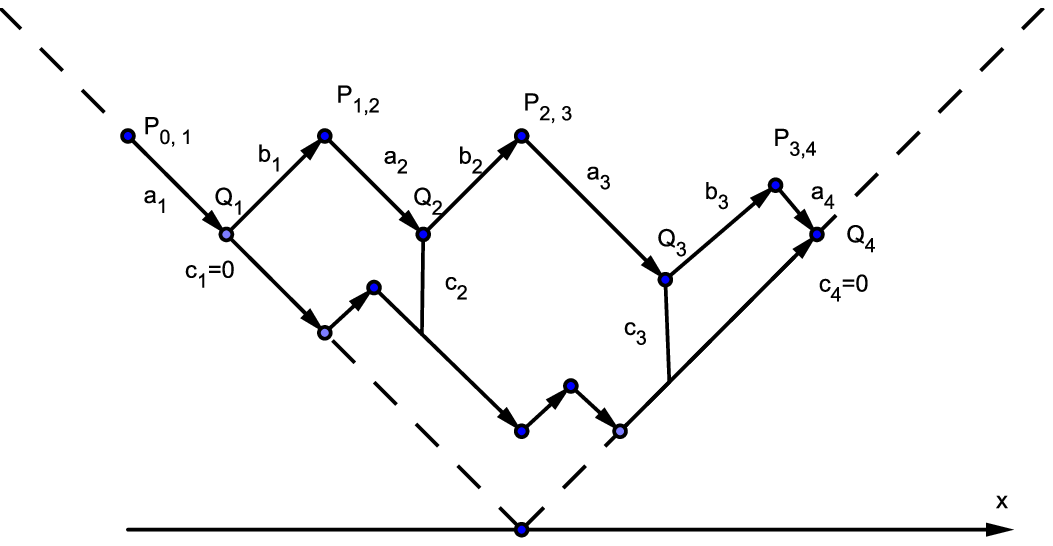}
\caption{\label{fig-multisegment15} }
\end{figure}

%
%
%

\begin{lemma}(cf. \cite{Z5})
For $\lambda, \mu\in \Omega^{r, \l}$, then  
\[
 \lambda\geq \mu \Longleftrightarrow \lambda(x)\geq \mu(x), \text{ for all }x.
\]
\end{lemma}

From now on until the end of this section, we let 
\[
 J=\{\sigma_i: i=1, \cdots, r-1\}\cup \{\sigma_i: i=r+1\cdots, r+\ell-1\},
\]
and
\[
 \a: =\a_{\Id}^{J, \emptyset}=\{\Delta_1, \cdots, \Delta_r, \cdots, \Delta_{r+\l}\}
\]
be a multisegment of parabolic type $(J, \emptyset)$, where 
\[
e(\Delta_i)=k-1, \text{ for } i=1, \cdots, r,
\]
and 
\[
 e(\Delta_i)=k, \text{ for } i=r+1, \cdots, r+\l. 
\]

\begin{definition} 
Then to each partition $\lambda\in \Omega^{r_1, \l_1}$ such that 
$r_1\geq r$ and $r_1+\l_1=r+\l$, we associate 
\begin{align*}
 \a_{\lambda}&=\sum_{i=1}^{b_0}[b(\Delta_{i}), k]+\sum_{i=b_0+1}^{b_0+a_1}[b(\Delta_{i}), k-1]
 +\cdots\\
 &+\sum_{i=b_0+a_1\cdots+b_{j-1}+1}^{b_0+a_1\cdots+b_{j-1}+a_{j}}[b(\Delta_{i}), k-1]+
 \sum_{i=b_0+\cdots +b_{j-1}+a_{j}+1}^{b_0+a_1\cdots+b_j}[b(\Delta_{i}), k]+\cdots. 
\end{align*} 
 
\end{definition}

\begin{definition}
Let $r, n\in \N$ such that $r\leq n$. Let 
\[
 R_r(n)=\{(x_1, \cdots, x_r): 1\leq x_1<\cdots< x_r\leq n \}.
\]

\begin{description}
 \item[(1)] 
 Let $x=(x_1, \cdots, x_{r_1})\in R_{r_1}(n)$ and $x'=(x_1', \cdots, x_{r_2}')\in R_{r_2}(n)$
such that $r_1\geq r_2$. We say
$x\supseteq x'$ if 
$\{x_1, \cdots, x_{r_1}\}\supseteq \{x_1', \cdots, x_{r_2}'\}$.
\item[(2)]Let $x=(x_1, \cdots, x_{r})\in R_{r}(n)$ and $x'=(x_1', \cdots, x_{r}')\in R_{r}(n)$.
We say $x\geq x'$ if $x_i\geq x_i'$ for all $i=1, \cdots, r$.
\item[(3)] We define $x\succeq y$, if $x\geq y'\supseteq y$ for some $y'$.
\end{description}
\end{definition}

\remk The set $R_r(n)$ is a poset with respect to the relation $\geq$.
And the set $\cup_{r\leq n}R_r(n)$ is a poset with respect to the relation $\supseteq$.

\begin{prop}
For $J=\{\sigma_i: i=1, \cdots, r-1\}\cup \{\sigma_i: i=r+1\cdots, r+\l-1\}$,
we have an isomorphism of posets
$$\varsigma_1: S_{r+\l}^{J, \emptyset}\rightarrow R_r(r+\l),$$
by associating the element
$w$ with $x_w:=(w^{-1}(1), \cdots, w^{-1}(r))$.  
\end{prop}

\begin{proof}
Note that 
by definition
\[
S_{r+\l}^{J, \emptyset}=\{w\in S_{r+\l}: w^{-1}(1)<\cdots< w^{-1}(r) \text{ and } w^{-1}(r+1)<\cdots <w^{-1}(r+\l)\}.
\]
Therefore,
$\varsigma$ is a bijection. This preserves the partial order, for a proof, 
see \cite{BF} proposition 2.4.8.
\end{proof}

\begin{definition}
For $\lambda\in \Omega^{r, \l}$ and $\lambda'\in \Omega^{r_1, \l_1}$ 
such that $r+\l=r_1+\l_1$.
We define 
$\lambda\supseteq \lambda'$ if and only if $x_{\lambda}\supseteq x_{\lambda'}$,
and $\lambda\succeq \lambda'$ if and only if $x_{\lambda}\succeq x_{\lambda'}$.

\end{definition}

\begin{definition}

Let $\lambda=(a_1, \cdots, a_m; b_0, \cdots, b_{m-1})$, 
consider the set 
$$\{b(\Delta): \Delta\in \a_{\lambda}, e(\Delta)=k-1\}=\{x_1<\cdots<x_r\},$$
here we have $r$ segments ending in $k-1$ since $\sum_ia_i=r$,
we associate $\lambda$ with the element 
\[
 x_{\lambda}: =(x_1,\cdots, x_r).
\]
This allows us to get a morphism $\varsigma_2: \Omega^{r, \l}\rightarrow R_r(r+\l)$
sending $\lambda$ to $x_{\lambda}$.
\end{definition}

\begin{lemma}
 The map $\varsigma_2$ is an isomorphism of posets. 
\end{lemma}
 
\begin{proof}
To see that $\varsigma$ is a bijection, we only need to construct an inverse.
Given $x=(x_1, \cdots, x_r)\in R_r(r+\l)$, we have 
$y=(y_1, \cdots, y_{\l})\in R_{\l}(r+\l)$  such that $\{1, \cdots, r+\l\}=\{x_1, \cdots, x_r, y_1, \cdots, y_{\l}\}$.
We can associate a multisegment to $x$
\[
 \a_x=\sum_{j=1}^r[b(\Delta_{x_{j}}), k-1]+\sum_{j=1}^{\l}[b(\Delta_{y_j}), k].
\]
Note that this allows us to construct a 
partition $\lambda(x)\in \Omega^{r, \l}$ by counting the segments ending in 
$k$ and $k+1$ alternatively. 

A simple calculation shows that 
if we write $\lambda=(\l_1, \cdots, \l_r)$ with $0\leq \l_1\leq \cdots\leq \l_r$,
then 
\[
 \varsigma_2(\lambda)=(\l_1+1, \cdots, \l_r+r), 
\]
as described in \cite{Br}. This shows that 
\[
 \mu\geq \lambda\Leftrightarrow \varsigma_2(\mu)\geq \varsigma_2(\lambda).
\]
\end{proof}

\begin{prop}\label{prop: 7.3.3}
For $\lambda\in \Omega^{r, \l}$,
we have $\a_{\lambda}\in S(\a)$, moreover, all the elements in $S(\a)$ are of this form. 
Moreover, we have $S(\a_{\lambda})=\{a_{\mu}: \mu\geq \lambda\}$.
\end{prop}
 
\begin{proof}
Let $w\in S^{J, \emptyset}$, by definition, we have 
\[
 w^{-1}(1)<\cdots<w^{-1}(r), ~w^{-1}(r+1)<\cdots <w^{-1}(r+\l).
\]
By definition, we have 
\begin{align*}
\Phi_{J, \emptyset}(w)=&\sum_{j}[b(\Delta_j), e(\Delta_{w(j)})]\\
=&\sum_{j}[b(\Delta_{w^{-1}(j)}), e(\Delta_{j})]\\
=&\sum_{j=1}^{r}[b(\Delta_{w^{-1}(j)}), k-1]+\sum_{j=r+1}^{r+\l}[b(\Delta_{w^{-1}(j)}), k]\\
=\a_{\varsigma_2^{-1}(x_w)}
\end{align*}

Now that $\varsigma_2^{-1}\circ\varsigma_1$ preserves
the partial order, we have 
\[
 S(\a_{\lambda})=\{\a_{\mu}: \mu\geq \lambda\}
\]
by proposition \ref{prop: 6.2.10}.

\end{proof}

\begin{example}

For example, for $r=1, \ell=3$, with 
$J=\{\sigma_2, \sigma_3\}$ and 
\[
 \a=\a_{\Id}^{J,\emptyset}=[1, 4]+[2, 5]+[3, 5]+[4, 5].
\]

Let $\lambda=(a_1, a_2; b_0, b_1)=(1, 0; 2, 1)$, then $\a_{\lambda}=[1, 5]+[2, 5]+[3, 4]+[4, 5]$.
This corresponds to the element $\varsigma_1^{-1}\circ \varsigma_2(\lambda)=\sigma_1\sigma_2$
in $S_4^{J, \emptyset}$.

\begin{figure}[!ht]
\centering
\includegraphics{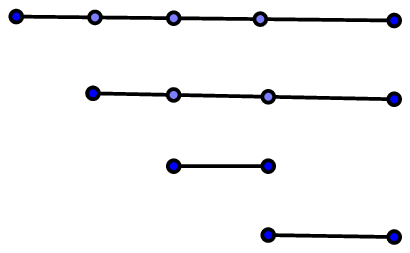}
\caption{\label{fig-multisegment14} }
\end{figure}

\end{example}

\begin{prop}
Let $\lambda, \mu\in \Omega^{r, \l}$ such that $\lambda<\mu$.
We have 
\[
 P_{\a_{\lambda}, \a_{\mu}}(q)=P_{\lambda, \mu}(q).
\] 
\end{prop}

\begin{proof}
 We can also prove this proposition in the following way. 
Let $w, v\in S_{r+\l}^{J, \emptyset}$, 
such that 
\[
 \lambda=\varsigma_2^{-1}\varsigma_1(w), ~\mu=\varsigma_2^{-1}\varsigma_1(v).
\]
Let $P_{J}$ be the parabolic
subgroup of $GL_n$, then by fixing an element in $V_0\in Gr_r(\C^{r+\l})$,  
we can identify $P_J\backslash GL_n$ with $Gr_r(\C^{r+\l})$. 
Moreover, the $B$-orbits $P_J\backslash wB$ corresponds to the varieties
$X_{\lambda}$, see \cite{Br} for a precise description.
Hence we have 
\[
 P_{\lambda, \mu}(q)=P^{J, \emptyset}_{w, v}(q)=P_{a_{\lambda}, \a_{\mu}}(q).
\]
\end{proof}

\remk One can surely prove this result using 
the open immersion we constructed in section 3.3. 

\begin{definition}
Let $\lambda\in \Omega^{r, \l}$. 
\begin{description}
 \item[(1)]
We define  
\[
 \Gamma(\lambda)=\{\mu\in \Omega^{r_1, \l_1}: r_1+\l_1=r+\l, r_1\geq r,~
\mu\succeq \lambda\}.
\]
and 
\[
 \Gamma^{\mu}(\lambda)=\{\mu': \mu\geq \mu', \mu'\succeq \lambda\},
\]
\[
  \Gamma_1^{\mu}(\lambda)=\{\mu': \mu\geq \mu', \mu'\supseteq \lambda\}.
\]

\item[(2)] For $\mu\in \Gamma(\lambda)$, we define 
\[
 S^{\mu}(\lambda)=\{\lambda'\in \Omega^{r, \l}: \lambda'\geq \lambda, \mu\succeq\lambda'\},
\]
and let 
\[
 S_1^{\mu}(\lambda)=\{\lambda'\in \Omega^{r, \l}: \lambda'\geq \lambda, \mu\supseteq \lambda'\}.
\]
\end{description}

\end{definition}

\begin{prop}\label{prop: 7.3.9}
Let $\lambda\in \Omega^{r, \l}$ and $\mu\in \Omega^{r_1, \l_1}$ with $r_1\geq r$
and $r_1+\l_1=r+\l$.  
Then $\pi(\a_{\mu})$ appears as a summand of $\D^k(\pi(\a_{\lambda}))$
if and only if $\mu\in \Gamma(\lambda)$.  
\end{prop}

\begin{proof}
Let $x_{\lambda}=(x_1^{\lambda}, \cdots, x_r^{\lambda})=\varsigma_2(\lambda)$
and $y_{\lambda}=(y_1^{\lambda}, \cdots, y_{\l}^{\lambda})\in R_r(r+\l)$ such that 
\[
 \{1, \cdots, r+\l\}=\{x_1^{\lambda}, \cdots, x_r^{\lambda}, y_1^{\lambda}, \cdots, y_{\l}^{\lambda}\}.
\]
As described in proposition \ref{prop: 7.3.3}, we have 
\[
 \a_{\lambda}=\sum_{j=1}^r[b(\Delta_{x_{j}^{\lambda}}), k-1]+\sum_{j=1}^{\l}[b(\Delta_{y_j^{\lambda}}), k].
\]
therefore
\[
 \D^k(\pi(\a_{\lambda}))=\pi(\a_{\lambda})+\sum_{y\supseteq x_{\lambda}}\pi(\a_{\varsigma_2^{-1}(y)}).
\]
Now by lemma \ref{lem: 3.0.8},
we know that $\pi(\mu)$ is a summand of 
$\D^k(\pi(\a_{\lambda}))$ if and only if 
$\mu\geq \varsigma_2^{-1}(y)$ for some $y\supseteq x_{\lambda}$, i.e,
$\mu\succeq \lambda$.
\end{proof}

\begin{cor}
We have $\mu\succeq \lambda$ if and 
only if $\a_{\mu}\preceq_k \a_{\lambda}$. 
 
\end{cor}

\begin{proof}
By corollary \ref{cor: 7.4.3}, we know that 
 $\a_{\mu}\preceq_k \a_{\lambda}$ if and only if $\D^k(\pi(\a_{\lambda}))-\pi(\a_{\mu})\geq 0$
 in $\mathcal{R}$, which is equivalent to say that $\mu\preceq \lambda$ by 
 the previous proposition. 
\end{proof}

\begin{prop}
Let $\lambda\in \Omega^{r, \ell}$ 
and $\mu\in \Omega^{r_1, \ell_1}$. Then we have  
$\a_{\mu}=(\a_{\lambda})_{\Gamma}$
 for some $\Gamma\subseteq \a_{\lambda}(k)$. 
if and only if we have $\mu\supseteq \lambda$. 
\end{prop}

\begin{proof}
Let $x_{\lambda}=(x_1^{\lambda}, \cdots, x_r^{\lambda})=\varsigma_2(\lambda)$
and $y_{\lambda}=(y_1^{\lambda}, \cdots, y_{\l}^{\lambda})\in R_r(r+\l)$ such that 
\[
 \{1, \cdots, r+\l\}=\{x_1^{\lambda}, \cdots, x_r^{\lambda}, y_1^{\lambda}, \cdots, y_{\l}^{\lambda}\}.
\]
As described in proposition \ref{prop: 7.3.3}, we have 
\[
 \a_{\lambda}=\sum_{j=1}^r[b(\Delta_{x_{j}^{\lambda}}), k-1]+\sum_{j=1}^{\l}[b(\Delta_{y_j^{\lambda}}), k].
\] 
And we have 
\[
 \a_{\lambda}(k)=\sum_{j=1}^{\l}[b(\Delta_{y_j^{\lambda}}), k].
\]
Let 
$\Gamma=\sum_{m=1}^{t}[b(\Delta_{y_{j_m}^{\lambda}}), k]$. 
If $\a_{\mu}=(\a_{\lambda})_{\Gamma}$, then 
\[
 \a_{\mu}=\sum_{j=1}^r[b(\Delta_{x_{j}^{\lambda}}), k-1]+\sum_{m=1}^{t}[b(\Delta_{y_{j_m}^{\lambda}}), k-1]
 +\sum_{ j\notin \{j_1, \cdots, j_t\}}[b(\Delta_{y_j^{\lambda}}), k].
\]
Therefore 
\[
 x_{\mu}\supseteq x_{\lambda}
\]
as a set.
The converse is also true. 
\end{proof}

\section{Grassmanian case}

As before, let 
\[
 J=\{\sigma_i: i=1, \cdots, r-1\}\cup \{\sigma_i: i=r+1\cdots, r+\ell-1\},
\]
and
\[
 \a: =\a_{\Id}^{J, \emptyset}=\{\Delta_1, \cdots, \Delta_r, \cdots, \Delta_{r+\l}\}
\]
be a multisegment of parabolic type $(J, \emptyset)$, where 
\[
e(\Delta_i)=k-1, \text{ for } i=1, \cdots, r,
\]
and 
\[
 e(\Delta_i)=k, \text{ for } i=r+1, \cdots, r+\l. 
\]
Moreover, for $\lambda\in \mathcal{P}(\ell, r)$, let 
$x_{\lambda}=(x_1^{\lambda}, \cdots, x_r^{\lambda})=\varsigma_2(\lambda)\in R_r(r+\l)$
and $y_{\lambda}=(y_1^{\lambda}, \cdots, y_{\l}^{\lambda})\in R_\l(r+\l)$ such that 
\[
 \{1, \cdots, r+\l\}=\{x_1^{\lambda}, \cdots, x_r^{\lambda}, y_1^{\lambda}, \cdots, y_{\l}^{\lambda}\}.
\]
As described in proposition \ref{prop: 7.3.3}, we have 
\[
 \a_{\lambda}=\sum_{j=1}^r[b(\Delta_{x_{j}^{\lambda}}), k-1]+\sum_{j=1}^{\l}[b(\Delta_{y_j^{\lambda}}), k].
\]    
Let $0<r_0\leq \ell $ and $r_1=r+r_0, ~ \ell_1=\ell-r_0$.

\begin{prop}\label{prop: 7.6.1}
Let $\mu\in \mathcal{P}(\ell_1, r_1)$. Then 
there exists $\mu^{\flat}\in \mathcal{P}(\ell, r)$, such that 
\[
 \{\b\in S(\a): \a_{\mu}\preceq_k \b\}=\{\a_{\lambda}: \lambda\in \mathcal{P}(\ell, r),~ \lambda\leq \mu^{\flat}\}.
\]
More explicitly, if 
$x_{\mu}=(x_1^{\mu}, \cdots, x_{r_1}^{\mu})=\varsigma_2(\mu)$, then 
\[
 x_{\mu^{\flat}}=\varsigma_2(\mu^{\flat})=(x_{r_0+1}^{\mu}, \cdots, x_{r_1}^{\mu}).
\]

\end{prop}
     
\begin{proof}
By lemma \ref{lem: 7.6.18}, we know that the set  
\[
  \{\b\in S(\a): \a_{\mu}\preceq_k \b\}
\]
contains a unique minimal element 
$\a_{\mu^{\flat}}\in S(\a)$ for some $\mu^{\flat}\in \mathcal{P}(\ell, r)$. 
Therefore we have 
\[
 \{\b\in S(\a): \a_{\mu}\preceq_k \b\}=\{\a_{\lambda}: \lambda\in \mathcal{P}(\ell, r),~ \lambda\leq \mu^{\flat}\}.
\]
Note that if we write 
\[
 \a_{\mu}=\sum_{j=1}^{r_1}[b(\Delta_{x_j^{\mu}}), k-1]+\sum_{j=1}^{\l_1}[b(\Delta_{y_j^{\mu}}), k],
\]
then 
\[
 \a_{\mu^{\flat}}=\sum_{j=1}^{r_0}[b(\Delta_{x_j^{\mu}}), k]+\sum_{j=r_0+1}^{r_1}[b(\Delta_{x_j^{\mu}}), k-1]+\sum_{j=1}^{\l_1}[b(\Delta_{y_j^{\mu}}), k]
\]
is the minimal element in $S(\a)$ satisfying
\[
 \a_{\mu}=(\a_{\mu^{\flat}})_{\Gamma}
\]
for some $\Gamma\subseteq \a_{\mu^{\flat}}(k)$. 
\end{proof}

\begin{definition} 

Let 
\[
 J_1=\{\sigma_i: i=1, \cdots, r-1\}\cup \{\sigma_i: i=r+1, \cdots, r_1-1\}\cup \{\sigma_i: r_1+1, \cdots, r+\ell-1\},
\]
and 
\[
 \a_1=: \a_{\Id}^{J_1, \emptyset}=\{\Delta_1, \cdots, \Delta_{r_1}, \Delta_{r_1+1}^{+}, \cdots, \Delta_{r+\l}^+\},
\]
where $\a=\{\Delta_1, \cdots, \Delta_{r+\l}\}$ with $\Delta_1\unlhd \Delta_2\unlhd \cdots \unlhd \Delta_r$ (cf. Def. \ref{def: 7.2.4}). 
\end{definition}

\begin{lemma}
Let $\d=\a+\l_1[k+1]$, then  
\begin{itemize}
 \item we have $\a=\a_1^{(k+1)}$;
 \item and $\mathfrak{X}_{\d}=\coprod_{w\in S_{r+\l}^{J_1, \emptyset}}O_{\a_{w}}$, where 
 $\a_w=\a_{w}^{J_1, \emptyset}\in S(\a_1)$ is the element associated to $w$ by lemma \ref{lem: 6.2.3}.
\end{itemize} 
\end{lemma}

\begin{proof}
Note that by definition we have  
 \[
  \a=\a_1^{(k+1)}.
 \]
And by definition of $\mathfrak{X}_{\d}$, we know that 
$\mathfrak{X}_{\d}$ consists of the orbit $O_{\c}$ with $\c\in S(\d)$ 
such that $\varphi_{e(\c)}(k)+\l_1=\varphi_{e(\a)}(k)$, and 
the latter condition implies that there exists $w\in S_{r+\l}^{J_1, \emptyset}$
such that $\c=\a_{w}^{J_1, \emptyset}$. 
\end{proof}

\begin{prop}\label{prop: 7.6.4}
Let  $\d=\a+\l_1[k+1]$ and $W\subseteq V_{\varphi_{\d}, k+1}$ such 
that $\dim(W)=\l_1$ (which implies that $W=V_{\varphi_{\d}, k+1}$). 
Then the composition of morphisms
\[
 \mathfrak{X}_{\d}=(\mathfrak{X}_{\d})_W \xrightarrow{p} E_{\a}''\xrightarrow{\beta''} E_{\varphi_{\a}},
\]
sends $O_{\a_w}\cap (\mathfrak{X}_{\d})_W$ to $O_{\a_w^{(k+1)}}$. 

\end{prop}

\begin{proof}
This is by definition. 
\end{proof}

\begin{prop}
Let $\mu\in \mathcal{P}(\l_1, r_1)$ and 
 $x_{\mu}=\varsigma_2(\mu)=(x_1^{\mu}, \cdots, x_{r_1}^{\mu})$, 
 $y_{\mu}=(y_1^{\mu}, \cdots, y_{\l_1}^{\mu})$ such that 
\[
 \{1, \cdots, r_1+\l_1\}=\{x_1^{\mu}, \cdots, x_{r_1}^{\mu}, y_1^{\mu}, \cdots, y_{\l_1}^{\mu}\}.
\]
 Then 
\[
 (\a_{\mu^{\flat}})^{\sharp}=\sum_{j=1}^{r_0}[b(\Delta_{x_j^{\mu}}), k]+
 \sum_{j=r_0+1}^{r_1}[b(\Delta_{x_j^{\mu}}), k-1]+\sum_{j=1}^{\l_1}[b(\Delta_{y_j^{\mu}}), k+1],
\]
 for definition of $(\a_{\mu^{\flat}})^{\sharp}$, cf. lemma \ref{lem: 7.6.18}.
\end{prop}

\begin{proof}
Note that by proposition \ref{prop: 7.6.1}, we know that 
\[
 \a_{\mu^{\flat}}=\sum_{j=1}^{r_0}[b(\Delta_{x_j^{\mu}}), k]+\sum_{j=r_0+1}^{r_1}[b(\Delta_{x_j^{\mu}}), k-1]+\sum_{j=1}^{\l_1}[b(\Delta_{y_j^{\mu}}), k]
\] 
 and 
 \[
  \a_{\mu}=(\a_{\mu^{\flat}})_{\Gamma}
 \]
for $\Gamma=\sum_{j=1}^{r_0}[b(\Delta_{x_j^{\mu}}), k]$. 
Now by construction in lemma \ref{lem: 7.6.18}, we know that  
\[
 (\a_{\mu^{\flat}})^{\sharp}=\sum_{j=1}^{r_0}[b(\Delta_{x_j^{\mu}}), k]+
 \sum_{j=r_0+1}^{r_1}[b(\Delta_{x_j^{\mu}}), k-1]+\sum_{j=1}^{\l_1}[b(\Delta_{y_j^{\mu}}), k+1].
\] 
 \end{proof}

\begin{prop}
We have  
 \[
  n(\a_{\mu}, \a_{\mu^{\flat}})=\sharp \{\c\in S(\a_1): \c^{(k+1)}=\a_{\mu^{\flat}}, \c\geq (\a_{\mu^{\flat}})^{\sharp}\}.
 \] 
\end{prop}

\begin{proof}
Consider the composed morphism 
\[
 h: \mathfrak{X}_{\d}=(\mathfrak{X}_{\d})_W \xrightarrow{p} E_{\a}''\xrightarrow{\beta''} E_{\varphi_{\a}},
\]
then the orbits contained in $h^{-1}(O_{\a_{\mu^{\flat}}})$ is indexed by the set  
\[
 \{\c\in S(\a_1): \c^{(k+1)}=\a_{\mu^{\flat}}, \c\geq (\a_{\mu^{\flat}})^{\sharp}\}
\]
Note that 
by corollary \ref{cor: 7.4.19} and  proposition \ref{prop: 7.3.8}, 
the number  
 \[
  n(\a_{\mu}, \a_{\mu^{\flat}})=\sum_i \dim \mathcal{H}^{2i}(\beta''_*(IC(\line{E''_{\a}((\a_{\mu^{\flat}})^{\sharp}))}))_{x}
 \]
for some $x\in O_{\a_{\mu^{\flat}}}$. Finally, note that the morphism 
$\beta''$ is smooth when restricted to the variety $\beta''^{-1}(O_{\a_{\mu^{\flat}}})$.
Moreover, the fibers are open in some Schubert variety, therefore, we are reduced to the counting of orbits.
\end{proof}

More generally, we have 

\begin{definition}
 Let $w_{\mu}\in S_{r+\l}^{J_1, \emptyset}$
 be the element such that 
 \[
  \a_{w_{\mu}}=(\a_{\mu^{\flat}})^{\sharp}.
 \]

\end{definition}

\begin{prop}\label{prop: 7.6.8}
Let $P_J$ and $P_{J_1}$ be the parabolic subgroups 
corresponding to $J$, $J_1$ respectively. Consider the 
natural morphism 
\[
 \pi: P_{J_1}\backslash GL_{r+\l}\rightarrow P_{J}\backslash GL_{r+\l}.
\]
Then 
\[
 n(\a_{\mu}, \a_{\lambda})=\sum_i \dim \mathcal{H}^{2i}(\pi_*(IC(\line{P_{J_1}}w_{\mu}B)))_x
\]
for some $x\in P_Jt_{\lambda}B$, here 
$t_{\lambda}$ is the element in $S_{r+\l}^{J, \emptyset}$
associated to the partition $\lambda$.

\end{prop}

\begin{proof}
Consider the composed morphism 
\[
 h: \mathfrak{X}_{\d}=(\mathfrak{X}_{\d})_W \xrightarrow{p} E_{\a}''\xrightarrow{\beta''} E_{\varphi_{\a}}.
\]
This proposition can be deduced from 
a construction of fibration similar to the one we did 
in Chapter 2 for symmetric multisegments, cf.\S 2.5. 
 
\end{proof}

\section{Parabolic Case}

In this section, as in the Grassmannian case, we deduce a formula for calculating the coefficient $n(\b, \a)$. 

Let 
\[
J\subseteq S
\] 
be a subset of generators and 
$$
\a=\a_{\Id}^{J, \emptyset}
$$
be some multisegment of parabolic type $(J, \emptyset)$ associated to the identity,  satisfying $f_{e(\a)}(k)\neq 0, f_{e(\a)}(k+1)=0$.

\begin{notation}
For $k\in \Z$, we let $\l_k=f_{e(\a)}(k)$.
\end{notation}

\begin{definition}
Let $\a(k)=\{\Delta_1, \cdots, \Delta_{\l_k}\}$ with $\Delta_1\unlhd \cdots \unlhd \Delta_{\l_k}$ and $r_0\in \N$ with $1\leq r_0\leq \l_k$. Then let 
\begin{align*}
 \a_1&=(\a\setminus \a(k))\cup \{\Delta\in \a(k):  \Delta\unlhd \Delta_{\l_k-r_0}\}\cup \{\Delta^+\in \a(k): \Delta\unrhd \Delta_{\l_k-r_0+1}\},\\
 \a_2&=(\a\setminus \a(k))\cup \{\Delta^-\in \a(k):  \Delta\unlhd \Delta_{r_0}\}\cup \{\Delta\in \a(k): \Delta\unrhd \Delta_{r_0+1}\}
\end{align*}
and $J_i(r_0, k)(i=1, 2)$ be a subset of $S$ such that 
 $\a_i$ is a multisegment of parabolic type $(J_i(r_0, k), \emptyset)$. Moreover, let 
\[
\a_{\Id}^{J_i(r_0, k), \emptyset}=\a_i,   \text{ for } i=1,2.
\]
\end{definition}

\begin{lemma}
Let $\l_1=\l_k-r_0$ and $\d=\a+\l_1[k+1]$, then  
\begin{itemize}
 \item we have $\a=\a_1^{(k+1)}$;
 \item and $\mathfrak{X}_{\d}=\coprod_{w\in S_{n}^{J_1(r_0, k), \emptyset}}O_{\a_{w}}$, where 
 $\a_w=\a_{w}^{J_1(r_0, k), \emptyset}\in S(\a_1)$ is the element associated to $w$ by lemma \ref{lem: 6.2.3}.
\end{itemize} 
\end{lemma}

\begin{prop}\label{prop: 7.7.1}
Let $w\in S_{n}^{J_2(r_0, k), \emptyset}$. Then 
there exists $w^{\flat}\in S_{n}^{J, \emptyset}$, such that 
\[
 \{\b\in S(\a): \a_{w}\preceq_k \b\}=\{\a_{v}: v\in S_{n}^{J, \emptyset},~ v\leq w^{\flat}\}.
\]
More explicitly, if 
$\a_{w}(k-1)=\{\Delta_1, \cdots, \Delta_{\l_{k-1}}\}$ with $\Delta_1\unlhd \cdots \unlhd \Delta_{\l_{k-1}}$, then 
\[
 \a_{w^{\flat}}=(\a_{w}\setminus \a_{w}(k-1))\cup \{\Delta^+\in \a_{w}(k-1):  \Delta\unlhd \Delta_{r_0}\}\cup \{\Delta\in \a_{w}(k-1): \Delta\unrhd \Delta_{r_0+1}\}.
\]
\end{prop}

\begin{prop}
Let $w\in S_{n}^{J_2(r_0, k), \emptyset}$.
 Then 
\[
 (\a_{\mu^{\flat}})^{\sharp}=(\a_{w^{\flat}}\setminus \a_{w}(k))\cup \{\Delta^+: \Delta\in \a_{w}(k)\}
\]
 for definition of $(\a_{\mu^{\flat}})^{\sharp}$, cf. lemma \ref{lem: 7.6.18}.
\end{prop}

\begin{definition}
Let $t_w\in S_{n}^{J_1(\ell _k-r_0, k), \emptyset}$ be the element such that 
\[
\a_{t_w}=(\a_{w^{\flat}})^{\sharp}.
\]
\end{definition}

\begin{prop}\label{prop: 7.7.8}
Let $P_J$ and $P_{J_1(\l_{k}-r_0, k)}$ be the parabolic subgroups 
corresponding to $J$, $J_1(\l_{k}-r_0, k)$ respectively. Consider the 
natural morphism 
\[
 \pi: P_{J_1(\l_{k}-r_0, k)}\backslash GL_{n}\rightarrow P_{J}\backslash GL_{n}.
\]
Then 
\[
 n(\a_{w}, \a_{v})=\sum_i \dim \mathcal{H}^{2i}(\pi_*(IC(\line{P_{J_1(\l_{k}-r_0, k)}t_{w}B})))_x
\]
for some $x\in P_JvB$.

\end{prop}

\begin{proof}
Consider the composed morphism 
\[
 h: \mathfrak{X}_{\d}=(\mathfrak{X}_{\d})_W \xrightarrow{p} E_{\a}''\xrightarrow{\beta''} E_{\varphi_{\a}}.
\]
This proposition can be deduced from 
a construction of fibration similar to the one we did 
in Chapter 2 for symmetric multisegments, cf.\S 2.5. 
 
\end{proof}  

\section{Calculation of Partial Derivatives}
Again, as previous section, we restrict ourselves to the case of multisegment of parabolic type.

\begin{definition}
Let $J_1\subseteq J_2\subseteq S$ be two subsets of generators of $S_n$. Let 
$v\in S^{J_1, \emptyset}_n, w\in S_n^{J_2, \emptyset}$, we define $\theta_{J_2}^{J_1}(w, v)$ to be the 
multiplicities of $IC(\line{P_{J_2}wB})$ in $\pi_*(IC(\line{P_{J_1}vB}))$, where   
\[
\pi: P_{J1}\backslash GL_n \rightarrow P_{J_2}\backslash GL_n
\]
be the canonical projection. 
\end{definition}

\remk
By proposition \ref{prop: 5.3.12}, we know that in case where $J_1=\emptyset, J_2=\{s_i\}$ we have 
$\theta_{J_2}^{J_1}(w, v)=\mu(s_iw, v)$ if $\ell(v)\leq \ell(s_iv)$, where $\mu(x, y)$ is the 
coefficient of degree $(\ell(y)- \ell(x)-1)/2$ in $P_{x, y}(q)$.

%
%
%

\begin{prop}
Let $J\subseteq S$ be a subset of generators in $S_n$. 
Let $k\in \Z$ and $\a$ be a multisegment satisfies all the assumptions in the beginning of section 7.7.
Then for any $w\in S_n^{J, \emptyset}$, we have 
\[
\D^{k}(L_{\Phi(w)})=\sum_{r_0=0}^{\ell_k}\sum_{v\in S_n^{J_2(r_0, k), \emptyset}}\theta_{J}^{J_{1}(\l_k-r_0, k)}(w, t_v)L_{\Phi(v)}.
\]
\end{prop}

\begin{proof}
Note that by proposition \ref{prop: 7.4.2}
\[
\D^k(\pi(\Phi(w)))=\sum_{\b\preceq_k \Phi(w)}n(\b, \a)L_{\b}.
\]
Note that by proposition proposition \ref{prop: 7.3.5}, we know that 
$\b\preceq_k \Phi(w)$ implies that 
\[
\b=\Phi(v), 
\]
for some $v\in J_2(\l_k-r_0, k)$. 
Moreover, according to the proposition  \ref{prop: 7.7.8}
\[
n(\Phi(v), \Phi(w))=\sum_i \dim \mathcal{H}^{2i}(\pi_*(IC(\line{P_{J_1(\l_{k}-r_0, k)}t_{v}B})))_x
\]
for some $x\in P_JwB$.
In fact, by the decomposition theorem, we have 
\addtocounter{theo}{1}
\begin{equation}\label{eq: 7.42}
\pi_*(IC(\line{P_{J_1(\l_{k}-r_0, k)}t_{v}B})=\bigoplus_{u\in S_n^J} \oplus_{i}IC(\line{P_JuB})^{h_i(u, t_v)}[d_u^i]
\end{equation}
therefore 
\[
\theta_{J}^{J_1(\l_k-r_0, k)}(u, t_v)=\sum_{i}h_i(u, t_v).
\]
Furthermore, if we denote by 
\[
\theta_{J}^{J_1(\l_k-r_0, k)}(u, t_v)(q)=\sum_{i}h_i(u, t_v)q^{-d_u^i/2}
\]
by localizing at a point of $P_JwB$ and applying proper base change, we get 
\addtocounter{theo}{1}
\begin{equation}\label{eq: 7.6}
\sum_{\rho\in S_{J/J_1(\l_k-r_0, k)}}q^{\ell(\rho)}P^{J_1(\l_k-r_0, k), \emptyset}_{\rho w, t_v}(q)=\sum_{u} \theta_{J}^{J_1(\l_k-r_0, k)}(u, t_v)(q) P^{J, \emptyset}_{w, u}(q).
\end{equation}
Now we return to the formula 
\addtocounter{theo}{1}
\begin{equation}\label{eq: 7.5}
\pi(\Phi(w))=\sum_{u} P^{J, \emptyset}_{w, u}(1)L_{\Phi(u)}.
\end{equation}
By induction, we can assume that for $u>w$, we have that $L_{\Phi(v)}$ appears in $\D^k(L_{\Phi(u)})$ with multiplicity $\theta_{J}^{J_1(\l_k-r_0, k)}(u, t_v)$. Then 
by applying the derivation $\D^k$ to equation (\ref{eq: 7.5}), 
on the right hand side we get the multiplicity of $L_{\Phi(v)}$ given by 
\[
x+\sum_{u>w}  \theta_{J}^{J_1(\l_k-r_0, k)}(u, t_v) P^{J, \emptyset}_{w, u}(1),
\]
where $x$ denotes the multiplicity of $L_{\Phi(v)}$ in the derivative $\D^k(L_{\Phi(w)})$.  And on the right hand side, 
applying corollary \ref{cor: 4.6.17},
we get
\[
\sum_{\rho\in S_{J/J_1(\l_k-r_0, k)}}P^{J_1(\l_k-r_0, k), \emptyset}_{\rho w, t_v}(1).
\]
Now compare with the equation (\ref{eq: 7.6}) to get $x=\theta_{J}^{J_1(\l_k-r_0, k)}(w, t_v)$
 \end{proof}

From now on we consider the derivative $\D^k(L_{\c})$ for a general multisegment $\c$ such that 
$f_{e(\c)}(k)>0$.   

\begin{prop}
There exists a multisegment $\c'$ which is of parabolic type $(J_1(\c), \emptyset)$( cf. definition \ref{def: 6.3.1})
and  a sequence of integers $k_1, \dots, k_r, k_{r+1}, \dots, k_{r+\l}$ such that $L_{\c}$ is the minimal degree term with multiplicity one in 
\[
{^{k_1}\D}\cdots {^{k_r}\D}\D^{k_{r+1}}\cdots \D^{k_{r+\l}}(L_{\c'}),
\]
and 
\begin{align*}
&f_{e(\c')}(i)=f_{e(\c)}(i), \quad \text{ if } i\leq k,  \\
&f_{e(\c')}(k+1)=0,\\
&k_i>k+1, \quad \text{if } i>r.
\end{align*}
\end{prop}

\begin{proof}
Let $i_0=\min \{i:  f_{b(\c)}(i)>1\}$ and $\Delta_0=\max_{\prec}\{\Delta\in \c:   b(\Delta)=i_0\}$.
Then replace all segments $\Delta\in \c$ with $b(\c)<i_0$ by ${^{+}\Delta}$ 
and $\Delta_0$ by ${^+\Delta}$ to get a new multisegment $\c_1$. Then if we let $\{i\in b(\c): i<i_0\}=\{j_1<\cdots<j_r\}$, we have 
$L_{\c}$ is the minimal degree terms in 
\[
{^{j_1-1}\D}\cdots {^{j_r-1}\D}(L_{\c_1}),
\]
Repeat this procedure to get $\c_0$ and a sequence of integers $k_1, \cdots, k_r$ such that 
$L_{\c}$ is the minimal degree term with multiplicity one in 
\[
{^{k_1}\D}\cdots {^{k_r}\D}(L_{\c_0}).
\]
Suppose that $f_{e(\c_0)}(k+1)>0$. Then replace all segments $\Delta$ in 
$\c_0$ with $e(\Delta)>k$ by $\Delta^+$ to obtain $\c'$, we are done.
\end{proof}

\begin{definition}
We define
\[
\Gamma^i(\a, k)=\{\b\in \Gamma(\a, k): \deg(\b)+i=\deg(\a)\},
\]
where $\l_k=f_{e(\a)}(k)$.
\end{definition}

\begin{definition}
Let $\a$ be a multisegment and $k, k_1\in \Z$. Then we define 
\begin{align*}
\Gamma^i(\a, k)_{k_1}&=\{\b\in \Gamma^i(\a, k):  \b\in S(\b)_{k_1}, \b^{(k_1)}\in \Gamma^i(\a^{(k_1)}, k)\}, \\
\Gamma(\a, k)_{k_1}&=\cup_i \Gamma^i(\a, k)_{k_1}.
\end{align*}
More generally for a sequence of integers $k_1, \cdots, k_r$, we define 
\[
\Gamma(\a, k)_{k_1, \cdots, k_r}=\{\b\preceq_k \a:  \b^{(k_1, \cdots, k_{i-1})}\in \Gamma(\a^{(k_1, \cdots, k_{i-1})}, k)_{k_{i}} \text{ for } 1\leq i\leq r\}.
\]
Similarly, we can define 
${_{k_1}\Gamma(\a, k)}$ and 
${_{k_1, \cdots, k_r}\Gamma(\a, k)}$.
\end{definition}

\remk 
We can also talk about the set $_{k_{r+1}, \cdots, k_{r+\l}}(\Gamma(\a, k)_{k_1, \cdots, k_r})$.

\begin{lemma}\label{lem: 7.8.6}
Let $k_1\neq k-1$, then the map
\begin{align*}
\psi_{k_1}: &\Gamma(\a, k)_{k_1}\rightarrow \Gamma(\a^{(k_1)}, k)\\
& \b\mapsto \b^{(k_1)}
\end{align*}
is bijective.
\end{lemma}

\begin{proof}
In fact we have $\Gamma^i(\a, k)=S(\a_i)$ where $\a_i$ is constructed in the following way:
let $\a(k)=\{\Delta_1\succeq \cdots \succeq \Delta_r\}$, then 
\[
\a_{i}=(\a\setminus \a(k))\cup \{\Delta_j^{-}: j\leq i\}\cup \{\Delta_j: j>i\}.
\]
Note that $\Gamma^i(\a, k)=S(\a_i)$, which implies that we have 
\[
\Gamma^i(\a, k)_{k_1}=S(\a_i)_{k_1}.
\]
Finally, note that  by proposition \ref{cor: 3.2.3}  we have a bijection
\[
\psi_{k_1}: S(\a_{i})_{k_1}\rightarrow S(\a_{i}^{(k_1)}).
\] 
Note that $k_1\neq k-1, k$ implies that $\a_i^{(k_1)}\in \Gamma^i(\a^{(k_1)}, k)$ and 
\[
\Gamma(\a^{(k_1)}, k)=\bigcup_{i}S(\a_i^{(k_1)}).
\]
And if $k_1=k$, then 
\[
\Gamma(\a, k)_k=S(\a)_k, \quad \Gamma(\a^{(k_1)}, k)=S(\a^{(k)}).
\]
Hence we are done.
\end{proof}

\begin{lemma}
Let $k_1, k\in \Z$ then the map
\begin{align*}
{_{k_1}\psi}: &{_{k_1}\Gamma(\a, k)}\rightarrow \Gamma({^{(k_1)}\a}, k)\\
& \b\mapsto {^{(k_1)}\b}
\end{align*}
is bijective.
\end{lemma}

\begin{proof}
If $k_1\neq k$, the proof  is the same as that of the previous lemma. Consider the case where $k_1=k$.
Let $\a(k)=\{\Delta_1\succeq \cdots \succeq \Delta_{r_0}\succ \underbrace{[k]=\cdots =[k]}_{r_1}\}$.
Then for $i\leq \l_k$, we have 
\[
\a_{i}=(\a\setminus \a(k))\cup \{\Delta_j^{-}: j\leq i\}\cup \{\Delta_j: j>i\},
\]
where $\Delta_j=[k]$ if $j> r_0$.
And we have $\Gamma^i(\a, k)=S(\a_i)$. By definition,  we have $\b\in {_k\Gamma^i(\a, k)}$ if and only if 
\[
\b\in {_kS(\b)}, \quad ^{(k)}\b\in \Gamma^{i}(^{(k)}\a, k).
\]
Since ${^{(k)}\a}(k)=\{\Delta_1, \cdots, \Delta_{r_0}\}$, we know that for $\b\in {_k\Gamma^i(\a, k)}$, we must have 
$i\leq r_0$. Also, let
\[
(^{(k)}\a)_{i}=(^{(k)}\a\setminus {^{(k)}\a}(k))\cup \{\Delta_j^{-}: j\leq i\}\cup \{\Delta_j: r_0\geq j>i\}.
\]
And we have $\Gamma^{i}(^{(k)}\a, k)=S((^{(k)}\a)_{i})$. Then we have 
\[
^{(k)}\a_i=(^{(k)}\a)_i.
\]
Finally, we conclude that $\b\in {_k\Gamma^i(\a, k)}$ if and only if $\b\in {_kS(\a_i)}$.
Since the map 
\[
{_kS(\a_i)}\rightarrow S(^{(k)}\a_i)
\]
is bijective, we are done.
\end{proof}

\begin{prop}
Let $\b, \c$ be two multisegments and $k_1\in \Z$ such that 
\[
\b=^{(k_1)}\c, \quad \c\in {{_{k_1}}S(\c)}.
\]
If we write
\addtocounter{theo}{1}
\begin{equation}\label{eq: 7.8.8}
\D^k(L_{\c})=L_{\c}+\sum_{\d\in \Gamma(\c, k)\setminus \{\c\}} \tilde{n}(\d, \c)L_{\d},
\end{equation}
then 
\[
\D^k(L_{\b})=L_{\b}+\sum_{\d\in {_{k_1}\Gamma(\c, k)}\setminus \{\c\}}\tilde{n}(\d, \c)L_{{^{(k_1)}\d}}.
\]
\end{prop}
\begin{proof}
-Suppose that $\deg(\c)=\deg(\b)+1$.
In fact, by corollary \ref{cor: 3.5.2}, we have 
\[
^{k_1}\D(L_{\c})=L_{\c}+L_{\b}
\]
By applying the derivation $\D^k$ and using the fact $\D^k({^{k_1}\D})={^{k_1}\D}\D^k$, we have 
\[
\D^{k}(L_{\c})+\D^{k}(L_{\b})=L_{\c}+L_{\b}+\sum_{\d\in \Gamma(\c, k)\setminus \{\c\}}\tilde{n}(\d, \c){^{k_1}\D(L_{\d})}
\]
By assumption that $\deg(\b)+1=\deg(\c)$, we have 
\[
^{k_1}\D(L_{\d})=L_{\d}+L_{^{(k_1)}\d} \text{ or } L_{\d},
\]
where $^{k_1}\D(L_{\d})=L_{\d}+L_{^{(k_1)}\d}$ if and only if $\d\in {_{k_1}S(\d)}$ and $\deg(^{(k_1)}\d)=\deg(\d)-1$.
This is equivalent to say that $\d\in {_{k_1}\Gamma(\a, k)}$.

\bigskip
-For general case,  consider 
\[
\{\Delta\in \c:  b(\Delta)=k_1\}=\{\Delta_1\succeq \dots\succeq \Delta_r\}.
\]
Now by proposition \ref{teo: 3.0.6} and proposition \ref{prop: 7.4.2}, we know that 
\[
^{k_1}\D(L_{\c})=L_{\b}+\sum_{ f_{\d}(k_{1})>f_{\b}(k_1)}\tilde{n}(\d, \c)L_{\d},
\]
for some $\tilde{n}(\d, \c)\in \N$. 

If $k_1\neq k$, then 
We observe that for any $\d$ such that $f_{\d}(k_{1})>f_{\b}(k_1)$ and $\d'\preceq_k \d$, we have 
\[
f_{\d'}(k_1)>f_{\b}(k_1), 
\] 
which implies that $L_{\d'}$ can not be a summand of $\D^{k}(L_{\b})$.  Therefore
we know that 
\[
\D^k(L_{\b})
\]   
is the sum of all irreducible representations $L_{\d''}$ contained in $\D^k(^{k_1}\D)(L_{\c})$ satisfying 
\[
f_{\d''}(k_1)=f_{\b}(k_1).
\]  
Applying the derivation $^{k_1}\D$ to the equation (\ref{eq: 7.8.8}), we get 
\[
{^{k_1}\D} \D^k(L_{\c})={^{k_1}\D}(L_{\c})+\sum_{\d\preceq_k \c}\tilde{n}(\d, \c) ({^{k_1}\D})(L_{\d}).
\]
Note that in this case the sub-quotient of ${^{k_1}\D} \D^k(L_{\c})$ consisting of irreducible representations $L_{\d''}$ 
satisfying 
\[
f_{\d''}(k_1)=f_{\b}(k_1)
\] 
is given by 
\[
L_{\b}+\sum_{\d\in {_{k_1}\Gamma(\c, k)}\setminus \{\c\}}\tilde{n}(\d, \c)L_{{^{(k_1)}\d}}.
\]
Compare the equation ${^{k_1}\D} \D^k(L_{\c})=\D^k(^{k_1}\D)(L_{\c})$ gives the results.\\
If $k_1=k$,  consider 
\[
\{\Delta\in \c:  b(\Delta)=k_1\}=\{\Delta_1\succeq \dots\succeq \Delta_r\}.
\]
Let $\c'$ be the multisegment obtained by replacing all segments $\Delta$ in $\c$
such that $b(\Delta)<k_1$ by ${^{+}\Delta}$, and $\Delta_1$ by ${^{+}\Delta_1}$.
Then there exists 
\[
k_2=k_1-1>k_3>\cdots >k_r
\]
such that 
\[
\c={^{(k_r, \cdots, k_2)}\c'},
\]
and 
\[
\b={^{(k_r, \cdots, k_3, k_1, k_2, k_1)}\c'}.
\]
Let $\b'=^{(k_1)}\c'$, then by induction on $f_{b(\c)}(k)$, we can assume that 
\[
\D^k(L_{\b'})=L_{\b'}+\sum_{\d\in {_{k}\Gamma(\c', k)}\setminus \c'} \tilde{n}(\d, \c')L_{^{(k)}\d}.
\]
Applying what we have proved before, we get 
\[
\D^k(L_{\b})=L_{\b}+\sum_{\d\in _{k_r, \cdots, k_3, k, k_2, k}\Gamma(\c', k)\setminus \{\c'\}}\tilde{n}(\d, \c')L_{^{(k_r, \cdots, k_3, k, k_2, k)}\d}.
\]
Also, we have 
\[
\D^k(L_{\c})=L_{\c}+\sum_{\d\in _{k_r, \cdots, k_3, k_2}\Gamma(\c', k)\setminus \{\c'\}} \tilde{n}(\d, \c')L_{^{(k_r, \cdots, k_3, k_2)}\d}.
\]
Since for any multisegment $\d$, we have
\[
^{(k, k_r, \cdots, k_3, k_2)}\d=^{(k_r, \cdots, k_3, k, k_2, k)}\d,
\]
it remains to show that 
\[
_{k_r, \cdots, k_3, k, k_2, k}\Gamma(\c', k)=_{k, k_r, \cdots, k_3, k_2}\Gamma(\c', k).
\]
By definition and the following lemma, 
we can assume that $r=2$. In this case we argue by contradiction.
Suppose that $\d\in {_{k, k-1, k}\Gamma^i(\c', k)}$
and $\d\notin {_{k, k-1}\Gamma(\c', k)}$, which
is equivalent to say that $\d\notin {_{k,k-1}S(\d)}$.
 Note that $\d\notin {_{k, k-1}S(\d)}$ implies that 
there exists two linked segments $\{\Delta, \Delta'\}$, such that 
\[
b(\Delta)=k, \quad b(\Delta')=k-1.
\]
Then ${^{(k-1, k)}\d}$ contains the pair of segments $\{{^{-}\Delta}, {^{-}\Delta'}\}$. The fact that 
${^{(k-1, k)}\d}\in {_{k}S(^{(k-1, k)}\d)}$ implies  that  ${^{-}\Delta'}=\emptyset$, i.e. $\Delta'=[k-1]$. 
However, this implies that ${^{(k, k-1, k)}\d}\notin \Gamma^i(^{(k, k-1, k)}\c', k)$
since $\deg({^{(k, k-1, k)}\d})+i=\deg({^{(k, k-1, k)}\a})+1$,
which is a  contradiction.

Conversely, assume that $\d\in {_{k, k-1}\Gamma(\c', k)}$ and $\d \notin {_{k, k-1, k}\Gamma^i(\c', k)}$, 
which by definition is equivalent to $\d\notin {_{k, k-1, k}S(\d)}$. Note that $\d\notin {_{k, k-1, k}S(\d)}$ implies 
that $\d\notin {_{k}S(\d)}$, which contradicts to $\d\in {_{k, k-1}S(\d)}$.  
\end{proof}

\begin{lemma}
Let $k>k-1>k'$ be two integers. Then for any multisegment $\c$, we have 
\[
{_{k, k'}\Gamma(\c, k)}={_{k', k}\Gamma(\c, k)}.
\]

\end{lemma}
\begin{proof}
Note that since for any multisegment $\d$
\[
{^{(k', k)}\d}={{^{k, k'}}\d}, 
\]
the fact 
\[
{_{k, k'}\Gamma(\c, k)}={_{k', k}\Gamma(\c, k)}
\]
is equivalent to 
\[
\d\in {_{k, k'}S(\d)}\Leftrightarrow \d\in {_{k', k}S(\d)}
\]
for all $\d\in {_{k, k'}\Gamma(\c, k)}$. But for any multisegment $\d$ and $k>k-1>k'$, we have 
\[
\d\in {_{k, k'}S(\d)}\Leftrightarrow \d\in {_{k', k}S(\d)}.
\]
Hence we are done.
\end{proof}

\begin{prop}\label{prop: 7.8.10}
Let $k_1\neq k-1, k, k+1$.
Let $\b, \c$ be two multisegments such that 
\[
\b=\c^{(k_1)}, \quad \c\in S(\c)_{k_1}.
\]
If we write
\addtocounter{theo}{1}
\begin{equation}
\D^k(L_{\c})=L_{\c}+\sum_{\d\in \Gamma(\c, k)\setminus \{\c\}} \tilde{n}(\d, \c)L_{\d},
\end{equation}
then 
\[
\D^k(L_{\b})=L_{\b}+\sum_{\d\in \Gamma(\c, k)_{k_1}\setminus \{\c\}}\tilde{n}(\d, \c)L_{d^{(k_1)}}.
\]
\end{prop}
\begin{proof}
The proof is the same as the proposition above.
\end{proof}

Now let $\c'=\Phi(w)$ for some $w\in S_n^{J, \emptyset}$.

\begin{cor} \label{coro-fornula-derivative}
We have 
\[
\D^k(L_{\a})=\sum_{r_0=0}^{\l_k}\sum_{v\in S_n^{J_2(r_0, k), \emptyset}, \Phi(v)\in _{k_1, \cdots, k_r}(\Gamma(\Phi(w), k)_{k_{r+1}, \cdots, k_{r+\l}})}\theta_{J}^{J_{1}(\l_k-r_0, k)}(w, t_v)L_{^{(k_1, \cdots, k_r)}\Phi(v)^{(k_{r+1}, \cdots, k_{r+\l})}}.
\]
\end{cor}

\begin{notation}\label{nota7-8-12}
For $\b\preceq_k \a$, we denote 
\[
\theta_k(\b, \a)=\theta_{J}^{J_{1}(\l_k-r_0, k)}(w, t_v)
\]
if $\b=^{(k_1, \cdots, k_r)}\Phi(v)^{(k_{r+1}, \cdots, k_{r+\l})}$. Otherwise, put $\theta_k(\b, \a)=0$.

\end{notation}
\remk 
The same way we define ${_{k}\theta}(\b, \a)$ by the formula
\[
(^{k}\D)(L_{\a})=\sum_{\b}{_{k}\theta}(\b, \a)L_{\b}.
\]
And let 
\[
\Gamma(k, \a)=\{\b: {_{k}\theta}(\b, \c)\neq 0 \text{ for some }\c\in S(\a)\},
\]
it shares similar properties with $\Gamma(\a, k)$.

\chapter{Multiplicities in induced representations: case of a segment}

In this chapter we will consider the multiplicities $m(\c, \b, \a)$ of irreducible components 
in the induced representation $L_{\a}\times L_{\b}$,
\[
L_{\a}\times L_{\b}=\sum m(\c, \b, \a)L_{\c}.
\]
Our goal in this chapter is then to determine a formula for the coefficient $m(\c, \b, \a)$
in case where $\b=[k-i_0+1, k+1](i_0\geq 0)$ is a segment. 
Roughly speaking, there are two major cases to discuss
\begin{description}
\item[(1)] $\max b(\a)\leq k-i_0+1$,
\item[(2)] $\max b(\a)>k-i_0+1$.
\end{description} 
In \S\ref{se: 8.1} we treat the first case, which is simpler to deal with.  We have an explicit formula for the case 
where $\b=[k+1]$ (cf. lemma \ref{lem: 8.1.4} and proposition \ref{prop: 8.1.5}), and then we deduce by induction the general case (cf. proposition \ref{prop: 8.1.7}).  
For example the formula of proposition \ref{prop: 8.1.5} looks like
\[
L_{\a}\times L_{\b}=L_{\a+\b}+\sum_{\c\in \Gamma^{\l_k-1}(\a, k)} \Bigl ( \theta_k(\c, \a)-\theta_k(\c^{[k+1]_1}, \a+\b) \Bigr ) L_{\c^{[k+1]_1[ k]_{\l_k-1}}}.
\]
where the $\theta_k(\c, \a)$ are defined thanks to partial derivative, cf. notation \ref{nota7-8-12}.

Here our main tool is the 
derivatives for which we have complete formulas, cf. proposition \ref{coro-fornula-derivative}. 
Note that even in the case where $\b=[k+1]$ is a point, we come across the difficulty that we have 
$\D^k(L_{\c})=L_{\c}$ for certain multisegments, cf. example \ref{ex: 8.1.6}, which prevents us from 
applying the partial derivations.  Our idea here is to first treat the case where $f_{e(\a)}(k-1)=0$, cf. proposition \ref{prop: 8.1.5},  and then 
reduce everything to such case.

In \S 8.2, we describe a procedure to compute $m(\c, \b, \a)$ for the second case, combining the first case 
and partial derivatives.

\bigskip 
Finally, we remark that our method could be used to deduce the general multiplicities for case where $\b$ is not a segment. We intend
to study this general case in some future work.

\section{When $\max b(\a)\leq k-i_0+1$}

\label{se: 8.1}
In this section we consider the case $L_{\a}\times L_{\b}$ where $\b=[k-i_0+1, k+1]$, with $i_0\geq 0$, is a segment
and $\a$ is a multisegment satisfying 
\[
\max b(\a)\leq k-i_0+1.
\]

\begin{definition}
Let $\b$ be a multisegment such that $f_{e(\b)}(k+1)=0$. Then 
we denote by $\b^{[k+1]_i}$ the unique element in $S(\b+i[k+1])_{k}$ such that 
\[
\c=(\c^{[k+1]_i})^{(k+1)}.
\] 
\end{definition}

\begin{prop}\label{prop: 8.1.6}
Let $\a$ be a multisegment satisfying the condition  
\[
f_{e(\a)}(k-i_0-1)\neq 0.
\]
 If we assume that 
 \[
 \{t\in e(\a): t\leq k-i_0-1\}=\sum_{i=1}^{r}\l_{k_i}[k_i]
 \]
with $k_1<\cdots <k_r=k-1$, then 
\[
m(\c^{[k_r]_{\l_{k_r}}[k_{r-1}]_{\l_{k_{r-1}}}\cdots [k_1]_{\l_1}}, \b, \a)=m(\c, \b, \a^{(k_1, \cdots, k_r)}).
\]
\end{prop}

\begin{proof}
We prove by induction on $i$ that 
\[
m(\c^{[k_i]_{\l_{k_i}}[k_{i-1}]_{\l_{k_{i-1}}}\cdots [k_{1}]_{\l_{k_{1}}}}, \b, \a)=m(\c, \b, \a^{(k_1, \cdots, k_i)})
\]
For $i=1$, since $\a$ satisfies the hypothesis $H_{k_1}(\a)$, by proposition \ref{teo: 3.0.6}, 
$\D^{k_1}(L_{\a})$ contains a unique minimal degree term with multiplicity one, which is $L_{\a^{(k_1)}}$, 
now apply $\D^{k_1}$ to  
\[
L_{\a}\times L_{\b}=\sum_{\c}m(\c, \b, \a)L_{\c}
\]
and consider the minimal degree terms on both sides, we obtain
\[
L_{\a^{(k_1)}}\times L_{\b}=\sum_{\c\in S(\a+\b)_{k_1}}m(\c, \b, \a)L_{\c^{(k_1)}} 
\]
which gives the formula. Now for general $i<r$, assume that we have 
\[
m(\c^{[k_i]_{\l_{k_i}}[k_{i-1}]_{\l_{k_{i-1}}}\cdots [k_{1}]_{\l_{k_{1}}}}, \b, \a)=m(\c, \b, \a^{(k_1, \cdots, k_i)}),
\]
that is to say 
\[
L_{\a^{(k_1, \cdots, k_{i})}}\times L_{\b}=\sum_{\c\in S(\a+\b)_{k_1, \cdots, k_i}}m(\c, \b, \a)L_{\c^{(k_1, \cdots, k_i)}}.
\]
Now apply $\D^{k_{i+1}}$ and the same argument as in the case where $i=1$ gives 
\[
L_{\a^{(k_1, \cdots, k_{i+1})}}\times L_{\b}=\sum_{\c\in S(\a+\b)_{k_1, \cdots, k_{i+1}}}m(\c, \b, \a)L_{\c^{(k_1, \cdots, k_{i+1})}}.
\]

\end{proof}

\remk If we assume that $\a$ is of parabolic type, i.e
\[
\bigcap_{\Delta\in \a}\Delta\neq \emptyset
\]
then 
\[
S(\a)_{k_1, \cdots, k_r}=S(\a).
\]
Then by replacing $\a$ by $\a^{(k_1, \cdots, k_r)}$, we are reduced to the case where 
\[
f_{e(\a)}(k-i_0-1)=0.
\]

\begin{prop}\label{prop: 8.1.9}
Let $\a$ be a multisegment such that 
\[
f_{e(\a)}(k+1)\neq 0.
\]
And let 
\[
\{t\in e(\a): t\geq k+1\}=\sum_{i=1}^{s} \l_{k_i}[k_i]
\]
with $k_1<k_2<\cdots <k_s$. Then
\[
m(\c, \b, \a)=m(\c^{[k_r]_{\l_{k_r}}[k_{r-1}]_{\l_{k_{r-1}}}\cdots [k_1]_{\l_{k_1}}}, \b, \a^{[k_r]_{\l_{k_r}}[k_{r-1}]_{\l_{k_{r-1}}}\cdots [k_1]_{\l_{k_1}}})
\]
\end{prop}

\remk This proposition allows us to reduce to the case where 
\[
f_{e(\a)}(k+1)=0.
\]

\begin{proof}
The proof is the same as that of the proposition above.
\end{proof}

As usual, we reduce to the 
parabolic case by the following proposition. 

\begin{prop}\label{prop: 8.1.2}
Let $\a$ be a multisegment satisfying $\max b(\a)\leq k-i_0+1$, then there exists a sequence of integers 
$k_1, k_2, \cdots, k_r$  and a parabolic multisegments $\c$ of type $(J_1(\a), \emptyset)$such that 
\[
\a={^{(k_1, \cdots, k_r)}\c}, \quad \c\in {_{k_1, \cdots, k_r}S(\c)}
\]
and if 
\[
L_{\c}\times L_{\b}=\sum_{\d}m(\d, \c, \b)L_{\d}
\]
then 
\[
L_{\a}\times L_{\b}=\sum_{\d\in {_{k_1, \cdots, k_r}S(\c+\b)}}m(\d, \c, \b)L_{^{(k_1, \cdots, k_r)}\d}. 
\]
\end{prop}

\begin{proof}
The existence of $\c$ follows from proposition \ref{prop: 6.3.3}.
To deduce our result, it suffices to apply the derivation 
\[
(^{k_1}\D)(^{k_2}\D)\cdots (^{k_r}\D)
\]
to $L_{\c}\times L_{\b}=\sum_{\d}m(\d, \c, \b)L_{\d}$ and then apply 
proposition \ref{teo: 3.0.6}.
\end{proof}

\begin{prop}\label{prop: 8.1.5}
Assume that $\a$ is a \textbf{parabolic} multisegment such that 
\[
f_{e(\a)}(k-i+1)=0
\]
for some $1\leq i\leq i_0$. Then 
\[
m(\c, \b, \a)=m(\c^{(k-i+2, \cdots,k-1,  k)}, \b^{(k+1)}, \a^{(k-i+2, \cdots, k-1, k)}).
\]
\end{prop}

\begin{proof}
The proof is the same as that of proposition \ref{prop: 8.1.6}
\end{proof}

\remk Combining the proposition \ref{prop: 8.1.6}, \ref{prop: 8.1.2}, \ref{prop: 8.1.5}, and \ref{prop: 8.1.9},
the calculation of the coefficients $m(\c, \b, \a)$ for case (1) can be reduced to the case where $\a$ 
is a parabolic multisegment such that 
\[
f_{e(\a)}(k-i_0-1)=f_{e(\a)}(k+1)=0, \quad f_{e(\a)}(k-i+1)\neq 0, \text{ for all } 1\leq i\leq i_0+1.
\]

From now on until the end of the section, assume that 
\[
\a_{\Id}^{J, \emptyset}
\]
be a multisegment of type $(J, \emptyset)$ associated to the identity in $S_n$, which satisfies 
\[
f_{e(\a_{\Id}^{J, \emptyset})}(k-i_0-1)=f_{e(\a_{\Id}^{J, \emptyset})}(k+1)=0, \quad f_{e(\a_{\Id}^{J, \emptyset})}(i)>0 \text{ for } k-i_0\leq i\leq k,
\]
and fix a bijection
\[
\Phi^{i_0}: S_n^{J, \emptyset}\rightarrow S(\a_{\Id}^{J, \emptyset})
\]
and $\a_{i_0}=\Phi^{i_0}(w)$.
\begin{lemma}
Under the above assumption, we have 
\[
J_1(\l_{k-i_0}-r_0, k-i_0)=J_2(\l_{k-i_0}-r_0, k-i_0).
\]
\end{lemma}

\begin{proof}
This follows directly from the definition.
\end{proof}

\begin{lemma}\label{lem: 8.1.4}
Let $\b=[k+1]$ and $\l_k=f_{e(\a_0)}(k)$. Then  
\[
L_{\a_0}\times L_{\b}=L_{\a_0+\b}+\sum_{\c\in \Gamma^{\l_k-1}(\a_0, k)} \Bigl ( \theta_k(\c, \a_0)-\theta_k(\c^{[k+1]_1}, \a_0+\b) \Bigr ) 
L_{\c^{[k+1]_1[ k]_{\l_k-1}}}
\]
\end{lemma}

\begin{proof}
Note that 
\[
(^{k+1}\D)(L_{\a_0}\times L_{\b})=L_{\a_0}\times L_{\b}+L_{\a_0}.
\]
And for each $\c\in S(\a_0+\b)$ if $[k+1]\in \c$, then 
\[
(^{k+1}\D)L_{\c}=L_{\c}+L_{^{(k+1)}\c}.
\]
This implies that if $\c\neq \a_0+\b$ and $[k+1]\in \c$, then $L_{\c}$ can not be a direct summand of $L_{\a_0}\times L_{\b}$.
Furthermore, by assumption on $\a_0$, we know that for any  $\c\in S(\a_0+\b)$ and $[k+1]\notin \c$,
we have $\c\in S(\a_0+\b)_{k}$ and hence $\c\in S(\a_0+\b)_{k, k+1}$.   
Moreover, we know that $\c^{(k, k+1)}\in \Gamma^{\l_k-1}(\a_0, k)$. Therefore we have 
\[
L_{\a_0}\times L_{\b}=L_{\a_0+\b}+\sum_{\c\in \Gamma^{\l_k-1}(\a, k)}m(\c, \b, \a_0)L_{\c^{[k+1]_1[k]_{\l_k-1}}}.
\]
Now apply the derivation $\D^k$ to both sides of the equation to get 
\begin{align*}
\D^k(L_{\a_0}\times L_{\b})&=(\sum_{\c\preceq_{k}\a_0} \theta_k(\c, \a_0)L_{\c}) \times L_{\b}\\
&=\sum_{\c\in \Gamma^{\l_k-1}(\a_0, k)}\theta_k(\c, \a_0)L_{\c}\times L_{\b}+\text{ other degree terms }
\end{align*}
and the right hand side we get 
\[
\sum_{\c\in \Gamma^{\l_k-1}(\c, \a_0+\b)}\theta_k(\c, \a_0+\b)L_{\c}+\sum_{\c\in \Gamma^{\l_k-1}(\a_0, k)}m(\c, \b, \a_0)L_{\c^{[k+1]_1}}+\text{ other degree terms }.
\]
Now by the following lemma we know that for $\c\in \Gamma^{\l_k-1}(\a_0, k)$
\[
L_{\c}\times L_{\b}=L_{\c+\b}+L_{\c^{[k+1]_1}},
\]
therefore by comparing the two sides, we obtain that for $\c\in \Gamma^{\l_k-1}(\a_0, k)$
\[
m(\c, \b, \a_0)+\theta_k(\c^{[k+1]}, \a_0+\b)=\theta_k(\c, \a_0).
\]
Hence we are done.

\end{proof}

\begin{lemma}
Let $\a$ be a multisegment such that 
\[
\max b(\a)\leq k+1, \quad f_{e(\a)}(k)=1, \quad f_{e(\a)}(k+1)=0.
\]
Then we have 
\[
L_{\a}\times L_{[k+1]}=L_{\a+[k+1]}+L_{\a^{[k+1]_1}}.
\]
\end{lemma}

\begin{proof}
First of all, it is known  by Zelevinsky that $L_{\a+[k+1]}$ appears in $L_{\a}\times L_{[k+1]}$ with multiplicity one.
Also, since 
\[
\D^{k+1}(L_{\a}\times L_{[k+1]})=L_{\a}\times L_{[k+1]}+L_{\a},
\]
we know that $L_{\a^{[k+1]_1}}$ is the only element in $S(\a+[k+1])$ which appears as a subquotient
in $L_{\a}\times L_{[k+1]}$ and does not contain $[k+1]$ as a beginning.
Finally, since 
\[
{^{k+1}\D}(L_{\a}\times L_{[k+1]})=L_{\a}\times L_{[k+1]}+L_{\a},
\]
we conclude that $\a+[k+1]$ is the only multisegment in $S(\a+[k+1])$ which is a subquotient of $L_{\a}\times L_{[k+1]}$
and contains $[k+1]$ as a beginning. 
\end{proof}

In particular,gathering all the calculation in case where $\b=[k+1]$, we obtain the following formula.

\begin{cor}\label{prop: 8.1.1}
Let $\a$ be a parabolic multisegment satisfying the condition  
\[
f_{e(\a)}(k)\neq 0, \quad f_{e(\a)}(k-1)=f_{e(\a)}(k+1)=0,
\]
 and $\b=[k+1]$. Then 
\[
L_{\a}\times L_{\b}=L_{\a+\b}+\sum_{\c\in \Gamma^{\l_k-1}(\a, k)} \Bigl ( \theta_k(\c, \a)-\theta_k(\c^{[k+1]_1}, \a+\b) \Bigr ) L_{\c^{[k+1]_1[ k]_{\l_k-1}}}.
\]
\end{cor}

\remk
The proposition is no longer true if we remove the condition 
\[
f_{e(\a)}(k-1)=0.
\]
\begin{example}\label{ex: 8.1.6}
Let $\a=[0, 2]+[1, 3]+[2, 3]$ and $\b=[4]$, and 
\[
\c_1=[0, 3]+[1, 4]+[2], \quad \c_2=[0, 2]+[1, 4]+[2, 3], \quad \d=[0, 2]+[2]+[1, 3],
\]
then 
\[
L_{\a}\times L_{\b}=L_{\a+\b}+L_{\c_1}+L_{\c_2}
\] 
and 
\[
\D^{3}(L_{\a})=L_{\a}+L_{\d}, \quad \D^{3}(L_{\c_2})=L_{\c_2}.
\]
In this case we cannot compute the multiplicity of $L_{\c_2}$ using directly the partial derivatives.
\end{example}

\remk The proposition is also false if we remove the condition 
\[
f_{e(\a)}(k+1)=0
\]
\begin{example}
Let $\a=[1, 2]+[2, 3]$ and $\b=[3]$, then
\[
L_{\a}\times L_{\b}=L_{\a+\b}
\]
which contradicts our formula.
\end{example}

\begin{prop}\label{prop: 8.1.7}
Let $\a_{i_0}=\Phi^{i_0}(w)$ and $\b=[k-i_0, k+1]$. Then
\begin{align*}
&L_{\a_{i_0}}\times L_{\b}=\sum_{\e}m(\e, {^{(k-i_0+1)}\b}, \a)L_{^{[k-i_0+1]_1}\e}+\\
 &\sum_{\c}m(\c^{[k-i_0+1]_{\l_{k-i_0+1}}[k-i_0]_{\l_{k-i_0}-1}}, \b, \a_{i_0})L_{\c^{[k-i_0+1]_{\l_{k-i_0+1}}[k-i_0]_{\l_{k-i_0}-1}}},
\end{align*}
with 
\begin{align*}
m(\c^{[k-i_0+1]_{\l_{k-i_0+1}}[k-i_0]_{\l_{k-i_0}-1}}, \b, \a_{i_0})=&\sum_{\d \in \Gamma^{\l_{k-i_0}-1}(\a, k-i_0)_{k-i_0+1}}\theta_{k-i_0}(\d, \a)m(\c, \b, \d^{(k-i_0+1)})-\\
                                                                                                       &\sum_{\e}\theta_{k-i_0}(\c^{[k-i_0+1]_{\l_{k-i_0+1}}}, ^{[k-i_0+1]_1}\e)m(\e, {^{(k-i_0+1)}\b}, \a_{i_0})
\end{align*}
where $\c$ runs through all the terms such that 
$m(\c, \b, \d^{(k-i_0+1)})\neq 0$ for some $\d$ and $f_{b(\c)}(k-i_0+1)=0$, $\e$ runs through all the terms such 
that  $m(\e, {^{(k-i_0+1)}\b}, \a_{i_0})\neq 0$ .

\end{prop}

\begin{proof}
Consider the formula
\addtocounter{theo}{1}
\begin{equation}\label{eq8.1}
L_{\a_{i_0}}\times L_{\b}=\sum_{\c}m(\c, \b, \a_{i_0})L_{\c}.
\end{equation}
In case $k-i_0+1\in b(\c)$, we know that $\c\in {_{k-i_0+1}S(\a+\b)}$, and moreover
\[
m(\c, \b, \a)=m({^{(k-i_0+1)}\c}, {^{(k-i_0+1)}\b}, \a),
\]
this gives the first part of the formula in our proposition. Now if $k-i_0+1\notin b(\c)$,
then we have 
\[
f_{e(\c)}(i)=f_{e(\a)}(i), \text{ for all } k-i_0+1\leq i\leq k, \quad f_{e(\c)}(k-i_0)=f_{e(\a)}(k-i_0)-1.
\]
In this case, we apply the derivative $\D^{k-i_0+1}\D^{k-i_0}$ to the equation (\ref{eq8.1}) and consider 
terms of degree equal to $\deg(\c^{(k-i_0, k-i_0+1)})$.
On the left hand side we find 
\[
\sum_{\c}\sum_{\d \in \Gamma^{\l_{k-i_0}-1}(\a, k-i_0)_{k-i_0+1}}\theta_{k-i_0}(\d, \a)m(\c, \b, \d^{(k-i_0+1)})L_{\c}.
\]
While for fix $\c$, on the right hand side we find 
\begin{multline*}
(\sum_{\e}\theta_{k-i_0}(\c^{[k-i_0+1]_{\l_{k-i_0+1}}}, ^{[k-i_0+1]_1}\e)m(\e, {^{(k-i_0+1)}\b}, \a) \\
+m(\c^{[k-i_0+1]_{\l_{k-i_0+1}}[k-i_0]_{\l_{k-i_0}-1}}, \b, \a_{i_0}))L_{\c}
\end{multline*}
here $\e$ runs through all the terms such 
that  $m(\e, {^{(k-i_0+1)}\b}, \a)\neq 0$. The first part in the coefficient comes from the part 
\[
\sum_{\e}m(\e, {^{(k-i_0+1)}\b}, \a)L_{^{[k-i_0+1]_1}\e}
\]
in the induction $L_{\a_{i_0}}\times L_{\b}$ so that by taking the difference, we get our results.

\end{proof}

\remk In general the multisegment $\d^{(k-i_0+1)}$ in the the formula does not satisfies 
the condition 
\[
f_{e(\d^{(k-i_0+1)})}(i)=0, \text{ for all } k-i_0\leq i\leq k.
\]
In order to proceed our calculation, we have to apply proposition \ref{prop: 8.1.5}.

%
%
%
%

\remk 
Combining all the propositions above, we finish the 
computation of $m(\c, \b, \a)$ in case where 
\[
\b=[k-i_0+1, k+1], \quad \max b(\a)\leq k-i_0+1.
\]

\section{General case}

Now we consider the case (2) in the introduction of this chapter. 

\begin{prop}
Let $k\in \Z$ and $\a$ be a multisegment. Then there exists a multisegment $\a'$ and 
a sequence of integers $k_1, \cdots, k_r$ such that 
\[
\a={^{(k_1, \cdots, k_r)}\a'}, \quad \a'\in {_{k_1, \cdots, k_r}S(\a')},
\]
and for any $1\leq i\leq r$,
\[
\deg({^{(k_i, \cdots, k_r)}\a})=\deg(^{(k_{i+1}, \cdots, k_r)}\a)-1, \quad \max b(\a')\leq k.
\]
\end{prop}

\begin{proof}
This is proved by applying successively the truncation functor, which is the same as 
the proof of proposition \ref{prop: 6.3.3}.
\end{proof}

\begin{prop}
Let $\a$ be a multisegment such that 
\[
\a\in {_{k-i_0+1}S(\a)}, \quad f_{e(\a)}(k-i_0+1)=1.
\]
If we assume that $\b=[k-i_0+1, k+1](i_0\geq 0)$ and 
\[
L_{\a}\times L_{\b}=\sum_{\c}m(\c, \b, \a)L_{\c},
\]
then 
\[
m(\d, \b, ^{(k-i_0+1)}\a)=\sum_{\c}m(\c, \b, \a)({_{k-i_0+1}\theta}(\d, \c))-m(\d, ^{(k-i_0+1)}\b, \a).
\]
\end{prop}
\begin{proof}
Note that by assumption we have 
\[
{^{k-i_0+1}\D}L_{\a}=L_{\a}+L_{^{(k-i_0+1)\a}}.
\]
If we apply $^{k-i_0+1}\D$ to 
\[
L_{\a}\times L_{\b}=\sum_{\c}m(\c, \b, \a)L_{\c},
\]
on the left hand side get 
\[
L_{\a}\times L_{\b}+L_{^{(k-i_0+1)}\a}\times L_{\b},
\]
while on the right hand side we get 
\[
\sum_{\d}\sum_{\c}m(\c, \b, \a)({_{k-i_0+1}\theta}(\d, \c))L_{\d}
\]
by comparing the two hand side, we get our result.
\end{proof}

\begin{prop}
Let $k_1\neq k-i_0+1$ and $\a$ be a multisegment such that 
\[
\a\in {_{k}S(\a)}, \quad f_{e(\a)}(k)=1.
\]
If we assume that $\b=[k-i_0+1, k+1](i_0\geq 0)$ and 
\[
L_{\a}\times L_{\b}=\sum_{\c}m(\c, \b, \a)L_{\c},
\]
then 
\[
m(\d, \b, {^{(k_1)}\a})=m(^{[k_1]_1}\d, \b, \a).
\]
\end{prop}
\begin{proof}
The proof is the same as that of proposition \ref{prop: 8.1.2}.
\end{proof}

\remk Combining the three proposition we get the computation of $m(\c, \b, \a)$ for any $\a$
and $\b$ a segment.

\bibliographystyle{plain}
\bibliography{biblio}
 
 \end{document}